\documentclass[12pt]{article}
\usepackage{scrextend} 

\usepackage{hyperref}
\usepackage{amsmath}
\usepackage{amsfonts}
\usepackage{amssymb}
\usepackage{mathtools}
\mathtoolsset{showonlyrefs}
\usepackage{graphicx}
\usepackage{subcaption}
\usepackage{natbib}
\usepackage{authblk}
\usepackage{color}
\usepackage{xcolor}
\usepackage{xr}
\usepackage{mathrsfs}
\usepackage[ruled]{algorithm2e}

\usepackage{bm}
\usepackage[referable]{threeparttablex}
\usepackage{booktabs, ltablex}
\keepXColumns
\usepackage{lineno}

\newtheorem{theorem}{Theorem}[section]
\newtheorem{proposition}{Proposition}[section]
\newtheorem{corollary}{Corollary}[section]
\newtheorem{definition}{Definition}[section]

\newtheorem{remark}{Remark}[section]
\newtheorem{lemma}{Lemma}[section]
\newenvironment{proof}[1][Proof]{\noindent\textbf{#1.} }{\hfill \rule{0.5em}{0.5em}}
\numberwithin{equation}{section}

\newcommand{\norm}[1]{\left\lVert#1\right\rVert}
\newcommand{\abs}[1]{\left\lvert#1\right\rvert}
\newcommand\nonumberthis{\nonumber\refstepcounter{equation}}

\DeclareMathOperator{\dist}{dist}

\DeclarePairedDelimiterX{\inp}[2]{\langle}{\rangle}{#1, #2}

\DeclareMathOperator*{\argmax}{argmax}

\SetKwInput{KwInput}{Input}
\SetKwInput{KwOutput}{Output} 

\begin{document}

\title{Analytical and statistical properties of local depth functions motivated by clustering applications}

\author[1,2]{Giacomo Francisci}
\author[1]{Claudio Agostinelli}
\author[2]{Alicia Nieto-Reyes}
\author[3]{Anand N. Vidyashankar}

\affil[1]{Dipartimento di Matematica, Universit\`a di Trento, Trento, Italy}
\affil[2]{Departamento de Matem\'aticas, Estad{\'\i}stica y Computaci\'on, Universidad de Cantabria, Santander, Spain}
\affil[3]{Department of Statistics, George Mason University, Fairfax, Virginia}

\date{\today}

\maketitle

\begin{abstract}
Local depth functions (LDFs) are used for describing the local geometric features of multivariate distributions, especially in multimodal models. In this paper, we undertake a rigorous systematic study of the LDFs and use it to develop a theoretically validated algorithm for clustering. For this reason, we establish several analytical and statistical properties of LDFs. First, we show that, when the underlying probability distribution is absolutely continuous,  under an appropriate scaling that converge to zero (referred to as \emph{extreme localization}), LDFs converge uniformly to a power of the density and obtain a related rate of convergence result. Second, we establish that the centered and scaled sample LDFs converge in distribution to a centered Gaussian process, uniformly in the space of bounded functions on $\mathbb{R}^p \times [0, \infty]$, as the sample size diverges to infinity. Third, under an extreme localization that depends on the sample size, \emph{we determine the correct centering and scaling} for the sample LDFs to possess a limiting normal distribution.
Fourth, invoking the above results, we develop a new clustering algorithm that uses the LDFs and their differentiability properties. Fifth, for the last purpose, we establish several results concerning the gradient systems related to LDFs. Finally, we illustrate the finite sample performance of our results using simulations and apply them to two datasets.
\end{abstract}


\section{Introduction}

Investigation of data depths is gaining momentum due to its applicability in a variety of machine learning problems such as non-parametric classification and clustering. This concept, formalized in \citet{Liu-1990} and \citet{Zuo-Serfling-2000}, dates back to \citet{Tukey-1975} and serves to identify a center for multivariate distributions and a multidimensional center-outward order similar to that of a real line. This ordering enables a description of quantiles of multivariate distributions (see \citet{Zuo-Serfling-2000-c}) and aids in using depth functions (DFs) for clustering as described, for instance, in \cite{Torrente-2020, Joernsten-2004} (see also the references therein). The focus of these methods is to improve k-means based clustering via the use of DFs by identifying the ``deepest points'', the depth median of a dataset. The current paper develops the intuitive notion that local depths possess properties that help in identifying peaks and valleys, and hence clustering based on such identification can improve the quality and stability of the clustering algorithm.

The notion of local depth, as first described in \citet{Agostinelli-2011} and investigated in one-dimension, provides a framework to describe the local multidimensional features of multivariate distributions. Section \ref{section_local_depth_and_extreme_localization} of this paper extends this analysis to multidimensions and provides a detailed study of the analytical properties of local depth functions (LDFs) and their scaled versions (referred to as $\tau$-approximation). These results form the basic ingredients for a gradient system analysis developed in the later sections. Recent additional work on local depth include that of \citet{Dutta-2016} and \citet{Chandler-2020}, where  the focus is on different types of data (spatial and functional, for instance) and construction of \emph{supervised} classification methods. However, none of these methods apply to the more challenging problem studied here, {\it{viz.,}} unsupervised classification or clustering. We now turn to a precise description of our contributions starting with depth and local depth.

In the following, we denote by ${\mathscr{B}}(\mathscr{X})$  the Borel subsets of a topological space $\mathscr{X}$ and by $(\mathbb{R}^p)^{k}$ the $k$-fold cartesian product of the $p$-dimensional Euclidean space $\mathbb{R}^p$. Let $P$ be a probability distribution on $\mathbb{R}^p$. Then for a set $Z^G(x) \in {\mathscr{B}((\mathbb{R}^p)^{k})}$, where $x \in \mathbb{R}^p$, the bounded non-negative function $GD(x, P)$ defined by
\begin{equation} \label{general_depth}
GD(x,P) = \int_{(\mathbb{R}^p)^{k}} \mathbf{I} ((x_1, \dots, x_k) \in Z^G(x)) \, dP(x_1) \dots dP(x_k)
\end{equation}
is referred to as the depth function at $x$ through $Z^G(x)$ for $P$. Depending on the choice of $P$ and $Z^G(\cdot)$, they can also satisfy the following properties: affine invariance, maximality at the center, monotonically non-increasing along the rays from the center, and vanishing at infinity. These are sometimes stated as required properties of a ``good'' depth function (see \citet{Liu-1990} and \citet{Zuo-Serfling-2000}). As an example, if $P$ is absolutely continuous with respect to the Lebesgue measure and is angularly symmetric, the so-called simplicial depth (see below) satisfies all the properties of a good depth function.

As expected, different choices of $Z^G(\cdot)$ yield different DFs. Specifically, choosing $k=2$ and $Z^G(\cdot)$ to be
\begin{equation*}
	Z^G(x) = \{ (x_1, x_2) \in \mathbb{R}^p \times \mathbb{R}^p \, : \max_{ i=1,2} \norm{x-x_i} \leq \norm{x_1 - x_2} \} \eqqcolon Z^L(x)
\end{equation*}
we obtain the lens depth, $LD(x, P)$ \citep{Liu-2011}. Also, for the choice $Z^G(\cdot)$ (with $k=2$) given by
\begin{equation*}
Z^{G}(x) = \{ (x_1, x_2) \in \mathbb{R}^p \times \mathbb{R}^p \, : \norm{2 x-(x_1+x_2)} \leq \norm{x_1 - x_2} \} \eqqcolon Z^B(x)
\end{equation*}
we get the spherical depth, $BD(x, P)$ \citep{Elmore-2006}. Finally, choosing $k=(p+1)$ and
\begin{equation*}
	Z^G(x) = \{ (x_1, \dots, x_{p+1}) \in (\mathbb{R}^p)^{(p+1)} \, : \, x \in \triangle[x_1, \dots, x_{p+1}] \}, 
\end{equation*}
where $\triangle[x_1, \dots, x_{p+1}]$ is the closed simplex with vertices $x_1, \dots, x_{p+1} \in \mathbb{R}^p$, we get the well-known simplicial depth, $SD(x, P)$ given in \citet{Liu-1990}. We denote the corresponding set by $Z^S(\cdot)$.

The local versions of the depths described above, involve an additional nonnegative parameter $\tau$, referred to as the \emph{localizing parameter}, yielding the LDFs, namely, $LGD(\cdot, P, \tau)$, where $G=L, B, S$. The corresponding sets now take the form $Z^G_{\tau}(x)$. As in \eqref{general_depth}, the LDFs can be expressed as:
\begin{equation} \label{general_local_depth}
LGD(x,P,\tau) = \int_{(\mathbb{R}^p)^{k}} \mathbf{I} ((x_1, \dots, x_k) \in Z^G_{\tau}(x)) \, dP(x_1) \dots dP(x_k).
\end{equation}
For the sake of clarity, we provide the expressions for ``the local set'' for different choices of $G=L, B, S$. Specifically, the local set associated with lens depth is
\begin{equation*}
	Z_{\tau}^L(x) = \{ (x_1, x_2) \in \mathbb{R}^p \times \mathbb{R}^p \, : \max_{ i=1,2} \norm{x-x_i} \leq \norm{x_1 - x_2} \leq \tau \}, 
\end{equation*}
while that for the spherical depth is given by
\begin{equation*}
Z_{\tau}^{B}(x) = \{ (x_1, x_2) \in \mathbb{R}^p \times \mathbb{R}^p \, : \norm{2 x-(x_1+x_2)} \leq \norm{x_1 - x_2} \leq \tau \}.
\end{equation*}
The local set for the simplicial depth is given by
\begin{equation*}
	Z_{\tau}^S(x) = \{ (x_1, \dots, x_{p+1}) \in (\mathbb{R}^p)^{(p+1)} \, : \, x \in \triangle[x_1, \dots, x_{p+1}], \, \max_{\substack{ i,j=1, \dots, p+1 \\ i > j}} \norm{x_i-x_j} \le \tau \}.
\end{equation*}
We observe here that while the definitions of LSD and LLD first appear in \citet{Agostinelli-2008} and \citet{Kleindessner-2017} respectively, the  LBD, as defined here, seems new. Of course, when $p=1$ all of the above three local depths coincide. Finally, when $\tau=\infty$, the LDFs reduce to the DFs.

A useful and important aspect of local depth is its behavior under \emph {extreme localization}, {\it{i.e.\ }}when $\tau \to 0^+$. When $P$ is absolutely continuous, LDFs investigated in this paper, under appropriate scaling, converge to a power of the probability density of $P$. Under additional conditions, one obtains convergence to the derivatives of the density which facilitates an enquiry into the modes of the density {\emph{via}} a gradient system analysis. This, in turn, allows one to characterize the related \emph{stable manifolds} paving the way for cluster analysis. Related ideas about clustering appear in \citet{Chazal-2013}, \citet{Chen-2016}, and \citet{Genovese-2016}. Our methodology differs from the existing literature in that we take advantage of the local depth notion, specifically the $\tau$-approximation and its properties, developed in Sections \ref{section_local_depth_and_extreme_localization} and \ref{section_clustering} below.

Statistical enquiry about local depth requires an investigation into their sample versions, {\it{viz.,}} estimators of local depth,  which are described in Subsection \ref{subsection_sample_local_depth}. We establish that, for all choices of $G=L,B, S$, the estimators are uniformly consistent across the pair $(x, \tau)$. The key issue is the additional uniformity in $\tau$ which requires investigation of the behavior of the classes of sets $\{Z^G_{\tau}(x): (x, \tau) \in \mathbb{R}^p \times [0, \infty]\}$  under the empirical measure. Specifically, we study the behavior of the estimator of LDFs in the space $\ell^{\infty}(\Gamma)$, of bounded functions on subsets $\Gamma$ of $\mathbb{R}^p \times [0, \infty]$, which is not separable. Borrowing tools from empirical process theory, we establish uniform consistency for all choices of $G=L, B, S$. Turning to the limit distributions of the estimator, we establish a uniform central limit theorem, in $\ell^{\infty}(\Gamma)$. This involves establishing convergence of finite dimensional distributions, admissible Suslin property, and use of Alexander's Lemma (see \citet{Alexander-1987}) for establishing the finiteness of the bracketing entropy for the above described VC classes of sets. Both the uniform consistency and the uniform central limit results rely on the Hoeffding's decomposition of U-statistics representation of the local depth which incidentally is a critical component of our analysis.

A natural next question concerns the distributional behavior of sample local depth when the \emph{localizing parameter} $\tau$
converges to zero as a function of the sample size. In this case, contrary to expectations, the correct centering is not the limiting density. In Subsection \ref{subsection_sample_local_depth}, we identify the correct centering and scaling for the sample local depth to converge, in distribution, to a centered normal distribution.

As the title of the manuscript indicates, the primary motivation of the paper is clustering in multidimensions. This involves two distinct but connected steps: (i) identifying the modes and (ii) identifying the stable manifold generated by them. For (i), we establish sufficient conditions in terms of the eigenvalues of the Hessian matrix. For (ii), which involves defining the clusters, one needs to identify the points whose trajectories (which result from the gradient system) converge to the mode, also referred to as the $\omega$-limit of the path. Thus from a practical perspective, an algorithm can be developed based on the above two considerations. While this is feasible, what is unclear is that the resulting sets indeed form non-trivial clusters; that is, they have positive Lebesgue measure. We establish that this is the case using tools from dynamical systems. Specifically, since $\omega$-limits are stationary points of the density, to validate the algorithm, it is required to show that the trajectories that converge to non-modal points have  Lebesgue measure zero. In contrast, those that converge to modal points have positive Lebesgue measure. Lyapunov’s stability Theorem plays a significant first step for this analysis. To the best of our knowledge, the current manuscript seems to be the first to provide such strong theoretical guarantees for clustering in multidimensional problems.

We now provide a brief description of the organization of the rest of the paper. In Subsection \ref{subsection_analytic_properties}, we establish the analytical properties of LLD while the statistical properties are established in \ref{subsection_sample_local_depth}. In Subsection \ref{subsection_other_instances_of_local_depth}, we discuss the modifications required for other local depths, including local half-space depth and local half-region depth.  In Subsections \ref{subsection_mathematical_background}, \ref{subsection_convergence_gradient_system} and \ref{subsection_modes_identification}, we focus on the mathematical underpinnings behind the algorithm for clustering which itself is described in Subsection \ref{subsection_numerical_implementation}. In Section \ref{section_simulations_and_data_analysis}, we provide several numerical experiments illustrating the finite sample behavior of the proposed methods and apply them to data examples while, in Section \ref{section_concluding_remarks}, we provide a few concluding remarks. The proofs of all the results are relegated to the appendix. We end this introduction with a comment about  the notations. When there is no scope for confusion, we suppress $P$ in $GD(x, P)$ and $LGD(x, P, \tau)$ and use the notation $GD(x)$ and $LGD(x, \tau)$ for all $G=L, S, B$.

\section{Local depth and extreme localization} \label{section_local_depth_and_extreme_localization}

We recall that the local lens depth is given by 
\begin{equation} \label{local_depth}
	LLD(x,\tau) = \int_{\mathbb{R}^p \times \mathbb{R}^p} \mathbf{I}( (x_1, x_2) \in Z_{\tau}(x) ) \, dP(x_1) \, dP(x_2),
\end{equation}
where $Z_{\tau}(x) = \{ (x_1, x_2) \in \mathbb{R}^p \times \mathbb{R}^p \, : \max_{ i=1,2} \norm{x-x_i} \leq \norm{x_1 - x_2} \leq \tau \}$, and the index L is suppressed for simplicity.

\subsection{Analytic properties} \label{subsection_analytic_properties}
We begin this section by describing properties of the local depth. These results will be required when investigating the properties of the estimator of the local depth. 
\begin{proposition} \label{proposition_local_depth}
  (i) For any fixed $x \in \mathbb{R}^p$, $LLD(x,\, \cdot \, )$ is a non-decreasing and right-continuous function of $\tau$ satisfying
\begin{equation*}
    \lim_{\tau \to 0^{+}} LLD(x,\tau) = P^{2}(\{ x \}) \text{ and } \lim_{\tau \to \infty} LLD(x,\tau) = LD(x).
\end{equation*}
(ii) For $\tau \in [0,\infty]$, $\sup_{x \in \mathbb{R}^p \, : \, \norm{x} \geq M } LLD(x,\tau) \rightarrow 0 \text{ as } M \rightarrow \infty$. \\
(iii) For $\tau \in [0,\infty]$, $LLD(\, \cdot \, ,\tau)$ is upper semicontinuous as a function of $x$. If $P$ is absolutely continuous with respect to the Lebesgue measure, then it is jointly continuous in $x$ and $\tau$. \\
(iv) If $P$ is absolutely continuous with respect to the Lebesgue measure, with $k$-times continuously differentiable density $f(\cdot)$, then $LLD(\,\cdot \,, \tau)$ is $k$-times continuously differentiable for each fixed $\tau \in [0,\infty)$.
\end{proposition}
(i) of the above proposition was studied for the case $p=1$ in \citet{Agostinelli-2011}[Proposition 6] and (ii) was investigated for the extreme case, namely $\tau=\infty$, in \citet{Liu-2011}[Theorem 2].
If $P$ is absolutely continuous with respect to the Lebesgue measure  on $\mathbb{R}^p$ with density $f(\cdot)$, then Proposition \ref{proposition_local_depth} (i) shows that LLD converges to $0$ as $\tau \rightarrow 0^{+}$. Notice that for $x \in \mathbb{R}^p$, $Z_{\tau}(x)=(x,x)+Z_{\tau}(0)$ and $Z_{\tau}(0)=\tau Z_1(0)$. Also, by a change of variable in \eqref{local_depth}, it follows that
\begin{equation} \label{local_depth_for_absolutely_continuous_distribution}
\begin{split}
	LLD(x,\tau) &= \int_{ Z_{\tau}(x) } \, f(x_1) \, f(x_2) \, d x_1 \, d x_2 \\
	 &=  \int_{Z_{\tau}(0)} f(x+x_1) f(x+x_2) \, dx_1 dx_2 \\
	&= \int_{Z_{1}(0)} \tau^{2 p} f(x + \tau x_1) f(x + \tau x_2) \, dx_1 dx_2.
\end{split}
\end{equation}
For $\tau>0$, let $\Lambda_{\tau} \coloneqq \lambda^{\otimes 2}(Z_{\tau}(0))$ denote the Lebesgue measure in $\mathbb{R}^p \times \mathbb{R}^p$ of $Z_{\tau}(0)$. By the translation and dilatation invariance of the Lebesgue measure, it follows that $\lambda^{\otimes 2}(Z_{1}(x))=\Lambda_1$ and $\Lambda_{\tau}=\tau^{2p} \Lambda_{1}$. Our next result is concerned with extreme localization. Specifically, we establish the rate of convergence of the local depth to the square of the density as $\tau$ approaches $0$. In the following, we use \emph{a.e.\ }to mean almost everywhere with respect to the Lebesgue measure on $\mathbb{R}^p$ and $\inp{\cdot}{\cdot}$ is the inner  product on $\mathbb{R}^p$. Also, $\nabla f$ and $H_f$ represent the gradient and the Hessian matrix of $f(\cdot)$, respectively.

\begin{theorem} \label{theorem_local_depth_tau_to_0}
Let $P$ be absolutely continuous with respect to the Lebesgue measure on $\mathbb{R}^p$, with density $f(\cdot)$. Then the following hold: \\
(i) Let $x \in \mathbb{R}^p$ be fixed, then
\begin{equation} \label{local_depth_tau_to_0}
	\lim_{\tau \rightarrow 0^{+}} \frac{1}{\tau^{2p} \Lambda_1} LLD(x,\tau) = f^{2}(x), \text{ a.e.\ }
\end{equation}
Furthermore, if $f(\cdot)$ is continuous, then \eqref{local_depth_tau_to_0} holds for all $x \in \mathbb{R}^p$. \\
(ii) Suppose that $f(\cdot)$ is three times continuously differentiable. Then
\begin{equation*}
	\lim_{\tau \rightarrow 0^{+}} \frac{1}{\tau^2} \left( \frac{1}{\tau^{2 p}} LLD(x,\tau) - \Lambda_1 f^2(x) \right) = h(x), \text{ where}
\end{equation*}
\begin{equation*}
\begin{split}
h(x) &\coloneqq f(x) \int_{Z_1(0)} x_{1}^{\top} H_f(x) x_{1} \, dx_1 dx_2 + \int_{Z_1(0)} \inp{\nabla f(x)}{x_1} \inp{\nabla f(x)}{x_2} \, dx_1 dx_2.
\end{split}
\end{equation*}
\end{theorem}
It is worth noticing here that $h(\cdot)$ vanishes outside the support of $f(\cdot)$, $S \coloneqq \{ x \in \mathbb{R}^p \, : \, f(x)>0 \}$, while it is, in general, non-zero in the interior of the support of $f(\cdot)$. When $p=1$ certain simplifications occur. Specifically, $\Lambda_1=1$. This is summarized in the following corollary.

\begin{corollary} \label{corollary_local_depth_tau_to_0}
Under the conditions of Theorem \ref{theorem_local_depth_tau_to_0} and for $p=1$, 
\begin{equation} \label{simplified_form_local_depth}
	LLD(x,\tau) = 2 \int_{T_{++}^{\tau}} f(x+x_1) f(x-x_2) \, dx_1 dx_2,
\end{equation}
where $T_{++}^{\tau} \coloneqq \{ (x_1, x_2) \, : x_1 \geq 0, \, x_2 \geq 0, \, x_1 + x_2 \leq \tau \}$. Furthermore, for any $x \in \mathbb{R}$,
\begin{equation} \label{lim_local_depth_tau_to_0_p_1}
	\lim_{\tau \rightarrow 0^{+}} \frac{1}{\tau^2} LLD(x,\tau) = f^2(x) \text{ a.e.\ }
\end{equation}
If $f(\cdot)$ is continuous then \eqref{lim_local_depth_tau_to_0_p_1} holds everywhere.
\end{corollary}
One of the motivations of this paper is to perform clustering using LLD. With this aim, we introduce the $\tau$-approximation of $f(\cdot)$ as follows:

\begin{definition}[$\tau$-approximation] \label{definition_tau_approximation}
For any $\tau>0$,
\begin{equation*}
f_{\tau}(x) \coloneqq \frac{1}{\tau^p \sqrt{\Lambda_1} } \sqrt{LLD(x,\tau)}. 
\end{equation*}
\end{definition}
\begin{remark} \label{remark_f_tau_continuous}
From Proposition \ref{proposition_local_depth} (iii), it follows that when $P$ has a density $f(\cdot)$ then, $f_\tau(\cdot)$ is continuous. Let $S_{\tau} \coloneqq \{ x \in \mathbb{R}^p \, : \, f_{\tau}(x)>0 \}$ be the support of $f_{\tau}(\cdot)$. If $f(\cdot)$ is $k$-times continuously differentiable, Proposition \ref{proposition_local_depth} (iv) implies that $f_{\tau}(\cdot)$ is $k$-times continuously differentiable in $S_{\tau}$.
\end{remark}
We note that in the approach to clustering adopted in this paper we will encounter convergence of level sets (described in Section \ref{section_clustering} below). This involves uniform convergence of $f_\tau(\cdot)$ to $f(\cdot)$ and other ideas. To this end, we denote by $\overline{B}_{\epsilon}(x) \coloneqq \{ y \in \mathbb{R}^p \, : \, \norm{x-y} \leq \epsilon \}$ the closed ball with center $x \in \mathbb{R}^p$ and radius $\epsilon > 0$.
\begin{theorem} \label{theorem_uniform_convergence_of_tau_approximation}
Let $P$ be absolutely continuous with respect to the Lebesgue measure on $\mathbb{R}^p$ with density $f(\cdot)$. Then the following hold: \\
(i) If $f(\cdot)$ is uniformly continuous and bounded, then
\begin{equation} \label{uniform_convergence_of_tau_approximation}
\lim_{\tau \to 0^{+}} \sup_{ x \in \mathbb{R}^p } \abs{ f_{\tau}(x) - f(x) } = 0.
\end{equation}
(ii) If $f(\cdot)$ is continuous, then for all compact sets $K \subset \mathbb{R}^p$
\begin{equation} \label{uniform_convergence_compact_set_of_tau_approximation}
\lim_{\tau \to 0^{+}} \sup_{ x \in K } \abs{ f_{\tau}(x) - f(x) } = 0.
\end{equation}
In particular, for all $x \in \mathbb{R}^p$
\begin{equation*}
\lim_{\tau, \epsilon \rightarrow 0^{+}} \sup_{y \in \overline{B}_{\epsilon}(x)} \abs{ f_{\tau}(y) - f(x) } = 0.
\end{equation*}
\end{theorem}
\begin{remark} \label{remark_theorem_uniform_consistency_of_tau_approximation}
The above theorem implies that the $\tau$-approximation converges uniformly to the density under extreme localization. We also note that continuity is not enough in Theorem \ref{theorem_uniform_convergence_of_tau_approximation} (i) (see Appendix \ref{sm:section_necessity_of_conditions_in_theorem_uniform_consistency_of_tau_approximation} for a counterexample).
\end{remark}

\subsection{Sample local depth} \label{subsection_sample_local_depth}

Let $\{X_1, \dots, X_n\}$ be independent and identically distributed (i.i.d.) from $P$ on $\mathbb{R}^p$; then the estimate of LLD is the empirical distribution of the pair $(X_i,X_j)$ belonging to the set $Z_{\tau}(x)$; that is, for $x \in \mathbb{R}^p$ and $\tau \in [0,\infty]$
\begin{equation} \label{sample_local_depth}
	LLD_n(x,\tau) \coloneqq \frac{1}{{n \choose 2}} \sum_{1 \leq i < j \leq n} \mathbf{I}( (X_i, X_j) \in Z_{\tau}(x) ).
\end{equation}
As an immediate consequence of Proposition \ref{proposition_local_depth} one obtains the Corollary \ref{corollary_sample_local_depth} below.
\begin{corollary} \label{corollary_sample_local_depth}
(i) For any fixed $x \in \mathbb{R}^p$ and sample size $n \in \mathbb{N}$, $LLD_n(x, \, \cdot \,)$ is a non-decreasing and right-continuous function of $\tau$ satisfying
\begin{equation*}
    \lim_{\tau \to 0^{+}} LLD_n(x,\tau) = \frac{1}{ { n \choose 2 }} \sum_{1 \leq i < j \leq n} \mathbf{I}( (X_i, X_j)=(x,x) )
\end{equation*}
and
\begin{equation*}
    \lim_{\tau \to \infty} LLD_n(x,\tau) = \frac{1}{ { n \choose 2 }} \sum_{1 \leq i < j \leq n} \mathbf{I}( (X_i, X_j) \in Z(x) ).
\end{equation*}
(ii) For $\tau \in [0,\infty]$, $\sup_{x \in \mathbb{R}^p \, : \, \norm{x} \geq M } LLD_n(x,\tau) \rightarrow 0 \text{ as } M \rightarrow \infty$. \\
(iii) For $\tau \in [0,\infty]$, $LLD_n(\, \cdot \, ,\tau)$ is a upper semicontinuous function of $x$.
\end{corollary}

We notice here that $LLD_n$ is a non-degenerate U-statistics of order $2$ \citep{Korolyuk-2013} with the kernel
\begin{equation} \label{kernel}
\mathcal{K}_{x,\tau}(x_1,x_2) \coloneqq \mathbf{I}( (x_1, x_2) \in Z_{\tau}(x))
\end{equation}
whose first projection is
\begin{equation} \label{kernel_projection}
\mathcal{J}_{x,\tau}(x_1) \coloneqq \int \mathbf{I}( (x_1, x_2) \in Z_{\tau}(x)) \, dP(x_2) = P( (X_1, X_2) \in Z_{\tau}(x) \, | X_1=x_1).
\end{equation}
Furthermore, $LLD_n$ is an unbiased estimator of $LLD$, that is, $E[LLD_n(x,\tau)] =  LLD(x,\tau)$. Using \citet[Lemma A, Section 5.2.1]{Serfling-2009}, it follows that 
\begin{equation} \label{variance_sample_local_depth}
	Var[ LLD_n(x,\tau) ] = \frac{1}{{n \choose 2}} [ a^2(x,\tau) + 2(n-2) b^2(x,\tau) ],
\end{equation} 
where $a^2(x,\tau) \coloneqq Var \left[ \mathcal{K}_{x,\tau}(X_1, X_2) \right] = LLD(x, \tau)(1-LLD(x, \tau))$, and
\begin{equation} \label{b_square}
b^2(x,\tau) \coloneqq Var \left[ \mathcal{J}_{x,\tau}(X_1) \right] = \int \left( \mathcal{J}_{x,\tau}(x_1) \right)^2 \, dP(x_1)  - \left( LLD(x,\tau) \right)^2.
\end{equation}
From \eqref{variance_sample_local_depth}, it follows that for all $n \in \mathbb{N}$
\begin{equation*}
    Var[ \sqrt{n} \, LLD_n(x,\tau) ] = 4 \, \frac{n-2}{n-1} \,  b^2(x,\tau) + O \left( \frac{1}{n} \right) \xrightarrow[ n \rightarrow \infty]{} 4 \, b^2(x,\tau).
\end{equation*}
The above calculation shows that $LLD_n$ is a consistent estimator of $LLD$. In typical applications, the choice of $x$ and $\tau$ vary and in exploratory analyses, different choices of $x$ and $\tau$ may be investigated. Our next result shows that the $LLD_n$ is uniformly consistent over $x$ and $\tau$. In the following, we use the notation \emph{a.s.\ }to mean almost surely with respect to $P$.

\begin{theorem} \label{theorem_uniform_consistency}
\begin{equation*}
	\sup_{ \substack{x \in \mathbb{R}^p \\ \tau \in [0,\infty]}} \abs{ LLD_n(x,\tau) - LLD(x,\tau) } \xrightarrow[ n \rightarrow \infty]{} 0 \text{ a.s.} 
\end{equation*}
\end{theorem}
 
We now turn to the uniform central limit theorem for $LLD_n$ over a subset $\Gamma$ of $\mathbb{R}^p \times [0,\infty]$. Let $\ell^{\infty} (\Gamma)$ denote the space of all bounded functions  $\bar{h}(\cdot): \Gamma \rightarrow \mathbb{R}$ . To study the convergence in distribution in $\ell^{\infty}(\Gamma)$, one needs to address the  measurability problems that are encountered due to the non-separability of $\ell^{\infty} (\Gamma)$. We address this using Theorem 4.9 in \cite{Arcones-Gine-1993}. In their paper they handle the issue by requiring ${\cal{F}}$ (not defined here) be a ``measurable'' class, where measurable was described on Page 1497. Indeed, in our proof, and as also stated in their paper, we address this issue by establishing that the class of kernels related to the U-statistics  is image admissible Suslin (see \citet{Dudley-2014}) and (1.9) of their paper holds. In the following, convergence in distribution in $\ell^{\infty} (\Gamma)$ is in the sense of Hoffmann-J{\o}rgensen \citep[Definition 3.7.22]{Gine-2016}. 

\begin{theorem} \label{theorem_uniform_asymptotic_normality_of_local_depth}
  If $\Gamma \subset \mathbb{R}^p \times [0,\infty]$ is such that $b^2(x,\tau) > 0$ for all $(x,\tau) \in \Gamma$, then
\begin{equation*}
	\sqrt{n} \left( LLD_n(\, \cdot \,, \, \cdot \,) - LLD(\, \cdot \, , \, \cdot \,) \right) \xrightarrow[ n \rightarrow \infty]{d} 2 \,  W(\, \cdot \, ,\, \cdot \,) \text{ in } \ell^{\infty} (\Gamma) 
\end{equation*}
where $\{ W(x,\tau) \}_{(x, \tau) \in \Gamma }$ is a centered Gaussian process with covariance function $\gamma: \Gamma \times \Gamma \rightarrow \mathbb{R}$ given by
\begin{equation} \label{covariance_function_of_normalized_local_depth}
\gamma((x,\tau),(y,\nu)) =  \int \mathcal{J}_{x,\tau}(x_1) \mathcal{J}_{y,\nu}(x_1) \, dP(x_1) - \, LLD(x,\tau) LLD(y,\nu).
\end{equation}
\end{theorem}

\begin{remark} \label{remark_b_square_positive}
  Notice that the variance of $W(x,\tau)$ is $\gamma((x,\tau),(x,\tau)) = b^2(x,\tau)$, and it is positive for $\tau>0$ provided $P$ is absolutely continuous with respect to the Lebesgue measure and $x$ is in its support.
\end{remark}

In the rest of this subsection, we assume that $P$ is absolutely continuous with respect to the Lebesgue measure with density $f(\cdot)$. The plug-in estimator of $f_{\tau}(\cdot)$ is given by
\begin{equation} \label{sample_tau_approximation}
f_{\tau,n}(x) \coloneqq \frac{1}{\tau^p \sqrt{\Lambda_1} } \sqrt{LLD_n(x,\tau)}. 
\end{equation}
Our first result summarizes basic properties of $f_{\tau,n}(\cdot)$ and it is a consequence of Theorems \ref{theorem_uniform_convergence_of_tau_approximation} and \ref{theorem_uniform_consistency}.
\begin{corollary} \label{corollary_uniform_consistency_of_sample_tau_approximation} Let $P$ be absolutely continuous with respect to the Lebesgue measure on $\mathbb{R}^p$ with density $f(\cdot)$. Then the following hold: \\
(i) If $f(\cdot)$ is uniformly continuous and bounded, then
\begin{equation} \label{uniform_consistency_of_sample_tau_approximation}
\lim_{\tau \to 0^{+}} \lim_{n \to \infty} \sup_{ x \in \mathbb{R}^p} \abs{ f_{\tau,n}(x) - f(x)} = 0 \text{ a.s.}
\end{equation}
(ii) If $f(\cdot)$ is continuous, then for all compact sets $K \subset \mathbb{R}^p$
\begin{equation} \label{uniform_convergence_compact_set_of_sample_tau_approximation}
\lim_{\tau \to 0^{+}} \lim_{n \to \infty} \sup_{ x \in K} \abs{ f_{\tau,n}(x) - f(x)} = 0  \text{ a.s.}
\end{equation}
In particular, for all $x \in \mathbb{R}^p$
\begin{equation} \label{uniform_continuous_consistency_of_sample_tau_approximation}
\lim_{\tau, \epsilon \rightarrow 0^{+}} \lim_{n \to \infty} \sup_{ y \in \overline{B}_{\epsilon}(x)}  \abs{ f_{\tau,n}(y) - f(x)} = 0 \text{ a.s.}
\end{equation}
\end{corollary}
It is well known that extreme localization is an important concept in depth analysis, however, the fluctuations of $f_{\tau,n}(\cdot)$ are unknown. Our main result in this section characterizes the asymptotic variance and establishes a related limit distribution. To this end, for $x_1 \in \mathbb{R}^p$, we denote the projection of $Z_{\tau}(x)$ onto the first component of $\mathbb{R}^p \times \mathbb{R}^p$ by
\begin{equation*}
Z_\tau(x)\rvert_{x_1} \coloneqq \{ x_2 \in \mathbb{R}^p \, : \, \max_{i=1,2} \norm{x-x_i} \leq \norm{x_1-x_2} \leq \tau \}
\end{equation*}
and $\Lambda_1^{*2} \coloneqq \int \lambda^2(Z_1(0) \rvert_{x_1} ) \, d x_1$.

\begin{theorem} \label{theorem_asymptotic_normality_of_sample_tau_approximation} Let $P$ be absolutely continuous with respect to the Lebesgue measure on $\mathbb{R}^p$ with continuous density $f(\cdot)$. Let $x$ belong to the support of $f(\cdot)$ and $\{ \tau_n \}_{n=1}^{\infty}$ be a sequence of positive scalars converging to zero. If $\sqrt{n} \, \tau_n^{\frac{3}{2}p} \xrightarrow[n \to \infty]{} \infty$, then
\begin{equation*}
  \sqrt{n} \, \tau_n^{\frac{1}{2}p}  \left( f_{\tau_n,n}(x) - f_{\tau_n}(x)  \right) \xrightarrow[ n \rightarrow \infty]{d} N\left( 0, \frac{\Lambda_1^{*2}}{\Lambda_1^2} f(x) \right).
\end{equation*}
\end{theorem}

\begin{remark}
We notice that the limit distribution in Theorem \ref{theorem_asymptotic_normality_of_sample_tau_approximation} with $f_{\tau_n}(\cdot)$ replaced by $f(\cdot)$ cannot hold. In fact, the deterministic term $f_{\tau_n}(x) - f(x)$ is, by Theorem \ref{theorem_local_depth_tau_to_0} (ii), of order $O(\tau_n)$, while the term $f_{\tau_n,n}(x) - f_{\tau_n}(x)$ converges to a normal distribution at rate $1 / ( \sqrt{n} \tau_{n}^{\frac{1}{2}p})$. Since, necessarily, $\sqrt{n} \, \tau_n^{\frac{3}{2}p} \xrightarrow[n \to \infty]{} \infty$, $f_{\tau_n}(x)-f(x)$ is the dominant term.
\end{remark}

The constant $\Lambda_1$ appearing in the limiting variance in Theorem \ref{theorem_asymptotic_normality_of_sample_tau_approximation} can be calculated numerically by computing the percentage of uniformly distributed random points in the centered ball in $\mathbb{R}^p \times \mathbb{R}^p$ with radius $\sqrt{2}$ that lie in $Z_1(0)$ and multiplying the result by its volume $(2^p \pi^p)(p!)^{-1}$. Similarly, the constant $\Lambda_1^{*2}$ can be calculated by approximating the integral with a sum. An alternative form for Theorem \ref{theorem_asymptotic_normality_of_sample_tau_approximation} without the factor $f(x)$ in the variance term is given in the following corollary.

\begin{corollary} \label{corollary_asymptotic_normality_of_sample_tau_approximation}
Under the hypothesis of Theorem \ref{theorem_asymptotic_normality_of_sample_tau_approximation},
\begin{equation*}
  \sqrt{n} \, \tau_n^{\frac{1}{2}p}  \left( \sqrt{f_{\tau_n,n}(x)} - \sqrt{f_{\tau_n}(x)}  \right) \xrightarrow[ n \rightarrow \infty]{d} N\left( 0, \frac{\Lambda_1^{*2}}{4 \Lambda_1^2} \right).
\end{equation*}
\end{corollary}

An extension of Theorem \ref{theorem_asymptotic_normality_of_sample_tau_approximation} uniformly over $S$, namely,
\begin{equation*}
\sqrt{n} \, \tau_n^{\frac{1}{2}p} (f_{\tau_n,n}(\cdot)-f_{\tau_n}(\cdot))  \xrightarrow[ n \rightarrow \infty]{d} \frac{\Lambda_1^{*}}{\Lambda_1} \, W(\cdot) \text{ in } \ell^{\infty} (S),
\end{equation*}
where $\{W(x)\}_{x \in S}$ is a centered Gaussian process with the covariance function $\gamma : S \times S \rightarrow \mathbb{R}$ given by $\gamma(x,y) = \sqrt{f(x)f(y)}$, requires an extension of the results of \citet{Arcones-Gine-1993} to triangular arrays and it is beyond the scope of the present paper. A result in this direction is given by \citet{Schneemeier-1989}, but this is not sufficient in this context since the sets $\{ Z_{\tau_n}(x) \}_{n=1}^{\infty}$ depend on $n$ and $x$.

\subsection{Other instances of local depth} \label{subsection_other_instances_of_local_depth}

In the previous subsections, we established results for the LLD. However, as explained in the introduction, our results remain valid for other LDFs such as LBD and LSD with appropriate modifications. Also, some of the results hold for a local version of half-space depth \citep{Tukey-1975}, denoted by LHD, and a local version of half-region depth \citep{Lopez-2011}, denoted by LRD. We highlight in this section the common features and the main differences between the other LDFs and LLD. In Section \ref{section_simulations_and_data_analysis}, we compare the performance of LLD with LSD in clustering.

We begin by comparing LLD and LSD. An alternative definition of LSD that bounds the volume of the simplex is introduced and studied in \citet{Agostinelli-2011}. However, in this definition, the maximum of the component lengths may remain large even if the volume of the simplex converges to $0$ (as an example consider the triangle). This prevents the extreme localization analysis for $p>1$ which requires the maximum of the component lengths to converge to zero. Thus, we use the definition of LSD provided in the introduction. As before, let $Z_{\tau}^S(x)\rvert_{x_1} $, for $x_1 \in \mathbb{R}^p$, denote the projection of $Z_{\tau}^S(x)$ onto the first component of $(\mathbb{R}^p)^{(p+1)}$.

The main difference between LSD and LLD is that the geometry of the sets $Z_{\tau}^S(x)$ and $Z_{\tau}^S(x) \rvert_{x_1}$ are different from those of $Z_{\tau}(x)$ and $Z_{\tau}(x) \rvert_{x_1}$ (recall that $L$ is suppressed for LLD). Specifically, note that $Z_{\tau}^S(x)$ is a subset of $(\mathbb{R}^p)^{(p+1)}$ while $Z_{\tau}(x)$ is a subset of $(\mathbb{R}^p)^{2}$. This difference affects the corresponding estimator, since the relevant U-statistics is now of order $(p+1)$ as against two for LLD.  The Lebesgue measure of the sets $Z_{\tau}(x)$ and $Z_{\tau}(x) \rvert_{x_1}$, namely, $\Lambda_1$ and $\Lambda_1^{*}$, appear in many of the asymptotic results for LLD. This, in the case of LSD, is replaced by the Lebesgue measures of the corresponding sets; that is, $\Lambda_1^S= \lambda(Z_1^S(0))$ and $(\Lambda_1^{*S})^2=\int \left( \lambda^{\otimes (p-1)} \right)^{2}(Z_{1}^S(0) \rvert_{x_1})\, dx_1$. This change also affects the $\tau$-approximation, since the factor $\tau^2$ for LLD is replaced by $\tau^{(p+1)}$ for LSD. Similar comments hold for other results as well.

On the other hand, except for the difference in the geometry of sets $Z_{\tau}^B(x)$ and $Z_{\tau}(x)$ (and hence their projections), both LBD and LLD have similar analytical and probabilistic properties. Indeed, BD and LD arise as particular instances of a broader class of DFs \citep{Yang-2018}, which is based on the concept of $\beta$-skeleton region \citep{Kirkpatrick-1985}. These DFs are therefore called $\beta$-skeleton DFs. For $\beta \geq 1$, the $\beta$-skeleton depth is a particular case of \eqref{general_depth} with $k=2$, $G=K_{\beta}$ and 
\begin{equation*}
Z^{K_{\beta}}(x) = \{ (x_1,x_2) \in \mathbb{R}^p \times \mathbb{R}^p \, : \, \max_{(i,j) \in \{ (1,2), (2,1) \}} \norm{ x_i + (2/\beta-1) x_j - 2/\beta \, x} \leq \norm{x_1-x_2} \}.
\end{equation*}
The corresponding local version ($LK_{\beta}D$) is then obtained by replacing $Z^G_{\tau}(x)$ with 
\begin{equation*}
Z^{K_{\beta}}_{\tau}(x) = \{ (x_1,x_2) \in \mathbb{R}^p \times \mathbb{R}^p \, : \, \max_{(i,j) \in \{ (1,2), (2,1) \}} \norm{ x_i + (2/\beta-1) x_j - 2/\beta \, x } \leq \norm{x_1-x_2} \leq \tau \}.
\end{equation*}
Notice that,  $LK_{1}D(\, \cdot \, , \, \cdot \,)=LBD( \, \cdot\, , \, \cdot \,)$ and  $LK_{2}D( \, \cdot \, , \, \cdot \,)=LLD( \, \cdot \, , \, \cdot \,)$. In particular, for $p=1$, the LDFs $LK_{\beta}D$, $LLD$, $LBD$, and $LSD$ coincide.

Turning to LHD (see \citet{Agostinelli-2011}[Definition 12]), notice that, as in LSD, the slabs proposed in there cannot be included in a ball of radius $r$ for any $r$ positive. For this reason, we propose an alternative definition of this local depth. As before, let $\tau \ge 0$ and $x \in \mathbb{R}^p$. Then the local half-space depth with respect to $P$ is given by 
\begin{equation} \label{local_half-space_depth}
LHD(x,P,\tau) = \inf_{u \in S^{p-1}} P(Z_{\tau}^H(x,u)),
\end{equation}
where $Z_{\tau}^H(x, u)$ is the hypercube with center $x+\frac{\tau}{2} u$ given by
\begin{equation*}
Z_{\tau}^H(x, u) = \bigg\{ y \in \mathbb{R}^p \, : \, \max_{i=1,\dots,p} \abs{x^{(i)} + \frac{\tau}{2} u^{(i)} - y^{(i)}} \leq \frac{\tau}{2} \bigg\}
\end{equation*}
and $S^{p-1}$ is the unit sphere in $\mathbb{R}^p$.

Similar to the notion of local half-space depth is the local half-region depth (LRD) introduced and studied in \citet{Agostinelli-2018}. This is given by
\begin{equation*}
  LRD(x,P,\tau) = \min( P(Z^{R-}_{\tau}(x)), P(Z^{R+}_{\tau}(x) ),
\end{equation*}
where $Z^{R-}_{\tau}(x) = \{ y \in \mathbb{R}^p \, : \, -\tau \leq y^{(i)} - x^{(i)} \leq 0,\, i=1, \dots, p \}$ and $Z^{R+}_{\tau}(x) = \{ y \in \mathbb{R}^p \, : \, 0 \leq y^{(i)} - x^{(i)} \leq \tau, \,i=1, \dots, p \}$. In particular, for $p=1$, LHD and LRD coincide.

We now describe some features of LSD, LBD, $LK_{\beta}D$, LHD, and LRD that are shared with LLD allowing one to obtain similar results. Recalling that the superscript $G$ represents general depth ($G = L, S, B, K_{\beta}, H, R$), we note that the expression defining LLD, LSD, LBD, and $LK_{\beta}D$ is  an integral of the form \eqref{general_local_depth}, where $Z_\tau^G(x) \subset (\mathbb{R}^p)^{k}$ is a closed set that satisfies the invariance properties and symmetry conditions
\begin{align} \label{invariance1_of_z_tau_x_general}
  Z_{\tau}^G(x)&=(x, \dots, x)+Z_{\tau}^G(0) \\ \label{invariance2_of_z_tau_x_general}
   Z_{\tau}^G(0)&=\tau Z_1^G(0) \\ \label{symmetry1_of_z_tau_x_general}
(x,\dots,x) + (x_1,\dots,x_k) \in Z_{\tau}^G(x)  &\Longleftrightarrow (x,\dots,x) - (x_1,\dots,x_k) \in Z_{\tau}^G(x) \\ \label{symmetry2_of_z_tau_x_general}
(x,\dots,x) + (x_1,\dots,x_k) \in Z_{\tau}^G(x) &\Longleftrightarrow (x,\dots,x) + (x_{i_1},\dots,x_{i_k}) \in Z_{\tau}^G(x)
\end{align}
for every permutation $(i_1,\dots,i_k)$ of $(1,\dots,k)$.  For LLD, LSD, LBD and $LK_{\beta}D$ there are sets $L \subset \mathbb{R}^p$ (lenses, simplices, balls, intersection of balls) such that
$\{ L(x_1,\dots,x_k) \, : \, x_1,\dots,x_k \in \mathbb{R}^p \}$ is a VC-class and $(x_1,\dots,x_k) \in Z_{\tau}^G(x)$ if and only if $x \in L(x_1,\dots,x_k)$. This property proves useful in establishing Theorems \ref{theorem_uniform_consistency} and \ref{theorem_uniform_asymptotic_normality_of_local_depth} for LSD, LBD and $LK_{\beta}D$. The invariance properties \eqref{invariance1_of_z_tau_x_general} and \eqref{invariance2_of_z_tau_x_general} are needed to establish Proposition \ref{proposition_local_depth} (i), Corollary \ref{corollary_sample_local_depth} (i) and the results concerning the density function, for example, Theorems \ref{theorem_local_depth_tau_to_0} and \ref{theorem_uniform_convergence_of_tau_approximation}. The symmetry properties \eqref{symmetry1_of_z_tau_x_general} and \eqref{symmetry2_of_z_tau_x_general} are required in Theorem \ref{theorem_local_depth_tau_to_0} (ii), and Theorems \ref{theorem_local_central_symmetry} and \ref{theorem_modes_local_central_symmetry} in the next section. A detailed description of the modifications required for the results concerning LLD to hold for LSD is provided in Appendix \ref{sm:section_local_simplicial_depth}. The results for LBD and $LK_{\beta}D$ are very similar to those of LLD and therefore omitted.

We notice that LHD and LRD also share many properties with $LLD$. Indeed, LHD is the infimum over $u \in S^{p-1}$ of integrals of the form \eqref{general_local_depth} with $k=1$ and $Z_{\tau}^G(x)$ replaced by $Z_{\tau}^H(x, u)$. If $P$ is absolutely continuous with respect to the Lebesgue measure, this infimum is actually a minimum. Hence, most of the properties and the applications to clustering hold for LHD since (i) we can replace $u$ by the minimizer, (ii) $Z_{\tau}^H(x, u)$ satisfies \eqref{invariance1_of_z_tau_x_general}, \eqref{invariance2_of_z_tau_x_general}, and \eqref{symmetry2_of_z_tau_x_general}, and (iii) the class of hypercubes $\{ Z_{\tau}^H(x, u) \, : \, x \in \mathbb{R}^p, u \in S^{p-1}, \tau \in [0,\infty] \}$ forms a VC-class.  We refer to Appendix \ref{sm:section_local_half-space_depth}, for further details. Similar comments also hold for LRD, as it is the minimum of the probability of two hypercubes satisfying \eqref{invariance1_of_z_tau_x_general}, \eqref{invariance2_of_z_tau_x_general} and \eqref{symmetry2_of_z_tau_x_general}. In particular, it has the same behavior under extreme localization (see Appendix \ref{sm:section_local_half-space_depth}).

\section{Clustering} \label{section_clustering}

In this section, we describe a new methodology for clustering high-dimensional data. The basic idea relies on defining clusters as the stable manifolds generated by the mode, which are obtained using the limiting trajectory of a gradient system (see \citet{Chacon-2015}). The fact that such manifolds yield non-trivial meaningful clusters can then be established using the Lyapunov's theory in dynamical systems. The development of the algorithm needs an investigation of the
behavior of the level sets. \citet{Chacon-2015} uses density based methods for investigating these clusters in one dimension only. Our methods, which involve the use of the $\tau$-approximation based on LLD, are for multidimensional distributions. 

It is {\it{a priori}} unclear why such a procedure would yield meaningful clusters. Indeed, using analytical and statistical results on extreme localization described in Sections 2 and 3, it is plausible that the proposal so described will yield meaningful clusters. In the rest of this section, we describe step-by-step analytical tools to fill in the gap between local depths and stable manifolds generated by the modes. Towards this aim, we begin by providing a precise notion of the key ideas. 

Let $f(\cdot)$ be a twice-continuously differentiable density function. Then for any $x_0 \in S$, the stable manifold generated by $x_0$ is
\begin{equation} \label{stable_manifold}
    C(x_0) \coloneqq \{ x \in S \, : \, \lim_{t \rightarrow \infty} u_x(t) = x_0 \}, 
\end{equation}
where $u_{x}(t)$ is the solution at time $t$ of the gradient system
\begin{equation} \label{gradient_system}
	u^{\prime}(t) = \nabla f (u(t))
\end{equation}
with initial value $u(0)=x$. For any choice of $x_0$, it is not required for the stable manifold so-defined to be non-trivial; i.e.\ the Lebesgue measure of $C(x_0)$ is non-zero. However, if $x_0$ is chosen as a local maximum of $f(\cdot)$, then, under additional conditions, one can verify that the resulting manifold has positive Lebesgue measure. In the next subsections, we establish several results leading up to the verification that using local maxima of $f(\cdot)$ and $f_{\tau}(\cdot)$ yield non-trivial clusters. These results are stated for LLD, but also hold for other notions of depth such as spherical depth, simplicial depth, and $\beta$-skeleton depth. Some details of the changes needed for simplicial depth are provided in Appendix \ref{sm:section_local_simplicial_depth}.

\subsection{Mathematical background for cluster identification}{\label{subsection_mathematical_background}}

In this subsection, we collect relevant mathematical background that allows one to develop the algorithm for identification of clusters, also referred to as stable manifolds or basins of attraction. We begin with the definition of level sets and super level sets.
\begin{definition} \label{definition_level_sets_and_regions}
  For $\alpha>0$, the level sets of $f(\cdot)$ and $f_{\tau}(\cdot)$ are $L^{\alpha} = \{x \in \mathbb{R}^p \, : \, f(x)=\alpha \}$ and $L^{\alpha}_{\tau} = \{x \in \mathbb{R}^p \, : \, f_{\tau}(x)=\alpha \}$, respectively. The super level sets of $f(\cdot)$, $f_{\tau}(\cdot)$ and $f_{\tau,n}(\cdot)$ are $R^{\alpha} \coloneqq \{ x \in \mathbb{R}^p \, : \, f(x) \geq \alpha \}$, $R^{\alpha}_{\tau} \coloneqq \{ x \in \mathbb{R}^p \, : \, f_{\tau}(x) \geq \alpha \}$ and $R^{\alpha}_{\tau,n} \coloneqq \{ x \in \mathbb{R}^p \, : \, f_{\tau,n}(x) \geq \alpha \}$, respectively.
\end{definition}
 The next proposition shows that in the limit the super level sets induced by $f_{\tau}(\cdot)$ and $f_{\tau,n}(\cdot)$ coincide with those induced by $f(\cdot)$. We use the notation $\mathring{A}$ for the interior of a set $A$.
\begin{proposition} \label{proposition_convergence_of_level_sets}
Suppose that $f(\cdot)$ is uniformly continuous and bounded. Let $\{ \alpha_k \}_{k=1}^{\infty}$ be a sequence converging to $\alpha>0$ and $\{ \tau_k \}_{k=1}^{\infty}$ a sequence of positive scalars converging to $0$. Then
\begin{equation} \label{inclusions_level_sets}
\mathring{R^{\alpha}} \subset \liminf_{ k \to \infty } R^{\alpha_k}_{\tau_k} \subset \limsup_{ k \to \infty } R^{\alpha_k}_{\tau_k} \subset R^{\alpha},
\end{equation}
and, if $\lambda(L^{\alpha})=0$,
\begin{equation} \label{convergence_level_sets}
\lim_{ k \to \infty } R^{\alpha_k}_{\tau_k} = R^{\alpha} \text{ a.e.}
\end{equation}
If additionally $\lambda(L^{\alpha}_{\tau_k})=0$ for $k \in \mathbb{N}$, then
\begin{equation} \label{convergence_sample_level_sets}
\lim_{ k \to \infty} \lim_{ n \to \infty} R^{\alpha_n}_{\tau_k,n} = R^{\alpha} \text{ a.s.}
\end{equation}
\end{proposition}
\begin{remark}
Proposition \ref{proposition_convergence_of_level_sets} also holds for continuous $f(\cdot)$ if we restrict $R^{\alpha}$, $R^{\alpha}_{\tau}$ and $R^{\alpha}_{\tau,n}$ to a compact subset of $\mathbb{R}^p$. Since LSD, LBD, $LK_{\beta}D$, LHD and LRD satisfy Proposition \ref{proposition_local_depth} (ii), Theorem \ref{theorem_uniform_convergence_of_tau_approximation} and Theorem \ref{theorem_uniform_consistency}, the proof of Proposition \ref{proposition_convergence_of_level_sets} shows that the super level sets based on all these LDFs satisfy \eqref{inclusions_level_sets}, \eqref{convergence_level_sets} and \eqref{convergence_sample_level_sets}.
\end{remark}
Proposition \ref{proposition_convergence_of_level_sets} is of independent interest, since a common approach in modal clustering is to define clusters as the connected components of the super level sets $R^{\alpha}$ for some $\alpha>0$ \citep{Menardi-2016}. Once the connected components are computed, the remaining points may be allocated to one of the clusters by using supervised classification techniques. A common approach is then to study how the clusters change as the parameter $\alpha$ varies, yielding cluster trees. 

In the rest of this section, we suppose that $f(\cdot)$ is twice continuously differentiable with finite number of stationary points in $S$. Additionally, we assume that the Hessian matrix $H_f$ has non-zero eigenvalues at its stationary points and that $R^{\alpha}$ is a bounded set for every $\alpha>0$. Since $f(\cdot)$ is continuous, $R^{\alpha}$ is compact. We notice that $R^{\alpha}$ is bounded if $f(\cdot)$ vanishes at infinity, that is, $\sup_{x \in \mathbb{R}^p \, : \, \norm{x} \geq M } f(x) \to 0$ as $M \to \infty$, which is satisfied, for example, if $S$ is bounded.

Even though well-known, to make the manuscript self-contained, we define the mode. The definition of cluster as basin of attraction of a mode is given in \citet{Chacon-2015}.
\begin{definition} \label{definition_modes_clusters}
A point $m \in S$ is said to be a mode (resp.\ an antimode) for $f(\cdot)$ if it is a stationary point of $f(\cdot)$ and $H_f(m)$ has only negative (resp.\ positive) eigenvalues, that is, $m$ is a local maximum (resp.\ minimum) for $f(\cdot)$. If $m_1, \dots, m_k$ are the modes of $f(\cdot)$, then the clusters induced by $m_1, \dots, m_k$ are the stable manifolds $C(m_1), \dots, C(m_k)$.
\end{definition}
Since the clusters are obtained as limits of trajectories induced by modes, we now summarize relevant properties of the gradient system \eqref{gradient_system} by using results from the theory of ordinary differential equations and dynamical systems \citep{Agarwal-1993,Hale-1980,Teschl-2012,Perko-2013}. We first note that $u_x(\cdot)$ exists and is unique for $t$ in some maximal time interval $(a,b)$ with $a<0<b$, where $a=-\infty$ or $b=\infty$ is allowed. To see this, observe that, as $f(\cdot)$ is twice continuously differentiable, for every $x \in S$ there exists a convex neighborhood $U_x$ of $x$ in which the second order partial derivatives are bounded. By applying \citet{Agarwal-1993}[Lemma 3.2.1] to $\nabla f$, it follows that $\nabla f$ is Lipschitz in $U_x$, and therefore $\nabla f$ is locally Lipschitz in $S$. Now, applying Picard-Lindel{\"o}f Theorem \citep{Teschl-2012}[Theorem 2.2 and 2.5] it follows that $u_x(\cdot)$ exists in some time interval, which can be chosen to be maximal in view of Theorem 2.13 in \cite{Teschl-2012}. 

We now show that, using the boundedness of $R^{\alpha}$, the solution $u_x(t)$ exists for all $t \ge 0$ and all $x \in S$. Furthermore, the solution starting in $R^{\alpha}$ cannot leave the set. To this end, notice that the equilibria of  \eqref{gradient_system} are the stationary points of $f(\cdot)$. The gradient   computed at each point gives the direction of most rapid increase of $f(\cdot)$. Hence, the trajectories $\{ u_x(t) \, : \, t \in \mathbb{R} \}$ for $x \in S$ that are not stationary points are curves of steepest ascent for $f(\cdot)$. More specifically, if $u_x(t) \in L^{\alpha}$ for some $x \in S$ and $t \in \mathbb{R}$, then any vector $v$ tangent to $L^{\alpha}$ at $u_x(t)$ satisfies $\inp{v}{u^{\prime}_x(t)}=0$ (see \citep{Hirsch-1974}[Chapter 9 \S4 Theorem 2] and \citep{Jost-2005}[Lemma 6.\ 4.\ 2.]). Hence, either $u_x(t)=x$ for all $t$ is a stationary point of the gradient system \eqref{gradient_system} or the trajectory $\{ u_x(t) \, : \, t \in \mathbb{R} \}$ crosses $L^{\alpha}$ orthogonally. This also implies that $u_x(t)$ cannot leave $R^{\alpha}$ for $t \geq 0$. Furthermore, this property shows that, for all $x \in S$, the solutions $u_x(t)$ exists for all $t \geq 0$, i.e.\ the maximal time interval in which $u_x$ is defined is $(a,\infty)$ for some $a<0$, where $a=-\infty$ is possible. To see this, for $x \in S$, choose an $\alpha>0$ such that $x \in R^{\alpha}$. Since $u_x(t)$ cannot leave $R^{\alpha}$ for $t \geq 0$ and $R^{\alpha}$ is compact, the result follows from \citet{Teschl-2012}[Corollary 2.15]. Recalling that our clusters are the stable manifolds generated by modes, we now link modes to the gradient system. This requires the notion of $\omega$-limit which we now define.
\begin{definition}
The $\omega$-limit of a point $x \in S$ is the set of points $y \in S$ such that $u_x(t)$ goes to $y$ as $t \rightarrow \infty$, in symbols
\begin{equation*}
\omega(x) = \{ y \in S \, : \, \lim_{t \to \infty} u_x(t) = y \}.
\end{equation*}
\end{definition}
The following definition is also required (\citet{Hirsch-1974}[Chapter 9 \S3 Theorem 1], \citet{Teschl-2012}[Section 6.6]). We use the notation that for any function $W\, : \, U \to \mathbb{R}$, $W^{\prime}(u_x(t))=\frac{d \, W(u_x(t))}{dt}$.
\begin{definition}
Let $u_0 \in S$ be an equilibrium point for \eqref{gradient_system} and $U \subset S$ a neighborhood of $u_0$. A  differentiable function $W \, : \, U \to \mathbb{R}$ is a strict Lyapunov function if (i) $W(u_0) = 0$ and $W(u) > 0$ for $u \neq u_0$, and (ii) $W^{\prime}(u_x(t)) < 0$ when $u_x(t) \in U \setminus \{ x_0 \} $.
\end{definition}
Let $V(\cdot) \coloneqq - f(\cdot)$. If $m$ is a mode for $f(\cdot)$, there exists a neighborhood $U_m$ of $m$ such that, for all $u \in U_m \setminus \{ m \}$, $V(u)-V(m) > 0$ and
\begin{equation*}
\left( V(u)-V(m) \right)^{\prime} = - \left( f(u) \right)^{\prime} = - \inp{\nabla f(u)}{u^{\prime}} = - \norm{\nabla f(u)}^2 < 0.
\end{equation*}
Hence, $V(u)-V(m)$ is a strict Lyapunov function in $U_m$. By the Lyapunov stability Theorem (see \citep{Hirsch-1974}[Chapter 9 \S3 Theorem 1] and \citep{Hale-1980}[Chapter X.1 Theorem 1.1]) $m$ is asymptotically stable, that is, there is a neighborhood $U^{*}_m \subset U_m$ of $m$ such that each solution starting from a point $x \in U^{*}_m$ converges to $m$, i.e., for all $x \in U^{*}_m$, $\omega(x)=\{ m \}$. 

Given a mode $m$, we now describe the stable manifold generated by $m$. This is equivalent to describing the properties of the points that have $\omega$-limit $m$. Indeed, if $0<\alpha < f(m)$ is such that the connected component of $m$ in $R^{\alpha}$ contains no equilibria other than $m$, then, since each solutions of \eqref{gradient_system} starting in that component cannot leave it, by LaSalle's invariance principle (see \citep{Hirsch-1974}[Chapter 9 \S3 Theorem 2] and \citep{Teschl-2012}[Theorem 6.14]) applied to the strict Lyapunov function $V(\cdot)-V(m)$, all the points in that component have $m$ as an $\omega$-limit point. On the other hand, if $m$ is an antimode for $f(\cdot)$, there exists a neighborhood $U_m$ of $m$ such that for all $u \in U_m \setminus \{ m \}$, $V(m)-V(u) > 0$ and $(V(m)-V(u))^{\prime} > 0$. This implies that $m$ is unstable \citep{Hale-1980}[Chapter X.1 Theorem 1.2]: for every neighborhood $U^{*}_m \subset U_m$ of $m$ every solution $u_x$ starting from a point $x \in U^{*}_m$ eventually leaves $U^{*}_m$. Furthermore, any $\omega$-limit point of gradient system \eqref{gradient_system} is an equilibrium point, that is a stationary points of $f(\cdot)$ (see  \citep{Hirsch-1974}[Chapter 9 \S4 Theorem 4] and \citep{Hale-1980}[Chapter X.1 Theorem 1.3], and \citep{Jost-2005}[Lemma 6.\ 4.\ 4.] in a different context).

For a stationary point $x_0$ of $f(\cdot)$, recall from \eqref{stable_manifold}
that $C(x_0)$ is the stable manifold induced by $x_0$, that is, the set of points with $\omega$-limit $x_0$. The hypothesis that $H_{f}$ (which is symmetric), at stationary points, has non-zero eigenvalues, implies that the sets $C(x_0)$ have dimension equal to the number of negative eigenvalues of $H_f (x_0)$ (see \citep{Teschl-2012}[Section 9]).

As in Definition \ref{definition_modes_clusters}, let $m_1, \dots, m_k$ be the modes of $f(\cdot)$. We are now ready to verify that the clusters $C(m_1), \dots, C(m_k)$ are non-trivial. We first observe that, by the uniqueness of the limit, $C(m_1), \dots, C(m_k)$ are disjoint. Also, if $x_0$ is not a mode, then $C(x_0)$ has dimension smaller than $p$ and the set $S \setminus \left( C(m_1) \cup \dots \cup C(m_k) \right)$ has Lebesgue measure $0$. However, in practice, the density $f(\cdot)$ is generally unknown and it has to be estimated. A common approach is to use kernel density estimators. However, we replace $f(\cdot)$ by $f_{\tau}(\cdot)$ in \eqref{gradient_system} and consider instead the gradient system
\begin{equation} \label{gradient_system_tau}
	u^{\prime}(t) = \nabla f_{\tau} (u(t)).
\end{equation}
The domain of this new system is $S_{\tau}$, the support of $f_{\tau}(\cdot)$.

\subsection{Convergence of the gradient system under extreme localization} \label{subsection_convergence_gradient_system}

We summarize in this subsection the main properties of the gradient system \eqref{gradient_system_tau} as $\tau \to 0^{+}$.
\begin{lemma} \label{lemma_support_S_and_S_tau}
For all $0 < \tau_{1} \leq \tau_{2}$, we have that $S_{\tau_1} \subset S_{\tau_2}$. Additionally, if $f(\cdot)$ is continuous, then for all $\tau>0$, $S \subset S_{\tau}$ and $\lim_{\tau \to 0^{+}} S_{\tau}=S$.
\end{lemma}
The next theorem shows that the gradient and the Hessian matrix of $f_{\tau}(\cdot)$ converge to those of $f(\cdot)$. Given a function $g : D \subset \mathbb{R}^p \rightarrow \mathbb{R}$ for $j=1, \dots, p$ we denote by $\partial_j g$ its partial derivative with respect to the $j$-component.
\begin{theorem} \label{theorem_first_and_second_partial_derivaties}
Let $\tau>0$. If $f(\cdot)$ has continuous first order partial derivatives, then, for all $j=1,\dots,p$, $\partial_j f_{\tau}$ is continuous in $S_{\tau}$ and, for all $x \in S$,
\begin{equation*}
\lim_{ \tau \rightarrow 0^{+} } \partial_j f_{\tau}(x) =  \partial_j f(x).
\end{equation*}
If also the second order partial derivatives are continuous, then, for all $i,j=1,\dots,n$, $\partial_i \partial_j f_{\tau}$ is continuous in $S_{\tau}$ and, for all $x \in S$,
\begin{equation*}
\lim_{ \tau \rightarrow 0^{+} } \partial_i \partial_j f_{\tau}(x) = \partial_i \partial_j f(x).
\end{equation*}
\end{theorem}
\begin{remark}
Lemma \ref{lemma_support_S_and_S_tau} and Theorem \ref{theorem_first_and_second_partial_derivaties} hold also for LSD, LBD, $LK_{\beta}D$, LHD and LRD.
\end{remark}
The sets $\{ S_{\tau} \}_{\tau>0}$ monotonically decrease to $S$ by Lemma \ref{lemma_support_S_and_S_tau}. Furthermore, because of Theorem \ref{theorem_first_and_second_partial_derivaties}, $f_{\tau}(\cdot)$ is twice continuously differentiable in $S_{\tau}$ and its gradient and Hessian matrix converge to those of $f(\cdot)$ in $S$. If it exists, we denote by $u_{x,\tau}(t)$ the solution of \eqref{gradient_system_tau} with initial point $u_{x,\tau}(0)=x$. Since $f_{\tau}(\cdot)$ is continuous, for $\alpha>0$, the sets $R^{\alpha}_{\tau}=\{ x \in \mathbb{R}^p \, : \, f_{\tau}(x) \geq \alpha \}=f_{\tau}^{-1}([\alpha,\infty))$ are closed. The next lemma shows that they are also bounded. For a set $A \subset \mathbb{R}^{p}$, we define $(A)^{+\delta} \coloneqq \{ x \in \mathbb{R}^p \, : \, \inf_{y \in A} \norm{x-y} \leq \delta \}$ and $(A)^{-\delta} \coloneqq \mathbb{R}^p \setminus \left( \mathbb{R}^p \setminus A \right)^{+\delta} = \{ x \in \mathbb{R}^p \, : \, \inf_{y \in \mathbb{R}^p \setminus A} \norm{x-y} > \delta \}$.
\begin{lemma} \label{lemma_R_alpha_tau_bounded}
For all $\tau>0$ and $\alpha>0$, $(R^{\alpha})^{-\tau} \subset R_{\tau}^{\alpha} \subset (R^{\alpha})^{+\tau}$. In particular, if $R^{\alpha}$ is bounded for $\alpha >0$, then $R^{\alpha}_{\tau}$ is also bounded for any $\tau>0$.
\end{lemma}
As we have shown for the gradient system \eqref{gradient_system}, Lemma \ref{lemma_R_alpha_tau_bounded}, along with the boundedness of $R^{\alpha}$ for all $\alpha >0$ and the fact that $f_{\tau}(\cdot)$ is twice continuously differentiable, implies that, for all $x \in S$, $u_{x,\tau}$ exists and is unique in  a maximal time interval $(a,\infty)$, for some $-\infty \leq a < 0$.

Theorem \ref{theorem_local_depth_tau_to_0} (i) shows that, in the limit for $\tau \to 0^{+}$, the function $f_{\tau}(\cdot)$ behaves like $f(\cdot)$ (see also Theorem \ref{theorem_uniform_convergence_of_tau_approximation}) and Theorem \ref{theorem_first_and_second_partial_derivaties} shows that the partial derivatives of $f_{\tau}(\cdot)$ up to the second order converge to those of $f(\cdot)$. We exploit this in the next theorem to show that the solutions of the gradient system \eqref{gradient_system_tau} converge for $\tau \to 0^{+}$ to those of the gradient system \eqref{gradient_system}.
\begin{theorem} \label{theorem_convergence_solution_estimated_gradient_system}
Suppose that $f(\cdot)$ is continuously differentiable in $\mathbb{R}^p$ and that for all $\alpha>0$ $R^{\alpha}$ is compact. Then, for all $t \geq 0$ and $x \in S$
\begin{equation*}
\lim_{\tau \rightarrow 0^{+}} u_{x,\tau}(t) = u_{x}(t).
\end{equation*}
\end{theorem}
\begin{remark}
Theorem \ref{theorem_convergence_solution_estimated_gradient_system} holds also for LSD, LBD, $LK_{\beta}D$, LHD and LRD using Theorem \ref{theorem_first_and_second_partial_derivaties} and Lemma \ref{lemma_R_alpha_tau_bounded}. For $LK_{\beta}D$ with $\beta>2$ a small modification is required in Lemma \ref{lemma_R_alpha_tau_bounded}; namely, for all $\tau>0$ and $\alpha>0$, $(R^{\alpha})^{-\frac{\beta}{2}\tau} \subset R_{\tau}^{\alpha} \subset (R^{\alpha})^{+\frac{\beta}{2}\tau}$. Similarly, since hypercubes in $\mathbb{R}^p$ have diameter $\sqrt{p}$, for LHD and LRD we have that $(R^{\alpha})^{-\sqrt{p} \tau} \subset R_{\tau}^{\alpha} \subset (R^{\alpha})^{+\sqrt{p}\tau}$.
\end{remark}
In the next subsection we provide some conditions under which the stationary points (resp.\ modes, antimodes) of $f(\cdot)$ are also stationary points (resp.\ modes, antimodes) of $f_{\tau}(\cdot)$ for $\tau>0$.
\subsection{Identification of modes} \label{subsection_modes_identification}
In this subsection we consider an absolutely continuous distribution with respect to the Lebesgue measure, with density $f(\cdot)$. The key criteria for the identification of the modes is the notion of symmetry proposed below.
\begin{definition} \label{definition_tau_symmetry}
Given $\tau>0$, a density function $f(\cdot)$ with support $S$ is said to be $\tau$-centrally symmetric about $\mu \in S$ if, for all $x \in \mathbb{R}^p$ with $\norm{x} \leq \tau$, $f(\mu+x)=f(\mu-x)$.
\end{definition}
In particular, for $p=1$, $f(\cdot)$ is $\tau$-centrally symmetric about $\mu \in \mathbb{R}$ if $f(\mu-x)=f(\mu+x)$ for all $x \in [0,\tau]$. If $f(\cdot)$ has a continuous derivative, a direct computation using Corollary \ref{corollary_local_depth_tau_to_0} shows that $f_{\tau}^{'}(\mu) = 0$ (see Appendix \ref{section_stationary_points_for_univariate_densities} for more details). For example, if $f(\cdot)$ is a mixture of two normal densities with means $\mu_1$ and $\mu_2$, variances equal to $\sigma^2$ and weights equal to $1/2$, that is,
\begin{equation} \label{mixture_normals_density}
  f(x) = \frac{1}{\sqrt{2 \pi} \cdot 2 \sigma} \left[ e^{\frac{-(x-\mu_1)^2}{2 \sigma^2}} + e^{-\frac{(x-\mu_2)^2}{2 \sigma^2}} \right],
\end{equation}
then $f(\cdot)$ is $\tau$-centrally symmetric about $\mu=\frac{\mu_1+\mu_2}{2}$ for all $\tau \geq 0$ and therefore $\mu$ is a stationary point of $f_{\tau}(\cdot)$. On the other hand, due to the interference of the tail of one component to the other around the means, $f(\cdot)$ is not $\tau$-centrally symmetric about them. However, if the two means are far away from each other, the modes of $f_{\tau}(\cdot)$ are a good approximation of those of $f(\cdot)$ (see Figure \ref{mixture_normals_plot} and Propositiion \ref{proposition_modes} in Appendix \ref{section_stationary_points_for_univariate_densities}).
\begin{figure}[!htbp]
	\includegraphics[width= \linewidth]{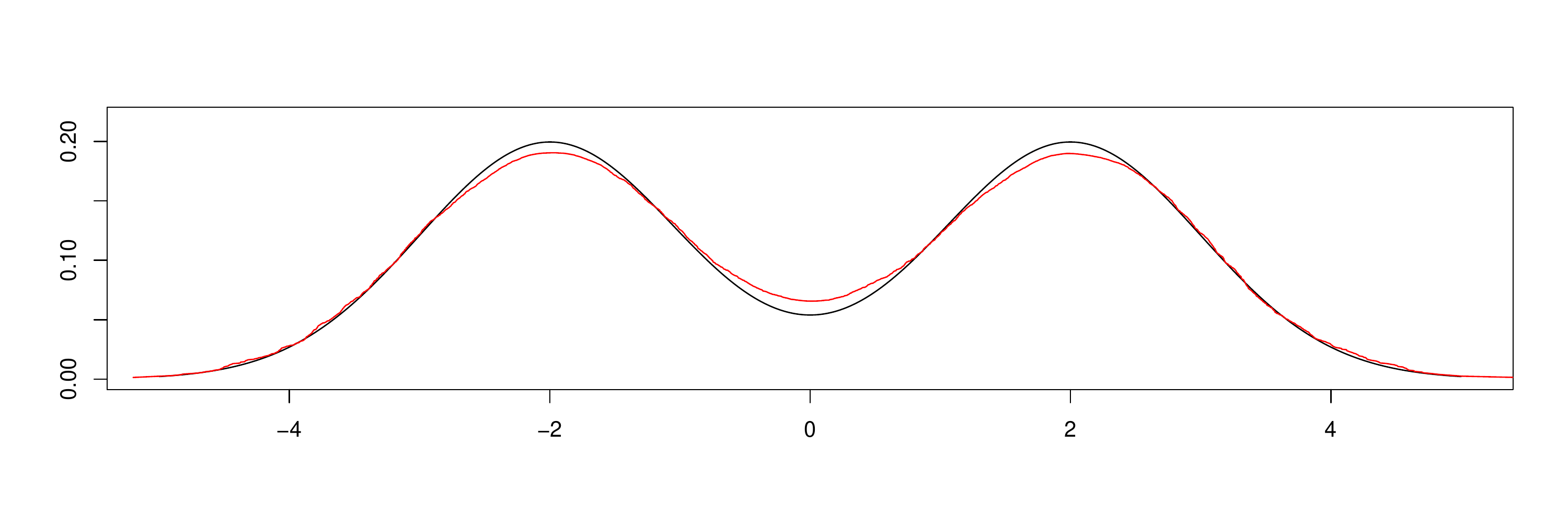}
	\caption{ In black, the density \eqref{mixture_normals_density} with $\mu_1=-2$, $\mu_2=2$ and $\sigma=1$. In red, its sample $\tau$-approximation $f_{\tau,n}(\cdot)$ for $\tau=1$ and $n=6000$. }
 	\label{mixture_normals_plot}
\end{figure}
We begin by identifying the stationary points of $f_{\tau}$.
\begin{theorem} \label{theorem_local_central_symmetry}
Suppose that for some $\tau>0$ and $\mu \in S$, $f(\cdot)$ has continuous first order partial derivatives in $\overline{B}_{\tau}(\mu) = \{ x \in \mathbb{R}^p \, : \, \norm{x-\mu} \leq \tau \}$. Then, $\mu$ is a stationary point for $f_{\tau}(\cdot)$ if and only if
\begin{equation} \label{equivalent_condition_for_mu_stationary_point}
	\int_{Z_1(0)} \nabla f(\mu+\tau x_1) f(\mu+\tau x_2) \, dx_1 dx_2 = 0,
\end{equation}
where the integral of a vector is the vector of the integrals. Furthermore, if $f(\cdot)$ is $\tau$-centrally symmetric about $\mu$, then $\nabla f_{\tau}(\mu)=0$.
\end{theorem}
Our next result, which is about the Hessian matrix, gives sufficient conditions for $m$ to be a mode of $f_{\tau}(\cdot)$.
\begin{theorem} \label{theorem_modes_local_central_symmetry}
Let $\tau>0$. Suppose that $f(\cdot)$ is $\tau$-centrally symmetric about a mode (resp.\ an antimode) $m$ and has continuous second order partial derivatives in $\overline{B}_{\tau}(m)$. \\
If, for all $x_1, x_2 \in \overline{B}_{\tau}(m)$, the matrix 
\begin{equation*}
G_f(x_1,x_2) \coloneqq H_{f}(x_1) f(x_2) + \nabla f(x_1) \nabla f(x_2)^{\top}
\end{equation*}
is negative (resp.\ positive) definite, then $m$ is also a mode (resp.\ an antimode) for $f_{\tau}(\cdot)$.
\end{theorem}
Notice that $G_f(m,m)=H_{f}(m) f(m)$ is negative (resp.\ positive) definite and therefore the second condition of Theorem \ref{theorem_modes_local_central_symmetry} is satisfied by $f(\cdot)$, for $\tau$ small.
\begin{remark}
  Theorem \ref{theorem_local_central_symmetry} and Theorem \ref{theorem_modes_local_central_symmetry} hold also for LSD, LBD and $LK_{\beta}D$. For $LK_{\beta}D$ with $\beta>2$ we assume that $f(\cdot)$ has continuous derivatives in $\overline{B}_{\frac{\beta}{2} \tau}(\mu)$, $\overline{B}_{\frac{\beta}{2} \tau}(m)$ and that it is $\frac{\beta}{2}\tau$-centrally symmetric.
\end{remark}

\subsection{Numerical implementation} \label{subsection_numerical_implementation}

In this section, we describe the algorithm for the numerical approximation of the clusters induced by the system \eqref{gradient_system_tau} based on the theoretical developments in Sections \ref{subsection_convergence_gradient_system} and \ref{subsection_modes_identification}. Since the sample $\tau$-approximation is not differentiable in $x$, we use a finite difference approximation that converges to the directional derivative. To this end, for $x \in \mathbb{R}^p$, $\tau >0$, $n \in \mathbb{N}$, $h>0$ and a unit vector $v \in \mathbb{R}^p$, we use the finite difference approximations of the directional derivatives of $f_{\tau}(\cdot)$ and $f_{n,\tau}(\cdot)$ along $v$ given by
\begin{equation*}
  \nabla_v^{h} f_\tau(x) = \frac{f_{\tau}(x+hv)-f_{\tau}(x)}{h} \quad \text{ and } \quad \nabla_v^{h} f_{\tau,n}(x) = \frac{f_{\tau,n}(x+hv)-f_{\tau,n}(x)}{h}.
\end{equation*}
Using the proof of Corollary \ref{corollary_uniform_consistency_of_sample_tau_approximation} and  the convergence of partial derivatives in Theorem \ref{theorem_first_and_second_partial_derivaties}, it follows that the finite difference approximations converge to the directional derivative. 

The next step towards identifying the modes, is finding a local maximum of a function.  To this end, we use the steepest ascent or gradient ascent idea; that is, starting from a point in the space, the next point is chosen in the direction given by the gradient of the function at that point. This procedure is repeated until convergence to a local maximum is achieved. In the context of mode hunting, this procedure is often combined with kernel density estimators to find the modes of the density underlying some given data points and to find the clusters associated to these modes \citep{Fukunaga-1975,Menardi-2016}. In the following, we propose a similar technique, which does not necessarily require the existence of the gradient and only considers data points as candidate points toward to which to move in the next step. \emph{This yields an efficient procedure, since  the local depth function has a higher potential to differentiate the local features of the distribution, especially in the inner part of the data cloud}. 

Turning to the development of the algorithm, starting from a point $x \in \mathbb{R}^p$, we search,  in a given neighborhood of $x$, the point $y$ that yields the largest directional derivative $\nabla_v^{h} f_{\tau,n}$ with $h=\norm{y-x}$ and $v = (y-x)/\norm{y-x}$. Since,
  \begin{equation*}
   \nabla_v^{h} f_{\tau}(x) \, \tau^p \, \sqrt{\Lambda_1} = \frac{\sqrt{LLD(x+hv,\tau)}-\sqrt{LLD(x,\tau)}}{h}
  \end{equation*}
  and
  \begin{equation*}
   \nabla_v^{h} f_{\tau,n}(x) \, \tau^p \, \sqrt{\Lambda_1} = \frac{\sqrt{LLD_n(x+hv,\tau)}-\sqrt{LLD_n(x,\tau)}}{h},
  \end{equation*}
the constant $\tau^p \, \sqrt{\Lambda_1}$ does not influence the choice of the point $y$ which maximizes both the finite differences $\nabla_v^{h} f_{\tau}(x)$ and $\nabla_v^{h} f_{\tau,n}(x)$. This allows one to ignore the constants in the specification of the algorithm. That is,  the finite difference approximation of the directional derivative of the square root of the local depth can be used instead, avoiding the computation of the constant $\Lambda_1$. 

We now show, in fact, that the constant $\tau^p \, \sqrt{\Lambda_1}$ also does not influence the clusters induced by the system \eqref{gradient_system_tau}. Since $\tau, \Lambda_1 > 0$, if, for $x \in \mathbb{R}^p$, $u_{x,\tau}: \mathbb{R} \rightarrow \mathbb{R}^p$ is a solution of the system \eqref{gradient_system_tau} with $u_{x,\tau}(0)=x$, then $\tilde{u}_{x,\tau}: \mathbb{R} \rightarrow \mathbb{R}^p$ given by $\tilde{u}_{x,\tau}(t) \coloneqq u_{x,\tau}(\tau^p \sqrt{\Lambda_1} t)$ also satisfies $\tilde{u}_{x,\tau}(0)=x$ and it is a solution of the system
\begin{equation} \label{tau_local_depth_gradient_system}
    \tilde{u}^{\prime}(t) = \nabla \left( \sqrt{LLD(\tilde{u}(t),\tau)} \right).
\end{equation}
Moreover, since $\lim_{t \to \infty} u_{x,\tau}(t) = \lim_{t \to \infty} \tilde{u}_{x,\tau}(t)$ for all $x \in \mathbb{R}^p$, the clusters induced by \eqref{gradient_system_tau} and \eqref{tau_local_depth_gradient_system} are the same. Hence, for $x,y \in \mathbb{R}^p$ with $y \neq x$ and $h=\norm{y-x} \leq r$ small enough, we consider the finite difference approximation of the directional derivatives of $\sqrt{LLD(x,\tau)}$ and $\sqrt{LLD_n(x,\tau)}$ along the direction $v=\frac{y-x}{\norm{y-x}}$ given by
\begin{equation} \label{finite difference_local_depth}
    d_{\tau}(x;y) \coloneqq \frac{\sqrt{LLD(y,\tau)} - \sqrt{LLD(x,\tau)}}{\norm{y-x}}
\end{equation}
and
\begin{equation} \label{finite difference_sample_local_depth}
    d_{\tau,n}(x;y) \coloneqq \frac{\sqrt{LLD_n(y,\tau)} - \sqrt{LLD_n(x,\tau)}}{\norm{y-x}}.
\end{equation}

Given $n$ data points $x_1, \dots, x_n$, the localization parameter $\tau$ used for the clustering procedure is chosen as the quantile of order $q$, for some $0 \leq q \leq 1$, of the empirical distribution of the ${n \choose 2}$ distances $\norm{x_i-x_j}$, $i>j$, $i,j \in \{1,2,\dots,n\}$. We now summarize the procedure for  computing the clusters in Algorithm \ref{algorithm_clusters}. It requires as input, the data points $\{ x_1, \dots, x_n \}$, the quantile $q$, and two additional parameters, $r$ and $s$. Additional points $\{ y_1, \dots, y_o \}$ may also be provided as input. Starting from any point $x \in \{ x_1, \dots, x_n \} \cup \{ y_1, \dots, y_o \}$, based on the finite difference \eqref{finite difference_sample_local_depth}, the algorithm moves to another data point $y \in \{ x_1, \dots, x_n \}$ (hence, except for the initial step, only data points are involved in \eqref{finite difference_sample_local_depth}). The parameter $r$ gives a bound on the norm $\norm{y-x}$ in \eqref{finite difference_sample_local_depth} in order to choose  only those points that are close to each other. The parameter $s$ is exploited, to avoid trivialities such that there are no points at distance from $x$ smaller than $r$ or the number of directional derivatives taken into account is too small to have a reliable estimate. If there are less than $s$ points at distance smaller than $r$ from $x$, then the closest $s$ points are considered instead. The data points toward which each point tends are returned as output. Typical choices for the parameters $q$, $s$ and $r$ are $q=0.05$, $s=30$ and $r=0.05$. The parameter $r$ does not significantly influence the output of Algorithm \ref{algorithm_clusters}. For the distributions and data considered in this manuscript, it is fixed to be $r=0.05$. On the other hand, the choice of $s$ plays an important role in Algorithm \ref{algorithm_clusters}, except when there is one true cluster in which case the choice of $s$ does not matter, provided it is not too small. However, if the number of true clusters is at least $2$, then $s$ cannot be too big or Algorithm \ref{algorithm_clusters} will return only one cluster. If the number of true clusters is large, say, $10$ or $20$, then smaller values of $s$ should be used. 

The quantity $s$ can also play the role of a smoothing parameter. If $q$ (and hence $\tau$) is small with a small sample size $n$, then the sample local depth can be noisy and have local peaks with a small basin of attraction that were not present in the original distribution. In this case, the choice of a larger $s$ helps to avoid these local maxima.
\begin{algorithm}
\DontPrintSemicolon
  
\KwInput{$\{ x_1, \dots, x_n \}$, $\{ y_1, \dots, y_o \}$ (optional), $q$, $s$, $r$ }
\KwOutput{Local maxima for input points: $\{ z_1, \dots, z_{n+o} \}$}
\nl Compute the quantile $\tau$ of order $q$ of all pairwise distances: $\norm{x_i-x_j}$, $i>j$, $i,j \in \{1,2,\dots,n\}$ \\
\nl Compute the local depth of $\{ x_1, \dots, x_n \}$ with the localization parameter $\tau$ \\
\nl Store $\{ x_1, \dots, x_n \}$, $\{ y_1, \dots, y_o \}$ in new variables \\
\For{$i=1$ to $n$}{
	$z^{*}_i \coloneqq x_i$
}
\For{$i=1$ to $o$}{
	$z^{*}_{i+n} \coloneqq y_i$
}
  
\nl For all points, compute the corresponding local maxima \\
\For{$i=1$ to $n+o$}{
   \Repeat{$LLD_n(z^{*}_i,\tau) > LLD_n(z_i,\tau)$}{
\nl   $z_i \coloneqq z^{*}_i$ \\
\nl   Store the data points (different from $z_i$) at distance from $z_i$ smaller than $r$ or the $s$ closest data points if they are less than $s$ in new variables $w_1, \dots, w_l$ ($l \geq s$) \\
\nl   $z^{*}_i \coloneqq \argmax_{j=1, \dots, l} d_{\tau,n}(z_i;w_j)$
   }
}

\caption{Clustering with local depth}
\label{algorithm_clusters}
\end{algorithm}
    
\section{Simulations and data analysis} \label{section_simulations_and_data_analysis}

\subsection{Illustrative examples} \label{subsection_illustrative_example}

We begin this section with a one-dimensional example showing the flexibility of the $\tau$-approximation for different values of $\tau$. As described in Section 2, for small values of  $\tau$, $f_{\tau}(\cdot)$ ``resembles'' the underlying density, while for larger $\tau$ it becomes unimodal, as DFs are decreasing from the median of the distribution. We take this univariate distribution to be a mixture of four normal distributions with means $-2$, $0$, $3$, $4$, standard deviations $0.5$, $0.8$, $0.5$, $0.2$ and weights $0.25$, $0.5$, $0.15$ and $0.1$, respectively. The resulting density is quadrimodal and is depicted in Figure \ref{mixture_normals_plot2} along with its sample $\tau$-approximation for $\tau=0.5,1,2,4$. As can be seen from the Figure \ref{mixture_normals_plot2}, for $\tau=0.5$ the approximation has a similar shape to the density with approximately the same number of modes. For $\tau=1$,  the clusters corresponding to the modes at $x=3$ and $x=4$ merge yielding only three clusters. As we increase $\tau$ to  $2$, we notice that one can still identify two clusters, while, for $\tau=4$, the $\tau$-approximation has a unimodal shape.

\begin{figure}[!htbp]
	\includegraphics[width= \linewidth]{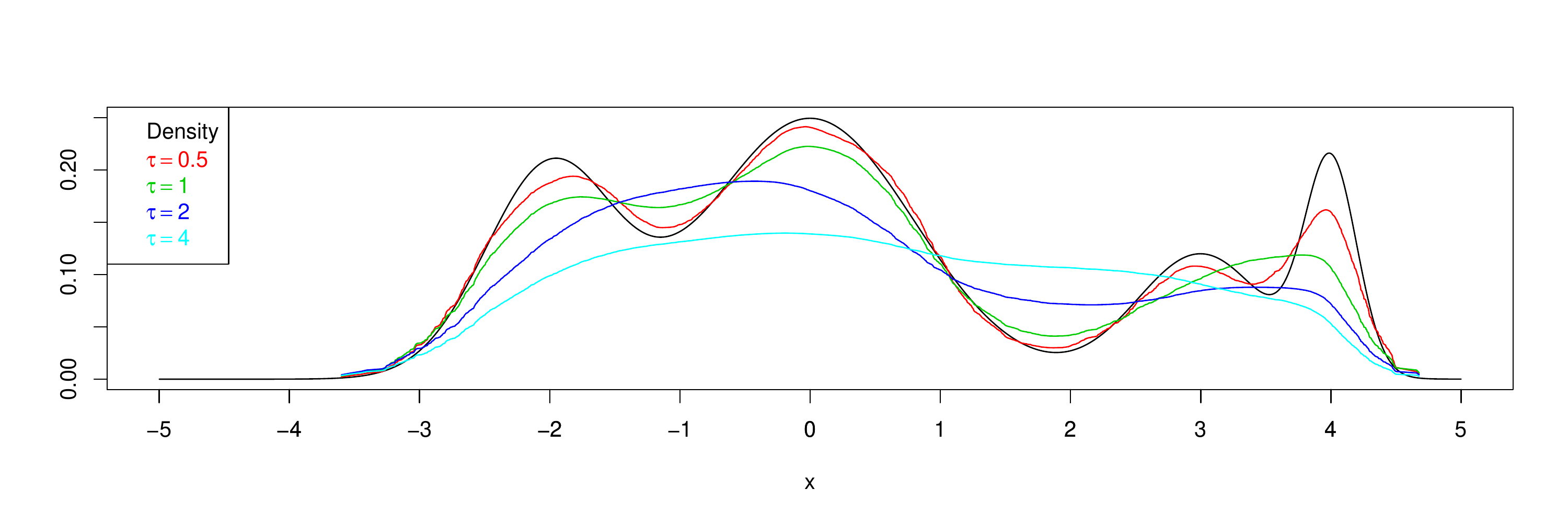}
	\caption{ In black the quadrimodal mixture density and in red, green, blue and cyan its sample $\tau$-approximation $f_{\tau,n}(\cdot)$ for $\tau=0.5,1,2,4$, respectively, and $n=6000$. }
 	\label{mixture_normals_plot2}
\end{figure}

Turning to bivariate examples studied in the literature (see \citet{Chacon-2015}),  we consider mixtures of bivariate normal distributions with the following characteristics: (i) two-mixture  with equal weights (Bimodal) and identity covariance matrix and (ii) the mixtures investigated in \citet{Wand-1993} and \citet{Chacon-2009} referred to as (H) Bimodal IV, (K) Trimodal III and \#10 Fountain (see Figure \ref{figure_true_clusters_ldc_kde}, first row, (K) Trimodal III is in Appendix \ref{sm:subsection_illustrative_examples}). Their analytical expression and the associated \emph{true clusters} are given in  Appendix \ref{sm:subsection_true_clusters}. We apply our algorithm to analyze these models and identify clusters; these results are displayed in Figure \ref{figure_true_clusters_ldc_kde} (second row). A comparison of our results with the clusters obtained using the kernel density estimator (obtained using the function \texttt{kms} in the R-package \emph{ks} \citep{Duong-2018}) are provided in Figure \ref{figure_true_clusters_ldc_kde} (third row). By a visual inspection of Figure \ref{figure_true_clusters_ldc_kde}, it is clear that a better cluster estimation is performed by LLD than by KDE.  A more detailed analysis of these and other distributions under extreme localization is provided in Appendix \ref{sm:subsection_illustrative_examples}.

\begin{figure}
\centering
\begin{minipage}[c]{0.1\textwidth}
   \caption*{ True \\ clusters}
\end{minipage}
\begin{minipage}[c]{0.88\textwidth}
\begin{subfigure}{0.32 \linewidth}
\centering
\caption*{Bimodal}
\includegraphics[width=\linewidth]{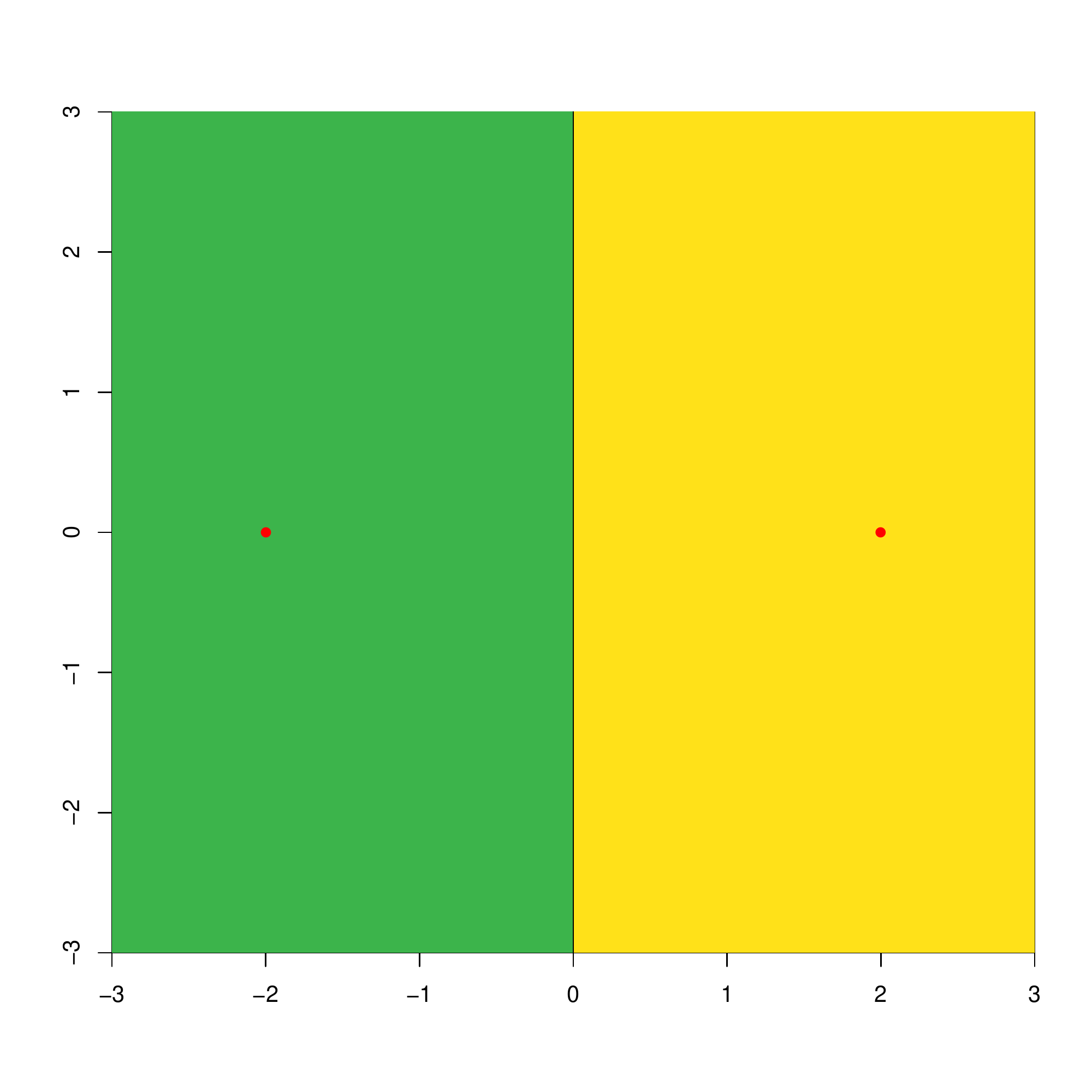}
\end{subfigure}
\begin{subfigure}{0.32 \linewidth}
\centering
\caption*{(H) Bimodal IV}
\includegraphics[width=\linewidth]{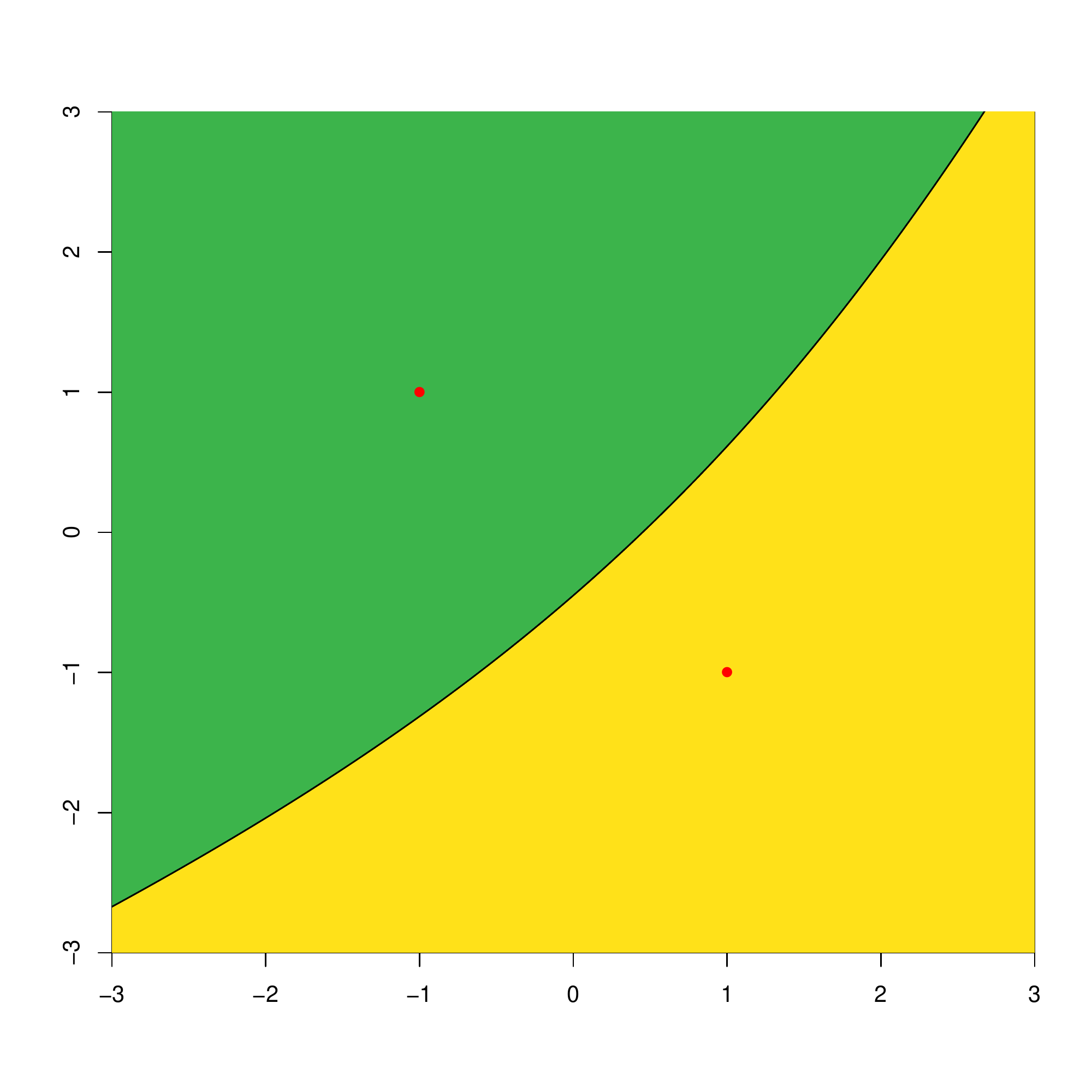}
\end{subfigure}
\centering
\begin{subfigure}{0.32 \linewidth}
\caption*{\#10 Fountain}
\includegraphics[width=\linewidth]{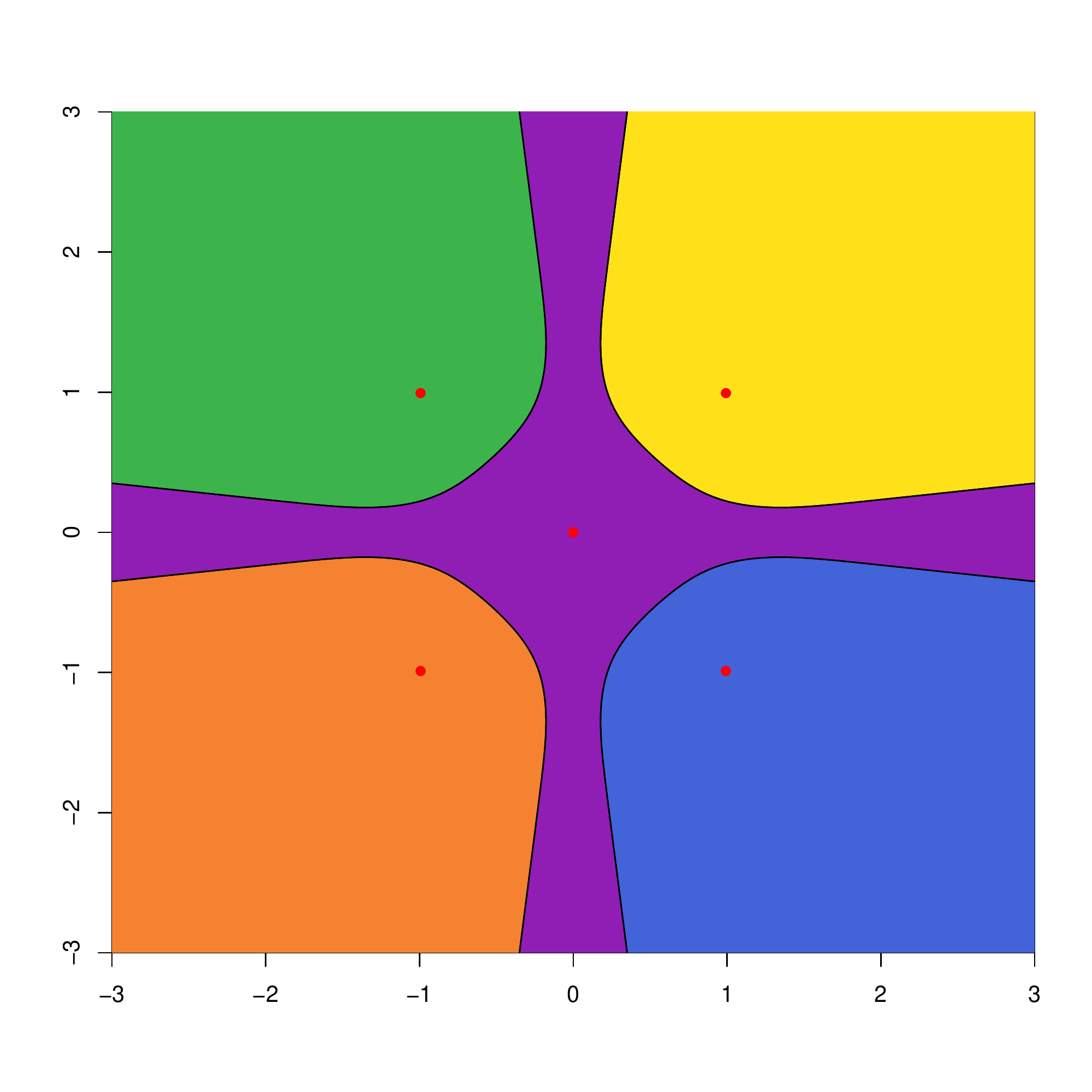}
\end{subfigure}
\end{minipage}
\begin{minipage}[c]{0.1\textwidth}
   \caption*{LLD \\ estimated \\ clusters}
\end{minipage}
\begin{minipage}[c]{0.88\textwidth}
\centering
\begin{subfigure}{0.32 \linewidth}
\centering
\includegraphics[width=\linewidth]{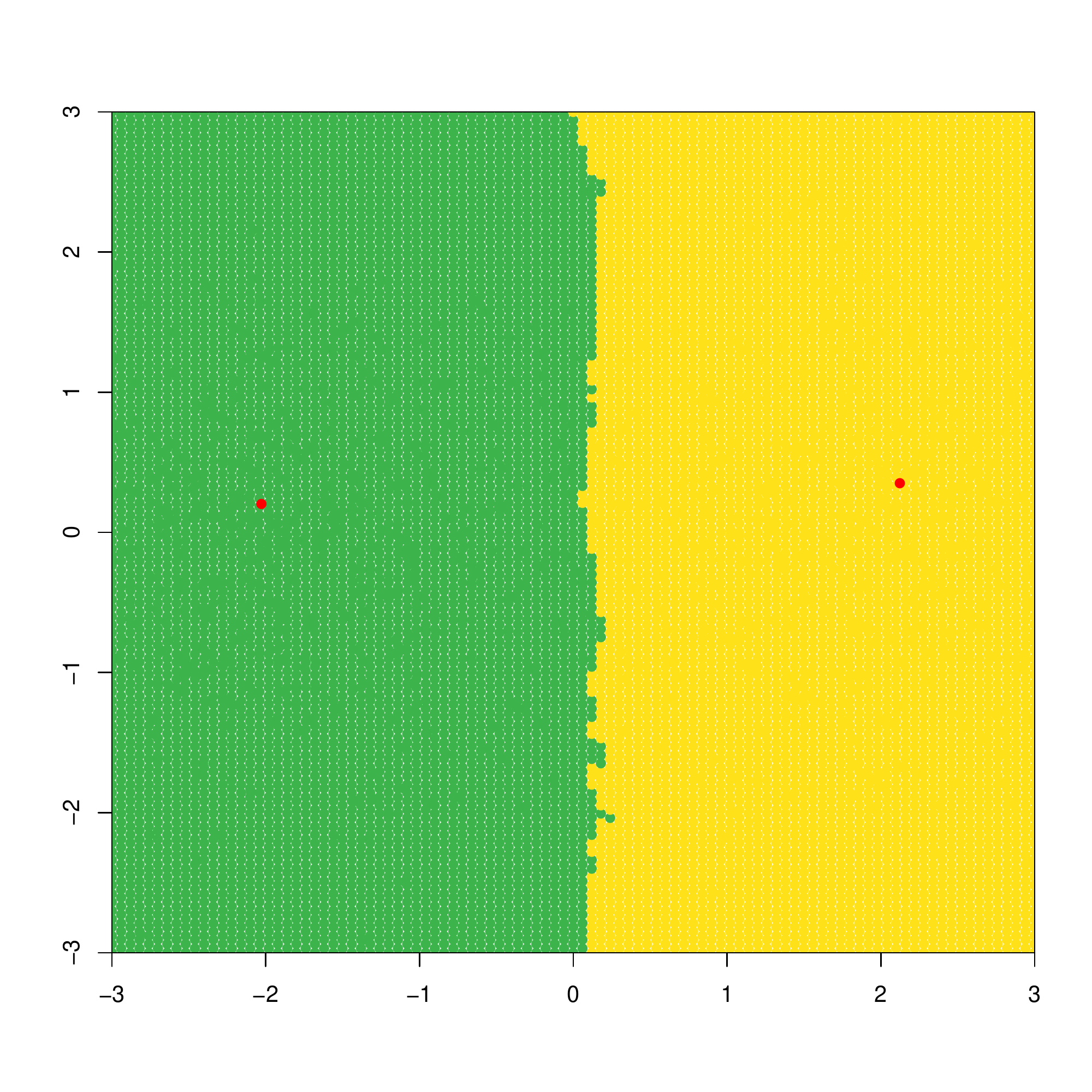}
\end{subfigure}
\begin{subfigure}{0.32 \linewidth}
\centering
\includegraphics[width=\linewidth]{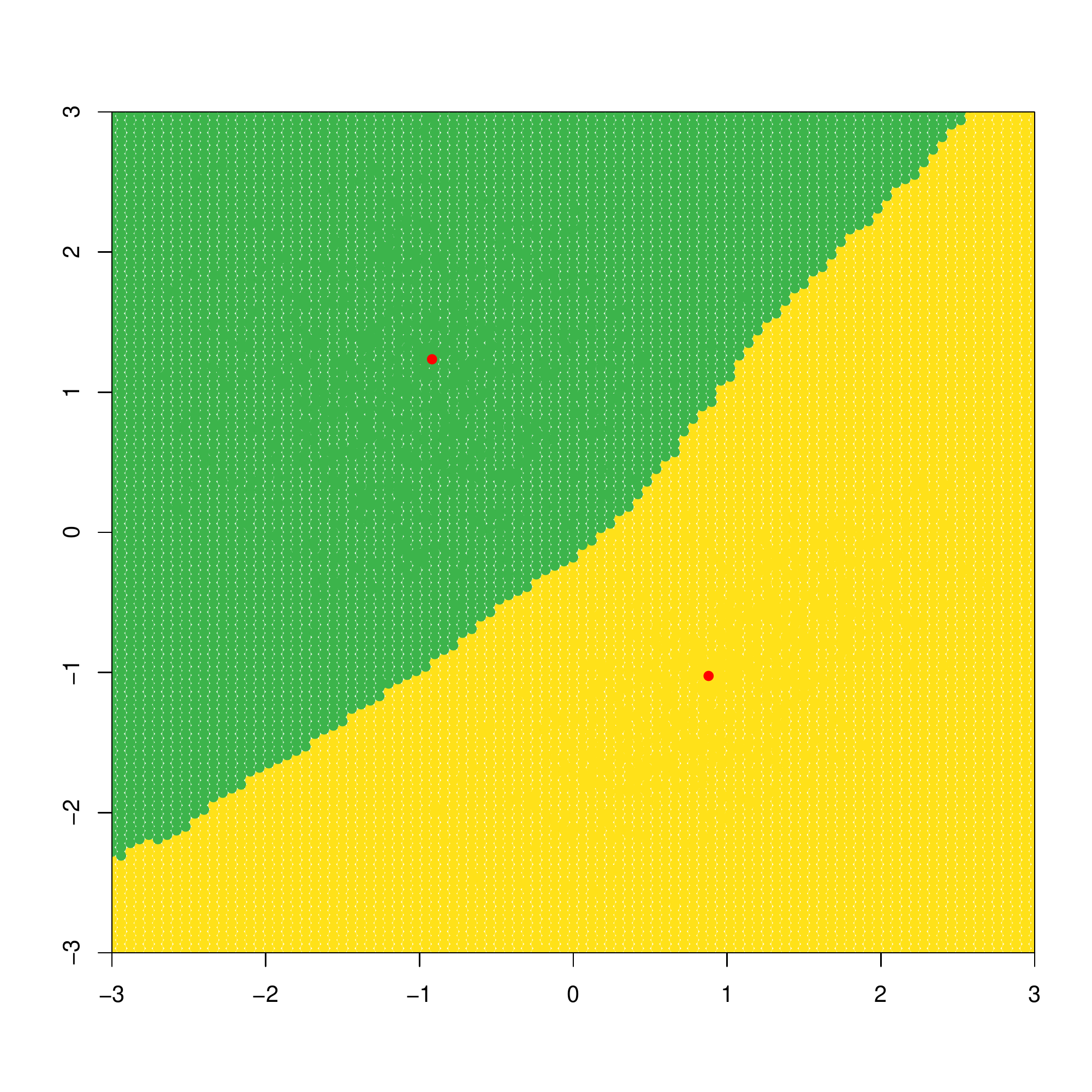}
\end{subfigure}
\begin{subfigure}{0.32 \linewidth}
\centering
\includegraphics[width=\linewidth]{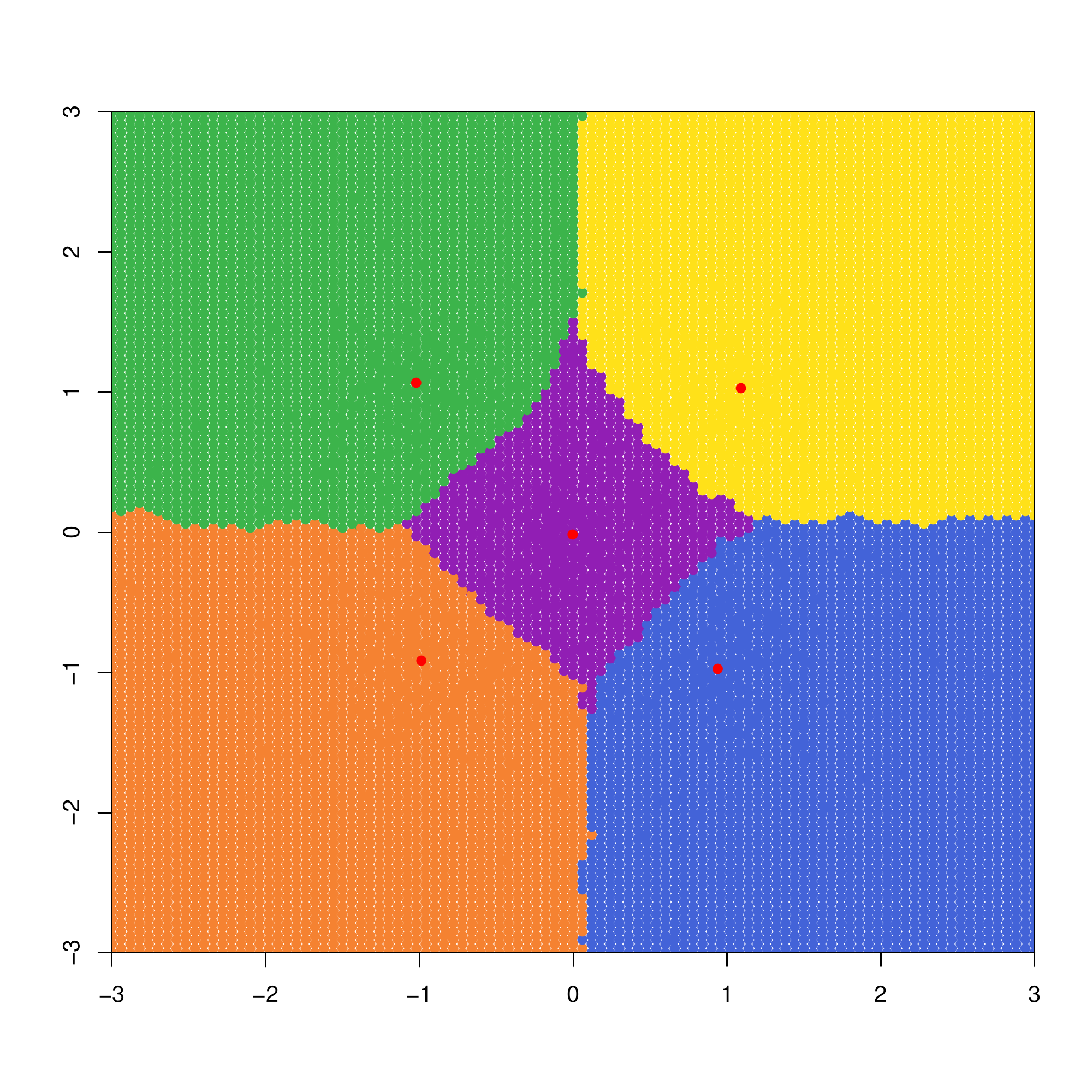}
\end{subfigure}
\end{minipage}
\begin{minipage}[c]{0.1\textwidth}
   \caption*{KDE \\ estimated \\ clusters}
\end{minipage}
\begin{minipage}[c]{0.88\textwidth}
\centering
\begin{subfigure}{0.32 \linewidth}
\centering
\includegraphics[width=\linewidth]{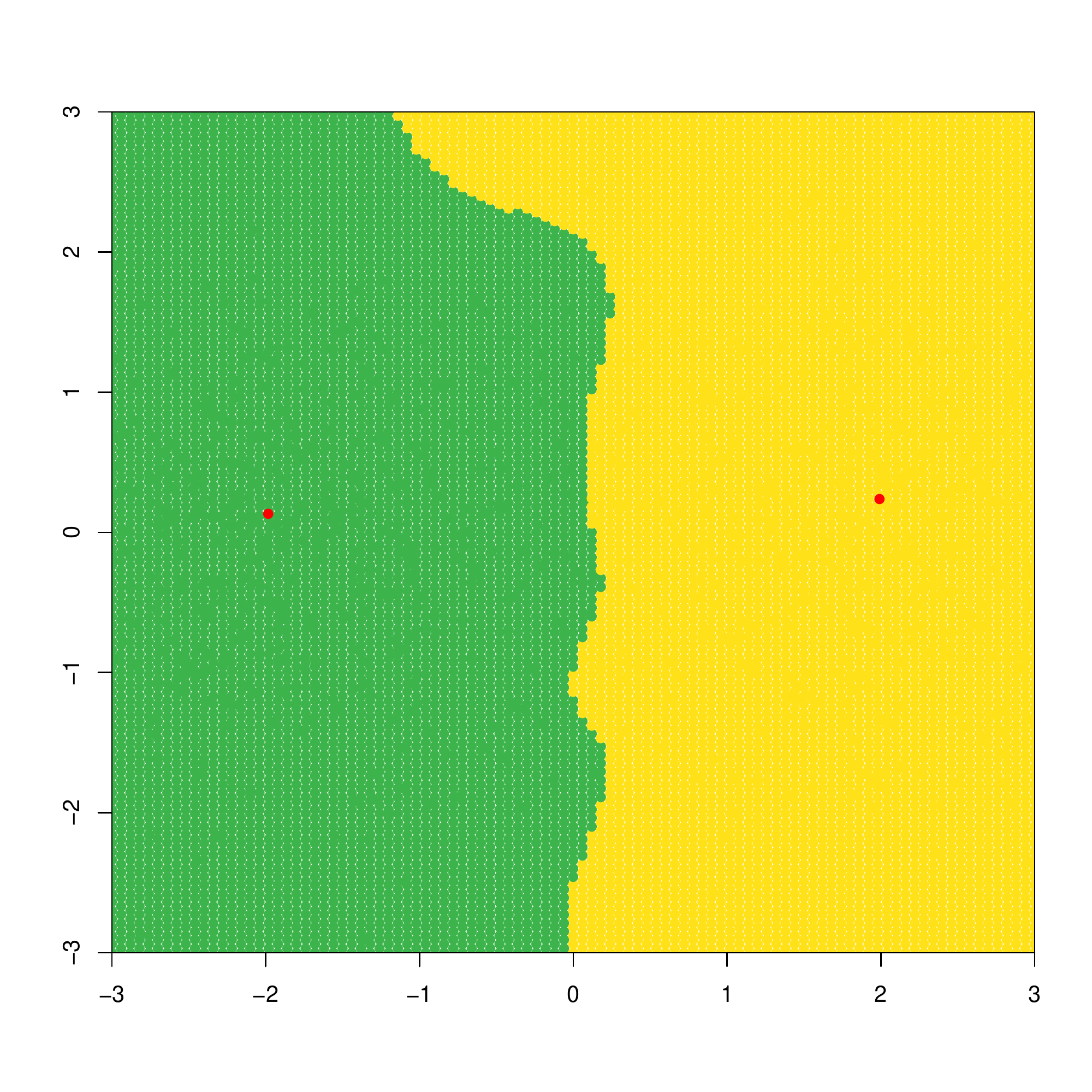}
\end{subfigure}
\begin{subfigure}{0.32 \linewidth}
\centering
\includegraphics[width=\linewidth]{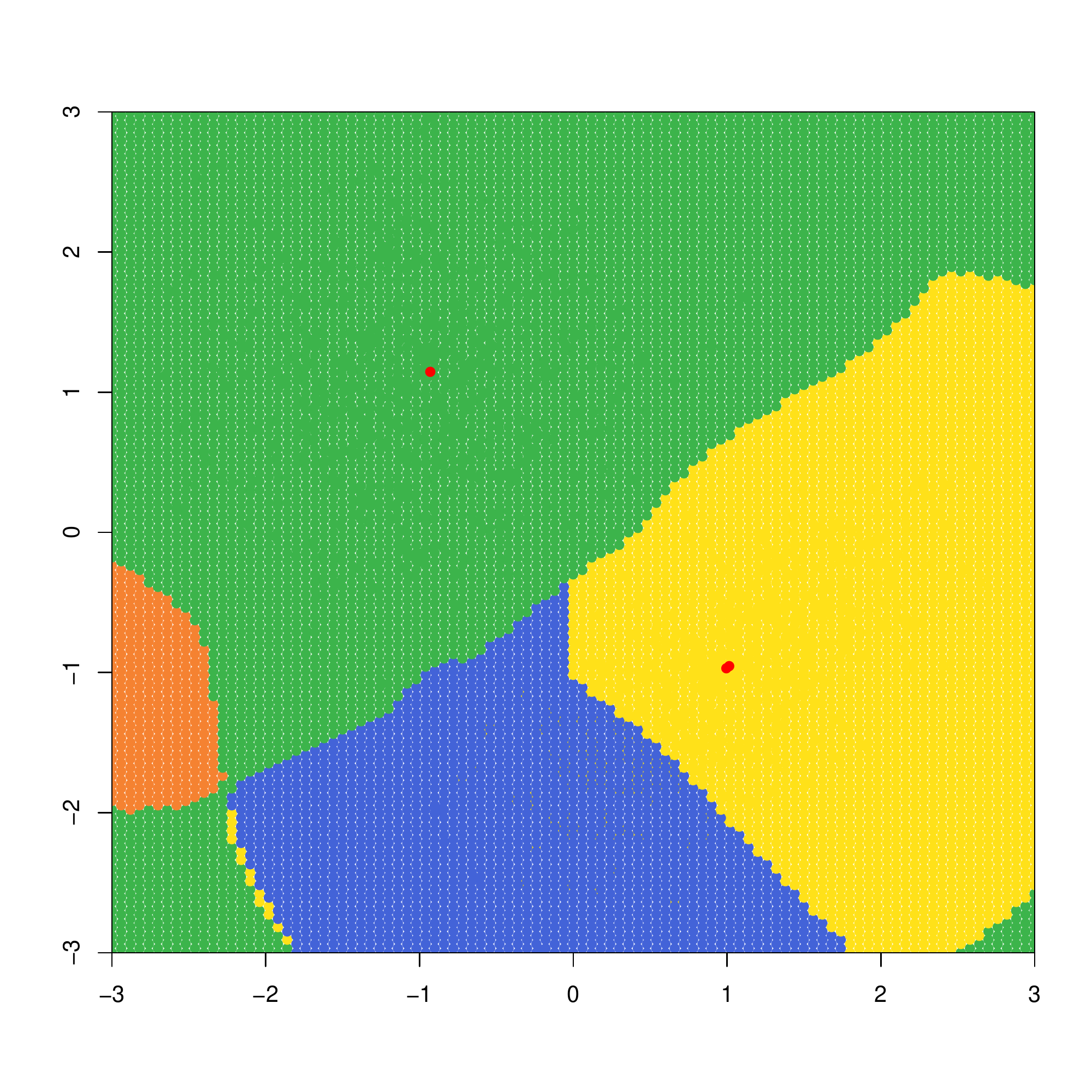}
\end{subfigure}
\begin{subfigure}{0.32 \linewidth}
\centering
\includegraphics[width=\linewidth]{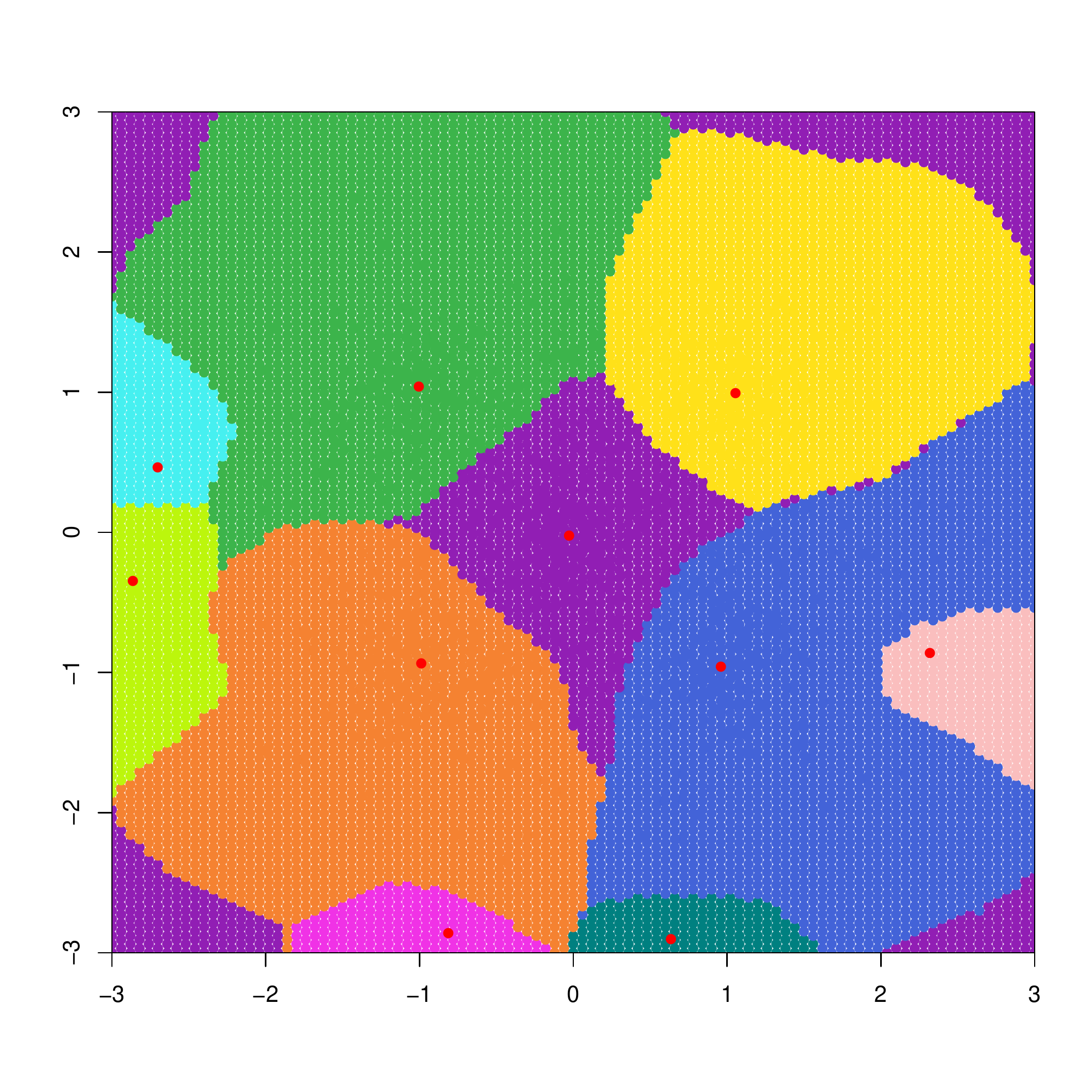}
\end{subfigure}
\end{minipage}

\caption{Clusters associated with the Bimodal (left), (H) Bimocal IV (middle) and \#10 Fountain (right) densities. True clusters (first row). Local depth clustering based on $n=1000$ samples from these densities and parameters $q=0.05$, $s=50$ and $r=0.05$ (second row). Kernel density estimator clustering (third row). The true modes (first row) and the predicted modes (second and third rows) are plotted in red.}
\label{figure_true_clusters_ldc_kde}
\end{figure}

\subsection{Numerical experiments} \label{subsection_numerical_experiments}
In this subsection, we provide simulation results of our method for identification of clusters. We evaluate the performance in three different ways:
(i) true number of  clusters identified by the algorithm, (ii) empirical Hausdorff distance between the ``true'' cluster and the estimated cluster, and (iii) empirical probability distance (see \citet{Chacon-2015}, for instance).  We recall that the symmetric difference between two clusters $C,D \subset \mathbb{R}^p$ is $C \Delta D = \left( \left( \mathbb{R}^p \setminus C \right) \cap D \right) \cup \left( C \cap \left( \mathbb{R}^p \setminus D \right) \right)$. Let $X_1,\dots, X_n$ be i.i.d.\ samples from some probability distribution $P$. The empirical probability distance between the clusters $\mathcal{C}=\{ C_1, \dots, C_l \}$ and $\mathcal{D}=\{ D_1, \dots, D_s \}$ with $l < s$ is given by
\begin{equation*}
\hat{d}_{P,\eta}(\mathcal{C},\mathcal{D}) = \frac{1}{2n} \min_{\pi \in \mathcal{P}_s} \left( \sum_{i=1}^l \sum_{j=1}^n \mathbf{I}(X_j \in C_i \Delta D_{\pi(i)}) + \eta \sum_{i=l+1}^s \sum_{j=1}^n \mathbf{I}(X_j \in D_{\pi(i)}) \right),
\end{equation*}
where $\mathcal{P}_s$ is the set of all permutations of $\{ 1, \dots, l \}$ and $\eta \geq 0$ is a penalization coefficient for clusters that do not match with any other. If $l=s$ the second term in the above expression is zero. In our numerical experiments, we choose $\eta=1$. Additional results for other values of $\eta$ are provided in Appendix \ref{sm:subsection_numerical_experiments}.

The empirical Hausdorff distance $\hat{d}_{H}(\mathcal{C},\mathcal{D}) $ is given by
\begin{equation*} \label{empirical_hausdorff_distance}
 \frac{1}{n} \max \left( \max_{i \in \{1,\dots,l\} } \min_{j \in \{1,\dots,s \} } \sum_{k=1}^n \mathbf{I}(X_k \in C_i \Delta D_j) , \max_{j \in \{1,\dots,s \} } \min_{i \in \{1,\dots,l\} } \sum_{k=1}^n \mathbf{I}(X_k \in C_i \Delta D_j) \right).
\end{equation*}

In applications, $\mathcal{C}$ is taken to be the set of true clusters while $\mathcal{D}$ is the estimated cluster, produced by the algorithm. If the estimated clusters coincide with the true clusters, then both these distances, {\it{viz}}.\ the \emph{clustering errors}, are zero. Thus, small values of these distances suggest a good performance. In this manuscript, as explained before, we consider the following distributions commonly used in the literature: Bimodal, (H) Bimodal IV, (K) Trimodal III and \#10 Fountain. To test the performance of our methodology in high dimensions, we also consider a bimodal and a quadrimodal density in dimension five. We refer to these distributions as Mult.\ Bimodal and Mult.\ Quadrimodal. Their analytic expressions are given in Appendix \ref{sm:subsection_true_clusters}. Our simulation results are based on a sample size of 1000 with 100 numerical experiments. A key ingredient for the procedure is the choice of $\tau$. We choose $\tau$ so that the corresponding quantiles $q$ are given by $0.01$, $0.05$ and $0.1$ (see Algorithm \ref{algorithm_clusters}). We compare our results with clustering based on the Kernel density estimator (KDE), on LSD, and hierarchical clustering (Hclust). The hierarchical clustering requires pre-specification of the number of clusters while the other methods do not, and it is reported here since it is one of the widely used methods for clustering. Thus, we compute it making use of the true number of clusters, which implies that the obtained results are not comparable with those of the other methodologies. For more details about the numerical implementation and the quantiles for LSD we refer to Appendix \ref{sm:section_local_simplicial_depth}. Further simulation results for these and other distributions are provided in Appendix \ref{sm:subsection_numerical_experiments}.

Table \ref{table_hausdorff_distance_and_probability_distance} provides clustering errors based on the Hausdorff distance and the probability distance. The best results are highlighted in bold. From these results it is clear that clustering errors based on the proposed LLD are in general much smaller than those based on the KDE, Hclust, and LSD. The improvement is substantial especially in the high-dimensional case as can be seen from the columns Mult.\ Bimodal and Mult.\ Quadrimodal. Table \ref{table_number_times} provides a comparison of the number of times the correct number of clusters is detected. The number of times the procedure identifies a lower number of clusters (on the left) and a higher number of clusters (on the right) is also provided. Again we notice that the proposed methods outperforms the competitors. It is possible to improve the performance of LSD for distributions in dimension 5, by choosing smaller values of $q$, as described in Subsection \ref{subsection_data_analysis}.

\begin{ThreePartTable}
\begin{TableNotes}[para,raggedright]
\item[*] \label{notetrueclusters} The true number of clusters is given in input.
\item[1] \label{note:q0.1:s30} $q=0.1$, $s=30$.
\item[2] \label{note:q0.01:s30} $q=0.01$, $s=30$.
\item[3] \label{note:q0.1:s50} $q=0.1$, $s=50$.
\item[4] \label{note:q0.05:s30} $q=0.05$, $s=30$.
\end{TableNotes}
\begin{tabularx}{10cm}{ | p{2.6cm} | p{3.5cm} | p{3.5cm} | p{3.5cm} | }
\hline
\multicolumn{4}{|c|}{\textbf{Clustering errors (Hausdorff distance)}} \\
\hline
& (H) Bimodal IV & (K) Trimodal III & \#10 Fountain \\
\hline
KDE & 0.19 (0.21) & 0.10 (0.11) & 0.12 (0.06) \\
\hline
LLD \tnotex{note:q0.1:s30} & {\bf{0.05 (0.10)}} & {\bf{0.01 (0.15)}} & {\bf{0.06 (0.01)}} \\
\hline
LSD \tnotex{note:q0.01:s30} & {\bf{0.05 (0.11)}} & 0.10 (0.15) & {\bf{0.06 (0.01)}} \\
\hline
Hclust \tnotex{notetrueclusters} & {\bf{0.05 (0.09)}} & 0.15 (0.09) & 0.29 (0.05) \\
\hline
& Bimodal & Mult.\ Bimodal & Mult.\ Quadrimodal \\
\hline
KDE & 0.09 (0.15) & 0.20 (0.21) & 0.09 (0.09) \\
\hline
LLD \tnotex{note:q0.1:s50} & 0.01 (0.03) & {\bf{0.01 (0.04)}} & {\bf{0.02 (0.01)}} \\
\hline
LSD \tnotex{note:q0.05:s30} & {\bf{0.00 (0.00)}} & 0.23 (0.18) & 0.38 (0.18) \\
\hline
Hclust \tnotex{notetrueclusters} & 0.06 (0.05) &  0.05 (0.03) & 0.07 (0.03) \\
\hline
\multicolumn{4}{|c|}{\textbf{Clustering errors (distance in probability)}} \\
\hline
& (H) Bimodal IV & (K) Trimodal III & \#10 Fountain \\
\hline
KDE & 0.37 (0.41) & 0.08 (0.08) & 0.42 (0.36) \\
\hline
LLD \tnotex{note:q0.1:s30} & 0.13 (0.28) & {\bf{0.06 (0.07)}} & {\bf{0.06 (0.01) }} \\
\hline
LSD \tnotex{note:q0.01:s30} & 0.12 (0.27) & 0.07 (0.09) & {\bf{0.06 (0.01)}} \\
\hline
Hclust \tnotex{notetrueclusters} & {\bf{0.05 (0.09)}} & 0.16 (0.09) & 0.35 (0.07) \\
\hline
& Bimodal & Mult.\ Bimodal & Mult.\ Quadrimodal \\
\hline
KDE & 0.05 (0.09) & 0.08 (0.12) & 0.31 (0.36) \\
\hline
LLD \tnotex{note:q0.1:s50} & 0.01 (0.02) & {\bf{0.01 (0.01) }} & {\bf{0.03 (0.01) }} \\
\hline
LSD \tnotex{note:q0.05:s30} & {\bf{0.00 (0.00) }} & 0.20 (0.17) & 0.45 (0.17) \\
\hline
Hclust \tnotex{notetrueclusters} & 0.06 (0.05) & 0.05 (0.03) & 0.10 (0.04) \\
\hline
\insertTableNotes \\
\caption{Mean of the clustering errors based on the Hausdorff distance and the distance in probability for the densities (H) Bimodal IV, (K) Trimodal III, \#10 Fountain, Bimodal, Mult.\ Bimodal and Mult.\ Quadrimodal. In parentheses the standard deviation. The true number of clusters is specified as input for the hierarchical clustering algorithm.} \label{table_hausdorff_distance_and_probability_distance}
\end{tabularx}
\end{ThreePartTable}

\begin{ThreePartTable}
\begin{tabularx}{10cm}{ | p{2.6cm} | p{3.5cm} | p{3.5cm} | p{3.5cm} | }
\hline
\multicolumn{4}{|c|}{\textbf{Number of times the true clusters are detected correctly}} \\
\hline
& (H) Bimodal IV & (K) Trimodal III  & \#10 Fountain \\
\hline
KDE & (0) 54 (46) & (5) 60 (35) & (0) 47 (53) \\
\hline
LLD \tnotex{note:q0.1:s30} & (0) 83 (17) & {\bf{(14) 79 (7)}} & {\bf{(0) 100 (0)}} \\
\hline
LSD \tnotex{note:q0.01:s30} & {\bf{(0) 85 (15)}} & (13) 75 (12) & {\bf{(0) 100 (0)}} \\
\hline
& Bimodal & Mult.\ Bimodal & Mult.\ Quadrimodal \\
\hline
KDE & (0) 76 (24) & (0) 56 (44) & (0) 63 (37) \\
\hline
LLD \tnotex{note:q0.1:s50} & (0) 99 (1) & {\bf{(0) 99 (1)}} & {\bf{(0) 100 (0)}} \\
\hline
LSD \tnotex{note:q0.05:s30} & {\bf{(0) 100 (0)}} & (12) 63 (25) & (77) 18 (5) \\
\hline
\caption{Number of times that the procedure identifies the true number of clusters for the densities (H) Bimodal IV, (K) Trimodal III , \#10 Fountain, Bimodal, Mult.\ Bimodal and Mult.\ Quadrimodal. In parentheses the number of times the procedure identifies a lower number of clusters (on the left) and a higher number of clusters (on the right). } \label{table_number_times}
\end{tabularx}
\end{ThreePartTable}

\subsection{Data analysis} \label{subsection_data_analysis}

In this subsection, we evaluate the performance of our methodology on two datasets taken from the UCI machine learning repository (\url{http://archive.ics.uci.edu/ml/}), namely, the Iris dataset and the Seeds dataset. As is well-known, the Iris dataset consists of $n=150$ observations from three classes (Iris Setosa, Iris Versicolour, and Iris Virginica) with four measurements each (sepal length, sepal width, petal length, and petal width). We compare our results to those based on KDE and Hclust. Our algorithm, based on both lens and simplicial depth, correctly identifies all three clusters (see Table \ref{table_iris_and_seeds_data}); furthermore, the Hausdorff distance and probability distance from our algorithm are smaller than those of the competitors.

Seeds dataset consists of $n=210$ observations concerning three varieties of wheat; namely, Kama, Rosa and Canadian. High quality visualization of the internal kernel structure was detected using a soft X-ray technique and seven geometric parameters of wheat kernels were recorded. They are area, perimeter, compactness, length of kernel,  width of kernel, asymmetry coefficient, and length of kernel groove. All of these geometric parameters were continuous and real-valued. Table \ref{table_iris_and_seeds_data} contains the results of our analysis. The best results are highlighted in bold and correspond to LLD. We notice that both of our methods, LLD and LSD, correctly identify the true number of clusters. 

It is worth mentioning here that Hclust was given as input the true number of clusters, three, as required by this methodology. However, the Haudorff distance and probability distance of our proposed methods are smaller than those of Hclust. KDE, in both the examples, overestimates the true number of clusters.

\begin{ThreePartTable}
\begin{TableNotes}[para,raggedright]
\item[5] \label{note:q0.0001:s20} $q=10^{-4}$, $s=20$.
\item[6] \label{note:q0.00001:s20} $q=10^{-5}$, $s=20$.
\end{TableNotes}
\begin{tabularx}{10cm}{ | p{2.6cm} | p{3.5cm} | p{3.5cm} | p{3.5cm} | }
\hline
\multicolumn{4}{|c|}{\textbf{Clustering errors for Iris data}} \\
\hline
& Number of clusters & Distance in prob.\ & Hausdorff distance \\
\hline
KDE & 7 & 0.37 & 0.30 \\
\hline
LLD \tnotex{note:q0.05:s30} & {\bf{3}} & {\bf{0.10}} & {\bf{0.10}} \\
\hline
LSD \tnotex{note:q0.0001:s20} & {\bf{3}} & {\bf{0.10}} & {\bf{0.10}}  \\
\hline
Hclust \tnotex{notetrueclusters} &  & 0.16 & 0.16 \\
\hline
\multicolumn{4}{|c|}{\textbf{Clustering errors for Seeds data}} \\
\hline
& Number of clusters & Distance in prob.\ & Hausdorff distance \\
\hline
KDE & 23 & 0.70 & 0.33 \\
\hline
LLD \tnotex{note:q0.05:s30} & {\bf{3}} & {\bf{0.10}} & {\bf{0.10}} \\
\hline
LSD \tnotex{note:q0.00001:s20} & {\bf{3}} & 0.17 & 0.17 \\
\hline
Hclust \tnotex{notetrueclusters}  &  & 0.20 & 0.20 \\
\hline
\insertTableNotes \\
\caption{Mean of the clustering errors based on the Hausdorff distance and distance in probability for the Iris and Seeds data. The true number of clusters is specified as input for the hierarchical clustering algorithm.} \label{table_iris_and_seeds_data}
\end{tabularx}
\end{ThreePartTable}

\section{Concluding Remarks} \label{section_concluding_remarks}
We have studied the properties of local depth in the multivariate setting. Our results include uniform convergence and uniform central limit theorem for the sample local depth. Based on the local depth, we proposed a population and sample approximation of the density that converges uniformly and uniformly almost surely to the true density. We made use of it in a modal clustering approach (via a gradient system) where the density is replaced by the population approximation. Convergence results show that the approximated approach provides, in the limit, the same clusters as those given by the true density. In particular, we have shown that, for symmetric distributions, our approximated approach correctly detects the true modes. Finally, we proposed an algorithm for the numerical computation of the clusters at sample level. 

Yet another work on local depth is that of \citet{Paindaveine-2013} and is based on the concept of symmetric neighborhoods. It is legitimate to wonder if the methods proposed in the current paper go over to this alternative notion. While we believe such an extension is feasible, we have not pursued all the technical details required for such analysis and is left for future research. Additionally, it will be interesting to study a smoothed version of the sample approximation and replace the gradient system by the stochastic differential equation based on this approximation, and consequently study the convergence of the corresponding solutions.

\medspace

\appendix
\renewcommand{\thesection}{\Alph{section}}

\section{Proofs} \label{section_proofs}

We begin this section summarizing relevant properties of $Z_{\tau}(x)=\{ (y_1, y_2) \in \mathbb{R}^p \times \mathbb{R}^p \, : \max_{ i=1,2} \norm{x-y_i} \leq \norm{y_1 - y_2} \leq \tau \}$ that are required in the proof of several main results on the paper. In the proofs  we will use $P^{\otimes k}$ for the $k$-fold product measure.

\begin{lemma} \label{lemma_interior_exterior_of_z_tau_x}
For $x \in \mathbb{R}^p$ and $\tau \in [0,\infty]$, $Z_{\tau}(x)$ is a closed set and its interior, exterior and boundary are given by
\begin{equation*}
\begin{split}
    \mathring{Z}_{\tau}(x) &= \{ (y_1, y_2) \in \mathbb{R}^p \times \mathbb{R}^p \, : \, \max_{i=1,2} \norm{x-y_i} < \norm{y_1-y_2} < \tau \} \\
    Z^{e}_{\tau}(x) &= \{ (y_1, y_2) \in \mathbb{R}^p \times \mathbb{R}^p \, : \, \norm{y_1-y_2} > \tau \} \\
    & \cup \{ (y_1, y_2) \in \mathbb{R}^p \times \mathbb{R}^p \, : \, \max_{i=1,2} \norm{x-y_i} > \norm{y_1-y_2} \le \tau \} \\
    \partial Z_{\tau}(x) &= \{ (y_1, y_2) \in \mathbb{R}^p \times \mathbb{R}^p \, : \, \max_{i=1,2} \norm{x-y_i} \le \norm{y_1-y_2} = \tau \} \\
    & \cup \{ (y_1, y_2) \in \mathbb{R}^p \times \mathbb{R}^p \, : \, \max_{i=1,2} \norm{x-y_i} = \norm{y_1-y_2} < \tau \}.
\end{split}
\end{equation*}
\end{lemma}

\begin{proof}
For the closeness, let $(y_{1,n}, y_{2,n}) \in Z_{\tau}(x)$ and suppose that $(y_{1,n}, y_{2,n})  \xrightarrow[ n \rightarrow \infty ]{} (y_{1}, y_{2})$. We will show that $(y_{1}, y_{2}) \in Z_{\tau}(x)$. Since $\max_{ i=1,2} \norm{x-y_{i,n}}$ $\leq \norm{y_{1,n} - y_{2,n}} \leq \tau$, for all $n \in \mathbb{N}$, it follows that $\lim_{n \to \infty} \max_{ i=1,2} \norm{x-y_{i,n}} \leq \lim_{n \to \infty} \norm{y_{1,n} - y_{2,n}} \leq \tau$. Using the continuity of the norm we get
$\max_{ i=1,2} \norm{x-y_i}$ $\leq \norm{y_1 - y_2} \leq \tau$, i.e.\ $(y_1,y_2) \in Z_{\tau}(x)$. Next, notice that
\begin{equation*}
\begin{split}
    Z^{e}_{\tau}(x) &= ( \mathbb{R}^p \times \mathbb{R}^p ) \setminus \overline{Z_{\tau}(x)} = ( \mathbb{R}^p \times \mathbb{R}^p ) \setminus Z_{\tau}(x) \\
    &= \{ (y_1, y_2) \in \mathbb{R}^p \times \mathbb{R}^p \, : \, \norm{y_1-y_2} > \tau \, \lor \, \max_{i=1,2} \norm{x-y_i} > \norm{y_1-y_2} \le \tau \} \\
    &= \{ (y_1, y_2) \in \mathbb{R}^p \times \mathbb{R}^p \, : \, \norm{y_1-y_2} > \tau \} \\
    & \cup \{ (y_1, y_2) \in \mathbb{R}^p \times \mathbb{R}^p \, : \, \max_{i=1,2} \norm{x-y_i} > \norm{y_1-y_2} \le \tau \}.
\end{split}
\end{equation*}
Now, by the definition of the boundary we have that
\begin{equation*}
\partial Z_{\tau}(x) = \overline{Z_{\tau}(x)} \setminus \overline{ \left( ( \mathbb{R}^p \times \mathbb{R}^p ) \setminus Z_{\tau}(x) \right) } = Z_{\tau}(x) \cap \overline{Z^{e}_{\tau}(x)}
\end{equation*}
where
\begin{equation*}
\begin{split}
\overline{Z^{e}_{\tau}(x)} &= \{ (y_1, y_2) \in \mathbb{R}^p \times \mathbb{R}^p \, : \, \norm{y_1-y_2} \geq \tau \} \\
    & \cup \{ (y_1, y_2) \in \mathbb{R}^p \times \mathbb{R}^p \, : \, \max_{i=1,2} \norm{x-y_i} \geq \norm{y_1-y_2} \le \tau \}.
\end{split}
\end{equation*}
Therefore,
\begin{equation*}
\begin{split}
    \partial Z_{\tau}(x)
    &= \{ (y_1, y_2) \in \mathbb{R}^p \times \mathbb{R}^p \, : \, \max_{i=1,2} \norm{x-y_i} \le \norm{y_1-y_2} = \tau \} \\
    & \cup \{ (y_1, y_2) \in \mathbb{R}^p \times \mathbb{R}^p \, : \, \max_{i=1,2} \norm{x-y_i} = \norm{y_1-y_2} \leq \tau \} \\
    &= \{ (y_1, y_2) \in \mathbb{R}^p \times \mathbb{R}^p \, : \, \max_{i=1,2} \norm{x-y_i} \le \norm{y_1-y_2} = \tau \} \\
    & \cup \{ (y_1, y_2) \in \mathbb{R}^p \times \mathbb{R}^p \, : \, \max_{i=1,2} \norm{x-y_i} = \norm{y_1-y_2} < \tau \}.
\end{split}
\end{equation*}
Finally, the interior of $Z_{\tau}(x)$ is
\begin{equation*}
\begin{split}
\mathring{Z}_{\tau}(x) &= Z_{\tau}(x) \setminus \partial Z_{\tau}(x) \\
&= \{ (y_1, y_2) \in \mathbb{R}^p \times \mathbb{R}^p \, : \, \max_{i=1,2} \norm{x-y_i} < \norm{y_1-y_2} < \tau \}.
\end{split}
\end{equation*}
\end{proof} \\

\noindent The next lemma deals with symmetry properties of the set $Z_{\tau}(x)$. Its proof follows directly from the definition of $Z_{\tau}(x)$ and it is omitted.

\begin{lemma} \label{lemma_symmetries_of_z_tau_x}
For $x \in \mathbb{R}^p$ and $\tau \in [0,\infty]$, $Z_{\tau}(x)$ satisfies the following symmetry conditions
\begin{equation} \label{symmetry1_of_z_tau_x}
	(x,x) + (x_1,x_2) \in Z_{\tau}(x) \Longleftrightarrow (x,x) - (x_1,x_2) \in Z_{\tau}(x)
\end{equation}
\begin{equation} \label{symmetry2_of_z_tau_x}
	(x,x) + (x_1,x_2) \in Z_{\tau}(x) \Longleftrightarrow (x,x) + (x_2,x_1) \in Z_{\tau}(x).
\end{equation}
In particular, \eqref{symmetry1_of_z_tau_x} and \eqref{symmetry2_of_z_tau_x} imply that
\begin{equation} \label{symmetry3_of_z_tau_x}
	(x,x) + (x_1,x_2) \in Z_{\tau}(x) \Longleftrightarrow (x,x) - (x_2,x_1) \in Z_{\tau}(x).
\end{equation}
\end{lemma}

\noindent We now proceed with the proofs of the results in the previous sections. \\

\begin{proof}[Proof of Proposition \ref{proposition_local_depth}]
  We start by proving (i). For the monotonicity, observe that for $\nu \ge \tau$, $Z_{\nu}(x) \supset Z_{\tau}(x)$ and therefore $LLD(x,\nu) \ge LLD(x,\tau)$. Now, let $\{ \tau_n \}_{n=1}^{\infty}$ be a sequence of scalars converging to $\tau$ from above. Then, since $Z_{\tau}(x)$ is closed,
 \begin{equation*}
 \lim_{n \to \infty} \mathbf{I}( (x_1, x_2) \in Z_{\tau_n}(x)) = \mathbf{I}( (x_1, x_2) \in Z_{\tau}(x)).
 \end{equation*}
Hence, by Lebesgue dominated convergence Theorem (LDCT) we get that
\begin{equation*}
\begin{split}
\lim_{n \to \infty} LLD(x,\tau_n) &= \int \lim_{n \to \infty} \mathbf{I}( (x_1, x_2) \in Z_{\tau_n}(x)) \, dP(x_1) dP(x_2) \\
&= \int \mathbf{I}( (x_1, x_2) \in Z_{\tau}(x)) \, dP(x_1) dP(x_2) = LLD(x,\tau).
\end{split}
\end{equation*}
In particular, for $\tau=0$, it follows that
\begin{equation*}
\lim_{n \to \infty} LLD(x,\tau_n) = LLD(x,0) = P^{2}(\{ x \}).
\end{equation*}
Similarly, if $\{ \tau_n \}_{n=1}^{\infty}$ diverges to $\infty$, then
\begin{equation*}
\begin{split}
\lim_{n \to \infty} LLD(x,\tau_n) &= \int \lim_{n \to \infty} \mathbf{I}( (x_1, x_2) \in Z_{\tau_n}(x)) \, dP(x_1) dP(x_2) \\
&= \int \mathbf{I}( (x_1, x_2) \in Z(x)) \, dP(x_1) dP(x_2) = LD(x).
\end{split}
\end{equation*}
For (ii) observe that
\begin{equation*}
	\sup_{x \in \mathbb{R}^p \, : \, \norm{x} \geq M } LLD(x,\tau) \leq \int \sup_{x \in \mathbb{R}^p \, : \, \norm{x} \geq M } \mathbf{I}( (x_1, x_2) \in Z_{\tau}(x)) \, dP(x_1) \, dP(x_2).
\end{equation*}
Since for $(x_1,x_2) \in Z_{\tau}(x)$, it follows that
\begin{equation*}
  \norm{x} \leq \max_{i=1,2} \left( \norm{x_i} + \norm{x_i-x} \right) \leq \max_{i=1,2} \norm{x_i} + \norm{x_1-x_2} \leq 3 \max_{i=1,2} \norm{x_i},
\end{equation*}
we have that
\begin{align*}
  \sup_{x \in \mathbb{R}^p \, : \, \norm{x} \geq M } \mathbf{I}( (x_1, x_2) \in Z_{\tau}(x)) \leq \sup_{x \in \mathbb{R}^p \, : \, \norm{x} \geq M } \mathbf{I}( \norm{x} &\leq 3 \max_{i=1,2} \norm{x_i} ) \leq \mathbf{I}( M \leq 3 \max_{i=1,2} \norm{x_i}),
\end{align*}
where $\mathbf{I}( M \leq 3 \max_{i=1,2} \norm{x_i} )$ converges to $0$ as $M \rightarrow \infty$. Now, (ii) follows from LDCT. \\
We now prove (iii). For the continuity of LLD at $(x, \tau) \in \mathbb{R}^p \times [0, \infty)$, we consider for $(y, \nu) \in \mathbb{R}^p \times [0, \infty)$ the difference
\begin{equation*}
\begin{split}
	& \abs{ LLD(x,\tau) - LLD(y,\nu) }  \leq \\
	& \int \abs{ \mathbf{I}( (x_1, x_2) \in Z_{\tau}(x) ) - \mathbf{I}( (x_1, x_2) \in Z_{\nu}(y) ) } \, dP(x_1) \, dP(x_2).
\end{split}
\end{equation*}
The result follows, using LDCT, if we can show that the absolute value of the difference of the indicators inside the integral converges pointwise to $0$ almost surely as $(y, \nu)$ converges to $(x,\tau)$. This in turn follows from the fact that for $(x_1, x_2) \in \mathbb{R}^p \times \mathbb{R}^p$ fixed, if $(x_1, x_2) \notin \partial Z_{\tau}(x)$ (by the absolute continuity $P^{\otimes 2}(\partial Z_{\tau}(x))=0$), then there exists $\epsilon > 0$ such that, for all $y \in \mathbb{R}^p$ with $\norm{x-y} \leq \epsilon$ and all $\nu \in [0, \infty)$ with $\abs{\tau-\nu} \le \epsilon$, $\mathbf{I}( (x_1, x_2) \in Z_{\tau}(x) ) = \mathbf{I}( (x_1, x_2) \in Z_{\nu}(y) )$. 

We now establish this as follows. Recall the interior, exterior and boundary of $Z_{\tau}(x)$ as given in Lemma \ref{lemma_interior_exterior_of_z_tau_x}. If $(x_1,x_2) \in \mathring{Z}_{\tau}(x)$, then $\max_{i=1,2} \norm{x-x_i} < \norm{x_1-x_2} < \tau$. Since,  $\norm{y-x_i} \leq \norm{y-x} + \norm{x-x_i}$, for $i=1,2$, it follows that $\max_{i=1,2} \norm{y-x_i} < \norm{x_1-x_2} < \nu$ by taking $\epsilon < \min \left( \norm{x_1-x_2}-\max_{i=1,2} \norm{x-x_i}, \tau-\norm{x_1-x_2} \right)$; this yields $(x_1,x_2) \in \mathring{Z}_{\nu}(y)$. On the other hand, if $(x_1,x_2) \in Z^{e}_{\tau}(x)$, then either $\norm{x_1-x_2} > \tau$ or $\max_{i=1,2} \norm{x-x_i} > \norm{x_1-x_2} \le \tau$. In the first case, by taking $0 < \epsilon < \norm{x_1-x_2}-\tau$ we see that also $\norm{x_1-x_2} > \nu$. In the second case, $\norm{x-x_i} > \norm{x_1-x_2}$ for $i=1$ or $i=2$, and since $\norm{y-x_i} \geq \abs{ \norm{x-x_i} - \norm{y-x} }$, by taking $0 < \epsilon < \norm{x-x_i} - \norm{x_1-x_2} $ it follows that $\norm{y-x_i} > \norm{x_1-x_2}$. In both cases we conclude that $(x_1,x_2) \in Z^{e}_{\nu}(y)$. 

We now show that $LLD(\cdot,\tau)$ is upper semicontinuous. Since, for $x \in \mathbb{R}^p$, $Z_{\tau}(x)$ is a closed set, it follows as above that, if $(x_1,x_2) \in Z^{e}_{\tau}(x)$, then $\limsup_{y \to x} \mathbf{I}( (x_1, x_2) \in Z_{\tau}(y) )=0$. Therefore, $\limsup_{y \to x} \mathbf{I}( (x_1, x_2) \in Z_{\tau}(y) ) \leq \mathbf{I}( (x_1, x_2) \in Z_{\tau}(x) )$, and by LDCT 
\begin{equation*}
\begin{split}
    \limsup_{y \to x} LLD(y,\tau) & \le \int \limsup_{y \to x} \mathbf{I}( (x_1, x_2) \in Z_{\tau}(y) )  \, dP(x_1) \, dP(x_2) \\
	&\leq \int \mathbf{I}( (x_1, x_2) \in Z_{\tau}(x) )  \, dP(x_1) \, dP(x_2) = LLD(x,\tau).
\nonumberthis
\end{split}
\end{equation*}
The proof of (iv) is rather routine and, for the skeptic, it is provided in Appendix \ref{section_proof_proposition_local_depth}.
\end{proof} \\

\begin{proof}[Proof of Theorem \ref{theorem_local_depth_tau_to_0}]
  A direct proof of (i), when the density $f(\cdot)$ is continuous, is obtained from \eqref{local_depth_for_absolutely_continuous_distribution}. Since $Z_{1}(0)$ is compact, it follows that for $\epsilon>0$
  \begin{equation*}
\sup_{\tau \in [0,\epsilon]} \sup_{(x_1,x_2) \in Z_{1}(0)} f(x + \tau x_1) f(x + \tau x_2) < \infty.
 \end{equation*}
Hence, from LDCT it follows that
\begin{equation*}
\begin{split}
	\lim_{\tau \rightarrow 0^{+}} \frac{1}{\Lambda_1 \tau^{2 p}} LLD(x,\tau) &= \frac{1}{\Lambda_1} \int_{Z_{1}(0)} \lim_{\tau \rightarrow 0^{+}} f(x + \tau x_1) f(x + \tau x_2) \, dx_1 dx_2 = f^{2}(x).
\end{split}
\end{equation*}
The general statement follows from Lebesgue differentiation Theorem \citep{Benedetto-2010}, since $Z_{\tau}(x) \subset \overline{B}_{\tau}(x) \times \overline{B}_{\tau}(x) \subset \overline{B}^{\otimes 2}_{\sqrt{2} \tau}(x)$, where $\overline{B}^{\otimes 2}_{\sqrt{2}\tau}(x) \coloneqq \{ (w,z) \in \mathbb{R}^p \times \mathbb{R}^p \, : \, \norm{w-x}^2 + \norm{z-x}^2 \leq 2 \tau^2 \}$ is the closed ball in $\mathbb{R}^p \times \mathbb{R}^p$ with center $(x,x)$ and radius $\sqrt{2} \tau$, and 
\begin{equation*}
\frac{\lambda^{\otimes 2}(Z_{\tau}(x))}{\lambda^{\otimes 2}(\overline{B}^{\otimes 2}_{\sqrt{2} \tau}(x))} = \frac{\lambda^{\otimes 2}(Z_{1}(0))}{2^{p} \lambda^{\otimes 2}(\overline{B}^{\otimes 2}_{1}(0))} = \frac{ \Lambda_1 p!}{2^p \pi^p} >0.
\end{equation*}
For (ii), notice from \eqref{local_depth_for_absolutely_continuous_distribution} that
\begin{equation}
\label{rate_of_convergence_integral_form}
	\frac{1}{\tau^{2 p}} LLD(x,\tau) - \Lambda_1 f^2(x) = \int_{Z_1(0)} \left[ f(x+\tau x_1) f(x+\tau x_2) - f^2(x) \right] \, dx_1 dx_2.
\end{equation}
Since $f(\cdot)$ is twice differentiable, by multivariate Taylor's expansion, for $i=1,2$,
\begin{equation*}
f(x+\tau x_i) = f(x) + \tau \inp{\nabla f(x)}{x_i} + \frac{\tau^2}{2} x_{i}^{\top} H_f(x) x_i + O(\tau^2).
\end{equation*}
Therefore,
\begin{equation} \label{taylor_expansion_of_product}
\begin{split}
&f(x+\tau x_1) f(x+\tau x_2) = f^2(x) + \tau f(x) \inp{\nabla f(x)}{x_1+x_2} \\
&+ \frac{\tau^2}{2} f(x) \left[ x_{1}^{\top} H_f(x) x_{1} + x_{2}^{\top} H_f(x) x_{2} \right] + \tau^2 \inp{\nabla f(x)}{x_1} \inp{\nabla f(x)}{x_2} + O(\tau^2).
\end{split}
\end{equation}
The continuity of the third order partial derivatives implies that the remainder $O(\tau^2)$ is a bounded and continuous function of $x_1$ and $x_2$ for $\tau \in [0,\tau_1]$. By substituting \eqref{taylor_expansion_of_product} in \eqref{rate_of_convergence_integral_form}, we see that
\begin{equation*}
\begin{split}
	\frac{1}{\tau^{2 p}} LLD(x,\tau) - \Lambda_1 f^2(x) &= \tau f(x) \int_{Z_1(0)} \inp{\nabla f(x)}{x_1+x_2} \, dx_1 dx_2 \\
	&+ \frac{\tau^2}{2} f(x) \int_{Z_1(0)} \left[ x_{1}^{\top} H_f(x) x_{1} + x_{2}^{\top} H_f(x) x_{2} \right] \, dx_1 dx_2 \\
	& + \tau^2 \int_{Z_1(0)} \inp{\nabla f(x)}{x_1} \inp{\nabla f(x)}{x_2} \, dx_1 dx_2 + O(\tau^2).
\end{split}
\end{equation*}
From Lemma \ref{lemma_symmetries_of_z_tau_x} it follows that \begin{equation} \label{symmetry1_of_Z_tau_0}
(x_1, x_2) \in Z_{1}(0) \Longleftrightarrow (-x_1, -x_2) \in Z_{1}(0),
\end{equation}
and by this change of variable
\begin{equation*}
\int_{Z_1(0)} \inp{\nabla f(x)}{x_1+x_2} \, dx_1 dx_2 = - \int_{Z_1(0)} \inp{\nabla f(x)}{x_1+x_2} \, dx_1 dx_2.
\end{equation*}
Therefore
\begin{equation*}
\int_{Z_1(0)} \inp{\nabla f(x)}{x_1+x_2} \, dx_1 dx_2 = 0.
\end{equation*}
 Lemma \ref{lemma_symmetries_of_z_tau_x} also implies that
\begin{equation*}
(x_1, x_2) \in Z_{1}(0) \Longleftrightarrow (x_2, x_1) \in Z_{1}(0),
\end{equation*}
and therefore
\begin{equation*}
\int_{Z_1(0)} x_{2}^{\top} H_f(x) x_{2} \, dx_1 dx_2 = \int_{Z_1(0)} x_{1}^{\top} H_f(x) x_{1} \, dx_1 dx_2.
\end{equation*}
We conclude that
\begin{equation*}
	\lim_{\tau \rightarrow 0^{+}} \frac{1}{\tau^2} \left( \frac{1}{\tau^{2 p}} LLD(x,\tau) - \Lambda_1 f^2(x) \right) = h(x), \text{ where}
\end{equation*}
\begin{equation*}
\begin{split}
h(x) &= f(x) \int_{Z_1(0)} x_{1}^{\top} H_f(x) x_{1} \, dx_1 dx_2 + \int_{Z_1(0)} \inp{\nabla f(x)}{x_1} \inp{\nabla f(x)}{x_2} \, dx_1 dx_2.
\end{split}
\end{equation*}
\end{proof} \\

\begin{proof}[Proof of Corollary \ref{corollary_local_depth_tau_to_0}]
From \eqref{local_depth_for_absolutely_continuous_distribution}, we see that
\begin{equation*}
	LLD(x,\tau) = \int_{ Z_{\tau}(0) } f(x+x_1) f(x+x_2) \, dx_1 dx_2.
\end{equation*}
In two dimensions $Z_{\tau}(0)$ can be expressed as the union of two triangles $T_{-+}^{\tau}$ and $T_{+-}^{\tau}$; that is,
\begin{equation*}
\begin{split}
	T_{-+}^{\tau} & \coloneqq \{ (x_1, x_2) \in \mathbb{R}^2 \, : \, x_1 \leq 0, \, x_2 \geq 0, \, x_2 - x_1 \leq \tau \} \\
	T_{+-}^{\tau} & \coloneqq \{ (x_1, x_2) \in \mathbb{R}^2 \, : \, x_1 \geq 0, \, x_2 \leq 0, \, x_1 - x_2 \leq \tau \}.
\end{split}
\end{equation*}
Now, by a change of variables in the integrals over the triangles it follows that
\begin{equation*}
\begin{split}
	LLD(x,\tau) &= \int_{ T_{-+}^{\tau} } f(x+x_1) f(x+x_2) + f(x-x_1) f(x-x_2) \, dx_1 dx_2 \\
	&= \int_{T_{++}^{\tau}} f(x-x_1) f(x+x_2) + f(x+x_1) f(x-x_2) \, dx_1 dx_2 \\
	&= 2 \int_{T_{++}^{\tau}} f(x+x_1) f(x-x_2) \, dx_1 dx_2.
\end{split}
\end{equation*}
The last part follows directly from Theorem \ref{theorem_local_depth_tau_to_0} (i) and the fact that $Z_1(0)$ has area $1$.
\end{proof} \\

\begin{proof}[Proof of Theorem \ref{theorem_uniform_convergence_of_tau_approximation}]
We start by proving (i). Observe that
\begin{equation} \label{hoelder_continuity_square_root}
\abs{\sqrt{s}-\sqrt{t}} \leq \abs{s-t}^{\frac{1}{2}} \text{ for all } s,t \geq 0.
\end{equation}
Therefore, for $\tau>0$,
\begin{equation*}
\begin{split}
\sup_{ x \in \mathbb{R}^p } \abs{ f_{\tau}(x) - f(x) } &= \sup_{ x \in \mathbb{R}^p } \abs{ \sqrt{\frac{1}{\tau^{2p}\Lambda_1} LLD(x,\tau)} - \sqrt{f^2(x)} } \\
& \leq \sup_{ x \in \mathbb{R}^p } \abs{ \frac{1}{\tau^{2p}\Lambda_1} LLD(x,\tau) - f^2(x) }^{\frac{1}{2}} \\
&= \sup_{ x \in \mathbb{R}^p } \sqrt{G_{\tau}(x)},
\end{split}
\end{equation*}
where
\begin{equation*}
G_{\tau}(x) \coloneqq \abs{ \frac{1}{\Lambda_1} \int_{Z_{1}(0)} \left[ f(x+\tau x_1) f(x+\tau x_2) - f^2(x) \right] \, dx_1 dx_2}.
\end{equation*}
Since the square root is a continuous function, by definition of supremum there exist sequences $\{ x^k \}_{k=1}^{\infty}$ such that
\begin{equation} \label{upper_semicontinuity_square_root}
\begin{split}
\sup_{ x \in \mathbb{R}^p } \sqrt{G_{\tau}(x)} = \limsup_{k \to \infty} \sqrt{G_{\tau}(x^k)} \leq \sqrt{\limsup_{k \to \infty} G_{\tau}(x^k)} \leq \sqrt{ \sup_{ x \in \mathbb{R}^p } G_{\tau}(x)}.
\end{split}
\end{equation}
To complete the proof, it is enough to show that
\begin{equation} \label{uniform_lebesgue_differentiation}
\lim_{\tau \to 0^{+}} \sup_{ x \in \mathbb{R}^p } G_{\tau}(x) = 0.
\end{equation}
For this, observe that
\begin{equation*}
\sup_{ x \in \mathbb{R}^p } G_{\tau}(x) \leq \frac{1}{\Lambda_1} \int_{Z_{1}(0)} \sup_{ x \in \mathbb{R}^p } \abs{ f(x+\tau x_1) f(x+\tau x_2) - f^2(x) } \, dx_1 dx_2.
\end{equation*}
Since $f(\cdot)$ is uniformly continuous and bounded, for all $(x_1, x_2)$
\begin{equation*}
\lim_{\tau \to 0^{+}} \sup_{ x \in \mathbb{R}^p } \abs{ f(x+\tau x_1) f(x+\tau x_2) - f^2(x) } = 0.
\end{equation*}
\eqref{uniform_lebesgue_differentiation} now follows from LDCT, since $Z_{1}(0)$ is compact and the supremum is bounded. \\
Since a continuous function is uniformly continuous on a compact set, the proof of the first part of (ii) follows from the proof of (i) with $\mathbb{R}^p$ replaced by $K$. \\
For the second part of (ii), notice that
\begin{equation*}
\sup_{y \in \overline{B}_{\epsilon}(x) } \abs{ f_{\tau}(y) - f(x) } \leq \sup_{y \in \overline{B}_{\epsilon}(x) } \abs{ f_{\tau}(y) - f(y) } + \sup_{y \in \overline{B}_{\epsilon}(x) } \abs{ f(y) - f(x) }.
\end{equation*}
The result now follows from the first part of (ii) and continuity of $f(\cdot)$.
\end{proof} \\

\begin{proof}[Proof of Corollary \ref{corollary_sample_local_depth}]
The result follows from Proposition \ref{proposition_local_depth}, where the probability measure $P^{\otimes 2}$ is replaced by $P^{\otimes 2}_n \coloneqq \frac{1}{ { n \choose 2 }} \sum_{1 \leq i < j \leq n} \delta_{(X_i, X_j)}$.
\end{proof} \\

\begin{proof}[Proof of Theorem \ref{theorem_uniform_consistency}]
  Let $\mathcal{F} \coloneqq \{ \mathcal{K}_{x,\tau} \, : \, x \in \mathbb{R}^p, \tau \in [0,\infty] \}$, where
\begin{equation*}
 \mathcal{K}_{x,\tau}(x_1,x_2) = \mathbf{I}( (x_1, x_2) \in Z_{\tau}(x)) = \mathbf{I}( x \in L(x_1,x_2)) \mathbf{I}( \norm{x_1-x_2} \leq \tau)
\end{equation*}
and $L(x_1,x_2) \coloneqq \{ y \in \mathbb{R}^p \, : \max_{ i=1,2} \norm{y-x_i} \leq \norm{x_1 - x_2} \}$. We will show that
 \begin{equation*}
\sup_{f \in \mathcal{F}} \abs{ \int f(x_1,x_2) \, dP(x_1) dP(x_2) - \frac{2}{n(n-1)} \sum_{\substack{i,j=1 \\ i<j}}^n f(X_i,X_j) } \xrightarrow[n \to \infty]{} 0 \text{ a.s.}
  \end{equation*}
 To this end, we use Corollary 3.3 of \citet{Arcones-Gine-1993} and verify that (i) $\sup_{f \in \mathcal{F}} \abs{f(\cdot)} < \infty$ and $\sup_{f^{\prime} \in \mathcal{F^{\prime}}} \abs{f^{\prime}(\cdot)} < \infty$, where $\mathcal{F}^{\prime} \coloneqq \{ \mathcal{J}_{x,\tau} \, : \, x \in \mathbb{R}^p, \tau \in [0,\infty] \}$ and $\mathcal{J}_{x,\tau}$ is given in \eqref{kernel_projection}, (ii) ${\mathcal{F}}$ is image admissible Suslin \citep[p.\ 186]{Dudley-2014}, and (iii) ${\mathcal{F}}$ is a VC subgraph class. \\
Observe first that $\mathcal{F}$ and $\mathcal{F}^{\prime}$ are bounded by $1$, and hence (i) holds. Turning to (ii), notice that 
 \begin{equation*}
(x_1,x_2,x,\tau) \to \mathbf{I}( (x_1, x_2) \in Z_{\tau}(x))
 \end{equation*}
 is jointly Borel measurable by Lemma \ref{lemma_interior_exterior_of_z_tau_x} and  $Z_{\tau}(x)$ is a closed set in $\mathbb{R}^p \times \mathbb{R}^p$ for all $x \in \mathbb{R}^p$ and $\tau \in [0,\infty]$. Hence, by \citep[p.\ 186]{Dudley-2014}, the class $\mathcal{F}$ is image admissible Suslin via the onto Borel measurable map $\mathfrak{e} \, : \, (\mathbb{R}^p \times [0,\infty],\mathscr{B}(\mathbb{R}^p) \times \mathscr{B}([0,\infty])  \rightarrow \mathcal{F}$ given by $\mathfrak{e}(x,\tau)=\mathcal{K}_{x,\tau}$.
For (iii), let $\mathcal{L}$ be the class of lenses $\mathcal{L} \coloneqq \{ L(x_1,x_2) \, : \, (x_1, x_2) \in \mathbb{R}^p \times \mathbb{R}^p \}$ and let $\mathcal{L}^{\prime} \coloneqq \mathcal{L} \cup \{ \emptyset \}$, where $\emptyset$ is the empty set.  For $x \in \mathbb{R}^p$, define the function $\mathfrak{g}_x \, : \, \mathcal{L}^{\prime} \to \mathbb{R}$ by $\mathfrak{g}_x(L)=\mathbf{I}( x \in L)$. Similarly, for $\tau \in [0,\infty]$ define the evaluation function $\mathfrak{h}_{\tau} \, : \, \mathbb{R}^p \times \mathbb{R}^p \to \mathcal{L}^{\prime}$ by $\mathfrak{h}_{\tau}(x_1,x_2) = L(x_1,x_2)$ if $\norm{x_1-x_2} \leq \tau$ and $\mathfrak{h}_{\tau}(x_1,x_2)=\emptyset$ if $\norm{x_1-x_2} > \tau$.  Since, $\mathcal{K}_{x,\tau} = \mathfrak{g}_x \circ \mathfrak{h}_{\tau}$ for all $x \in \mathbb{R}^p$ and $\tau \in [0,\infty]$, it follows that one can identify ${\mathcal{F}}$  with ${\mathcal{G}} \coloneqq \{\mathfrak{g}_x: x \in \mathbb{R}^p\}$. Hence it is enough to show that $\mathcal{G}$ is VC-subgraph. \\
For this, let $\mathcal{B}$ be the class of balls in $\mathbb{R}^p$. Notice that $\mathcal{L}$ is a subset of
 \begin{equation*}
\mathcal{B} \cap \mathcal{B} \coloneqq \{ B_1 \cap B_2 \, : \, B_1, B_2 \in \mathcal{B} \}.
 \end{equation*}
By Theorem 1 in \citet{Dudley-1979}, $\mathcal{B}$ is a VC-class of sets. Applying Proposition 3.6.7 (ii) of \citet{Gine-2016}, it follows that $\mathcal{B} \cap \mathcal{B}$ is also a VC-class of sets. This implies that $\mathcal{L}$ and $\mathcal{L}^{\prime}=\mathcal{L} \cup \{ \emptyset \}$ are VC-classes of sets by Proposition 3.6.7 (iv) in \citet{Gine-2016}. Finally, using Definition 3.6.8 of \citet{Gine-2016}, it follows that $\mathcal{G}$ is a VC-subgraph class of functions and the proof follows.
\end{proof} \\

\begin{proof}[Proof of Theorem \ref{theorem_uniform_asymptotic_normality_of_local_depth}]
To prove Theorem \ref{theorem_uniform_asymptotic_normality_of_local_depth}, we will verify the conditions of Theorem 4.9 in \citet{Arcones-Gine-1993}.  To this end, first let $\mathcal{F} \coloneqq \{ \mathcal{K}_{x,\tau} \, : \, (x,\tau) \in \Gamma \}$, where $\mathcal{K}_{x,\tau}$ is given in \eqref{kernel} and $\mathcal{F}^{\prime} = \{ \mathcal{J}_{x,\tau} \, : \, (x,\tau) \in \Gamma \}$, where $\mathcal{J}_{x,\tau}$ is given in \eqref{kernel_projection}. 
Notice that, $\sup_{f \in \mathcal{F}} \abs{f(\cdot)} \leq 1$, $\sup_{f^{\prime} \in \mathcal{F^{\prime}}} \abs{f^{\prime}(\cdot)} \leq 1$ and $\mathcal{F}$ is image admissible Suslin \citep[p.\ 186]{Dudley-2014}. This then shows that ${\mathcal{F}}$ is a measurable class with a bounded envelope and (ii) of Theorem 4.9 in \citet{Arcones-Gine-1993} holds. To verify (iii) in \citet{Arcones-Gine-1993} we appeal to Lemma 4.4 and (4.2) in \citet{Alexander-1987} concerning the covering number $N(\mathcal{F}, d_{L^2(\mathcal{F},P)}, \, \cdot \, )$ of $\mathcal{F}$ with respect to the $L^2$-distance, $d_{L^2(\mathcal{F},P)}$, given by
\begin{equation*}
d_{L^2(\mathcal{F},P)}((x,\tau),(y,\nu)) = \left( \int (\mathcal{K}_{x,\tau}(x_1,x_2)-\mathcal{K}_{y,\nu}(x_1,x_2))^2 \, dP(x_1) dP(x_2) \right)^{\frac{1}{2}}.
\end{equation*}
For this, we observe that ${\mathcal{F}}$ is a VC-subgraph class as in the proof of Theorem \ref{theorem_uniform_consistency}. Thus to complete the proof, we need to verify (i) in \citet{Arcones-Gine-1993}. To this end, we need to show: (a) the finite dimensional distributions of $\sqrt{n}(LLD_n(x, P, \tau)-LLD(x, p, \tau))$ converges to a multivariate normal distribution and (b) for each $(x, \tau)$, the limiting normal random variable $\{ W(x,\tau) \}_{(x,\tau) \in \Gamma }$ admits a version whose sample paths are all bounded and uniformly continuous with respect to the distance $d_{\mathcal{F}^{\prime},P}^2$ on $\mathcal{F}^{\prime}$ given by
\begin{equation*}
d_{\mathcal{F}^{\prime},P}^2((x,\tau),(y,\nu)) = \int ( \mathcal{J}_{x,\tau}(x_1) - \mathcal{J}_{y,\nu}(x_1) )^2 \, dP(x_1) - ( LLD(x,\tau) - LLD(y,\nu) )^2,
\end{equation*}
where we identify a function $\mathcal{J}_{x,\tau}$ for $(x,\tau) \in \Gamma$ with its parameter $(x,\tau)$. In this sense, $d_{\mathcal{F}^{\prime},P}^2$ is a metric on $\Gamma$. 

Since $W(x, \tau)$ is Gaussian, we can apply \citet[Theorem 2.3.7]{Gine-2016} with $T=\Gamma$ and $d=d_{\mathcal{F}^{\prime},P}^2$. First, note that $\{ W(x,\tau) \}_{(x,\tau) \in \Gamma }$ is a sub-Gaussian process relative to $d_{\mathcal{F}^{\prime},P}^2$. Indeed, using Proposition \ref{proposition_asymptotic_normality_of_local_depth}, for $(x,\tau), (y,\nu) \in \Gamma$, $(W(x,\tau),W(y,\nu))^{\top}$ has a bivariate normal distribution with mean $(0,0)^{\top}$ and covariance matrix
\begin{equation*}
\begin{pmatrix}
	b^{2}(x,\tau) & \gamma((x,\tau),(y,\nu)) \\
	\gamma((y,\nu),(x,\tau)) & b^{2}(y,\nu).
\end{pmatrix}.
\end{equation*}
It follows that $W(x,\tau)-W(y,\nu)$ is normally distributed with mean $0$ and variance $b^{2}(x,\tau)+b^{2}(y,\nu) - 2 \, \gamma((x,\tau),(y,\nu))=d_{\mathcal{F}^{\prime},P}^2((x,\tau),(y,\nu))$. Therefore, for all $\alpha \in \mathbb{R}$
\begin{equation*}
E [ e^{\alpha (W(x,\tau)-W(y,\nu))}] = e^{\alpha^2 \frac{d_{\mathcal{F}^{\prime},P}^2((x,\tau),(y,\nu))}{2}}
\end{equation*}
and the process $\{ W(x,\tau) \}_{(x,\tau) \in \Gamma }$ is sub-Gaussian with respect to $d_{\mathcal{F}^{\prime},P}^2$. 

We next verify the integrability condition for the metric entropy. To this end, notice that, for $(x,\tau), (y,\nu) \in \Gamma$, the $L^2$-distance on $\mathcal{F}^\prime$, $d_{L^2(\mathcal{F}^{\prime},P)}$ is given by
\begin{equation*}
d_{L^2(\mathcal{F}^{\prime},P)}((x,\tau),(y,\nu)) = \left( \int (\mathcal{J}_{x,\tau}(x_1)-\mathcal{J}_{y,\nu}(x_1))^2 \, dP(x_1) \right)^{\frac{1}{2}}.
\end{equation*}
Now using yet another application of Lemma 4.4 of \citet{Alexander-1987}, it follows that there are constants $C_1, C_2 > 1$ such that
\begin{equation*}
N(\mathcal{F}, d_{L^2(\mathcal{F},P)},\sqrt{\epsilon}) \leq \left( \frac{C_1}{\sqrt{\epsilon}} \right)^{C_2}.
\end{equation*}
By Jensen's inequality, it follows that
\begin{equation*}
d_{L^2(\mathcal{F}^{\prime},P)}((x,\tau),(y,\nu)) \leq d_{L^2(\mathcal{F},P)}((x,\tau),(y,\nu)),
\end{equation*}
which in turn, implies that
\begin{equation*}
N(\mathcal{F}^{\prime}, d_{L^2(\mathcal{F}^{\prime},P)},\sqrt{\epsilon}) \leq N(\mathcal{F}, d_{L^2(\mathcal{F},P)},\sqrt{\epsilon}) \leq \left( \frac{C_1}{\sqrt{\epsilon}} \right)^{C_2},
\end{equation*}
where $N(\mathcal{F}^{\prime}, d_{L^2(\mathcal{F}^{\prime},P)}, \sqrt{\epsilon})$ is the covering number of $\mathcal{F}^{\prime}$ with respect to the distance $d_{L^2(\mathcal{F}^{\prime},P)}$. Thus, for any $0 < \epsilon \le 1$,
\begin{equation} \label{inequality_covering_numbers}
N(\mathcal{F}^{\prime}, d_{\mathcal{F}^{\prime},P}^2,\epsilon) \leq N(\mathcal{F}^{\prime}, d_{L^2(\mathcal{F}^{\prime},P)}^2,\epsilon) =  N(\mathcal{F}^{\prime}, d_{L^2(\mathcal{F}^{\prime},P)},\sqrt{\epsilon}) \leq \left( \frac{C_1}{\sqrt{\epsilon}} \right)^{C_2} \leq \left( \frac{C_1}{\epsilon} \right)^{C_2}.
\end{equation}
Hence, it follows that
\begin{align*}
  \int_0^1 \sqrt{ \log(N(\mathcal{F}^{\prime}, d_{\mathcal{F}^{\prime},P}^2,\epsilon)) } \, d \epsilon & \leq \sqrt{C_2} \int_0^1 \sqrt{ \log(C_1) - \log(\epsilon) } \, d \epsilon \\
  & \leq \sqrt{ C_2} \left( \sqrt{ \log(C_1)} + \int_0^{e^{-1}} \sqrt{ -\log(\epsilon) } \, d \epsilon + \int_{e^{-1}}^1 \sqrt{ - \log(\epsilon) } \, d \epsilon \right) \\
  & \leq \sqrt{ C_2} \left( \sqrt{ \log(C_1)} + \int_0^{e^{-1}} -\log(\epsilon) \, d \epsilon + 1-e^{-1} \right) \\
  & = \sqrt{ C_2} \left( \sqrt{ \log(C_1)} + e^{-1} + 1 \right) < \infty,
\end{align*}
where we have used that for $a,b \geq 0$, $\sqrt{a+b} \leq \sqrt{a} + \sqrt{b}$. The proof finally follows from Proposition A.1. given below.
\end{proof} \\

\noindent Before we state the following proposition, we recall that $\Gamma$ is a subset of $\mathbb{R}^p \times [0,\infty]$ such that, for $(x,\tau) \in \Gamma$, $b^2(x,\tau) = Var[\mathcal{J}_{x,\tau}(X_1)] >0$. For $k \geq 1$ and $(x^1,\tau^1), \dots, (x^k,\tau^k) \in \Gamma$, we also use the notations $\bm{LLD_n}(x^l,\tau^l) \coloneqq (LLD_n(x^1,\tau^1), \dots, LLD_n(x^k,\tau^k) )^{\top}$, $\bm{LLD}(x^l,\tau^l) \coloneqq (LLD(x^1,\tau^1), \dots, LLD(x^k,\tau^k) )^{\top}$ and $\bm{\mathcal{J}}_{x^l,\tau^l}(X_j) = (\mathcal{J}_{x^1,\tau^1}(X_j), \dots, \mathcal{J}_{x^k,\tau^k}(X_j))^{\top}$.

\begin{proposition} \label{proposition_asymptotic_normality_of_local_depth}
For $(x^1,\tau^1), \dots, (x^k,\tau^k) \in \Gamma$, $\sqrt{n} ( \bm{LLD_n}(x^l,\tau^l) - \bm{LLD}(x^l,\tau^l) )$ converges in distribution to a $k$-variate normal distribution with mean $0$ and covariance matrix whose $(l_1,l_2)$-element is given by $4 \, \gamma((x^{l_1},\tau^{l_1}),(x^{l_2},\tau^{l_2}))$, where $l_1,l_2 =1,\dots,k$.
\end{proposition}

\begin{proof}[Proof of Proposition \ref{proposition_asymptotic_normality_of_local_depth}] 
Using Hoeffding's decomposition of U-statistics (\citet[Equation 1.1.22]{Korolyuk-2013}), it follows that
\begin{equation} \label{representation_of_U-statistics}
\begin{split}
LLD_n(x,\tau)-LLD(x,\tau) &= \frac{2}{n} \sum_{j=1}^n \left[ \mathcal{J}_{x,\tau}(X_j) - LLD(x,\tau) \right] \\
&+\frac{1}{{n \choose 2}} \sum_{1 \leq i < j \leq n} \left[ \xi_{x,\tau}(X_i,X_j) + LLD(x,\tau) \right],
\end{split}
\end{equation}
where
\begin{equation} \label{function_xi}
\xi_{x,\tau}(x_1,x_2) \coloneqq \mathcal{K}_{x,\tau}(x_1,x_2) - \mathcal{J}_{x,\tau}(x_1)-\mathcal{J}_{x,\tau}(x_2).
\end{equation}
Hence, for $k \geq 1$ and $(x^1,\tau^1), \dots, (x^k,\tau^k) \in \Gamma$, it follows that
\begin{equation*}
 \bm{LLD_n}(x^l,\tau^l) - \bm{LLD}(x^l,\tau^l) = \frac{2}{n} \sum_{j=1}^n \left[ \bm{\mathcal{J}}_{x^l,\tau^l}(X_j) - \bm{LLD}(x^l,\tau^l) \right] + R_n,
\end{equation*}
where
\begin{equation*}
R_n=R_n(X_1, \dots, X_n) \coloneqq
\begin{pmatrix}
\frac{1}{{n \choose 2}} \sum_{1 \leq i < j \leq n} \left[ \xi_{x^1,\tau^1}(X_i,X_j) + LLD(x^1,\tau^1) \right]
 \\ \vdots \\ 
\frac{1}{{n \choose 2}} \sum_{1 \leq i < j \leq n} \left[ \xi_{x^k,\tau^k}(X_i,X_j) + LLD(x^k,\tau^k) \right]
\end{pmatrix}.
\end{equation*}
Now applying Cheybchev's inequality,  for all $\epsilon>0$,
\begin{equation*} \label{markov_inequality_remainder_of_local_depth}
P^{\otimes n}( \sqrt{n} \norm{R_n} > \epsilon) \leq 
 \frac{4n}{(n(n-1)\epsilon)^2}  \sum_{l=1}^{k} E \left[ \left( \sum_{1 \leq i < j \leq n} \left[ \xi_{x^l,\tau^l}(X_i,X_j)+LLD(x^l,\tau^l) \right] \right)^2 \right].
\end{equation*}
Observe that, for $l=1,\dots,k$, using the independence of $X_i$ and $X_j$ for $i \neq j$ and a routine calculation, each term inside the sum on the RHS of the above inequality is equal to
\begin{equation*}
E\left[ \sum_{1 \leq i < j \leq n} \left[ \xi_{x^l,\tau^l}(X_i,X_j)+LLD(x^l,\tau^l) \right]^2\right].
\end{equation*}
Now, the expectation of the above term is
\begin{equation*}
{n \choose 2} E \left[ \left( \xi_{x^l,\tau^l}(X_1,X_2)+LLD(x^l,\tau^l) \right)^2 \right].
\end{equation*}
Hence,
\begin{equation} \label{markov_inequality_remainder_of_local_depth_4}
P^{\otimes n}( \sqrt{n} \norm{R_n} > \epsilon) \leq \frac{1}{\epsilon^2} \, \frac{2}{n-1} \sum_{l=1}^{k}  E \left[ \left( \xi_{x^l,\tau^l}(X_1,X_2)+LLD(x^l,\tau^l) \right)^2 \right]
\end{equation}
and this converges to $0$ as $n \to \infty$, which implies that $R_n$ converges to $0$ in probability. \\
On the other hand, observe that $\frac{1}{n} \sum_{j=1}^n \left[ \bm{\mathcal{J}}_{x^l,\tau^l}(X_j) - \bm{LLD}(x^l,\tau^l) \right]$ is an average of i.i.d.\ random variables with mean $0$ and covariance matrix given by
\begin{equation}
\label{covariance_matrix_of_finite_dimensional_distributions}
E \left[ \left( \mathcal{J}_{x^{l_1},\tau^{l_1}}(X_j) - LLD(x^{l_1},\tau^{l_1}) \right) \left( \mathcal{J}_{x^{l_2},\tau^{l_2}}(X_j) - LLD(x^{l_2},\tau^{l_2}) \right) \right] = \gamma((x^{l_1},\tau^{l_1}), (x^{l_2},\tau^{l_2})).
\end{equation}
Therefore, by the multivariate central limit theorem, as $n \to \infty$,
\begin{equation*}
\sqrt{n} \, \frac{1}{n} \sum_{j=1}^n \left[ \bm{\mathcal{J}}_{x^l,\tau^l}(X_j) - \bm{LLD}(x^l,\tau^l) \right]
\end{equation*}
converges in distribution to a multivariate normal distribution with mean $0$ and covariance matrix given by \eqref{covariance_matrix_of_finite_dimensional_distributions}.
\end{proof} \\

\noindent An immediate consequence of Proposition \ref{proposition_asymptotic_normality_of_local_depth} is the following corollary.

\begin{corollary} \label{corollary_asymptotic_normality_of_local_depth}
If $x \in \mathbb{R}^p$ and $\tau \in (0,\infty]$ satisfy $b^2(x,\tau)>0$, then
\begin{equation} \label{asymptotic_normality_of_local_depth}
	\sqrt{n} \left( LLD_n(x,\tau) - LLD(x,\tau) \right) \xrightarrow[ n \rightarrow \infty]{d} N(0, 4 \, b^2(x,\tau))
\end{equation}
where $b^2(x,\tau)$ is as in \eqref{b_square}.
\end{corollary}

\begin{proof}[Proof of Corollary \ref{corollary_uniform_consistency_of_sample_tau_approximation}]
 For (i), observe that
\begin{equation*}
\sup_{ x \in \mathbb{R}^p } \abs{ f_{\tau,n}(x) - f(x)} \leq \sup_{ x \in \mathbb{R}^p } \abs{ f_{\tau,n}(x) - f_{\tau}(x)} + \sup_{ x \in \mathbb{R}^p } \abs{ f_{\tau}(x) - f(x)} 
\end{equation*}
and, by Theorem \ref{theorem_uniform_convergence_of_tau_approximation} (i), it is enough to show that
\begin{equation*}
\begin{split}
&\lim_{n \to \infty} \sup_{ \substack{x \in \mathbb{R}^p \\ \tau \in (0,T]}} \abs{ f_{\tau,n}(x) - f_{\tau}(x)}= \\
& \frac{1}{\tau^{p} \sqrt{\Lambda_1}} \lim_{n \to \infty} \sup_{ x \in \mathbb{R}^p } \abs{ \sqrt{LLD_n(x,\tau)} - \sqrt{LLD(x,\tau)} } \xrightarrow[n \rightarrow \infty ]{} 0 \text{ a.s.}
\end{split}
\end{equation*}
Now, using \eqref{hoelder_continuity_square_root}, it follows that
\begin{equation*}
\sup_{ x \in \mathbb{R}^p } \abs{ \sqrt{LLD_n(x,\tau)} - \sqrt{LLD(x,\tau)} } \leq \sup_{ x \in \mathbb{R}^p } \abs{ LLD_n(x,\tau) - LLD(x,\tau) }^{\frac{1}{2}},
\end{equation*}
and setting $G_{\tau,n}(x) \coloneqq \abs{ LLD_n(x,\tau) - LLD(x,\tau) }$, notice that
\begin{equation*}
\begin{split}
\sup_{ x \in \mathbb{R}^p} \sqrt{G_{\tau,n}(x)} \leq \sqrt{ \sup_{ x \in \mathbb{R}^p } G_{\tau,n}(x)}.
\end{split}
\end{equation*}
By Theorem \ref{theorem_uniform_consistency}, $\sup_{ x \in \mathbb{R}^p } G_{\tau,n}(x)$ converges to $0$ almost surely as $n \rightarrow \infty$, and since the square root is a continuous function, (i) follows from the continuous mapping theorem. Since a continuous function on a compact set is uniformly continuous, the proof of the first part of (ii) follows from the proof of (i) with $\mathbb{R}^p$ replaced by $K$. 

For the second part of the proof of (ii), notice that for all $x \in \mathbb{R}^p$
\begin{equation*}
\sup_{y \in \overline{B}_{\epsilon}(x)} \abs{ f_{\tau,n}(y) - f(x)} \leq \sup_{y \in \overline{B}_{\epsilon}(x)} \abs{ f_{\tau,n}(y) - f_{\tau}(y)} + \sup_{y \in \overline{B}_{\epsilon}(x)} \abs{ f_{\tau}(y) - f(x)},
\end{equation*}
where for all $\tau>0$ and $\epsilon>0$
\begin{equation*}
\sup_{y \in \overline{B}_{\epsilon}(x)} \abs{ f_{\tau,n}(y) - f_{\tau}(y)} \xrightarrow[n \rightarrow \infty ]{} 0 \text{ a.s.}
\end{equation*}
Using the compactness of $\overline{B}_{\epsilon}(x)$, the first part of (ii), \eqref{uniform_convergence_compact_set_of_sample_tau_approximation}, and  Theorem \ref{theorem_uniform_convergence_of_tau_approximation} (ii), \eqref{uniform_convergence_compact_set_of_tau_approximation}, it follows that
\begin{equation*}
\sup_{y \in \overline{B}_{\epsilon}(x)} \abs{ f_{\tau}(y) - f(x)} \xrightarrow[\epsilon, \tau \rightarrow 0^{+}]{} 0.
\end{equation*}
\end{proof} \\

\noindent Before proving Theorem \ref{theorem_asymptotic_normality_of_sample_tau_approximation},  we provide a lemma concerning the order of convergence of $b^2(x,\tau)$ to $0$, as $\tau \to 0^{+}$.

\begin{lemma} \label{lemma_order_of_convergence_b_square}
If $f(\cdot)$ is continuous, then
\begin{equation*}
  \lim_{\tau \to 0^{+}} \frac{b^2(x,\tau)}{\tau^{3p}} = \Lambda_1^{*2} f^3(x), \quad \text{where} \quad \Lambda_1^{*2} =  \int \lambda^2(Z_1(0)\rvert_{x_1}) \, dx_1.
\end{equation*}
\end{lemma}

\begin{proof}[Proof of Lemma \ref{lemma_order_of_convergence_b_square}]
  Let $\tau>0$. We compute
\begin{equation} \label{b_square_over_tau^3p}
  \frac{b^2(x,\tau)}{\tau^{3p}} = \frac{1}{\tau^{3p}} \int \left( \int_{Z_\tau(x)\rvert_{x_1}} f(x_2) \, dx_2 \right)^2 f(x_1) \, dx_1 - \left( \frac{LLD(x,\tau)}{\tau^{\frac{3}{2}p}} \right)^2,
\end{equation}
where, by Theorem \ref{theorem_local_depth_tau_to_0} (i),
\begin{equation*}
\frac{LLD(x,\tau)}{\tau^{2p}} \xrightarrow[\tau \to 0^{+}]{} \Lambda_1 f^2(x), \quad \text{and} \quad \frac{LLD(x,\tau)}{\tau^{\frac{3}{2}p}} \xrightarrow[\tau \to 0^{+}]{} 0.
\end{equation*}
We now focus on the first term in \eqref{b_square_over_tau^3p}. By changing variables twice, we note that
\begin{align*}
  \frac{1}{\tau^{3p}}  \int \left( \int_{Z_{\tau}(x)\rvert_{x_1}} f(x_2) \, dx_2 \right)^2 f(x_1) \, dx_1 &= \int \left( \frac{1}{\tau^{p}} \int_{Z_{\tau}(x)\rvert_{x+\tau x_1}} f(x_2) \, dx_2 \right)^2 f(x+\tau x_1) \, dx_1 \\
& = \int \left(  \int_{Z_{1}(0)\rvert_{x_1}} f(x+\tau x_2) \, dx_2 \right)^2 f(x+\tau x_1) \, dx_1.
\end{align*}
Since $f(\cdot)$ is continuous, it follows from LDCT that
\begin{equation*}
\lim_{\tau \to 0^{+}}  \int_{Z_{1}(0)\rvert_{x_1}} f(x+\tau x_2) \, dx_2 = f(x) \, \lambda(Z_1(0)\rvert_{x_1})
\end{equation*}
and
\begin{equation*}
  \lim_{\tau \to 0^{+}} \int \left(  \int_{Z_{1}(0)\rvert_{x_1}} f(x+\tau x_2) \, dx_2 \right)^2 f(x+\tau x_1) \, dx_1 = f^3(x) \int \lambda^2(Z_1(0)\rvert_{x_1}) \, dx_1.
\end{equation*}
\end{proof} \\

\begin{proof}[Proof of Theorem \ref{theorem_asymptotic_normality_of_sample_tau_approximation}]
Using Hoeffding's decomposition of U-statistics (\eqref{representation_of_U-statistics} with $\tau$ replaced by $\tau_n$), it follows that
\begin{equation} \label{representation_of_U-statistics_2}
\begin{split}
LLD_n(x,\tau_n)-LLD(x,\tau_n) &= 2 \, \frac{1}{n} \sum_{j=1}^n \left[ \mathcal{J}_{x,\tau_n}(X_j) - LLD(x,\tau_n) \right] \\
&+\frac{1}{{n \choose 2}} \sum_{1 \leq i < j \leq n} \left[ \xi_{x,\tau_n}(X_i,X_j) + LLD(x,\tau_n) \right].
\end{split}
\end{equation}
Since $x$ belongs to the the support of $f(\cdot)$ and $\tau_n >0$, by Remark \ref{remark_b_square_positive}, $b^2(x,\tau_n)>0$. Now, applying Lindeberg-Levy Theorem for triangular arrays \citep[Theorem 27.2]{Billingsley-2012} with $r_n=n$, $s_n = \sqrt{n} \, b(x,\tau_n)$ and
\begin{equation*}
S_n = \sum_{j=1}^n \left[ \mathcal{J}_{x,\tau_n}(X_j) - LLD(x,\tau_n) \right], \quad \text{it follows that,}
\end{equation*}
\begin{equation} \label{Lindeberg-Levy_theorem_for_triangular_arrays}
\sqrt{n} \, \frac{1}{n} \, \sum_{j=1}^n \left[ \mathcal{J}_{x,\tau_n}(X_j) - LLD(x,\tau_n) \right]/b(x,\tau_n) \xrightarrow[ n \rightarrow \infty]{d} N(0,1),
\end{equation}
provided the Lindeberg condition \citep[Equation (27.8)]{Billingsley-2012}
\begin{equation} \label{lindeberg_condition}
\lim_{n \to \infty} \frac{1}{n \, b^2(x,\tau_n)} \sum_{k=1}^n \int_{A_{\epsilon,n,k}} \left( \mathcal{J}_{x,\tau_k}(x_1) - LLD(x,\tau_k) \right)^2 \, f(x_1) \, dx_1 = 0
\end{equation}
holds for all $\epsilon>0$, where
\begin{equation*}
A_{\epsilon,n,k,x} \coloneqq \{ x_1 \in \mathbb{R}^p \, : \, \abs{ \mathcal{J}_{x,\tau_k}(x_1)-LLD(x,\tau_k)} \geq \epsilon \, \sqrt{n} \, b(x,\tau_n)  \}.
\end{equation*}
Notice that, due to $\sqrt{n} \, \tau_n^{\frac{3}{2}p} \xrightarrow[n \to \infty]{} \infty$ and Lemma \ref{lemma_order_of_convergence_b_square}, $\sqrt{n} \, b(x,\tau_n) \xrightarrow[n \to \infty]{} \infty$. Let $n^{*} \in \mathbb{N}$ be such that, for all $n \geq n^{*}$, $1<\epsilon \, \sqrt{n} \, b(x,\tau_n)$. Since $\abs{\mathcal{J}_{x,\tau_k}(x_1) - LLD(x,\tau_k)} \leq 1$ for all $x,x_1 \in \mathbb{R}^p$ and $k \in \mathbb{N}$, it follows that $A_{\epsilon,n,k,x}=\emptyset$, for all $n \geq n^{*}$ and all $1 \leq k \leq n$. Thus, \eqref{lindeberg_condition} holds true and we obtain \eqref{Lindeberg-Levy_theorem_for_triangular_arrays}. Finally, let
\begin{equation*}
R_n = R_n(X_1,\dots,X_n) \coloneqq \frac{1}{{n \choose 2}} \sum_{1 \leq i < j \leq n} \left[ \xi_{x,\tau_n}(X_i,X_j) + LLD(x,\tau_n) \right].
\end{equation*}
By \eqref{markov_inequality_remainder_of_local_depth_4} with $k=1$, $x^1$ replaced by $x$ and $\tau^1$ replaced by $\tau_n$,
\begin{equation*}
  P^{\otimes n}( \sqrt{n} R_n > \epsilon) \leq \frac{1}{\epsilon^2} \, \frac{2}{n-1}  E \left[ \left( \xi_{x,\tau_n}(X_1,X_2)+LLD(x,\tau_n) \right)^2 \right],
\end{equation*}
which implies that
\begin{equation*}
  P^{\otimes n} \left( \sqrt{n} \, R_n / b(x,\tau_n) > \epsilon \right) \leq \frac{1}{\epsilon^2} \, \frac{2}{(n-1) \, b^2(x,\tau_n)}  E \left[ \left( \xi_{x,\tau_n}(X_1,X_2)+LLD(x,\tau_n) \right)^2 \right].
\end{equation*}
Since $\sqrt{n} \, b(x,\tau_n) \xrightarrow[n \to \infty]{} \infty$,
\begin{equation} \label{markov_inequality_remainder_of_local_depth_3}
P^{\otimes n}\left( \sqrt{n} \, R_n / b(x,\tau_n) > \epsilon \right)  \xrightarrow[ n \rightarrow \infty]{} 0.
\end{equation}
From \eqref{representation_of_U-statistics_2}, \eqref{Lindeberg-Levy_theorem_for_triangular_arrays}, and \eqref{markov_inequality_remainder_of_local_depth_3}, it follows that
\begin{equation} \label{asymptotic_normality_local_depth_tau_n}
	\sqrt{n} \frac{LLD_n(x,\tau_n) - LLD(x,\tau_n)}{2 \, b(x,\tau_n)} \xrightarrow[ n \rightarrow \infty]{d} N(0, 1).
\end{equation}
Now, using the delta method we obtain
\begin{equation*}
	\sqrt{n} \frac{\sqrt{LLD(x,\tau_n)}}{b(x,\tau_n)} \left( \sqrt{LLD_n(x,\tau_n)} - \sqrt{LLD(x,\tau_n)} \right) \xrightarrow[ n \rightarrow \infty]{d} N\left( 0, 1 \right);
\end{equation*}
equivalently,
\begin{equation} \label{asymptotic_normality_sample_tau_approximation}
  Z_n \coloneqq \sqrt{n} \, \frac{ \tau_n^{2p} \, \Lambda_1 \, f_{\tau_n}(x)}{b(x,\tau_n)} \left( f_{\tau_n,n}(x) - f_{\tau_n}(x) \right) \xrightarrow[ n \rightarrow \infty]{d} N\left( 0, 1 \right).
\end{equation}
To complete the proof, since $x \in S$ and $\tau_n >0$, by Remark \ref{remark_b_square_positive} $b^2(x,\tau_n)>0$.
By Theorem \ref{theorem_local_depth_tau_to_0} (i),
\begin{equation} \label{ratio_tau_approximation_density}
  \frac{f^2_{\tau_n}(x)}{f^2(x)} = \frac{LLD(x,\tau_n)}{\Lambda_1 \, \tau_n^{2p} \, f^2(x)} \xrightarrow[ n \rightarrow \infty]{} 1
\end{equation}
and  by Lemma \ref{lemma_order_of_convergence_b_square}
\begin{equation} \label{b_over_tau_to_3_halves_p}
 \frac{b(x,\tau_n)}{\tau_n^{\frac{3}{2}p}} \xrightarrow[n \to \infty]{} \Lambda_1^{*} f^{\frac{3}{2}}(x) > 0.
\end{equation}
\eqref{ratio_tau_approximation_density} and \eqref{b_over_tau_to_3_halves_p} imply that
\begin{equation*}
 Y_n \coloneqq \frac{b(x,\tau_n)}{ \tau_n^{\frac{3}{2}p} \, f_{\tau_n}(x)} \cdot \frac{1}{\Lambda_1^{*} f^{\frac{1}{2}}(x)} = \frac{b(x,\tau_n)}{ \tau_n^{\frac{3}{2}p}}\cdot \frac{1}{ \Lambda_1^{*} f^{\frac{3}{2}}(x) } \cdot \frac{f(x)}{f_{\tau_n}(x)} \xrightarrow[ n \rightarrow \infty]{} 1.
\end{equation*}
From \eqref{asymptotic_normality_sample_tau_approximation} and Slutsky's Theorem it follows that
\begin{equation*}
Y_n \, Z_n  \xrightarrow[ n \rightarrow \infty]{d} N\left( 0, 1 \right),
\end{equation*}
completing the proof.
\end{proof} \\

\begin{proof}[Proof of Corollary \ref{corollary_asymptotic_normality_of_sample_tau_approximation}]
 We will show that
\begin{equation} \label{ratio_sample_tau_approximation_tau_approximation} 
\frac{f_{\tau_n,n}(x)}{f_{\tau_n}(x)} = \frac{\sqrt{LLD_n(x,\tau_n)}}{\sqrt{LLD(x,\tau_n)}}  \xrightarrow[ n \rightarrow \infty]{d} 1.
\end{equation}
For this, it is enough to verify that
\begin{equation} \label{ratio_sample_local_depth_local_depth}
\frac{LLD_n(x,\tau_n)}{LLD(x,\tau_n)} \xrightarrow[ n \rightarrow \infty]{d} 1.
\end{equation}
Notice that
\begin{equation*}
\frac{LLD_n(x,\tau_n)}{LLD(x,\tau_n)} = \frac{\tau_n^{2p}}{LLD(x,\tau_n)} \cdot \frac{LLD_n(x,\tau_n)-LLD(x,\tau_n)}{\tau_n^{2p}} + 1,
\end{equation*}
where, by Theorem \ref{theorem_local_depth_tau_to_0} (i),
\begin{equation*}
\frac{LLD(x,\tau_n)}{\tau_n^{2p}} \xrightarrow[n \to \infty]{} \Lambda_1 f^2(x) > 0.
\end{equation*}
On the other hand,
\begin{equation*}
\frac{LLD_n(x,\tau_n)-LLD(x,\tau_n)}{\tau_n^{2p}} = \frac{b(x,\tau_n)}{\tau_n^{\frac{3}{2}p}} \cdot \frac{1}{\sqrt{n} \tau_n^{\frac{1}{2}p}} \cdot \sqrt{n} \frac{LLD_n(x,\tau_n)-LLD(x,\tau_n)}{b(x,\tau_n)},
\end{equation*}
where $\sqrt{n} \, \tau_n^{\frac{1}{2}p}  \xrightarrow[n \to \infty]{} \infty$, by Lemma \ref{lemma_order_of_convergence_b_square},
\begin{equation*}
 \frac{b(x,\tau_n)}{\tau_n^{\frac{3}{2}p}} \xrightarrow[n \to \infty]{}  \Lambda_1^{*} f^{\frac{3}{2}}(x) >0,
\end{equation*}
and by \eqref{asymptotic_normality_local_depth_tau_n}
\begin{equation*}
	\sqrt{n} \frac{LLD_n(x,\tau_n) - LLD(x,\tau_n)}{b(x,\tau_n)} \xrightarrow[ n \rightarrow \infty]{d} N(0, 4).
\end{equation*}
Now applying Slutsky's Theorem
\begin{equation*}
  \frac{LLD_n(x,\tau_n)-LLD(x,\tau_n)}{\tau_n^{2p}} \xrightarrow[ n \rightarrow \infty]{d} 0,
\end{equation*}
and, hence \eqref{ratio_sample_local_depth_local_depth} and \eqref{ratio_sample_tau_approximation_tau_approximation} hold. Now, \eqref{ratio_sample_tau_approximation_tau_approximation} and \eqref{ratio_tau_approximation_density} imply that
\begin{equation} \label{ratio_sample_tau_approximation_density}
\frac{f_{\tau_n,n}(x)}{f(x)} = \frac{f_{\tau_n,n}(x)}{f_{\tau_n}(x)} \cdot \frac{f_{\tau_n}(x)}{f(x)}  \xrightarrow[ n \rightarrow \infty]{d} 1.
\end{equation}
By Theorem \ref{theorem_asymptotic_normality_of_sample_tau_approximation}
\begin{equation*}
Z^{*}_n \coloneqq \sqrt{n} \, \tau_n^{\frac{1}{2}p} \frac{1}{\sqrt{f(x)}} ( f_{\tau_n,n}(x)-f_{\tau_n}(x) ) \xrightarrow[ n \rightarrow \infty]{d} N\left( 0, \frac{\Lambda_1^{*2}}{\Lambda_1^2} \right),
\end{equation*}
where we can write $ f_{\tau_n,n}(x)-f_{\tau_n}(x)$ as
\begin{equation*}
f_{\tau_n,n}(x)-f_{\tau_n}(x) = \left( \sqrt{f_{\tau_n,n}(x)}-\sqrt{f_{\tau_n}(x)} \right) \left( \sqrt{f_{\tau_n,n}(x)}+\sqrt{f_{\tau_n}(x)} \right).
\end{equation*}
Also, by \eqref{ratio_tau_approximation_density} and \eqref{ratio_sample_tau_approximation_density}
\begin{equation*}
Y^{*}_n \coloneqq \frac{\sqrt{f_{\tau_n,n}(x)}+\sqrt{f_{\tau_n}(x)}}{\sqrt{f(x)}} = \sqrt{\frac{f_{\tau_n,n}(x)}{f(x)}} + \sqrt{ \frac{f_{\tau_n}(x)}{f(x)}} \xrightarrow[ n \rightarrow \infty]{d} 2,
\end{equation*}
and by another application of Slutsky's Theorem
\begin{equation*}
\frac{Z^{*}_n}{Y^{*}_n} \xrightarrow[ n \rightarrow \infty]{d} N\left( 0, \frac{\Lambda_1^{*2}}{4 \Lambda_1^2} \right),
\end{equation*}
proving the Corollary.
\end{proof} \\

Before proving Proposition \ref{proposition_convergence_of_level_sets}, we state a result concerning the super level sets of the local depth, which is an immediate consequence of the general theory on DFs (see \citep{Zuo-Serfling-2000-c}). Given a level of localization specified by $\tau>0$, a common approach in cluster analysis exploits LDFs to define clusters as the connected components of 
\begin{equation*}
	R^{*\alpha}_{\tau} \coloneqq \{ x \in \mathbb{R}^p \, : \, LLD(x,\tau) \geq \alpha \} 
\end{equation*}
for some $\alpha>0$. In practice, $R^{*\alpha}_{\tau}$ is obtained through its sample counterpart
\begin{equation*}
	R^{*\alpha}_{\tau, n} \coloneqq \{ x \in \mathbb{R}^p \, : \, LLD_n(x, P, \tau) \geq \alpha \}. 
\end{equation*}
By Proposition \ref{proposition_local_depth} (ii) and Theorem \ref{theorem_uniform_consistency}, we can apply \citet[Theorem 4.1]{Zuo-Serfling-2000-c} to obtain the following result which is required in the proof of Proposition \ref{proposition_convergence_of_level_sets}.

\begin{proposition}  \label{proposition_convergence_of_regions}
For all $\epsilon > 0$, $0 < \delta < \epsilon$, $\alpha > 0$ and sequence $\alpha_n \rightarrow \alpha$ there exists an $n_0$ such that, for all $n \geq n_0$, $R_{\tau}^{*\alpha+\epsilon} \subset R_{\tau, n}^{*\alpha_n+\delta} \subset R_{\tau, n}^{*\alpha_n} \subset R_{\tau, n}^{*\alpha_n-\delta}  \subset R_{\tau}^{*\alpha-\epsilon}$. Furthermore, if $P( \{ x \in \mathbb{R}^p \, : \, LLD(x,\tau) = \alpha \} ) = 0$, then $R_{\tau, n}^{*\alpha_n} \xrightarrow[ \substack{ n \rightarrow \infty } ]{} R_{\tau}^{*\alpha}$ a.s.\
\end{proposition}

\begin{proof}[Proof of Proposition \ref{proposition_convergence_of_level_sets}]
Since $\alpha_j \rightarrow \alpha$ as $j \to \infty$ and using Theorem \ref{theorem_uniform_convergence_of_tau_approximation} (i), we have, for all $m \in \mathbb{N}$, there exists a $k \in \mathbb{N}$ such that $\abs{\alpha_j - \alpha} < \frac{1}{m}$, for all $j \geq k$, and $\abs{f_{\tau_j}(x)-f(x)} < \frac{1}{m}$, for all $j \geq k$ and $x \in \mathbb{R}^p$. It follows, using the definition of limit inferior
\begin{equation*}
\begin{split}
\liminf_{k \to \infty} R_{\tau_k}^{\alpha_k} &= \cup_{k=1}^{\infty} \cap_{j=k}^{\infty} \{ x \in \mathbb{R}^p \, : \, f_{\tau_j}(x) \geq \alpha_j \} \\
&\supset \{ x \in \mathbb{R}^p \, : \, f(x) > \alpha+\frac{2}{m} \} = \mathring{R^{\alpha+\frac{2}{m}}} \uparrow_{m \to \infty} \cup_{m=1}^{\infty} \mathring{R^{\alpha+\frac{2}{m}}} = \mathring{R^{\alpha}}
\end{split}
\end{equation*}
and 
\begin{equation*}
\begin{split}
\limsup_{k \to \infty} R_{\tau_k}^{\alpha_k} &= \cap_{k=1}^{\infty} \cup_{j=k}^{\infty} \{ x \in \mathbb{R}^p \, : \, f_{\tau_j}(x) \geq \alpha_j \} \\
&\subset \{ x \in \mathbb{R}^p \, : \, f(x) \geq \alpha-\frac{2}{m} \} = R^{\alpha-\frac{2}{m}} \downarrow_{m \to \infty} \cap_{m=1}^{\infty} R^{\alpha-\frac{2}{m}} = R^{\alpha},
\end{split}
\end{equation*}
establishing \eqref{inclusions_level_sets}. For the second part, using $R^{\alpha}=L^{\alpha} \cup \mathring{R^{\alpha}}$ and \eqref{inclusions_level_sets}, it follows that
\begin{equation*}
\begin{split}
\liminf_{ k \to \infty } R^{\alpha_k}_{\tau_k}  &= \mathring{R^{\alpha}} \cup \left( \liminf_{ k \to \infty } R^{\alpha_k}_{\tau_k}  \cap L^{\alpha} \right) \text{ and } \\
\limsup_{ k \to \infty } R^{\alpha_k}_{\tau_k}  &= \mathring{R^{\alpha}} \cup \left( \limsup_{ k \to \infty } R^{\alpha_k}_{\tau_k}  \cap L^{\alpha} \right), \text{ where }
\end{split}
\end{equation*}
\begin{equation*}
\liminf_{ k \to \infty } R^{\alpha_k}_{\tau_k}  \cap L^{\alpha} \subset \limsup_{ k \to \infty } R^{\alpha_k}_{\tau_k}  \cap L^{\alpha} \subset L^{\alpha}
\end{equation*}
are sets of Lebesgue measure $0$. Therefore,
\begin{equation*}
\liminf_{ k \to \infty } R^{\alpha_k}_{\tau_k} = \limsup_{ k \to \infty } R^{\alpha_k}_{\tau_k} = R^{\alpha}
\end{equation*}
except for a set of Lebesgue measure $0$ and we obtain \eqref{convergence_level_sets}. \\
Next, to verify \eqref{convergence_sample_level_sets},  observe that $\lambda(L^{\alpha})=0$ implies 
\begin{equation*}
P(\{ x \in \mathbb{R}^p \, : \, f(x)=\alpha \}) = \int_{\{ x \in \mathbb{R}^p \, : \, f(x)=\alpha \}} f(y) \, dy = \alpha \lambda(L^{\alpha})=0,
\end{equation*}
and, therefore, \eqref{convergence_level_sets} holds with $\alpha_k \equiv \alpha$ implying
\begin{equation} \label{convergence_of_level_alpha_n_constant}
\lim_{ k \to \infty } R_{\tau_k}^{\alpha}=R^{\alpha} \textit{ a.s.}
\end{equation}
Also note that (using the notation of Proposition \ref{proposition_convergence_of_regions}) $R_{\tau_k}^{\alpha}=R_{\tau_k}^{*\alpha^2\tau_{k}^{2p}\Lambda_1}$, $R_{\tau_k}^{\alpha_n}=R_{\tau_k}^{*\alpha_{n}^2\tau_{k}^{2p}\Lambda_1}$ and $R_{\tau_k,n}^{\alpha_n}=R_{\tau_k,n}^{*\alpha_{n}^2\tau_{k}^{2p}\Lambda_1}$. Finally, since $\lambda(L_{\tau_{k}}^{\alpha}) = 0$ is equivalent to $\lambda(\{ x \in \mathbb{R}^p \, : \, LLD(x,\tau_k)=\alpha^2 \tau_{k}^{2p} \Lambda_1 \}) = 0$, and since $P$ is absolutely continuous with respect to the Lebesgue measure, it follows that $P(\{ x \in \mathbb{R}^p \, : \, LLD(x,\tau_{k})=\alpha^2 \tau_{k}^{2p} \Lambda_1 \})=0$. Now, applying Proposition \ref{proposition_convergence_of_regions}, we obtain 
\begin{equation} \label{convergence_of_sample_level_sets_tau_n}
\lim_{ n \to \infty } R^{\alpha_n}_{\tau_k,n} = R^{\alpha}_{\tau_k} \textit{ a.s.}
\end{equation}
The proof now follows from \eqref{convergence_of_level_alpha_n_constant} and \eqref{convergence_of_sample_level_sets_tau_n}.
\end{proof} \\

\begin{proof}[Proof of Lemma \ref{lemma_support_S_and_S_tau}]
We first observe that $x \in S_{\tau}$ if and only if $f_{\tau}(x)>0$ if and only if $LLD(x,\tau)>0$. Proposition \ref{proposition_local_depth} (i) implies that for $x \in \mathbb{R}^p$, $LLD(x,\tau_1) \leq LLD(x,\tau_2)$, from which it follows that $S_{\tau_1} \subset S_{\tau_2}$. Next, suppose that $f(\cdot)$ is continuous and let $x \in S$ and $\tau>0$. Since $f(\cdot)$ is continuous, $S$ is open and there exists an $\epsilon>0$ such that the closed ball $\overline{B}_{\epsilon}(x)$ is contained in $S$. Also, since $Z_{\tau}(x)$ is closed and bounded, it follows that $Z_{\tau}^{\epsilon}(x) \coloneqq Z_{\tau}(x) \cap \left( \overline{B}_{\epsilon}(x) \times \overline{B}_{\epsilon}(x) \right)$ is a compact subset of $S \times S$. Since $\lambda^{\otimes 2}(Z_{\tau}^{\epsilon}(x)) > 0$, it follows that
\begin{equation*}
\begin{split}
LLD(x,\tau) &= \int_{Z_{\tau}(x)} f(x_1) f(x_2) \, dx_1 dx_2 \geq \int_{Z_{\tau}^{\epsilon}(x) } f(x_1) f(x_2) \, dx_1 dx_2 > 0.
\end{split}
\end{equation*}
Thus $x \in S_{\tau}$ and $S \subset S_{\tau}$. For the last part, since the sets $\{ S_{\tau} \}_{\tau>0}$ are monotonically decreasing with $\tau$, we have that $\lim_{\tau \to 0^{+}} S_{\tau} = \cap_{\tau > 0} S_{\tau} \supset S$. For $x \in S$, we have to show that $x \in \cap_{\tau>0} S_{\tau}$. As before, there exists an $\epsilon>0$ such that $\overline{B}_{\epsilon}(x) \subset S$ and, for all $0<\tau\leq\epsilon$, $Z_{\tau}(x) \subset \overline{B}_{\epsilon}(x) \times \overline{B}_{\epsilon}(x) \subset S \times S$. From the compactness of $Z_{\tau}(x)$ and $\lambda^{\otimes 2}(Z_{\tau}(x))>0$ it follows that, for all $0<\tau\leq\epsilon$, $LLD(x,\tau)>0$; if  $\tau>\epsilon$, by Proposition \ref{proposition_local_depth} (i) we have that $LLD(x,\tau) \geq LLD(x,\epsilon) > 0$. Hence, for all $\tau>0$, $LLD(x,\tau)>0$ and $x \in \cap_{\tau>0} S_{\tau}$.
\end{proof} \\

\begin{proof}[Proof of Theorem \ref{theorem_first_and_second_partial_derivaties}]
  Since $f(\cdot)$ and its first order partial derivatives (second order partial derivatives for the second part) are continuous, by Remark \ref{remark_f_tau_continuous}, $f_{\tau}(\cdot)$ is continuously differentiable in $S_\tau$ (two times continuously differentiable for the second part). Its partial derivatives can be computed following the argument in the proof of Proposition \ref{proposition_local_depth} (iv) (see Appendix \ref{section_proof_proposition_local_depth}). \\
 The $j$-th component of the first partial derivative of $f_{\tau}(\cdot)$ is given by
\begin{equation*}
\partial_j f_{\tau}(x) = \frac{1}{2 \tau^p \sqrt{\Lambda_1}} \frac{1}{\sqrt{LLD(x,\tau)}} \partial_j LLD(x,\tau),
\end{equation*}
where, from the proof of Proposition \ref{proposition_local_depth} (iv), we have that
\begin{equation*}
	\partial_j LLD(x,\tau) = \int_{Z_{1}(0)} g_{*}^{j}(x;x_1,x_2) \, dx_1 dx_2
\end{equation*}
with
\begin{equation*}
g_{*}^{j}(x;x_1,x_2) = \tau^{2p} [ \partial_j f (x+\tau x_1) f(x+\tau x_2) + f(x + \tau x_1) \partial_j f(x + \tau x_2) ].
\end{equation*}
Therefore,
\begin{equation*}
\begin{split}
\partial_j f_{\tau}(x) &= \frac{1}{2 \sqrt{\Lambda_1}} \frac{1}{\sqrt{\frac{1}{\tau^{2p}} LLD(x,\tau) }} \left( \frac{\partial_j LLD(x,\tau)}{\tau^{2p}} \right) \\
& \xrightarrow[ \tau \rightarrow 0^{+} ]{} \frac{1}{2 \Lambda_1 f(x)} \left( 2 \, \Lambda_1 \, \partial_j f(x) \, f(x) \right) = \partial_j f(x).
\end{split}
\end{equation*}
The second partial derivative is computed by differentiating $\partial_j f_{\tau}$ with respect to the $i$-th component, namely
\begin{equation*}
\partial_i \partial_j f_{\tau}(x) = \frac{1}{2 \tau^p \sqrt{\Lambda_1}} \left[ \partial_i \left( \frac{1}{ \sqrt{LLD(x,\tau)}} \right) \partial_j LLD(x,\tau) + \frac{1}{ \sqrt{LLD(x,\tau)}} \partial_i \partial_j LLD(x,\tau) \right],
\end{equation*}
where
\begin{equation*}
\partial_i \left( \frac{1}{ \sqrt{LLD(x,\tau)}} \right) = - \frac{1}{2 (LLD(x,\tau))^{\frac{3}{2}}} \partial_i LLD(x,\tau)
\end{equation*}
and, by the proof of Proposition \ref{proposition_local_depth} (iv),
\begin{equation*}
\begin{split}
\partial_i \partial_j LLD(x,\tau) &= \int_{Z_{1}(0)} g_{*}^{ij}(x;x_1,x_2) \, dx_1 dx_2
\end{split}
\end{equation*}
with
\begin{equation*}
\begin{split}
g_{*}^{ij}(x;x_1,x_2) &= \tau^{2p} [ \partial_i \partial_j f (x+\tau x_1) f(x+\tau x_2) + \partial_j f (x+\tau x_1) \partial_i f(x+\tau x_2)  \\
& + \partial_i f(x + \tau x_1) \partial_j f(x + \tau x_2) + f(x + \tau x_1) \partial_i \partial_j f(x + \tau x_2) ].
\end{split}
\end{equation*}
Hence,
\begin{equation*}
\begin{split}
\partial_i \partial_j f_{\tau}(x) = \frac{1}{2 \sqrt{\Lambda_1}} & \left[ - \frac{1}{2 \left( \frac{1}{\tau^{2p}} LLD(x,\tau) \right)^{\frac{3}{2}}} \left( \frac{\partial_i LLD(x,\tau)}{\tau^{2p}} \right) \left( \frac{\partial_j LLD(x,\tau)}{\tau^{2p}} \right) \right. \\
& \left. + \frac{1}{ \sqrt{ \frac{1}{\tau^{2p}} LLD(x,\tau)}} \left( \frac{\partial_i \partial_j LLD(x,\tau)}{\tau^{2p}} \right) \right] \\
\xrightarrow[ \tau \rightarrow 0^{+} ]{} \frac{1}{2 } & \left[    - \frac{1}{2 f^{3}(x)} \left( 2 \partial_i f(x) \, f(x) \right) \left( 2 \partial_j f(x) \, f(x) \right) \right. \\
& \left. + \frac{1}{f(x)} \left( 2 \partial_i \partial_j f(x) \, f(x) + 2 \partial_i f(x) \, \partial_j f(x) \right) \right] = \partial_i \partial_j f(x).
\end{split}
\end{equation*}
\end{proof} \\

\begin{proof}[Proof of Lemma \ref{lemma_R_alpha_tau_bounded}]
Since $x \in (R^{\alpha})^{-\tau}$ satisfies $\inf_{y \in \mathbb{R}^p \setminus R^{\alpha}} \norm{x-y} > \tau$, we have that $\overline{B}_{\tau}(x) \subset R^{\alpha}$. From $Z_{\tau}(x) \subset \overline{B}_{\tau}(x) \times \overline{B}_{\tau}(x)  \subset R^{\alpha} \times R^{\alpha}$, it follows that
\begin{equation} \label{tau_approximation_leq_than_alpha}
f_{\tau}(x) = \frac{1}{\tau^p \Lambda_1} \sqrt{LLD(x,\tau)} \geq \alpha,
\end{equation}
and therefore $x \in R_{\tau}^{\alpha}$. Next, let $x \in R_{\tau}^{\alpha}$. Then there exists $(x_1,x_2) \in Z_{\tau}(x)$ such that $f(x_1) f(x_2) \geq \alpha^2$. In particular, since $Z_{\tau}(x) \subset \overline{B}_{\tau}(x) \times \overline{B}_{\tau}(x)$, there exists a point $z \in \overline{B}_{\tau}(x)$ with $f(z) \geq \alpha$. Hence $z \in R^{\alpha}$, and since $z \in \overline{B}_{\tau}(x)$, $\norm{x-z} \leq \tau$, implying that $x \in (R^{\alpha})^{+\tau}$. Finally, suppose that for $\alpha>0$, $R^{\alpha}$ is bounded. Then, there exists $r>0$ such that $R^{\alpha} \subset \overline{B}_{r}(0)$. It follows that, for $\tau>0$, $x \in R^{\alpha}_{\tau} \subset (R^{\alpha})^{+\tau}$ satisfies $\norm{x} \leq \inf_{y \in R^{\alpha}} \left( \norm{y}+\norm{y-x} \right) \leq r + \inf_{y \in R^{\alpha}} \norm{y-x} \leq r + \tau$. Hence $R^{\alpha}_{\tau} \subset \overline{B}_{r+\tau}(0)$ is bounded.
\end{proof} \\

\begin{proof}[Proof of Theorem \ref{theorem_convergence_solution_estimated_gradient_system}]
Fix $x \in S$ and $t \geq 0$. Let $\alpha>0$ be such that $x \in R^{\alpha}$. Since $R^{\alpha}\subset S$ is compact, $\mathbb{R}^p \setminus S$ is closed and these two sets are disjoint, we have that $\dist(R^{\alpha},\mathbb{R}^p \setminus S) > 0$, where, for $A,B \subset \mathbb{R}^p$, $\dist(A,B) \coloneqq \inf_{y \in A, \, z \in B} \norm{y-z}$. Let $\tau^{*} \coloneqq \dist(R^{\alpha},\mathbb{R}^p \setminus S)/3$ and notice that by definition
\begin{align*}
\dist((R^{\alpha})^{+\tau^{*}},\mathbb{R}^p \setminus S) &= \inf_{y \in \mathbb{R}^p \setminus S, \, z \in (R^{\alpha})^{+\tau^{*}}} \norm{y-z} \\
& = \inf_{y \in \mathbb{R}^p \setminus S} \inf_{z \in \mathbb{R}^p \, : \, \inf_{w \in R^{\alpha}} \norm{w-z} \leq \tau^{*} } \norm{y-z}.
\end{align*}
By the triangle inequality, we have that for $w \in R^{\alpha}$
\begin{equation*}
\norm{y-z} \geq \norm{y-w} - \norm{w-z} \geq \norm{y-w} - \tau^{*}.
\end{equation*}
It follows that
\begin{equation} \label{dist_compact_subset_of_S}
\dist((R^{\alpha})^{+\tau^{*}},\mathbb{R}^p \setminus S) \geq \dist(R^{\alpha},\mathbb{R}^p \setminus S) - \tau^{*} = 2 \tau^{*}>0.
\end{equation}
Lemma \ref{lemma_R_alpha_tau_bounded} implies that for all $0 < \tau \leq \tau^{*}$, $R_{\tau}^{\alpha} \subset (R^{\alpha})^{+\tau} \subset (R^{\alpha})^{+\tau^{*}} \subset S$. In the rest of the proof, we suppose that $0 < \tau \leq \tau^{*}$. Also, for all $s \geq 0$, $u_{x}(s) \in R^{\alpha}$ and $u_{x,\tau}(s) \in R^{\alpha}_{\tau}$; as shown in Subsection \ref{subsection_mathematical_background}, the solutions of the gradient system \eqref{gradient_system} cannot leave the regions $R^{\alpha}$, and the same is true for the gradient system \eqref{gradient_system_tau} and $R^{\alpha}_{\tau}$. In particular, for all $0 \leq s \leq t$, $u_{x}(s),u_{x,\tau}(s) \in K$, where $K \coloneqq (R^{\alpha})^{+\tau^{*}}$ is a compact subset of $S$. \\
Now noticing that the integral of a vector is the vector of the integrals of its components, we obtain
\begin{equation*}
u_{x, \tau}(t) - u_x(t) = \int_{0}^{t} \nabla f_{\tau} (u_{x, \tau}(s)) - \nabla f (u_{x}(s)) \, d s.
\end{equation*}
Next by adding and subtracting $\nabla f_{\tau} (u_{x}(s))$ inside the integral and taking the Euclidean norm on both sides we see that
\begin{equation*}
\norm{ u_{x,\tau}(t) - u_x(t) } \leq \int_{0}^{t} \norm{ \nabla f_{\tau} (u_{x, \tau}(s)) - \nabla f_{\tau} (u_{x}(s)) } \, d s +\int_{0}^{t} \norm{ \nabla f_{\tau} (u_{x}(s)) - \nabla f (u_{x}(s)) } \, d s.
\end{equation*}
Since $\nabla f_{\tau}$ is locally Lipschitz in $S$, it is Lipschitz in the compact subset $K$; that is, there exists a constant $L_{\tau} < \infty$ such that for all $y,z \in K$
\begin{equation} \label{gradient_of_f_tau_lipschitz_in_compact_set}
\norm{\nabla f_\tau(y) - \nabla f_\tau(z)} \leq L_{\tau} \norm{y-z}.
\end{equation}
It follows that
\begin{equation*}
\norm{ u_{x,\tau}(t) - u_x(t) } \leq L_{\tau} \int_{0}^{t} \norm{ u_{x, \tau}(s) - u_{x}(s) } \, d s + \int_{0}^{t} \norm{ \nabla f_{\tau} (u_{x}(s)) - \nabla f (u_{x}(s)) } \, d s.
\end{equation*}
We now apply Gr\"onwall's inequality \citep{Hale-1980}[Corollary 6.6] with $a=0$, $\beta(s)=L_{\tau}$, $0 \leq s \leq t$, $\alpha=\int_{0}^{t} \norm{ \nabla f_{\tau} (u_{x}(s)) - \nabla f (u_{x}(s)) } \, d s$ and $\varphi(t)=\norm{ u_{x,\tau}(t) - u_x(t) }$, and obtain that
\begin{equation} \label{gradient_system_solutions_converge_for_tau_to_0}
\norm{ u_{x,\tau}(t) - u_x(t) } \leq e^{ L_{\tau}} \int_{0}^{t} \norm{ \nabla f_{\tau} (u_{x}(s)) - \nabla f (u_{x}(s)) } \, d s.
\end{equation}
To conclude the proof we need to show that this converges to $0$ as $\tau \to 0^{+}$. To this end,  since $\nabla f$ is also locally Lipschitz in $S$, there exists a constant $L < \infty$ such that, for all $y,z \in K$
\begin{equation} \label{gradient_of_f_lipschitz_in_compact_set}
\norm{\nabla f(y) - \nabla f(z)} \leq L \norm{y-z}.
\end{equation}
Also by Theorem \ref{theorem_first_and_second_partial_derivaties}, $\nabla f_{\tau}$ converges pointwise to $\nabla f$, for all $y,z \in K$ 
\begin{equation*}
\norm{\nabla f_\tau(y) - \nabla f_\tau(z)} \xrightarrow[ \tau \rightarrow 0^{+} ]{} \norm{\nabla f(y) - \nabla f(z)},
\end{equation*}
and by \eqref{gradient_of_f_lipschitz_in_compact_set} for all $y,z \in K$ with $y \neq z$
\begin{equation*}
\lim_{\tau \to 0^{+}} \frac{\norm{\nabla f_\tau(y) - \nabla f_\tau(z)}}{\norm{y-z}} = \frac{\norm{\nabla f(y) - \nabla f(z)}}{\norm{y-z}} \leq L.
\end{equation*}
It follows that $\{ L_{\tau} \}_{0 < \tau \leq \tau^{*}}$ in \eqref{gradient_of_f_tau_lipschitz_in_compact_set} can be chosen in such a way that in the limit does not blow up. Therefore, from \eqref{gradient_system_solutions_converge_for_tau_to_0}, it follows that 
\begin{equation*}
\lim_{\tau \to 0^{+}} \norm{ u_{x,\tau}(t) - u_x(t) } = 0,
\end{equation*}
if we can show that
\begin{equation} \label{gradient_system_integral_of_solutions_converge_for_tau_to_0}
\int_{0}^{t} \norm{ \nabla f_{\tau} (u_{x}(s)) - \nabla f (u_{x}(s)) } \, d s \xrightarrow[ \tau \rightarrow 0^{+} ]{} 0.
\end{equation}
To show this, we first enlarge the compact set $K$ by $\tau^{*}$ in such a way that it is still contained in $S$ by considering the set $(K)^{+\tau^{*}} \subset (R^{\alpha})^{+2 \tau^{*}}$. As in \eqref{dist_compact_subset_of_S}, we see that
\begin{equation*}
\dist((R^{\alpha})^{+2\tau^{*}},\mathbb{R}^p \setminus S) \geq \dist(R^{\alpha},\mathbb{R}^p \setminus S) - 2\tau^{*} = \tau^{*}>0,
\end{equation*}
and $(K)^{+\tau^{*}}$ is indeed a compact subset of $S$. Furthermore, for all $y \in K$,  $\overline{B}_{\tau}(y) \subset \overline{B}_{\tau^{*}}(y) \subset (K)^{+\tau^{*}} \subset S$, and in particular, $Z_{\tau}(y) \subset \overline{B}_{\tau^{*}}(y) \times \overline{B}_{\tau^{*}}(y) \subset S \times S$. Now, by proceeding as in the proof of Theorem \ref{theorem_first_and_second_partial_derivaties}, we see that for $y \in K$, the $j$-th partial derivative of $f_{\tau}(\cdot)$ at $y$ is given by
\begin{equation*}
\partial_j f_{\tau}(y) = \frac{1}{2 \tau^{2p} \Lambda_1} \frac{1}{\sqrt{\frac{1}{\tau^{2p} \Lambda_1} LLD(x,\tau)}} \int_{Z_{\tau}(x)} \partial_j f(x_1) f(x_2) + f(x_1) \partial_j f(x_2) \, dx_1 dx_2
\end{equation*}
and
\begin{equation*}
\abs{\partial_j f_{\tau}(y)} \leq \frac{\alpha_0 \cdot \alpha_1^{j}}{\beta_0} < \infty
\end{equation*}
where $\alpha_0 \coloneqq \max_{z \in (K)^{+\tau^{*}}} f(z)$, $\beta_0 \coloneqq \min_{z \in (K)^{+\tau^{*}}} f(z)$ and $\alpha_1^j \coloneqq \max_{z \in (K)^{+\tau^{*}}} \partial_j f(z)$ satisfy $0<\alpha_0, \beta_0, \alpha_1^j < \infty$. It follows that, for $y \in K$,
\begin{align*}
\norm{\nabla f_{\tau}(y) - \nabla f(y)} &\leq \norm{\nabla f_{\tau}(y)} + \norm{\nabla f(y)} \\
&\leq \sqrt{p} \left( \frac{\alpha_0}{\beta_0} \right) \left( \sum_{j=1}^p \left( \alpha_1^{j} \right)^2 \right)^{1/2} + \sqrt{p} \left( \sum_{j=1}^p \left( \alpha_1^{j} \right)^2 \right)^{1/2} < \infty.
\end{align*}
Therefore, for all $0 \leq s \leq t$, $\norm{ \nabla f_{\tau} (u_{x}(s)) - \nabla f (u_{x}(s)) }$ is bounded and by Theorem \ref{theorem_first_and_second_partial_derivaties}, for all $0 \leq s \leq t$, $\norm{ \nabla f_{\tau} (u_{x}(s)) - \nabla f (u_{x}(s)) } \xrightarrow[ \tau \rightarrow 0^{+} ]{} 0$. Now, \eqref{gradient_system_integral_of_solutions_converge_for_tau_to_0} follows using LDCT.
\end{proof} \\

\begin{proof}[Proof of Theorem \ref{theorem_local_central_symmetry}]
As in the proof of Theorem \ref{theorem_first_and_second_partial_derivaties} we see that for $j=1,\dots,p$
\begin{equation*}
\partial_j f_{\tau}(x) = \frac{1}{2 \tau^p \sqrt{\Lambda_1}} \frac{1}{\sqrt{LLD(x,\tau)}} \partial_j LLD(x,\tau),
\end{equation*}
where
\begin{equation*}
	\partial_j LLD(x,\tau) = \int_{Z_{\tau}(0)} \partial_j f (x+x_1) f(x+x_2) + f(x+x_1) \partial_j f (x+x_2) \, dx_1 dx_2.
\end{equation*}
Also from Lemma \ref{lemma_symmetries_of_z_tau_x}, it follows that $Z_{\tau}(0)$ satisfies the following symmetry properties
\begin{equation} \label{symmetry2_of_Z_tau_0}
(x_1, x_2) \in Z_{\tau}(0) \Longleftrightarrow (x_2, x_1) \in Z_{\tau}(0)
\end{equation}
and
\begin{equation} \label{symmetry3_of_Z_tau_0}
	(x_1, x_2) \in Z_{\tau}(0) \Longleftrightarrow (-x_1, -x_2) \in Z_{\tau}(0).
\end{equation}
From \eqref{symmetry2_of_Z_tau_0} it follows that
\begin{equation*}
	\partial_j LLD(x,\tau) = 2 \int_{Z_{\tau}(0)} \partial_j f (x+x_1) f(x+x_2) \, dx_1 dx_2
\end{equation*}
which implies that $\partial_j f_{\tau}(\mu)=0$ if and only if 
\begin{equation*}
	\int_{Z_{\tau}(0)} \partial_j f (\mu+x_1) f(\mu+x_2) \, dx_1 dx_2 = 0,
\end{equation*}
and hence \eqref{equivalent_condition_for_mu_stationary_point} holds. Finally, if $f(\cdot)$ is $\tau$-centrally symmetric about $\mu$, then, for all $y \in \mathbb{R}^p$ with $\norm{y} \leq \tau$, $f(\mu+y)=f(\mu-y)$ and $\partial_j f(\mu-y) = -\partial_j f(\mu+y)$. By the change of variable in \eqref{symmetry3_of_Z_tau_0} it follows that, for all $1 \leq j \leq p$,
\begin{align*}
  \partial_j LLD(\mu,\tau) &= 2 \int_{Z_{\tau}(0)} \partial_j f (\mu+x_1) f(\mu+x_2) \, dx_1 dx_2 \\
  &= 2 \int_{Z_{\tau}(0)} \partial_j f(\mu-x_1) f(\mu-x_2) \, dx_1 dx_2 \\
  &= -2 \int_{Z_{\tau}(0)} \partial_j f(\mu+x_1) f(\mu+x_2) \, dx_1 dx_2 = 0,
\end{align*}
and therefore $\nabla f_{\tau}(\mu)=0$.
\end{proof} \\

\begin{proof}[Proof of Theorem \ref{theorem_modes_local_central_symmetry}]
Since $f(\cdot)$ is $\tau$-centrally symmetric about $m$, as in the proof of Theorem \ref{theorem_local_central_symmetry}, we see that 
\begin{equation} \label{derivative_local_depth_tau_symmetry}
    \partial_j LLD(m,\tau)=0 \text{ for } j=1,\dots,p
\end{equation}
and hence $m$ is a stationary point for $f_{\tau}(\cdot)$. Now, as in the proof of Theorem \ref{theorem_first_and_second_partial_derivaties}, we see that \eqref{derivative_local_depth_tau_symmetry} implies also that for $i,j=1,\dots,p$
\begin{equation*}
\partial_i \partial_j f_{\tau}(m) = \frac{1}{2 \tau^p \sqrt{\Lambda_1}} \frac{1}{ \sqrt{LLD(m,\tau)}} \partial_i \partial_j LLD(m,\tau),
\end{equation*}
where
\begin{equation*}
\begin{split}
\partial_i \partial_j LLD(m,\tau) &= 2 \int_{Z_{\tau}(0)} \partial_i \partial_j f (m+x_1) f(m+x_2) + \partial_j f (m+x_1) \partial_i f(m+ x_2) \, dx_1 dx_2,
\end{split}
\end{equation*}
due to the symmetry properties of $Z_{\tau}(0)$ (see \eqref{symmetry2_of_Z_tau_0}). Denoting by $H_{f_{\tau}}$ the Hessian matrix of $f_{\tau}(\cdot)$ and noticing that the integral of a matrix is the matrix of the integrals, we get that
\begin{equation*}
\begin{split}
H_{f_{\tau}}(m) &= C_{m,\tau} \int_{Z_{\tau}(0)}  H_f (m+x_1) f(m+x_2) + \nabla f(m+x_1) \nabla f(m+x_2)^{\top} \, dx_1 dx_2 \\
&= C_{m,\tau} \int_{Z_{\tau}(0)}  G_f (m+x_1,m+x_2) \, dx_1 dx_2,
\end{split}
\end{equation*}
where
\begin{equation*}
C_{m,\tau} = \frac{1}{\tau^p \sqrt{\Lambda_1}} \frac{1}{ \sqrt{LLD(m,\tau)}}.
\end{equation*}
Since the Hessian is symmetric, there exists an orthogonal matrix $Q$ such that
\begin{equation*}
\begin{split}
D &= Q^{\top} H_{f_{\tau}}(m) Q \\
&= Q^{\top} \left[ C_{m,\tau} \int_{Z_{\tau}(0)}  G_f (m+x_1,m+x_2) \, dx_1 dx_2 \right] Q \\
&= C_{m,\tau} \int_{Z_{\tau}(0)} Q^{\top} G_f (m+x_1,m+x_2) Q  \, dx_1 dx_2
\end{split}
\end{equation*}
is a diagonal matrix. Now, since $G_f (m+x_1,m+x_2)$ is negative (resp.\ positive) definite, for all $y \in \mathbb{R}^p \setminus \{ 0 \}$, $y^{\top} G_f (m+x_1,m+x_2) y < 0$ (resp.\ $>0$), and therefore the diagonal elements of $ Q^{\top} G_f (m+x_1,m+x_2) Q $ are negative (resp.\ positive). It follows that the diagonal elements of $D$ (that is the eigenvalues of $H_{f_{\tau}}(m)$) are negative (resp.\ positive) and $m$ is a mode (resp.\ an antimode) for $f_{\tau}(\cdot)$.
\end{proof}

\section{Necessity of conditions in Theorem \ref{theorem_uniform_convergence_of_tau_approximation}}
\label{sm:section_necessity_of_conditions_in_theorem_uniform_consistency_of_tau_approximation}

We show in this section that continuity is not sufficient in Theorem \ref{theorem_uniform_convergence_of_tau_approximation} (i). For $p=1$ consider the function $\tilde{f} \, : \, \mathbb{R} \rightarrow \mathbb{R}$ with support $\cup_{n=1}^{\infty} \left( n-\frac{1}{2n^3}, n+\frac{1}{2n^3} \right)$ defined by
\begin{equation*}
\tilde{f}(n+x_1) = \tilde{f}(n-x_1)= \begin{cases}
2n^4 \left( \frac{1}{2n^3} - x_1 \right) &\text{ if } 0 \leq x_1 < \frac{1}{2n^3} \\
0 &\text{ if } \frac{1}{2n^3} \leq x_1 \leq \frac{1}{2}.
\end{cases}
\end{equation*}
Notice that, since $\sum_{n=1}^{\infty} \frac{1}{n^2} = \frac{\pi^2}{6}$,
\begin{equation*}
\begin{split}
\int \tilde{f}(x) \, dx = \sum_{n=1}^{\infty} \int_{\left( n-\frac{1}{2n^3}, n+\frac{1}{2n^3} \right)} \tilde{f}(x) \, dx = \sum_{n=1}^{\infty} \frac{1}{2 n^2} = \frac{\pi^2}{12}.
\end{split}
\end{equation*}
Let $c \coloneqq \frac{\pi^2}{12}$. Then $f \, : \, \mathbb{R} \rightarrow \mathbb{R}$ defined by $f(x)=\frac{1}{c} \tilde{f}(x)$ is an unbounded, continuous density function. The $\tau$-approximation of $f(\cdot)$ is given by $f_{\tau}(x)=\frac{1}{\tau} \sqrt{LLD(x,\tau)}$, where, by Corollary \ref{corollary_local_depth_tau_to_0},
\begin{equation*}
\begin{split}
LLD(x,\tau) &= 2 \int_{T_{++}^{\tau}} f(x+x_1) f(x-x_2) \, dx_1 dx_2 \\
&= 2 \int_{0}^{\tau} \left[ \int_{0}^{\tau-x_1} f(x-x_2) \, dx_2 \right] f(x+x_1) \, dx_1.
\end{split}
\end{equation*}
Notice that $f(\cdot)$ is symmetric about $n \in \mathbb{N}$ and for $\frac{1}{n^3} \leq \tau \leq \frac{1}{2}$
\begin{equation*}
\begin{split}
LLD(n,\tau)&= 2 \int_{0}^{\min\left(\tau,\frac{1}{2 n^3}\right)} \left[ \int_{0}^{\min\left(\tau-x_1,\frac{1}{2n^3}\right)} \frac{2n^4}{c} \left( \frac{1}{2n^3} - x_2 \right) \, dx_2 \right] \frac{2n^4}{c} \left( \frac{1}{2n^3} - x_1 \right) \, dx_1 \\
&= \frac{2}{c^2} \int_{0}^{\frac{1}{2 n^3}} \left[ \int_{0}^{\frac{1}{2n^3}} 2n^4 \left( \frac{1}{2n^3} - x_2 \right) \, dx_2 \right] 2n^4 \left( \frac{1}{2n^3} - x_1 \right) \, dx_1 \\
&= \frac{2}{c^2} \int_{0}^{\frac{1}{2 n^3}} \frac{n^2}{2} \left( \frac{1}{2n^3} - x_1 \right) \, dx_1 = \frac{1}{8 c^2 n^4}.
\end{split}
\end{equation*}
For all $0<\tau \leq \frac{1}{2}$ fixed there exists $n \in \mathbb{N}$ such that $\frac{1}{n^2} \leq \tau$, and therefore
\begin{equation*}
\begin{split}
\sup_{x \in \mathbb{R}} \abs{f_{\tau}(x) - f(x)} & \geq \sup_{n \in \mathbb{N} \, : \, \frac{1}{n^2} \leq \tau} \abs{ f_{\tau}(n) - f(n) } = \sup_{n \in \mathbb{N} \, : \, \frac{1}{n^2} \leq \tau } \abs{ \frac{1}{\tau} \frac{1}{2\sqrt{2} c n^2} - n } \\
&\geq \sup_{n \in \mathbb{N} \, : \, \frac{1}{n^2} \leq \tau } \abs{ \frac{1}{2\sqrt{2} c} - n } = \infty.
\end{split}
\end{equation*}
The boundedness assumption in Theorem \ref{theorem_uniform_convergence_of_tau_approximation} (i) prevents $f(\cdot)$ to become arbitrarily large and allows one to show that the above supremum is bounded. On the other hand, uniform continuity ensures that the supremum converges to zero, thus allowing to use LDCT and obtain the statement.

\section{Stationary points for univariate densities} \label{section_stationary_points_for_univariate_densities}
If $p=1$, from Corollary \ref{corollary_local_depth_tau_to_0}, \eqref{simplified_form_local_depth}, we see that 
\begin{equation*}
f_{\tau}(x) = \frac{1}{\tau} \sqrt{LLD(x,\tau)}
\end{equation*}
where
\begin{equation*}
LLD(x,\tau)  =  2 \int_{T_{++}^{\tau}} f(x+x_1) f(x-x_2) \, dx_1 dx_2
\end{equation*}
and
\begin{equation*}
	T_{++}^{\tau} = \{ (x_1, x_2) \, : x_1 \geq 0, \, x_2 \geq 0, \, x_1+x_2 \leq \tau \}.
\end{equation*}
In particular, if  $f(\cdot)$ has a continuous derivative, it follows that
\begin{equation} \label{derivative_f_tau}
	f_{\tau}^{'}(x) = \frac{1}{\tau \sqrt{LLD(x,\tau)}} \int_{T_{++}^{\tau}} f^{'}(x+x_1) f(x-x_2) + f(x+x_1) f^{'}(x-x_2) \, dx_1 dx_2.
\end{equation}
Therefore, the sign of $f_{\tau}(x)$ depends on the sign of $f^{'}(\cdot)$ in the interval $(x-\tau, x+\tau)$. In particular, if $\mu \in \mathbb{R}$ satisfies $f(\mu-x) = f(\mu+x)$ for all $x \in (0, \tau)$, it follows that $f^{'}(\mu-x) = - f^{'}(\mu+x)$, yielding $f_{\tau}^{'}(\mu) = 0$. 

The next result, which also holds for asymmetric densities, shows that, for $p=1$ and $\tau$ small enough, $f_{\tau}(\cdot)$ has a mode in a neighborhood of a mode of $f(\cdot)$.
\begin{proposition} \label{proposition_modes}
Let $p=1$. Suppose that $f(\cdot)$ is twice continuously differentiable in a neighborhood of a mode (resp.\ an antimode) $m$. Then, there is $\tau^{*}>0$ such that, for all $0 < \tau \leq \tau^{*}$, $f_{\tau}(\cdot)$ has a mode (resp.\ an antimode) in $(m-\tau,m+\tau)$.
\end{proposition}

\begin{proof}[Proof of Proposition \ref{proposition_modes}]
 Since $m$ is a mode (resp.\ an antimode) for $f(\cdot)$ there exists $\epsilon > 0$ such that (i) $f^{\prime \prime}(x) < 0$ (resp.\ $>0$) for $x \in [m-\epsilon,m+\epsilon]$ and (ii) $f^{\prime \prime}(x) f(y) + f^{\prime}(x) f^{\prime}(y) <0$ (resp.\ $>0$) for $x,y \in [m-\epsilon,m+\epsilon]$. In particular, (i) implies that $f^{\prime}(x) > 0$ (resp.\ $<0$) for $x \in [m-\epsilon,m)$ and $f^{\prime}(x) < 0$ (resp.\ $>0$) for $x \in (m,m+\epsilon]$. \\
Let $\tau^{*} \coloneqq \frac{\epsilon}{2}$. As in the proof of Proposition \ref{proposition_local_depth} (iv) (see Appendix \ref{section_proof_proposition_local_depth}), it follows that $f_{\tau}(\cdot)$ is twice continuously differentiable in $[m-\tau^{*},m+\tau^{*}]$. In particular, its first derivative is given by \eqref{derivative_f_tau}. Since $T_{++}^{\tau} \subset [0,\tau] \times [0,\tau]$, for $(x_1,x_2) \in T_{++}^{\tau}$, $m-\tau+x_1 \in [m-\tau,m]$ and $m-\tau-x_2 \in [m-2\tau,m-\tau]$, yielding $f(m-\tau+x_1) > 0$ (resp.\ $<0$) and $f(m-\tau-x_2) > 0$ (resp.\ $<0$) \emph{a.e.\ }for $(x_1,x_2) \in T_{++}^{\tau}$. Hence $f_{\tau}^{\prime}(m-\tau) > 0$ (resp.\ $<0$). Similarly, $f^{\prime}_{\tau}(m+\tau) < 0$ (resp.\ $>0$). Therefore, there exists $m_{\tau} \in (m-\tau,m+\tau)$ such that $f_{\tau}^{\prime}(m_{\tau})=0$. Finally, by (ii)
\begin{align*}
  f_{\tau}^{\prime \prime}(m_{\tau}) = \frac{1}{\tau \sqrt{LLD(x,\tau)}} &\int_{T_{++}^{\tau}} f^{\prime \prime}(x+x_1) f(x-x_2) + f^{\prime}(x+x_1) f^{\prime}(x-x_2) \\
  &+ f^{\prime}(x+x_1) f^{\prime}(x-x_2) + f(x+x_1) f^{\prime \prime}(x-x_2) \, dx_1 dx_2 <0 
\end{align*}
(resp.\ $>0$) and $m_{\tau}$ is a mode (resp.\ an antimode) for $f_{\tau}(\cdot)$.
\end{proof}

\section{Local simplicial depth} \label{sm:section_local_simplicial_depth}
In this appendix we describe in detail the appropriate changes to be made in the statements and proofs about LLD for LSD.

\noindent { \bf{Properties of local simplicial depth:}}  We emphasize that for $p>1$, LLD and LSD are different, even though they coincide for $p=1$. However, as we explain below the results for the simplicial depth can be obtained from the proof of our results for LLD. In this appendix we suppress $S$ in $Z_{\tau}^S(x)$, $Z^S(x)$, $\Lambda_1^{S}$, $\Lambda_1^{*S}$ and other quantities related to LSD.

Proposition \ref{proposition_local_depth} remains valid for LSD, with the only difference being that
\begin{equation*}
    \lim_{\tau \to 0^{+}} LSD(x,P,\tau) = P^{(p+1)}(\{x\}) \text{ and } \lim_{\tau \to \infty} LSD(x,P,\tau) = SD(x,P),
\end{equation*}
where $Z(x) = \{ (x_1, \dots, x_{p+1}) \in (\mathbb{R}^p)^{(p+1)} \, : \, x \in \triangle[x_1, \dots, x_{p+1}] \}$, and $SD(x, P)$ is the simplicial depth of a point $x \in \mathbb{R}^p$ with respect to $P$ and is given by
\begin{equation*}
	SD(x, P) = \int_{Z(x)}  \, dP(x_1) \dots dP(x_{p+1}).
\end{equation*}
Next, if $P$ is absolutely continuous with respect to the Lebesgue measure with a density $f(\cdot)$, then
\begin{equation*}
\begin{split}
	LSD(x,\tau) &= \int_{ Z_{\tau}(x) } f(x_1) \dots \, f(x_{p+1}) \, d x_1 \dots \, d x_{p+1} \\
	 &=  \int_{Z_{\tau}(0)} f(x+x_1) \dots \, f(x+x_{p+1}) \, d x_1 \dots \, d x_{p+1} \\
	&= \int_{Z_{1}(0)} \tau^{p(p+1)} f(x+ \tau x_1) \dots \, f(x+\tau x_{p+1}) \, d x_1 \dots \, d x_{p+1}.
\end{split}
\end{equation*}
This then implies that Theorem \ref{theorem_local_depth_tau_to_0} (i) can be rephrased as follows: for almost every point $x \in \mathbb{R}^p$,
\begin{equation} \label{local_simplicial_depth_tau_to_0}
	\lim_{\tau \rightarrow 0^{+}} \frac{1}{\tau^{p(p+1)} \Lambda_1} LSD(x,\tau) = f^{(p+1)}(x).
\end{equation}
where $\Lambda_1 = \lambda^{\otimes p}(Z_{1}(0))$ is the Lebesgue measure of $Z_{1}(0)$. Furthermore, if $f(\cdot)$ is continuous then \eqref{local_simplicial_depth_tau_to_0} holds for all $x \in \mathbb{R}^p$.

Proceeding analogously, Theorem \ref{theorem_local_depth_tau_to_0} (ii) becomes
\begin{equation*}
\lim_{\tau \to 0^{+}} \frac{1}{\tau^2} \left( \frac{1}{\tau^{p(p+1)}} LSD(x,\tau) - \Lambda_1 f^{(p+1)}(x) \right) = h(x),
\end{equation*}
where now $h(\cdot)$ is given by
\begin{align*}
  h(x) &= \frac{p+1}{2} f^p(x) \int_{Z_1(0)} x_{1}^{\top} H_f(x) x_1 \, dx_1 \dots dx_{p+1} \\
  &+ \frac{p(p+1)}{2} f^{p-1}(x) \int_{Z_1(0)} \inp{\nabla f(x)}{x_1} \inp{\nabla f(x)}{x_2} \, dx_1 \dots dx_{p+1}.
\end{align*}
Finally, the $\tau$-approximation in Definition \ref{definition_tau_approximation} for LSD becomes
\begin{equation}
f_{\tau}(x) = \frac{1}{\tau^{p} \sqrt[(p+1)]{\Lambda_1}} \sqrt[(p+1)]{LSD(x,\tau)}
\end{equation}
and Theorem \ref{theorem_uniform_convergence_of_tau_approximation} holds as before. 

\noindent {\bf{Sample local simplicial depth}:} Turning to the sample version of LSD, again given $X_1, \dots, X_{n}$ i.i.d.\ random variables from $P$, LSD is estimated using the proportion
\begin{equation} \label{sample_local_simplicial_depth}
LSD_n(x,\tau) = \frac{1}{ {n \choose p+1}} \sum_{1 \leq i_1 < \dots i_{p+1} \leq n} \mathbf{I} ( (X_{i_1},\dots,X_{i_{p+1}}) \in Z_{\tau}(x)),
\end{equation}
which is a U-statistics with the kernel
\begin{equation*}
\mathcal{K}_{x,\tau}(x_1,\dots,x_{p+1}) = \mathbf{I}((x_1,\dots,x_{p+1}) \in Z_{\tau}(x))
\end{equation*}
and first projection
\begin{equation} \label{kernel_projection_simplicial_depth}
  \mathcal{J}_{x,\tau}(x_1) = \int \mathbf{I}((x_1,\dots,x_{p+1}) \in Z_{\tau}(x)) \, dP(x_2) \dots dP(x_{p+1}).
\end{equation}
The variance of the first projection, using the same notation as in the LLD case, is given by
\begin{equation} \label{b_square_simplicial_depth}
b^2(x,\tau) = Var[\mathcal{J}_{x,\tau}(X_1)] = \int (\mathcal{J}_{x,\tau}(x_1))^2 \, dP(x_1) - (LSD(x,\tau))^2.
\end{equation}
With this description, Corollary \ref{corollary_sample_local_depth} (i) can be stated as follows: for any fixed $x \in \mathbb{R}^p$, $LSD_n(x, \cdot)$ is a non-decreasing and right continuous function; {\it{i.e.}},
\begin{equation*}
\lim_{\tau \to 0^{+}} LSD_n(x,\tau) = \frac{1}{{n \choose p+1}} \sum_{1 \leq i_1 < \dots < i_{p+1} \leq n} \mathbf{I}((X_{i_1},\dots,X_{i_{p+1}})=(x,\dots,x))
\end{equation*}
and
\begin{equation*}
\lim_{\tau \to \infty} LSD_n(x,\tau) = \frac{1}{{n \choose p+1}} \sum_{1 \leq i_1 < \dots < i_{p+1} \leq n} \mathbf{I}((X_{i_1},\dots,X_{i_{p+1}}) \in Z(x)).
\end{equation*}
The other statements in Corollary \ref{corollary_sample_local_depth} and Theorem \ref{theorem_uniform_consistency} remain valid without any changes. Theorem \ref{theorem_uniform_asymptotic_normality_of_local_depth} requires a minor modification. Specifically, using the definition of LSD, \eqref{sample_local_simplicial_depth}, \eqref{kernel_projection_simplicial_depth} and \eqref{b_square_simplicial_depth}, if $\Gamma \subset \mathbb{R}^p \times [0,\infty]$ is such that $b^2(x,\tau) > 0$ for all $(x,\tau) \in \Gamma$, then
\begin{equation*}
	\sqrt{n} \left( LSD_n(\, \cdot \,, \, \cdot \,) - LSD(\, \cdot \, , \, \cdot \,) \right) \xrightarrow[ n \rightarrow \infty]{d} (p+1) \,  W(\, \cdot \, ,\, \cdot \,) \text{ in } \ell^{\infty} (\Gamma),
\end{equation*}
where $\{ W(x,\tau) \}_{(x, \tau) \in \Gamma }$ is a centered Gaussian process with covariance function $\gamma: \Gamma \times \Gamma \rightarrow \mathbb{R}$ given by
\begin{equation*}
\gamma((x,\tau),(y,\nu)) =  \int \mathcal{J}_{x,\tau}(x_1) \mathcal{J}_{y,\nu}(x_1) \, dP(x_1) - \, LSD(x,\tau) LSD(y,\nu).
\end{equation*}
The extreme case of this result, that is the case $\tau=\infty$, was studied in \citet{Duembgen-1992} and \citet{Arcones-Gine-1993}[Corollary 6.8]. We repeat here that uniformity in $\tau$ is not an issue in the other papers, but it is a critical issue in our work.

We now turn to the sample $\tau$-approximation, which is needed for clustering. Now, the sample $\tau$-approximation in \eqref{sample_tau_approximation} becomes
\begin{equation*}
f_{\tau,n}(x) = \frac{1}{\tau^{p} \sqrt[(p+1)]{\Lambda_1}} \sqrt[(p+1)]{LSD_n(x,\tau)},
\end{equation*}
and it is straightforward to see that the results in Corollary \ref{corollary_uniform_consistency_of_sample_tau_approximation} hold. As for extreme localization,  Theorem \ref{theorem_asymptotic_normality_of_sample_tau_approximation} continues to hold as before; {\it{viz.,}}
\begin{equation*}
  \sqrt{n} \, \tau_n^{\frac{p}{2}}  \left( f_{\tau_n,n}(x) - f_{\tau_n}(x)  \right) \xrightarrow[ n \rightarrow \infty]{d} N\left( 0, \frac{\Lambda_1^{*2}}{\Lambda_1^2} f(x) \right),
\end{equation*}
where the sequence $\{ \tau_n \}_{n=1}^{\infty}$ now satisfies $\tau_n \xrightarrow[n \to \infty]{} 0$ and $\sqrt{n} \, \tau_n^{(p^2+\frac{p}{2})} \xrightarrow[n \to \infty]{} \infty$. We notice that, while the limiting distribution is exactly the same of Theorem \ref{theorem_asymptotic_normality_of_sample_tau_approximation}, the quantities $\Lambda_1^2$ and $\Lambda_1^{*2} = \int \left( \lambda^{\otimes (p-1)} \right)^2(Z_1(0) \rvert_{x_1} ) \, d x_1$ are now based on $Z_1(0)=Z_1^S(0)$ and $Z_1(0) \rvert_{x_1} = Z_1^S(0) \rvert_{x_1}$. It is worth mentioning that $\sqrt{n} \, \tau_n^{(p^2+\frac{p}{2})} \xrightarrow[n \to \infty]{} \infty$ is a necessary condition for the limiting normality to hold. Finally, under these hypotheses, it follows that
\begin{equation*}
  \sqrt{n} \, \tau_n^{\frac{p}{2}}  \left( \sqrt{f_{\tau_n,n}(x)} - \sqrt{f_{\tau_n}(x)}  \right) \xrightarrow[ n \rightarrow \infty]{d} N\left( 0, \frac{\Lambda_1^{*2}}{4 \Lambda_1^2} \right).
\end{equation*}
\noindent {\bf{Clustering via local simplicial depth:}} We describe below how to adapt LSD for the clustering procedure described in Section \ref{section_clustering}. Since LSD satisfies Proposition \ref{proposition_local_depth} (ii) and Theorem \ref{theorem_uniform_consistency}, Proposition \ref{proposition_convergence_of_regions} still holds for LSD. This and the fact that $f_{\tau}(\cdot)$ satisfies Theorem \ref{theorem_uniform_convergence_of_tau_approximation} imply that the super level sets of Definition \ref{definition_level_sets_and_regions} based on LSD satisfy Proposition \ref{proposition_convergence_of_level_sets}. 

Since the algorithm for clustering involves identification of the stationary points and the modes, first, we need analogous of Theorems \ref{theorem_local_central_symmetry} and \ref{theorem_modes_local_central_symmetry}. Specifically, in Theorem \ref{theorem_local_central_symmetry} $\mu$ is a stationary point for $f_{\tau}(\cdot)$ if and only if
\begin{equation*}
	\int_{Z_1(0)} \nabla f(\mu+\tau x_1) f(\mu+\tau x_2) \dots f(\mu+\tau x_{p+1}) \, dx_1 dx_2  \dots dx_{p+1} = 0,
\end{equation*}
and $\nabla f_{\tau}(\mu)=0$ if $f(\cdot)$ is $\tau$-centrally symmetric about $\mu$. Next, turning to a mode (resp.\ an antimode) $m$ of $f(\cdot)$, Theorem \ref{theorem_modes_local_central_symmetry} can be described as follows: if for all $x_1, x_2 \in \overline{B}_{\tau}(m)$ the matrix 
\begin{equation*}
G_f(x_1,x_2) = H_{f}(x_1) f(x_2) + p \, \nabla f(x_1) \nabla f(x_2)^{\top}
\end{equation*}
is negative (resp.\ positive) definite, then $m$ is also a mode (resp.\ an antimode) for $f_{\tau}(\cdot)$. \\
Incidentally, it is worth noting here that Lemma \ref{lemma_support_S_and_S_tau}, Theorem \ref{theorem_first_and_second_partial_derivaties}, Lemma \ref{lemma_R_alpha_tau_bounded} and Theorem \ref{theorem_convergence_solution_estimated_gradient_system} all hold true for LSD. \\
\noindent {\bf{Numerical implementation:}} In this subsection we describe the changes to the numerical procedure required for clustering using LSD. First the quantity $d_{\tau,n}(\, \cdot \, ; \, \cdot \, )$ in \eqref{finite difference_sample_local_depth} is now replaced by
\begin{equation*}
    d_{\tau,n}(x;y) = \frac{\sqrt[(p+1)]{LSD_n(y,\tau)} - \sqrt[(p+1)]{LSD_n(x,\tau)}}{\norm{y-x}}.
\end{equation*}
As for LLD, $\tau$ can be chosen as the $q$-quantile of the empirical distribution of the ${n \choose p+1}$ maxima of the form $\max_{\substack{ j,l = 1, \dots, p+1 \\ j > l}} \norm{x_{i_j}-x_{i_l}}$ for all ${n \choose p+1}$ combinations of indices $i_1,\dots,i_{p+1}$ from $\{1,2,\dots,n\}$. 

\noindent {\bf{Computational complexity:}} We recall that $LLD_n$ is a U-statistics of order $2$, while $LSD_n$ is a U-statistics of order $(p+1)$. This means that the computational complexity of $LLD_n$ is of order $O({n \choose 2})$, while the computational complexity of the $LSD_n$ is of order $O({n \choose p+1})$, which makes a significant difference, especially in high dimensions. For large $p$ and $n$, an approximation to LSD can be made by considering  a large number of simplices sampled with replacement amongst all the ${n \choose p+1}$ simplices that define $LSD_n$; in our simulations we sample $10^8$ simplices to  reduce the computational cost.

\section{Local half-space depth} \label{sm:section_local_half-space_depth}
In this appendix we describe in detail the appropriate changes to be made in the statements about LLD that yield the corresponding results for LHD. Additionally, we also suppress $H$ in $Z_{\tau}^H(x)$ and write $Z_{\tau}(x)$.

\noindent { \bf{Properties of local half-space depth:}} We begin by showing that, for all fixed $\tau \in [0,\infty]$, $LHD(\cdot,\tau)$ is upper semicontinuous. This follows from the fact that the infimum of a collection of upper semicontinuous functions is upper semicontinuous and that the function $x \to P(Z_{\tau}(x,u))$ is upper semicontinuous for all $u \in S^{p-1}$. In fact, let  $\{ x_k \}_{k=1}^{\infty}$ be a sequence converging to $x \in \mathbb{R}^p$. Since $Z_{\tau}(x,u)$ is closed, for all $u \in S^{p-1}$
\begin{align*}
  \limsup_{k \to \infty} P(Z_{\tau}(x_k,u)) &= \limsup_{k \to \infty} \int \mathbf{I}(y \in Z_{\tau}(x_k,u)) \, dP(y) \\
  &\leq \int \limsup_{k \to \infty} \mathbf{I}(y \in Z_{\tau}(x_k,u)) \, dP(y) \\
  &\leq \int \mathbf{I}(y \in Z_{\tau}(x,u)) \, dP(y) = P(Z_{\tau}(x,u)).
\end{align*}
If $P$ assigns probability $0$ to all hyperplanes \citep{Masse-2004}[condition (S)] (for example if $P$ is absolutely continuous with respect to the Lebesgue measure), then the function $(x,u,\tau) \to P(Z_{\tau}(x,u))$ is jointly continuous. In fact, if $\{ \tau_k \}_{k=1}^{\infty} \subset [0,\infty]$, $\{ u_k \}_{k=1}^{\infty} \subset S^{p-1}$ and $\{ x_k \}_{k=1}^{\infty}$ are sequences converging to $\tau \in [0,\infty]$, $u \in S^{p-1}$ and $x \in \mathbb{R}^p$, respectively, then
\begin{equation*}
\abs{P(Z_{\tau_k}(x_k,u_k))-P(Z_{\tau}(x,u))} \leq \int \abs{\mathbf{I}(y \in Z_{\tau_k}(x_k,u_k)) - \mathbf{I}(y \in Z_{\tau}(x,u)) } \, dP(y) \xrightarrow[k \to \infty]{} 0
\end{equation*}
by the LDCT. This is because $\partial Z_{\tau}(x,u)$ has probability $0$, and if $y \notin \partial Z_{\tau}(x,u)$, then $\lim_{k \to \infty} \mathbf{I}(y \in Z_{\tau_k}(x_k,u_k)) = \mathbf{I}(y \in Z_{\tau}(x,u))$. In particular, since the infimum of a collection of continuous functions over a compact set is a continuous function, LHD is jointly continuous in $x$ and $\tau$. Hence Proposition \ref{proposition_local_depth} (iii) holds. The compactness of $S^{p-1}$ also implies that the infimum in \eqref{local_half-space_depth} is a minimum. The other statements in Proposition \ref{proposition_local_depth} remain true, the only change being that in part (i) we require $P$ to assign probability $0$ to all hyperplanes to obtain
\begin{equation*}
\lim_{\tau \to 0^{+}} LHD(x,P,\tau) = P(\{ x \}) = 0 \text{ and } \lim_{\tau \to \infty} LHD(x,P,\tau) = HD(x,P),
\end{equation*}
where $HD(x,P)$ is the half-space depth of $x \in \mathbb{R}^p$ with respect to $P$ given by
\begin{equation*}
HD(x,P) = \inf_{u \in S^{p-1}} P(Z(x,u)), \text{ with } Z(x,u) = \{ y \in \mathbb{R}^p \, : \, \inp{u}{x} \leq \inp{u}{y} \}.
\end{equation*}
If $P$ is absolutely continuous with respect to the Lebesgue measure with a density $f(\cdot)$, due to $Z_{\tau}(x,u)=x+Z_{\tau}(0,u)$ and $Z_{\tau}(0,u)=\tau Z_{1}(0,u)$, by a change of variable we see that
\begin{equation*}
LHD(x,\tau) = \min_{u \in S^{p-1}} \int_{Z_{\tau}(x,u))} f(y) \, dy = \tau^{p} \min_{u \in S^{p-1}} \int_{Z_{1}(0,u))} f(x+\tau y) \, dy.
\end{equation*}
This allows one to obtain in Theorem \ref{theorem_local_depth_tau_to_0}
\begin{equation*}
\lim_{\tau \to 0^{+}} \frac{1}{\tau^{p}} LHD(x,\tau) = f(x),
\end{equation*}
and for twice continuously differentiable $f(\cdot)$
\begin{equation*}
  \lim_{\tau \to 0^{+}} \frac{1}{\tau} \left( \frac{1}{\tau^{p}} LHD(x,\tau) - f(x) \right) = \min_{u \in S^{p-1}} \int_{Z_{1}(0,u)} \inp{\nabla f(x)}{y} \, dy.
\end{equation*}
Theorem  \ref{theorem_uniform_convergence_of_tau_approximation} remains valid for the $\tau$-approximation based on LHD given by
\begin{equation*}
f_\tau(x) = \frac{1}{\tau^p} LHD(x,\tau).
\end{equation*}
\noindent { \bf{Sample local half-space depth:} } Given $X_1, \dots, X_{n}$ i.i.d.\ random variables from $P$, LHD is estimated using the proportion
\begin{equation*}
LHD_n(x,\tau) = \inf_{u \in S^{p-1}} \frac{1}{n} \sum_{i=1}^n \mathbf{I} (X_i \in Z_{\tau}(x,u)).
\end{equation*}
With this definition, Corollary \ref{corollary_sample_local_depth} remains valid, but part (i) becomes
\begin{equation*}
\lim_{\tau \to 0^{+}} LHD_n(x,\tau) = \frac{1}{n} \, \sum_{i=1}^n \mathbf{I}(X_i=x)
\end{equation*}
and
\begin{equation*}
\lim_{\tau \to \infty} LHD_n(x,\tau) = \inf_{u \in S^{p-1}} \frac{1}{n} \, \sum_{i=1}^n \mathbf{I}(X_i \in Z(x,u)).
\end{equation*}
Since the classes of half-spaces and cubes in $\mathbb{R}^p$ are VC \citep{Despres-2014}, we also obtain Theorem \ref{theorem_uniform_consistency} and Corollary \ref{corollary_uniform_consistency_of_sample_tau_approximation}. Asymptotic normality results for the half-space depth are more involved \citep{Arcones-2006, Masse-2004} and are beyond the scope of this paper. The properties derived in this and in the previous sections allow to obtain all the results derived for clustering in Subsections \ref{subsection_mathematical_background} and \ref{subsection_convergence_gradient_system} for LHD. However, $Z_{\tau}(x,u)$ is not symmetric in the sense of \eqref{symmetry1_of_z_tau_x_general} (with $k=1$). Therefore, the results concerning the modes and stationary points of $f(\cdot)$ in Subsection \ref{subsection_modes_identification} do not hold for LHD.

\section{Proof of Proposition \ref{proposition_local_depth} (iv)} \label{section_proof_proposition_local_depth}
We provide in this section a detailed proof of Proposition \ref{proposition_local_depth} (iv). We first observe that if $\tau=0$ then $LLD(x,\tau)=0$ for all $x \in \mathbb{R}^p$ and the statement is trivial. Let $\tau>0$ and $k \geq 1$. We will show that, for all $0 \leq l \leq k$, the partial derivatives of LLD up to order $l$ exist and are given by the integral over $Z_1(0)$ of a $(k-l)$-times continuously differentiable function. To this end, notice that $Z_1(0)$ is bounded by definition and closed by Lemma \ref{lemma_interior_exterior_of_z_tau_x}, and hence compact. For $l=0$, by \eqref{local_depth_for_absolutely_continuous_distribution}, we have that
\begin{equation*}
LLD(x,\tau) = \int_{Z_1(0)} g_{*}(x;x_1,x_2) \, dx_1 dx_2,
\end{equation*}
where for $x,x_1,x_2 \in \mathbb{R}^p$
\begin{equation*}
g_{*}(x;x_1,x_2) \coloneqq \tau^{2p} g(x+\tau x_1,x+\tau x_2)
\end{equation*}
and for $y_1,y_2 \in \mathbb{R}^p$
\begin{equation*}
g(y_1,y_2) \coloneqq f(y_1) f(y_2).
\end{equation*}
Since $f(\cdot)$ is  $k$-times continuously differentiable, $g_{*}(x; x_1, x_2)$ is $k$-times continuously differentiable with respect to $x$, $x_1$, and $x_2$. 

Given a function $g : D \subset \mathbb{R}^p \rightarrow \mathbb{R}$ for $j=1, \dots, p$ we denote by $\partial_j g$ its partial derivative with respect to the $j$-component. Suppose by induction that the partial derivatives of the local depth up to order $l-1$ ($1 \leq l \leq k$) exist and for some choice of indices $1 \leq i_1, \dots, i_{l-1} \leq p$ are given by
\begin{equation*}
\partial_{i_{l-1}} \dots \partial_{i_1} LLD(x,\tau) = \int_{Z_1(0)} g_{*}^{i_{l-1} \dots i_1}(x;x_1,x_2) \, dx_1 dx_2,
\end{equation*}
where $g_{*}^{i_{l-1} \dots i_1}(x;x_1,x_2)=\tau^{2p} g^{i_{l-1} \dots i_1}(x+\tau x_1, x+\tau x_2)$ is $(k-(l-1))$-times continuously differentiable with respect to $x$, $x_1$ and $x_2$. In particular, since $Z_1(0)$ and $\overline{B}_1(x)$ are compact, for all $1 \leq i_l \leq p$, $z \in \overline{B}_1(x)$ and $(x_1,x_2) \in Z_1(0)$,
\begin{equation} \label{partial_derivative_dominating_constant}
\partial_{i_l} g_{*}^{i_{l-1} \dots i_1}(z;x_1,x_2) \leq C_x^{i_l \dots i_1} \coloneqq \sup_{(y_1,y_2) \in Z_1(0)} \sup_{y \in \overline{B}_1(x) } \tau^{2p} \partial_{i_l} g^{i_{l-1} \dots i_1}(y+\tau y_1,y+\tau y_2) < \infty.
\end{equation}
Let $\{ h_n \}_{n=1}^{\infty}$ be a sequence of scalars with $0 < h_n \leq 1$ and $h_n \xrightarrow[n \to \infty]{} 0$. For $x \in \mathbb{R}^p$ and $h>0$ we define the $i$-th partial finite difference of a function $g_*: \mathbb{R}^p \to \mathbb{R}$ by
\begin{equation*}
\partial_i^{h} g_*(x) = \frac{g_*(x+he_i)-g_*(x)}{h},
\end{equation*}
where $\{ e_i \, : \, i=1,\dots,p \}$ is the standard basis of $\mathbb{R}^p$. Clearly, $\partial_{i_l}^{h_n} g_{*}^{i_{l-1} \dots i_1}(x;x_1,x_2) \xrightarrow[n \to \infty]{} \partial_{i_l} g_{*}^{i_{l-1} \dots i_1}(x;x_1,x_2) = \tau^{2p} \partial_{i_l} g^{i_{l-1} \dots i_1}(x+\tau x_1, x+\tau x_2)$. Now, for $y_1,y_2 \in \mathbb{R}^p$, define
\begin{equation*}
g^{i_{l} \dots i_1}(y_1,y_2) \coloneqq \partial_{i_l} g^{i_{l-1} \dots i_1}(y_1,y_2)
\end{equation*}
and for $x,x_1,x_2 \in \mathbb{R}^p$ let
\begin{equation*}
g_{*}^{i_{l} \dots i_1}(x;x_1,x_2) \coloneqq \tau^{2p} g^{i_{l} \dots i_1}(x+\tau x_1,x+\tau x_2).
\end{equation*}
$g_{*}^{i_{l} \dots i_1}$ is $(k-l)$-times continuously differentiable with respect to $x$, $x_1$ and $x_2$ and by the mean value theorem there are scalars $0 \leq c_n \leq 1$ such that
\begin{equation*}
\partial_{i_l} g_{*}^{i_{l-1} \dots i_1}(x+c_n h_n e_i;x_1,x_2) = \partial_{i_l}^{h_n} g_{*}^{i_{l-1} \dots i_1}(x;x_1,x_2).
\end{equation*}
Since $x+c_n h_n e_i \in \overline{B}_1(x)$ for all $n \in \mathbb{N}$, the sequence of functions $\{ \partial_{i_l}^{h_n} g_{*}^{i_{l-1} \dots i_1}(x+c_n h_n e_i;\cdot,\cdot) \}_{n=1}^{\infty}$ is dominated on $Z_1(0)$ by the constant $C_x^{i_l \dots i_1}$ (see \eqref{partial_derivative_dominating_constant}). By LDCT it follows that
\begin{align*}
  \partial_{i_l} \dots \partial_{i_1} LLD(x,\tau) &= \lim_{n \to \infty} \partial_{i_l}^{h_n} \partial_{i_{l-1}} \dots \partial_{i_1} LLD(x,\tau) \\
  &= \lim_{n \to \infty} \int_{Z_1(0)} \partial_{i_l}^{h_n} g_{*}^{i_{l-1} \dots i_1}(x;x_1,x_2) \, dx_1 dx_2 \\
  &= \int_{Z_1(0)} \partial_{i_l} g_{*}^{i_{l-1} \dots i_1}(x;x_1,x_2) \, dx_1 dx_2 \\
  &= \int_{Z_1(0)} g_{*}^{i_{l} \dots i_1}(x;x_1,x_2) \, dx_1 dx_2,
\end{align*}
where $g_{*}^{i_{l} \dots i_1}$ is $(k-l)$-continuously differentiable with respect to $x$, $x_1$ and $x_2$. \\
In particular, for $l=k$ this means that for all $1 \leq i_1, \dots, i_k \leq p$ the partial derivatives $\partial_{i_k} \dots \partial_{i_1} LLD(x,\tau)$ exist for all $x \in \mathbb{R}^p$ and are continuous since they are given by the integral over the compact set $Z_1(0)$ of the continuous function $g_{*}^{i_{k} \dots i_1}$. Hence the statement. We finally observe that the functions $g_{*}^{i_{k} \dots i_1}$ can be computed explicitly as it is done, for instance, in the proof of Theorem \ref{theorem_first_and_second_partial_derivaties} for $k=1$ and $k=2$.

\section{Additional simulations}
\label{sm:section_additional_simulations}

\subsection{True clusters} \label{sm:subsection_true_clusters}

In this subsection we provide the analytical expression for the distributions Bimodal, (H) Bimodal IV, (K) Trimodal III, \#10 Fountain, Mult.\ Bimodal, and Mult.\ Quadrimodal considered in Section \ref{section_simulations_and_data_analysis} and the corresponding true clusters. We also consider two additional distributions: one is (L) Quadrimodal distribution in \citet{Wand-1993} and a four-mixture, which is defined below, and referred to as Quadrimodal (without (L)). We now describe these distributions.

(i) The Bimodal density is a two-mixture of normal distributions with equal weights, identity covariance matrix and means $(-2,0)$ and $(2,0)$. \\
(ii) The Quadrimodal density is a mixture of four normal distributions with means $(-2,2)$, $(-2,-2)$, $(2,-2)$ and $(2,2)$, and again equal weights and identity covariance matrix. \\
(iii) The (H) Bimodal IV density is a mixture of two normal distributions with equal weights, means $\mu_1=(1,-1)^{\top}$, $\mu_2=(-1,1)^{\top}$ and covariances
\begin{equation*}
  \Sigma_1 = \frac{4}{9} \begin{pmatrix}
    1 & \frac{7}{10} \\
    \frac{7}{10} & 1
  \end{pmatrix}
  \quad \text{ and } \quad
  \Sigma_2 = \frac{4}{9} \begin{pmatrix}
    1 & 0 \\
    0 & 1
  \end{pmatrix}.
\end{equation*}
(iv) The (K) Trimodal III density is a mixture of three normal distributions with weights $w_1=w_2=\frac{3}{7}$ and $w_3=\frac{1}{7}$; means $\mu_1=(-1,0)^{\top}$, $\mu_2=(1,2 \cdot \frac{\sqrt{3}}{3})^{\top}$ and $\mu_3=(1,-2 \cdot \frac{\sqrt{3}}{3})^{\top}$; and covariances
\begin{equation*}
  \Sigma_1 = \begin{pmatrix}
    \frac{9}{25} & \frac{7}{10} \cdot \frac{9}{25} \\
    \frac{7}{10} \cdot \frac{9}{25} & \frac{49}{100}
  \end{pmatrix}
  \quad \text{ and } \quad
  \Sigma_2 = \Sigma_3 = \begin{pmatrix}
    \frac{9}{25} & 0 \\
    0 & \frac{49}{100}
  \end{pmatrix}.
\end{equation*}
(v) The (L) Quadrimodal density is a mixture of four normal distributions with weights $w_1=w_3=\frac{1}{8}$ and $w_2=w_4=\frac{3}{8}$; means $\mu_1=(-1,1)^{\top}$, $\mu_2=(-1,-1)^{\top}$, $\mu_3=(1,-1)^{\top}$ and $\mu_4=(1,1)^{\top}$; and covariances
\begin{align*}
  \Sigma_1 &= \begin{pmatrix}
    \frac{4}{9} & \frac{2}{5} \cdot \frac{4}{9} \\
    \frac{2}{5} \cdot \frac{4}{9} & \frac{4}{9}
  \end{pmatrix}, \quad
  \Sigma_2 = \begin{pmatrix}
    \frac{4}{9} & \frac{3}{5} \cdot \frac{4}{9} \\
    \frac{3}{5} \cdot \frac{4}{9} & \frac{4}{9}
  \end{pmatrix}, \\
  \Sigma_3 &= \begin{pmatrix}
    \frac{4}{9} & -\frac{7}{10} \cdot \frac{4}{9} \\
    -\frac{7}{10} \cdot \frac{4}{9} & \frac{4}{9}
  \end{pmatrix}
  \quad \text{ and } \quad
  \Sigma_4 = \begin{pmatrix}
    \frac{4}{9} & -\frac{1}{2} \cdot \frac{4}{9} \\
     -\frac{1}{2} \cdot \frac{4}{9} & \frac{4}{9}
  \end{pmatrix}.
\end{align*}
(vi) The \#10 Fountain density is a mixture of six normal distributions with weights $w_1=\frac{1}{2}$ and $w_2=w_3=w_4=w_5=w_6=\frac{1}{10}$; means $\mu_1=\mu_2=(0,0)^{\top}$, $\mu_3=(-1,1)^{\top}$, $\mu_4=(-1,-1)^{\top}$, $\mu_5=(1,-1)^{\top}$ and $\mu_6=(1,1)^{\top}$; and covariances
\begin{align*}
  \Sigma_1 = \begin{pmatrix}
    1 & 0 \\
    0 & 1
  \end{pmatrix}, \quad
  \Sigma_2 = \Sigma_3 = \Sigma_4 = \Sigma_5 = \Sigma_6 = \begin{pmatrix}
    \frac{1}{16} & 0 \\
    0 & \frac{1}{16}
  \end{pmatrix}.
\end{align*}
The true clusters corresponding to these densities are in Figure \ref{figure_true_clusters_ldc_kde} (first row), in the main paper, and Figure \ref{sm:figure_true_clusters_ldc_kde} (first row), in Appendix \ref{sm:subsection_illustrative_examples}. Finally, the Mult.\ Bimodal and Mult.\ Quadrimodal densities are obtained as mixtures of normal densities with identity covariance matrix and equal weights. In particular, (vii) the Mult.\ Bimodal density is a mixture of two normal distributions with means $(-2,0,0,0,0)$ and $(2,0,0,0,0)$ and (viii) the  Mult.\ Quadrimodal density is a mixture of four normal distributions with means $(-2,2,0,0,0)$, $(-2,-2,0,0,0)$, $(2,-2,0,0,0)$ and $(2,2,0,0,0)$. The true clusters for these distributions can be deduced from those of the Bimodal and Quadrimodal densities, respectively.

In Section \ref{sm:subsection_numerical_experiments} we perform additional simulations on four more challenging circular densities, referred to as Circular 2, Circular 2 Cauchy, Circular 3 and Circular 4 Cauchy. They have densities proportional to $f_1(\cdot)$, $f_2(\cdot)$, $f_3(\cdot)$ and $f_4(\cdot)$ (respectively), where
\begin{align*}
f_1(x) &= 0.5 \exp\left(-12.5 \left(-2 +\norm{x}\right)^2\right) \left(1.1 - \frac{x^{(1)}}{\norm{x}}\right) \\
&+ 0.5 \exp\left(-12.5 \left(-0.5 +\norm{x}\right)^2\right) \left(1.1 + \frac{x^{(1)}}{\norm{x}}\right), \\
f_2(x) &= \left. \left( 0.5 \left( 1.1 + \frac{x^{(1)}}{\norm{x}}\right)\right) \right/ \left( 1 + 25 \left(-2 +\norm{x} \right)^2 \right) \\
&+ \left. \left(0.5 \left(1.1 + \frac{x^{(1)}}{\norm{x}}\right)\right) \right/ \left(1 + 25 \left(-0.5 +\norm{x} \right)^2 \right), \\
f_3(x) &= 0.3 \exp \left(-\frac{200}{9} \left(-1.5 +\norm{x}\right)^2\right) \left(1.1 - \frac{x^{(1)}}{\norm{x}}\right) \\
 &+ 0.15 \exp \left(-\frac{200}{9} \left(-2.5 +\norm{x}\right)^2\right) \left(1.1 + \frac{x^{(1)}}{\norm{x}}\right) \\
 &+ 0.55 \exp \left(-\frac{200}{9} \left(-0.5 +\norm{x}\right)^2\right) \left(1.1 + \frac{x^{(1)}}{\norm{x}}\right), \\
f_4(x) &= \left. \left(2 + \cos \left( 4 \arccos \left( \frac{x^{(1)}}{\norm{x}} \right) \right) \right) \right/ \left(1 + \left(-2 +\norm{x}\right)^2\right). \\
\end{align*}
Figure \ref{sm:figure_circular_densities} shows the functions $f_1(\cdot)$, $f_2(\cdot)$, $f_3(\cdot)$ and $f_4(\cdot)$. The true clusters associated with these densities are shown in Figure \ref{figure_true_clusters_circular}. Although the density Circular 4 Cauchy has a circular structure, it does not have clusters of a circular form, which makes it easier to identify the true clusters in the simulations.

\begin{figure}
\centering
\begin{subfigure}{0.56 \linewidth}
\centering
\includegraphics[width=\linewidth]{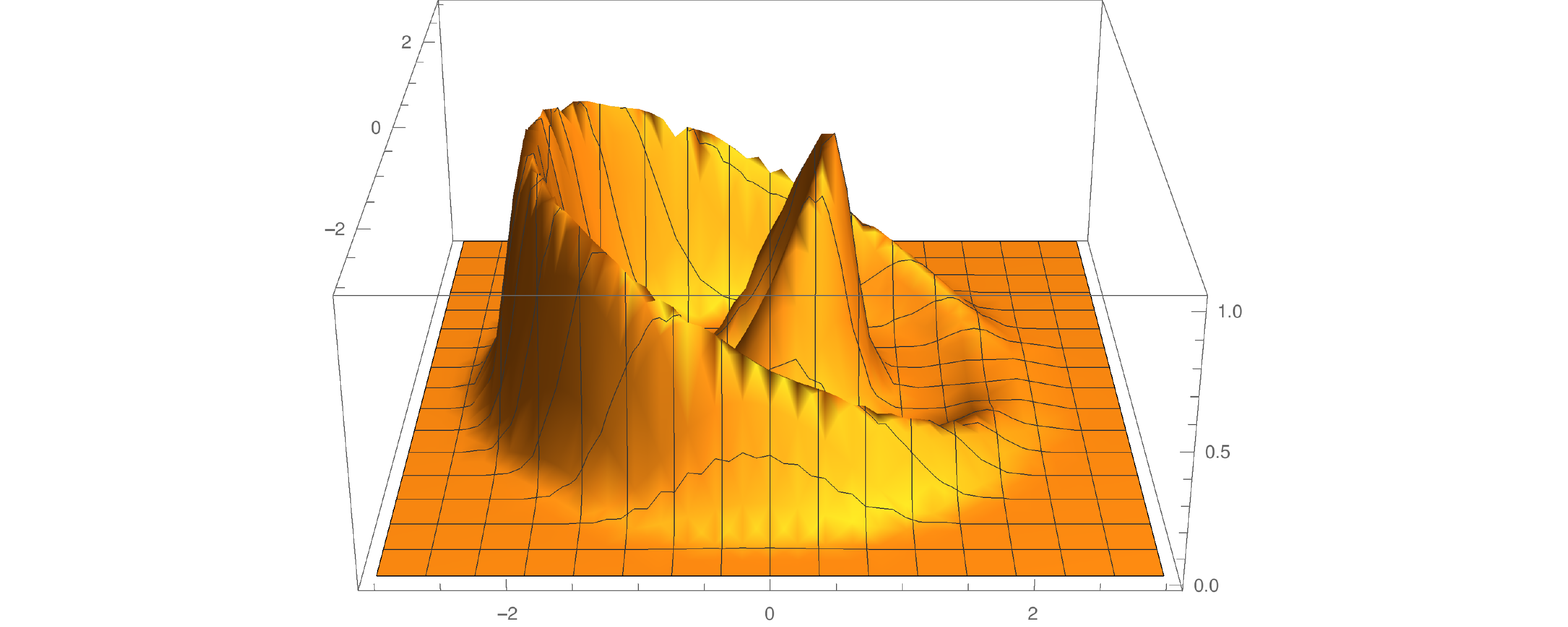}
\caption*{Circular 2}
\end{subfigure}
\begin{subfigure}{0.36 \linewidth}
\centering
\includegraphics[width=\linewidth]{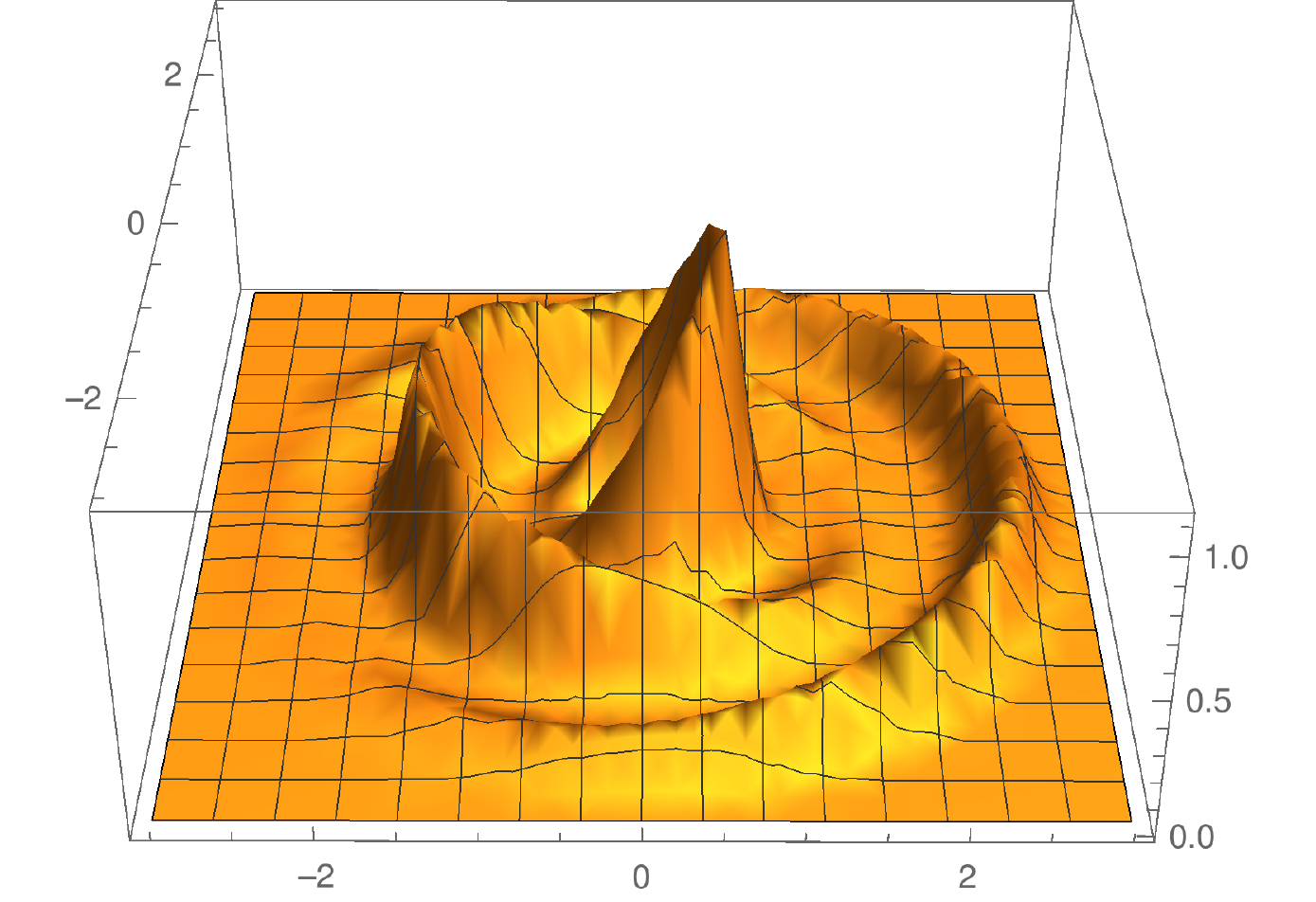}
\caption*{Circular 3}
\end{subfigure}
\break
\begin{subfigure}{0.64 \linewidth}
\centering
\includegraphics[width=\linewidth]{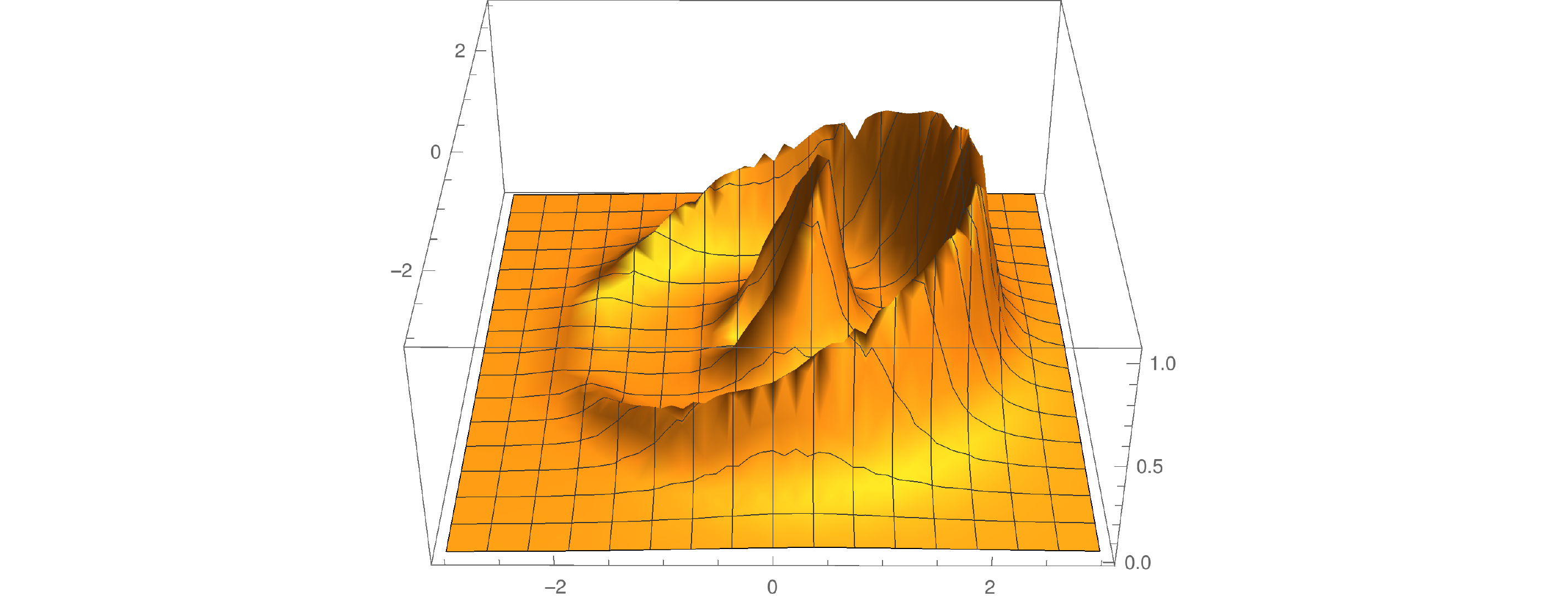}
\caption*{Circular 2 Cauchy}
\end{subfigure}
\begin{subfigure}{0.32 \linewidth}
\centering
\includegraphics[width=\linewidth]{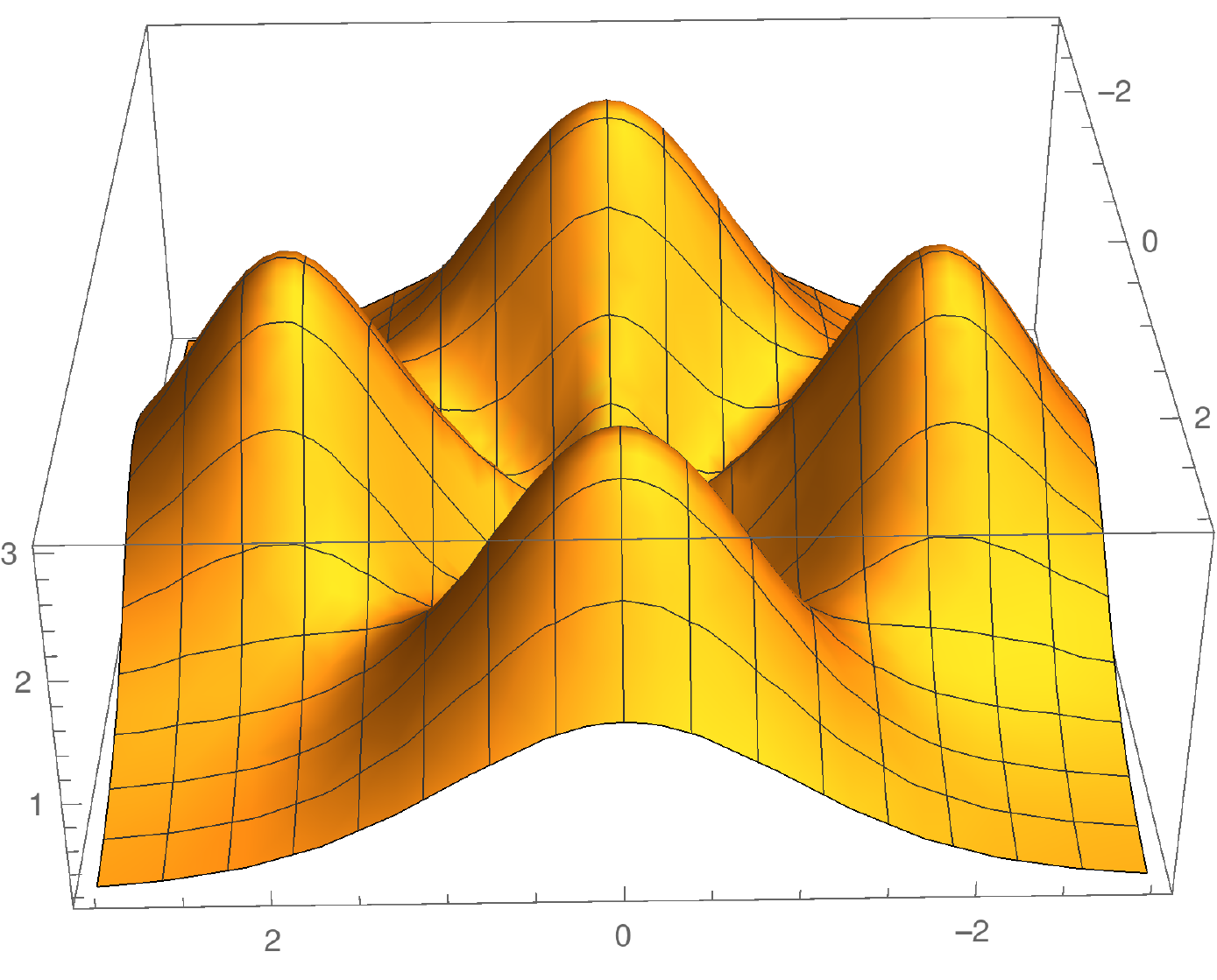}
\caption*{Circular 4 Cauchy}
\end{subfigure}
\begin{subfigure}{0.32 \linewidth}
\centering
\end{subfigure}

\caption{Plots of the functions $f_1(\cdot)$, $f_2(\cdot)$, $f_3(\cdot)$ and $f_4(\cdot)$ proportional to the Circular 2, Circular 2 Cauchy, Circular 3 and Circular 4 Cauchy densities, respectively.}
\label{sm:figure_circular_densities}
\end{figure}

\begin{figure}
\centering
\begin{subfigure}{0.32 \linewidth}
\centering
\includegraphics[width=\linewidth]{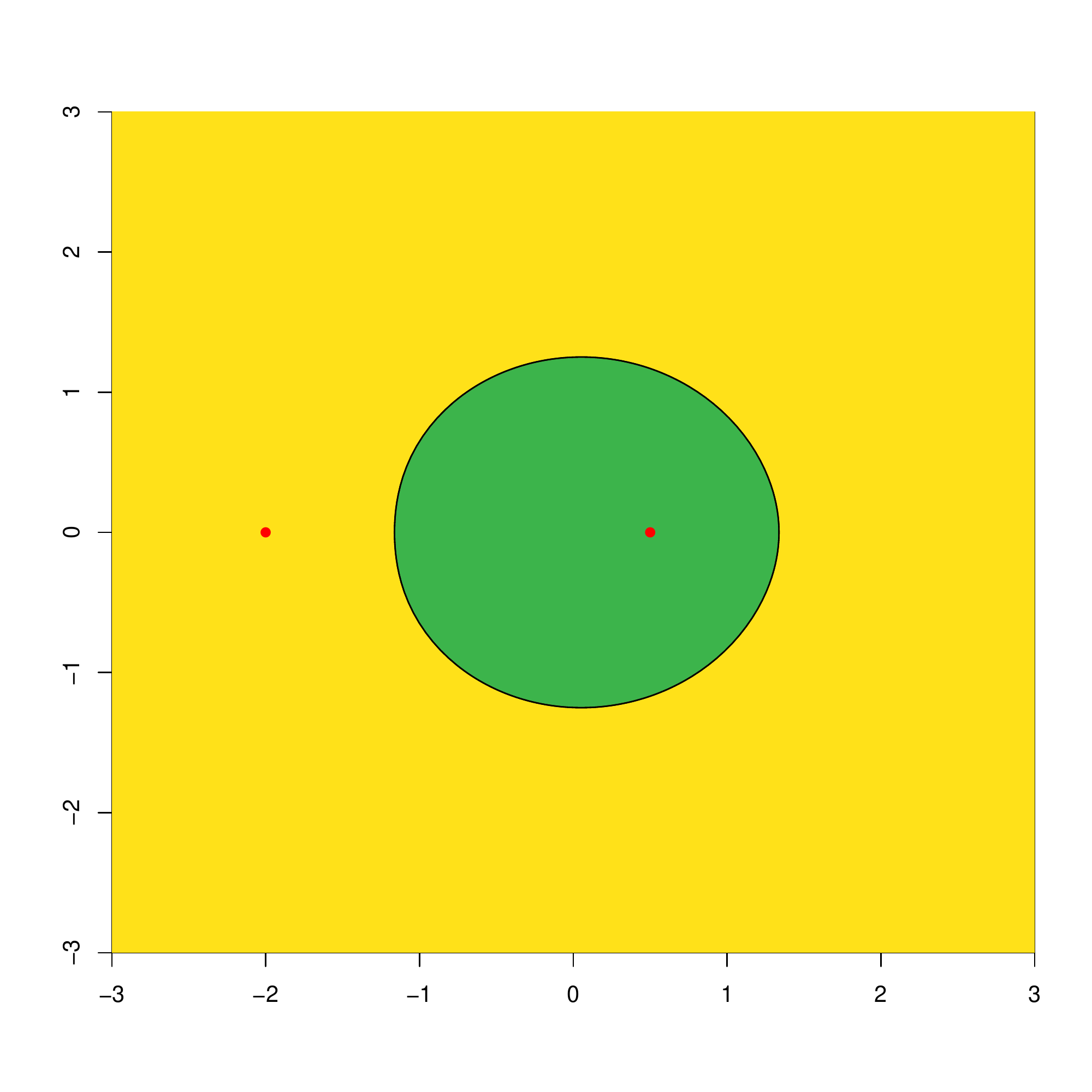}
\caption*{Circular 2}
\end{subfigure}
\begin{subfigure}{0.32 \linewidth}
\centering
\includegraphics[width=\linewidth]{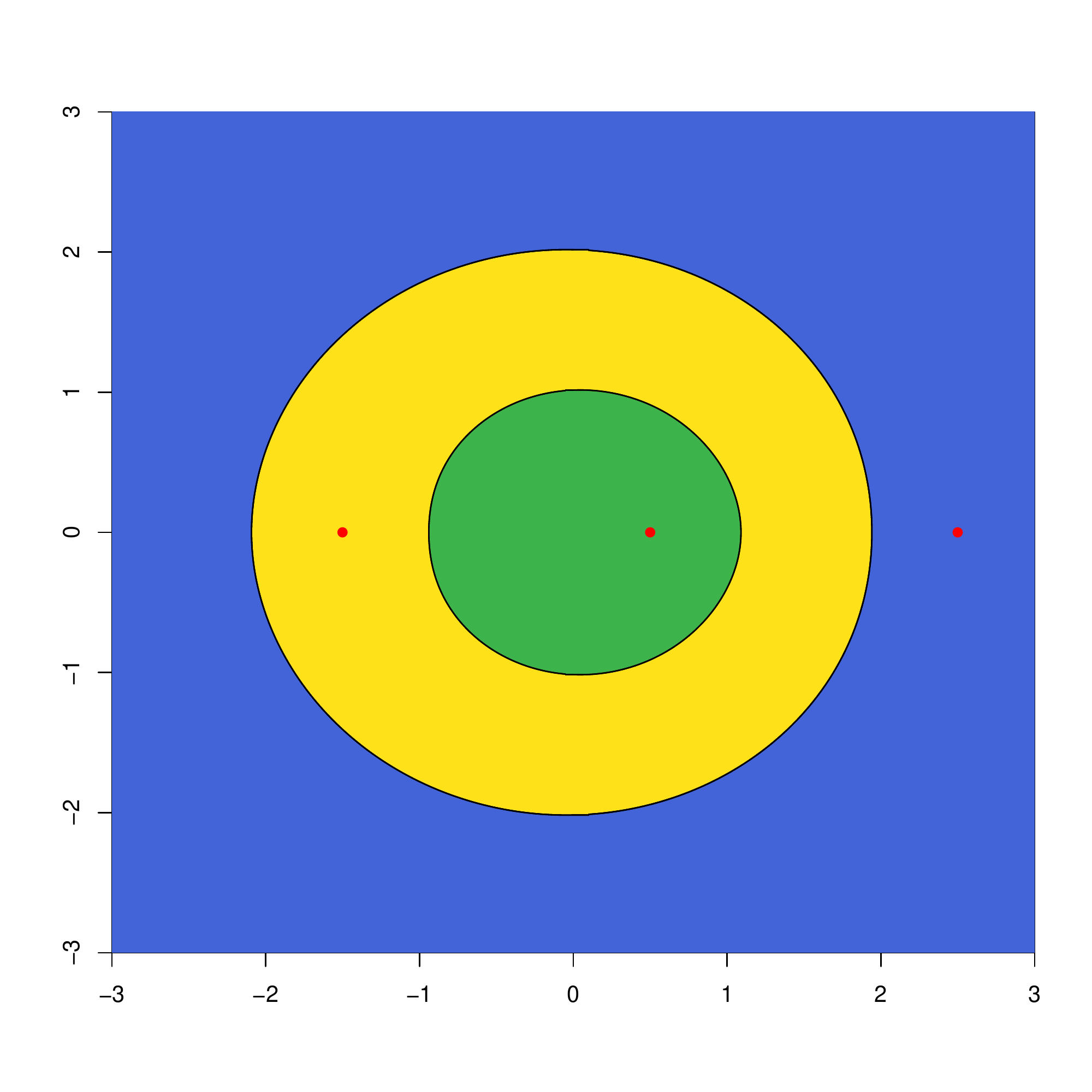}
\caption*{Circular 3}
\end{subfigure}
\break
\begin{subfigure}{0.32 \linewidth}
\centering
\includegraphics[width=\linewidth]{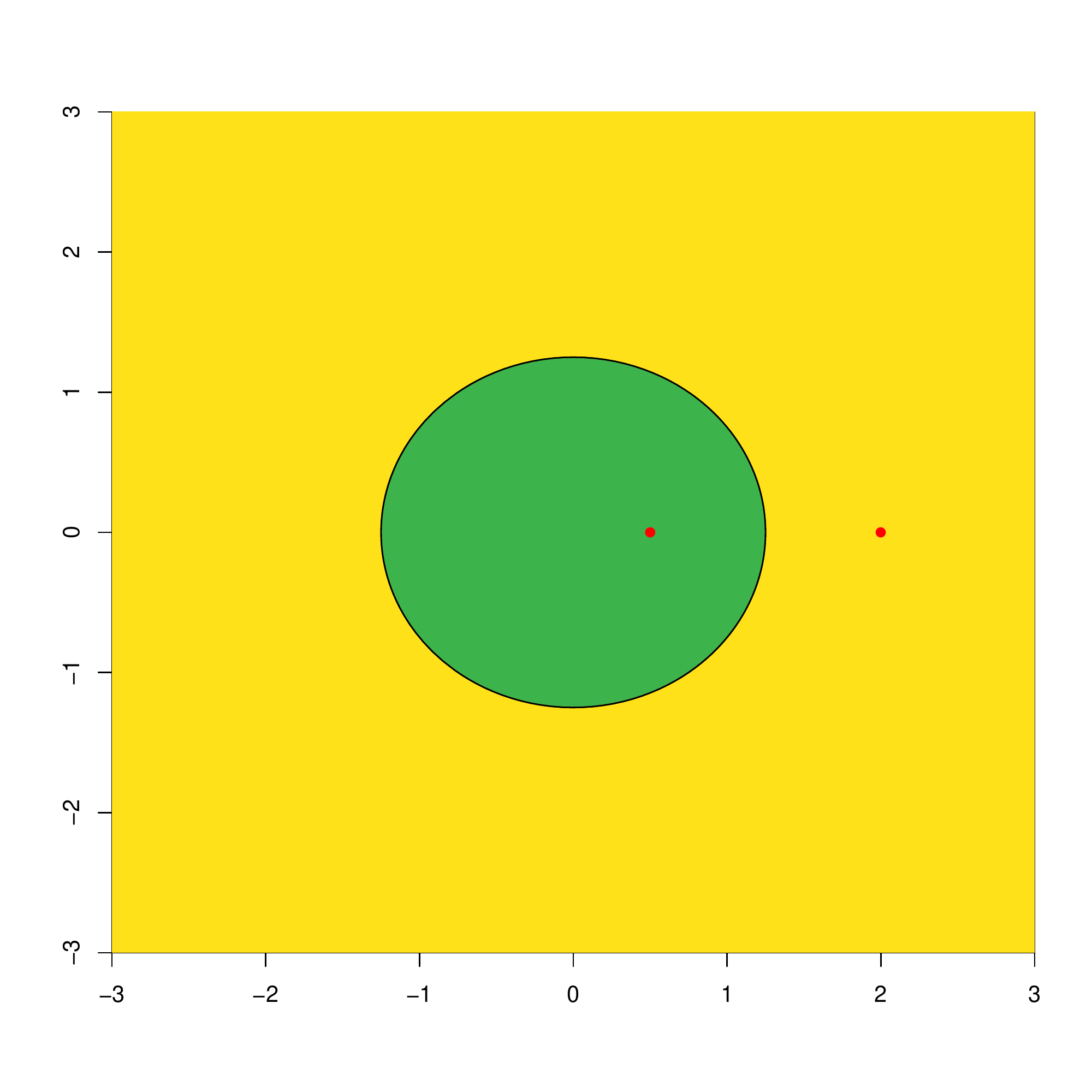}
\caption*{Circular 2 Cauchy}
\end{subfigure}
\begin{subfigure}{0.32 \linewidth}
\centering
\includegraphics[width=\linewidth]{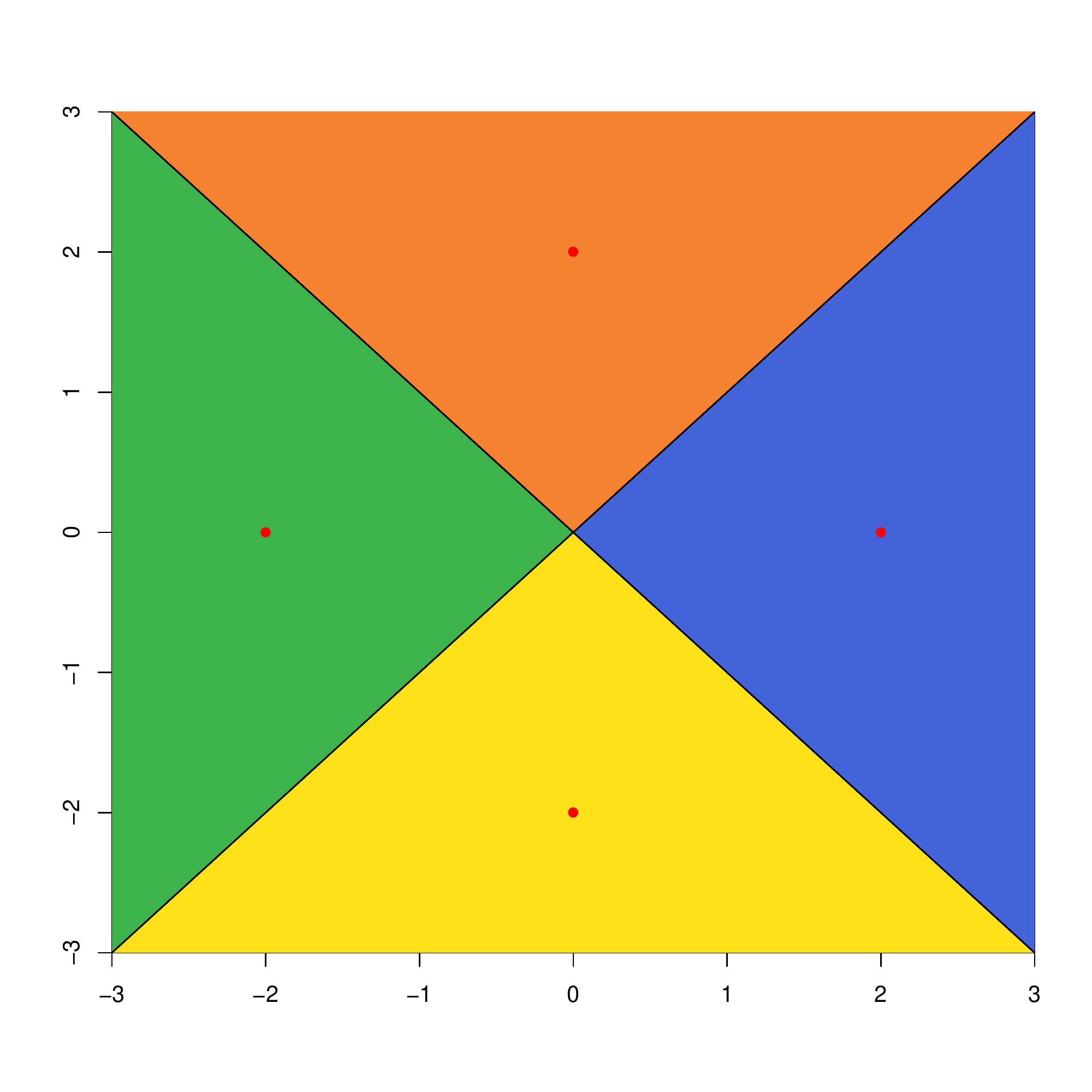}
\caption*{Circular 4 Cauchy}
\end{subfigure}
\begin{subfigure}{0.32 \linewidth}
\centering
\end{subfigure}

\caption{True clusters associated with the Circular 2, Circular 2 Cauchy, Circular 3 and Circular 4 Cauchy densities. The modes are plotted in red.}
\label{figure_true_clusters_circular}
\end{figure}

\subsection{Illustrative examples}
\label{sm:subsection_illustrative_examples}

In this subsection, we provide additional illustrations of clustering (as in Figure \ref{figure_true_clusters_ldc_kde}, Section \ref{section_simulations_and_data_analysis}), by considering the distributions (K) Trimodal III, Quadrimodal, and (L) Quadrimodal. For this, we compare the second and the third row in Figure \ref{sm:figure_true_clusters_ldc_kde} with the true clusters in the first row. Based on this comparison, we observe that the cluster estimates based on the proposed LLD method (second row) are better than those based on KDE (third row). Also, Figures \ref{sm:plot_clusters_bimodal_location_prob_all_points} through \ref{sm:plot_clusters_fountain_prob_all_points} provide further illustrations of clustering for different choices of $s$ and $q$, similar to the second rows of Figure \ref{figure_true_clusters_ldc_kde} and Figure \ref{sm:figure_true_clusters_ldc_kde}. These may be regarded as the bivariate analogue of Figure \ref{mixture_normals_plot2} for the densities Bimodal, (H) Bimodal IV, (K) Trimodal III, Quadrimodal, (L) Quadrimodal and \#10 Fountain. Since the parameter r does not affect the output of Algorithm 1, we leave it fixed at $r = 0.05$. However, the choice of $q$ affects the estimated clusters and a recommendation on its choice is  given in Section \ref{section_clustering}.
\begin{figure}
\centering
\begin{minipage}[c]{0.1\textwidth}
   \caption*{ True \\ clusters}
\end{minipage}
\begin{minipage}[c]{0.88\textwidth}
\begin{subfigure}{0.32 \linewidth}
\centering
\caption*{(K) Trimodal III}
\includegraphics[width=\linewidth]{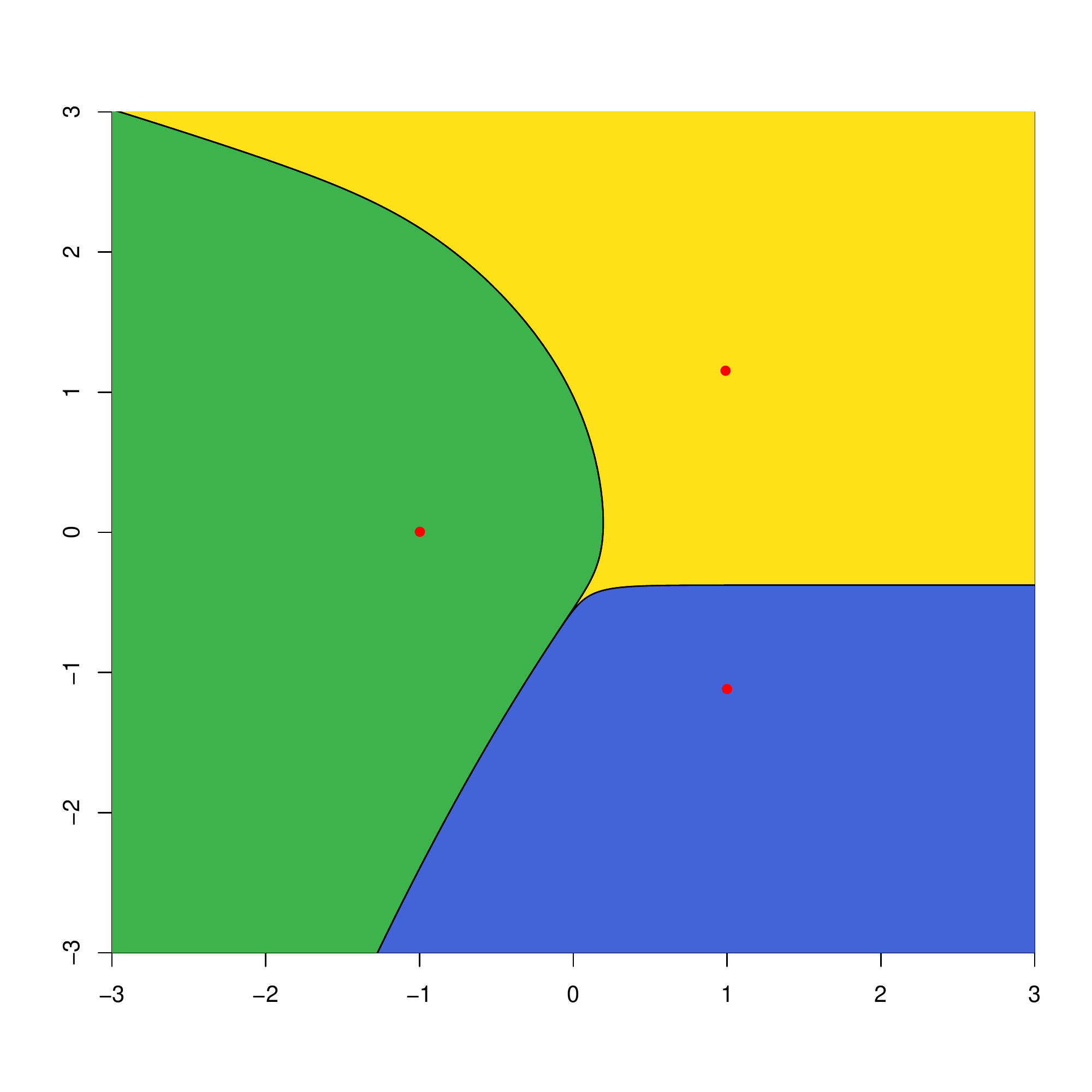}
\end{subfigure}
\begin{subfigure}{0.32 \linewidth}
\centering
\caption*{Quadrimodal}
\includegraphics[width=\linewidth]{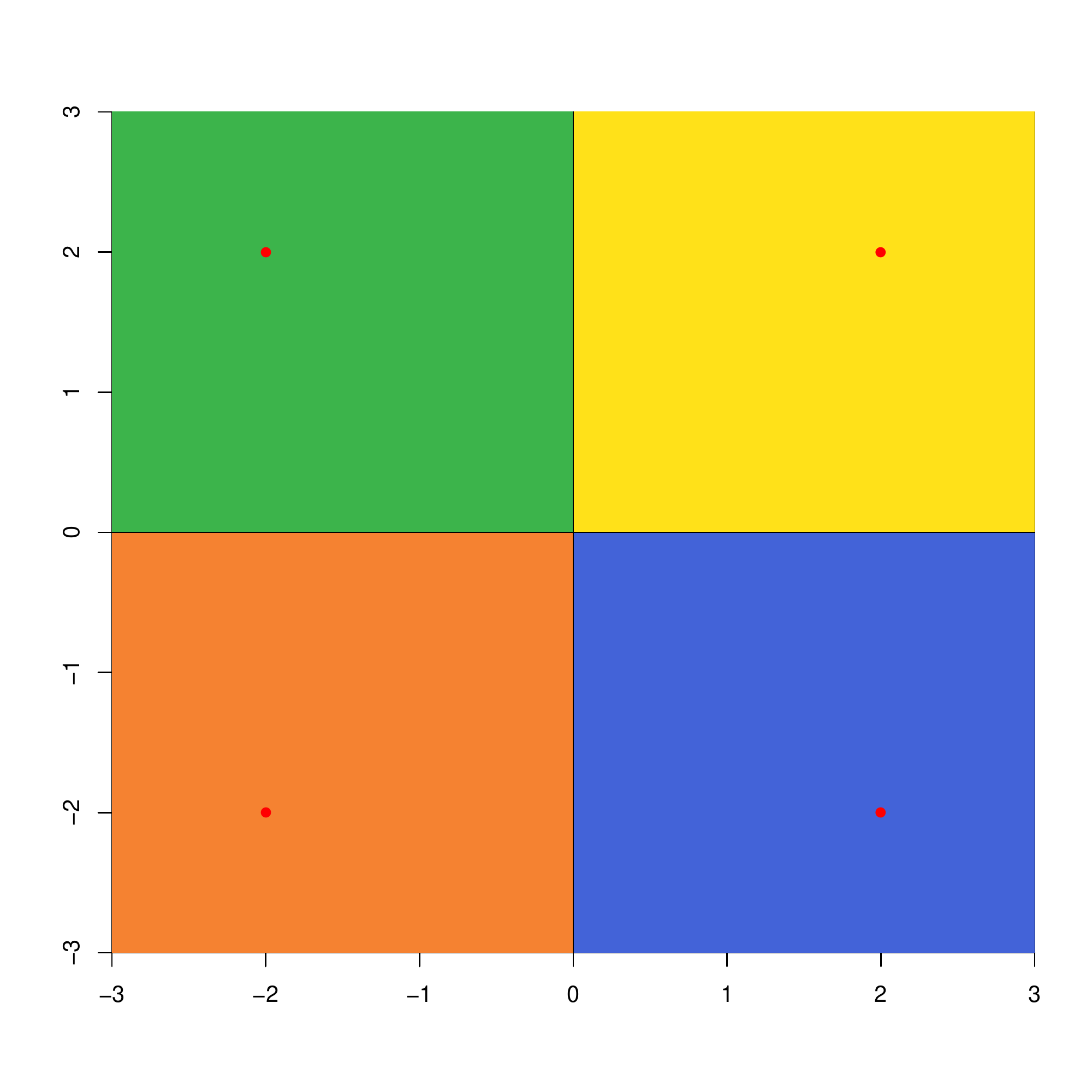}
\end{subfigure}
\centering
\begin{subfigure}{0.32 \linewidth}
\caption*{(L) Quadrimodal}
\includegraphics[width=\linewidth]{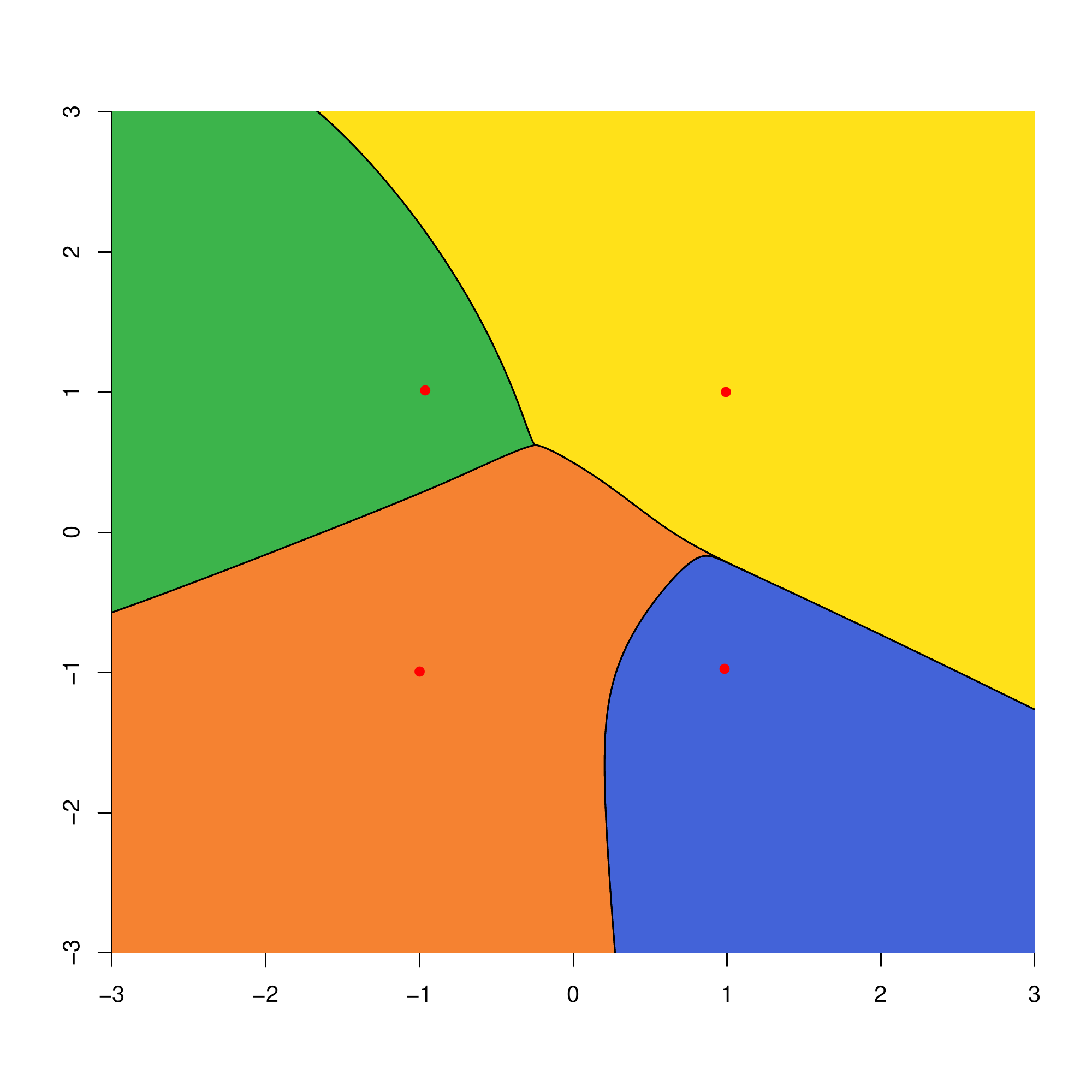}
\end{subfigure}
\end{minipage}
\begin{minipage}[c]{0.1\textwidth}
   \caption*{LLD \\ estimated \\ clusters}
\end{minipage}
\begin{minipage}[c]{0.88\textwidth}
\centering
\begin{subfigure}{0.32 \linewidth}
\centering
\includegraphics[width=\linewidth]{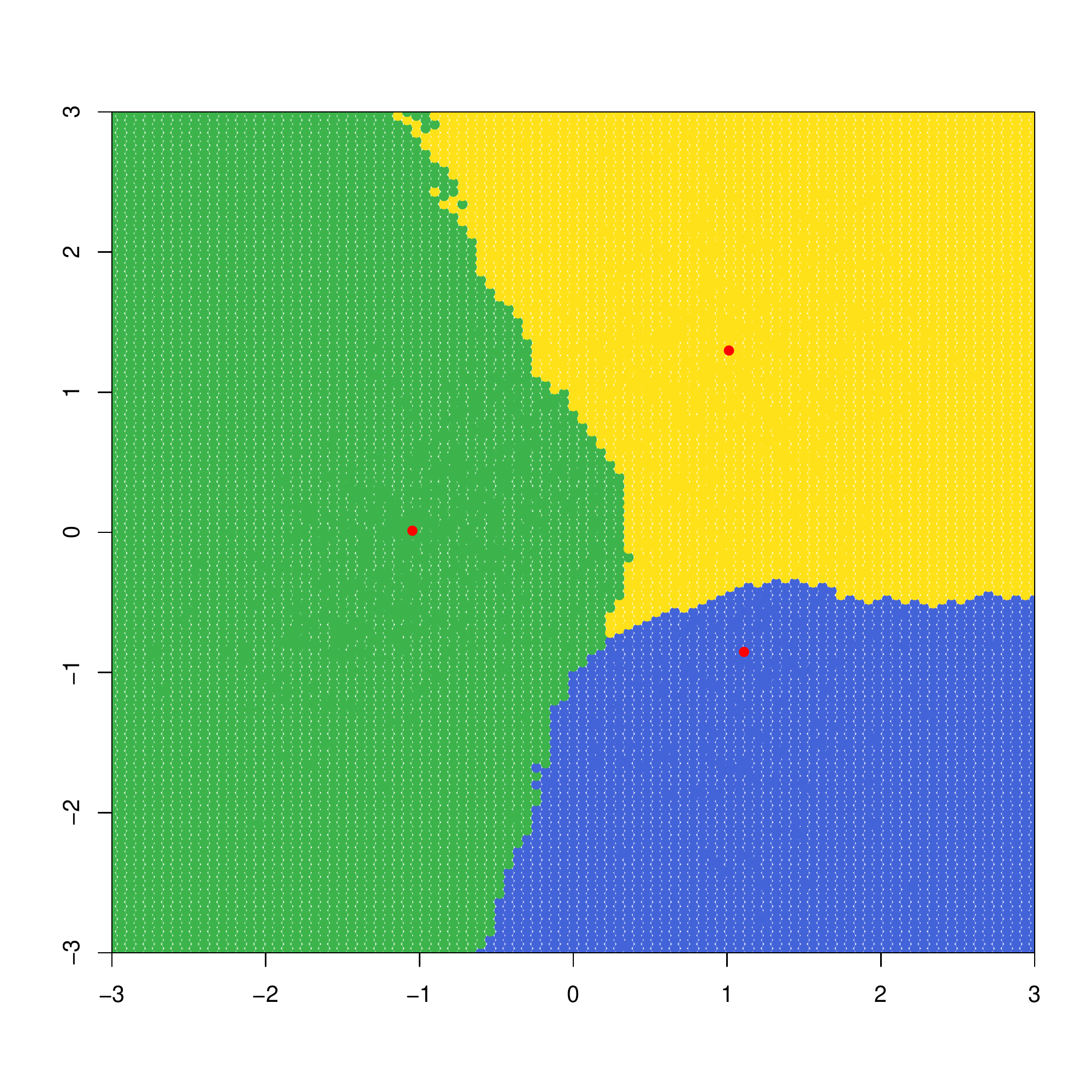}
\end{subfigure}
\begin{subfigure}{0.32 \linewidth}
\centering
\includegraphics[width=\linewidth]{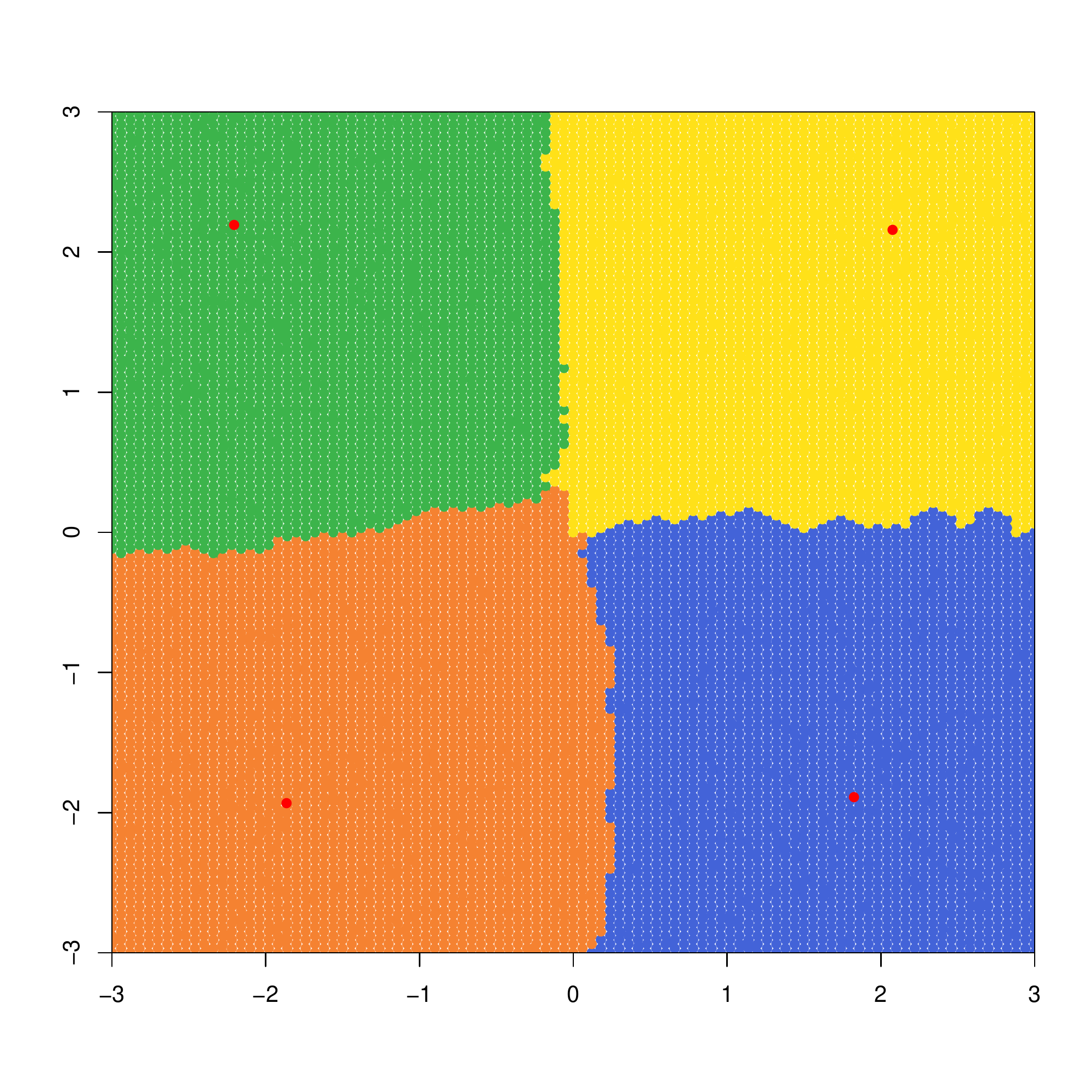}
\end{subfigure}
\begin{subfigure}{0.32 \linewidth}
\centering
\includegraphics[width=\linewidth]{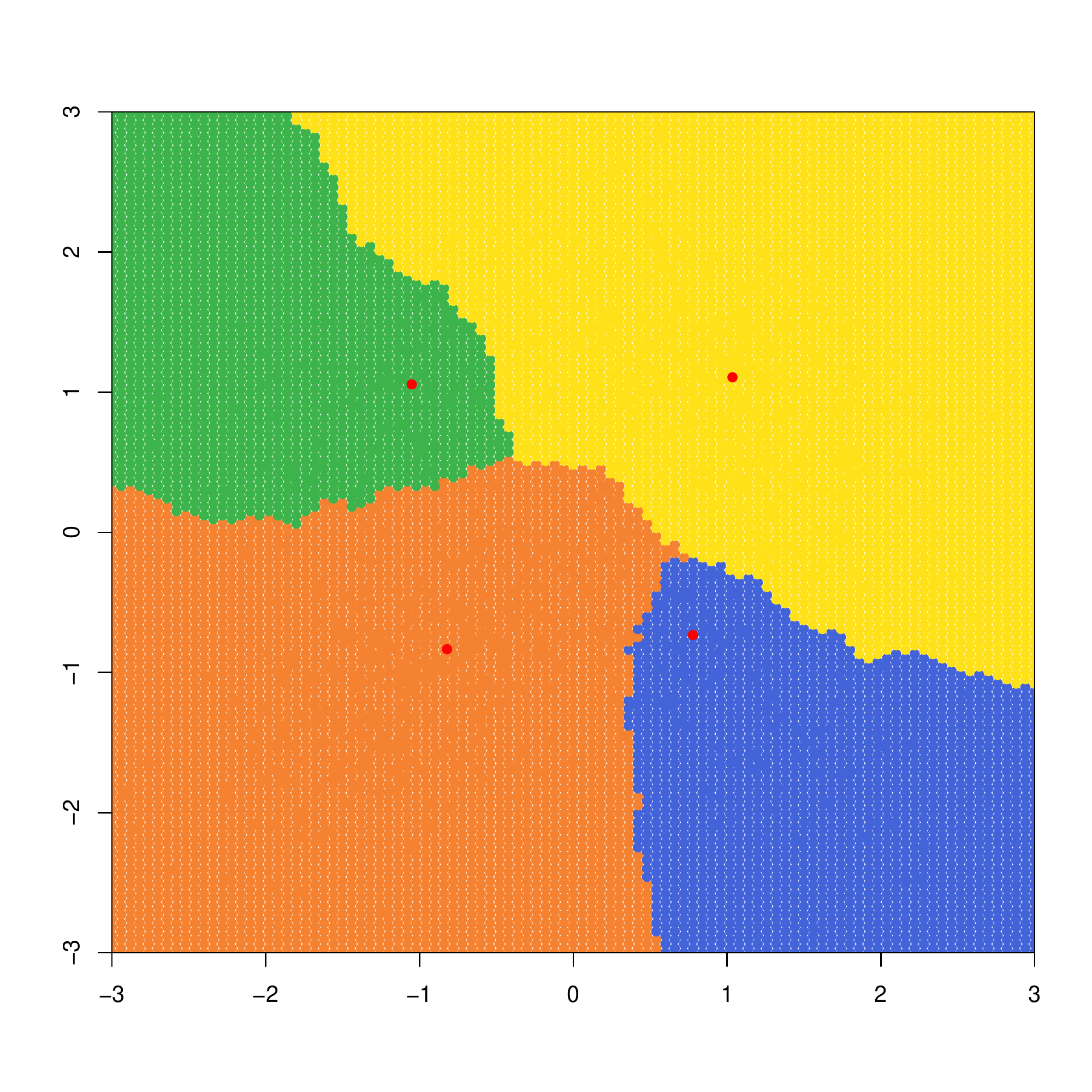}
\end{subfigure}
\end{minipage}
\begin{minipage}[c]{0.1\textwidth}
   \caption*{KDE \\ estimated \\ clusters}
\end{minipage}
\begin{minipage}[c]{0.88\textwidth}
\centering
\begin{subfigure}{0.32 \linewidth}
\centering
\includegraphics[width=\linewidth]{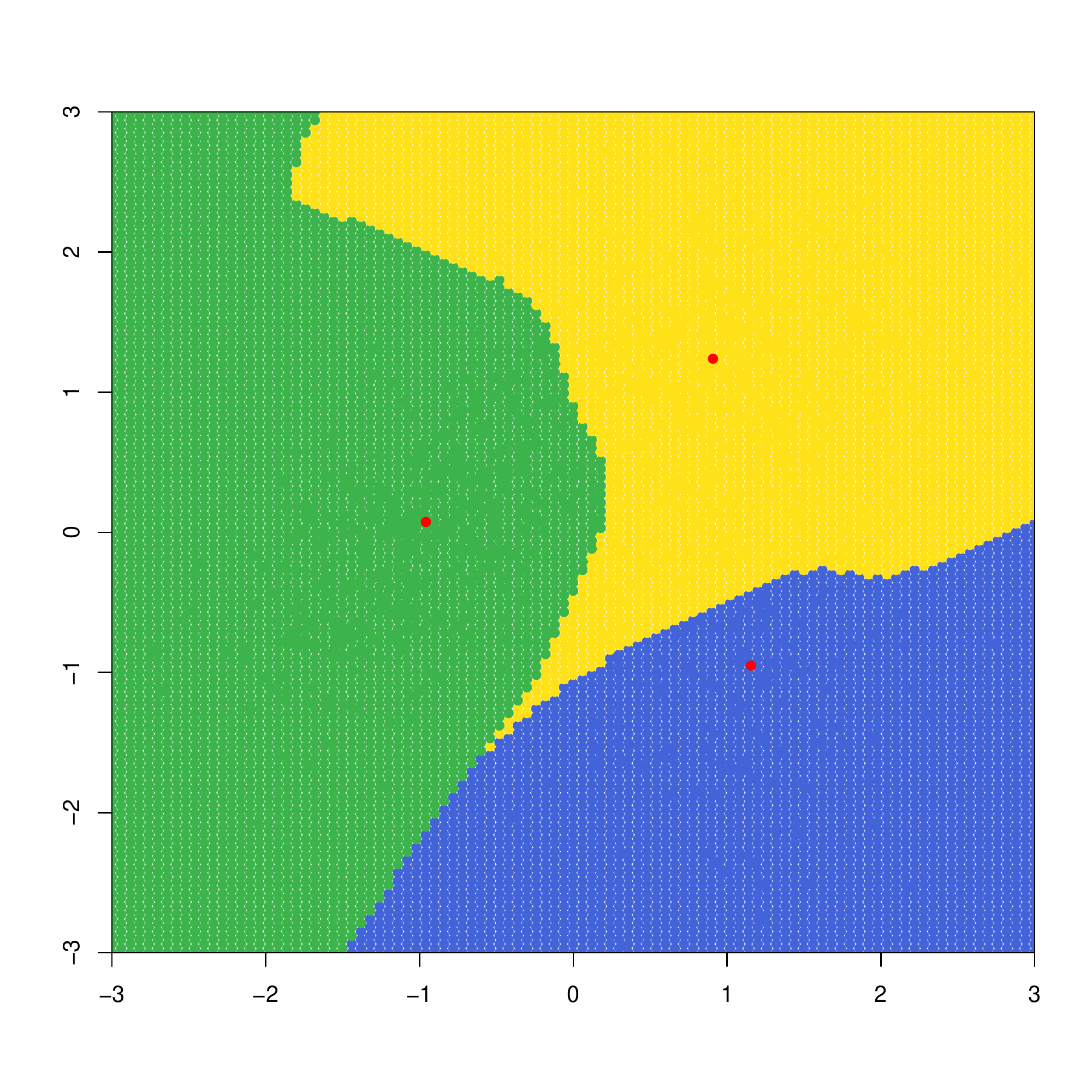}
\end{subfigure}
\begin{subfigure}{0.32 \linewidth}
\centering
\includegraphics[width=\linewidth]{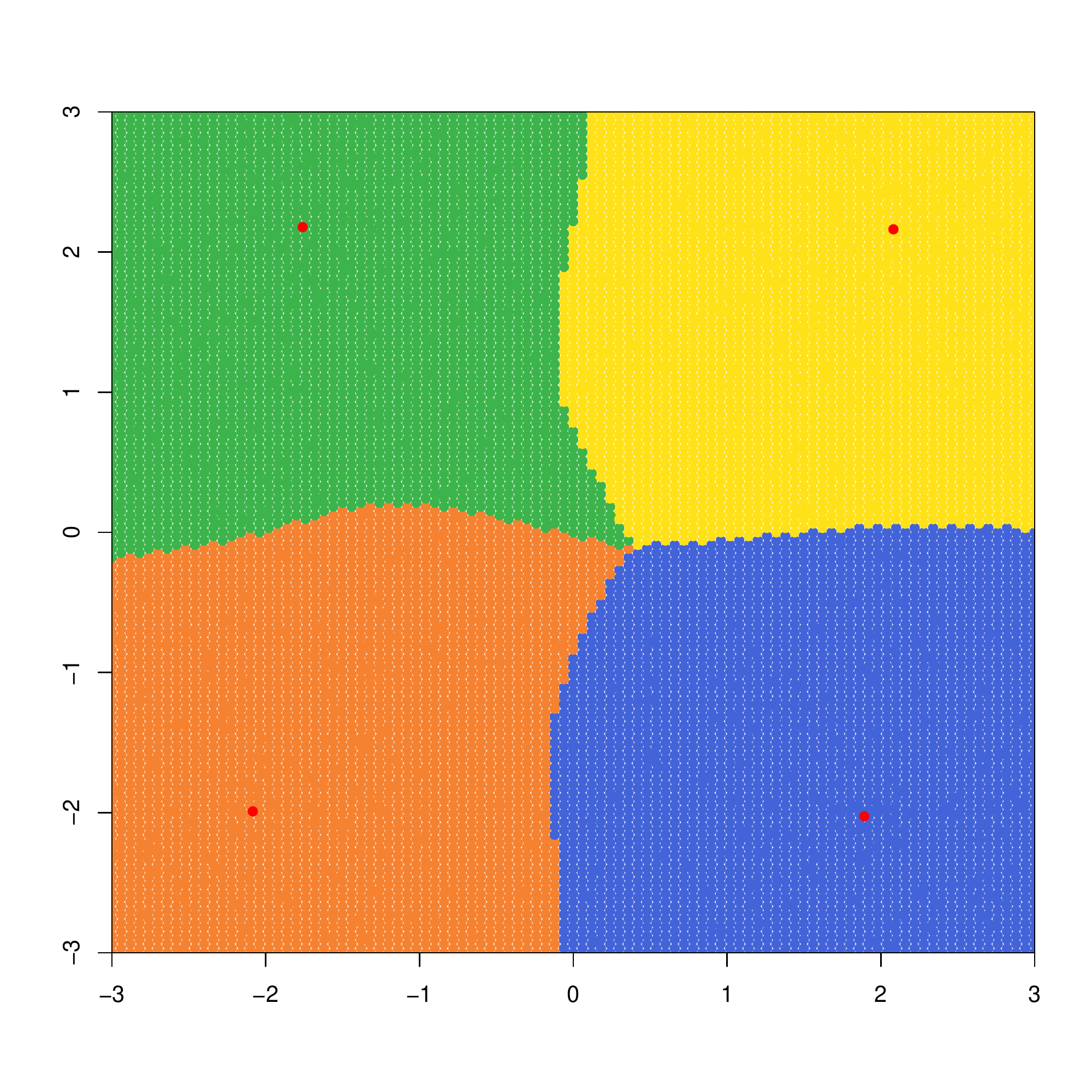}
\end{subfigure}
\begin{subfigure}{0.32 \linewidth}
\centering
\includegraphics[width=\linewidth]{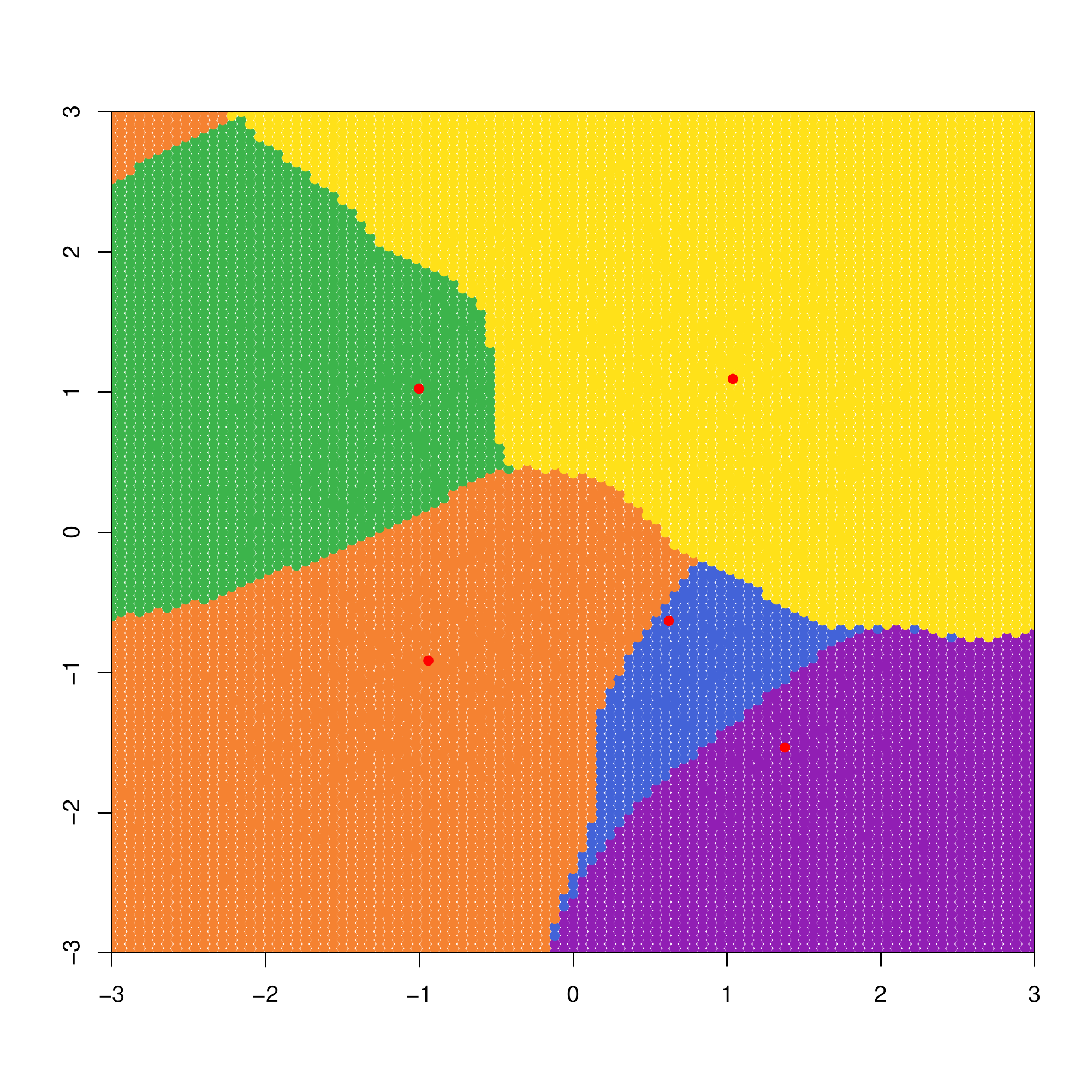}
\end{subfigure}
\end{minipage}

\caption{Clusters associated with the (K) Trimodal III (left), Quadrimodal (middle) and (L) Quadrimodal (right) densities. True clusters (first row). Local depth clustering based on $n=1000$ samples from these densities and parameters $q=0.05$, $s=50$ and $r=0.05$ (second row). Kernel density estimator clustering (third row). The true modes (first row) and the predicted modes (second and third rows) are plotted in red.}
\label{sm:figure_true_clusters_ldc_kde}
\end{figure}

\begin{figure}
\centering
\includegraphics[width=0.32\linewidth]{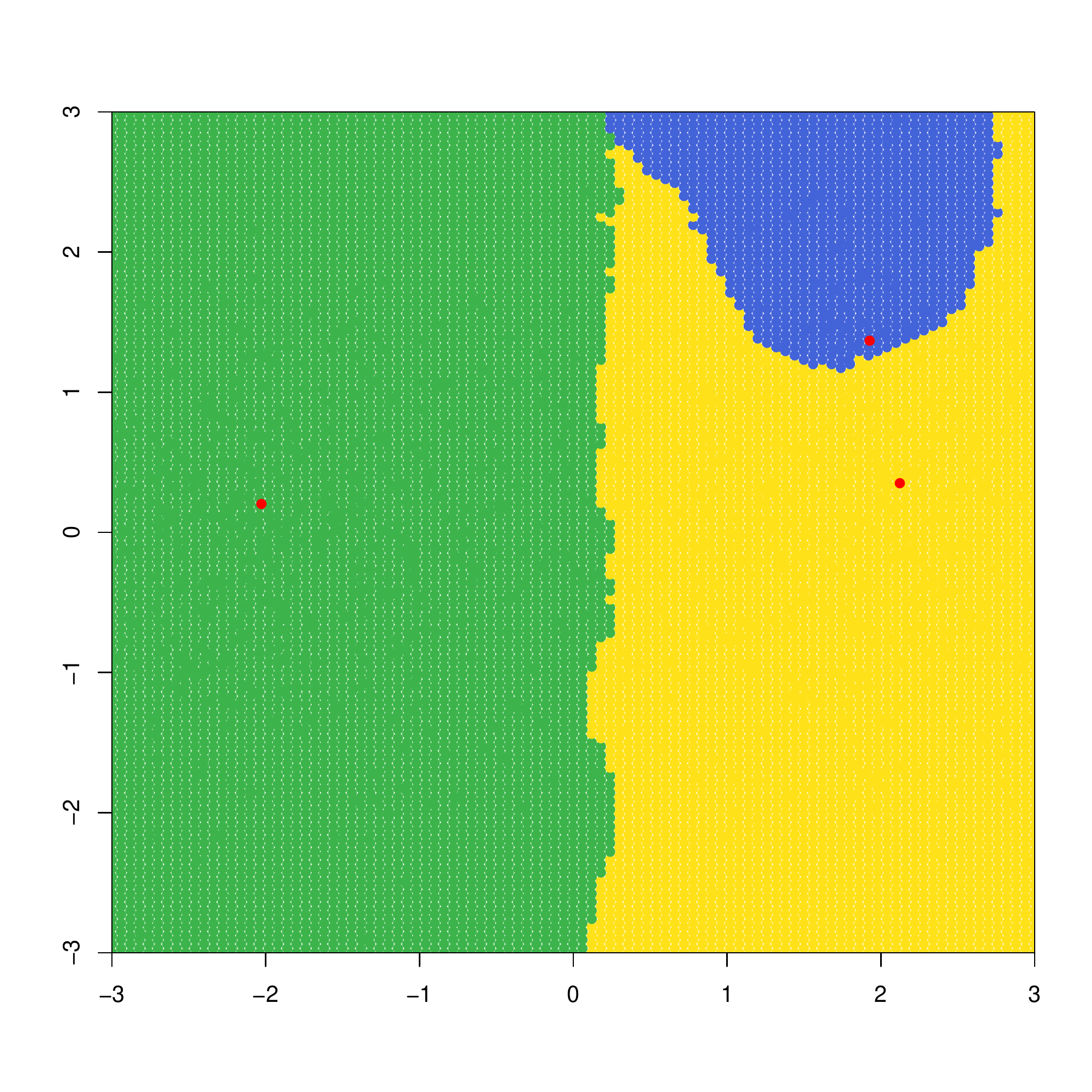}
\includegraphics[width=0.32\linewidth]{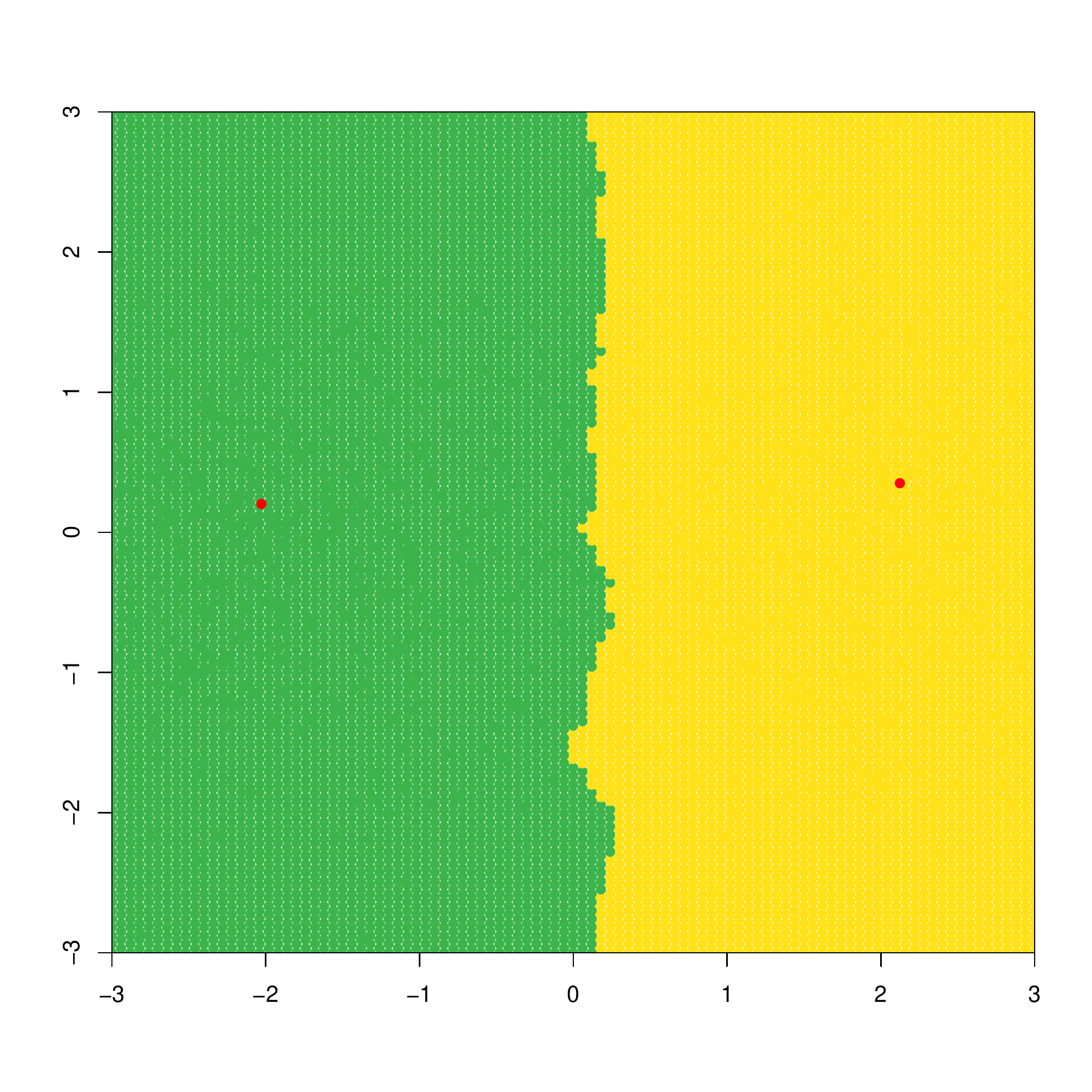}
\includegraphics[width=0.32\linewidth]{{lens-bimodal_location-prob-0.63-0.05-50-0.05-1000-ldc-all}.pdf}
\includegraphics[width=0.32\linewidth]{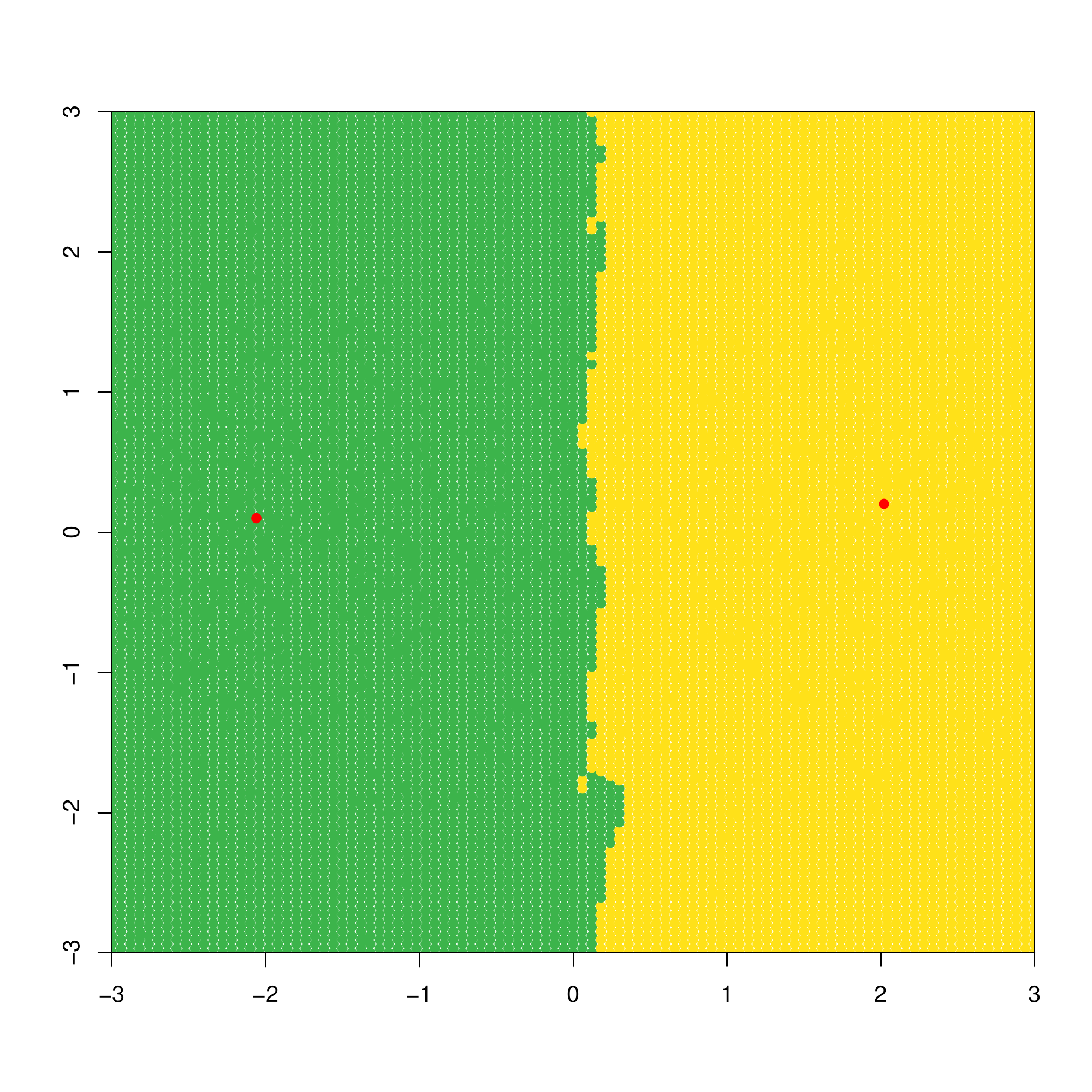}
\includegraphics[width=0.32\linewidth]{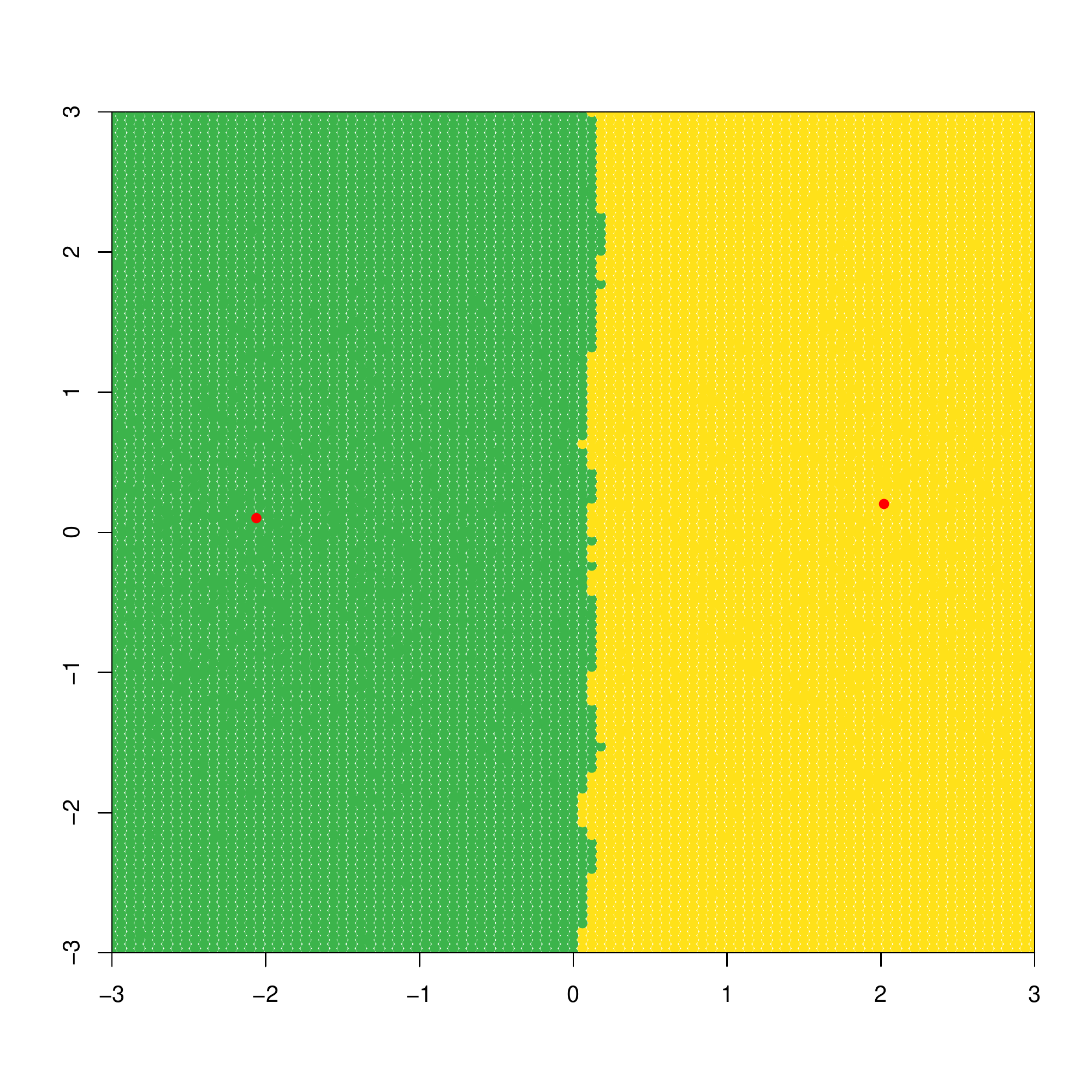}
\includegraphics[width=0.32\linewidth]{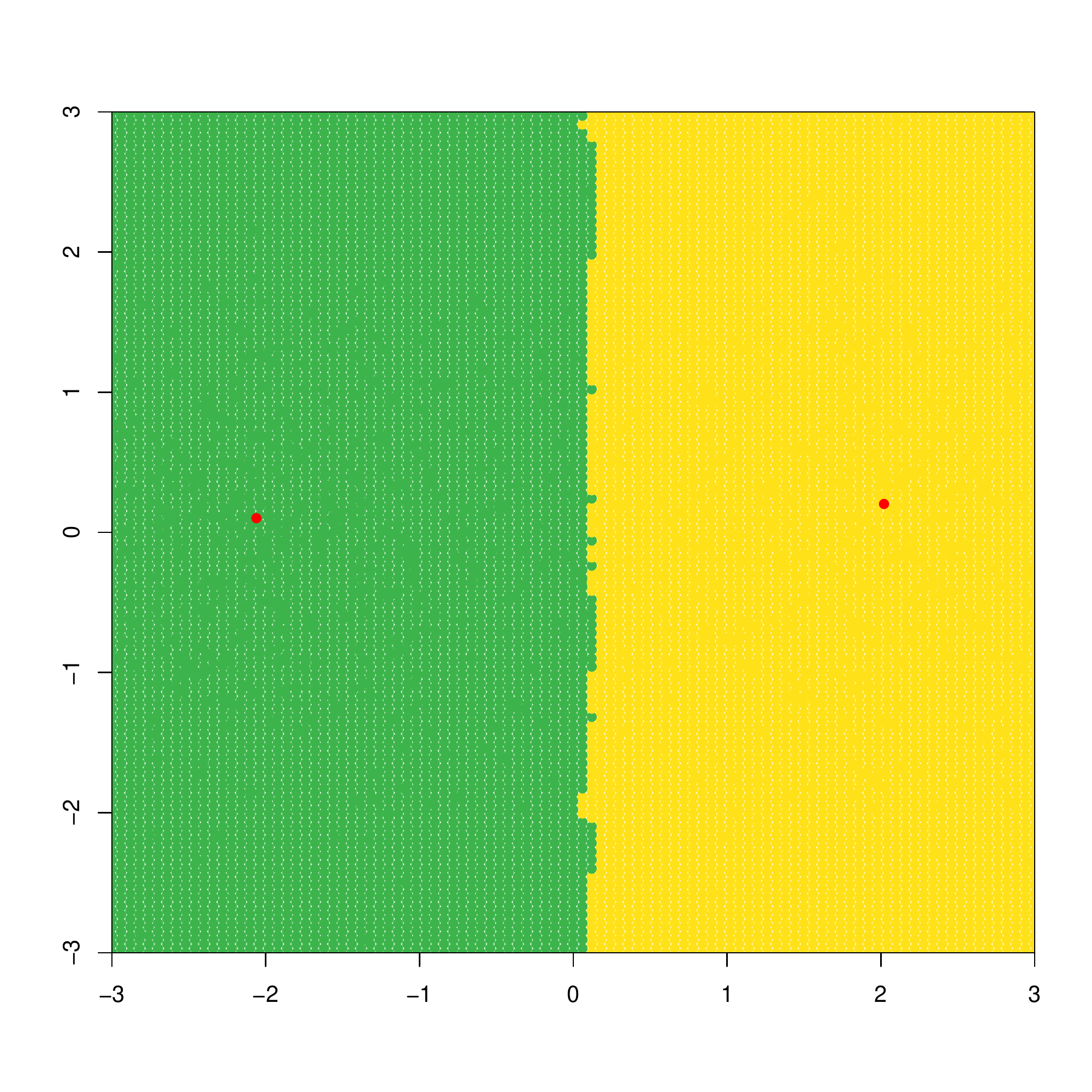}
\includegraphics[width=0.32\linewidth]{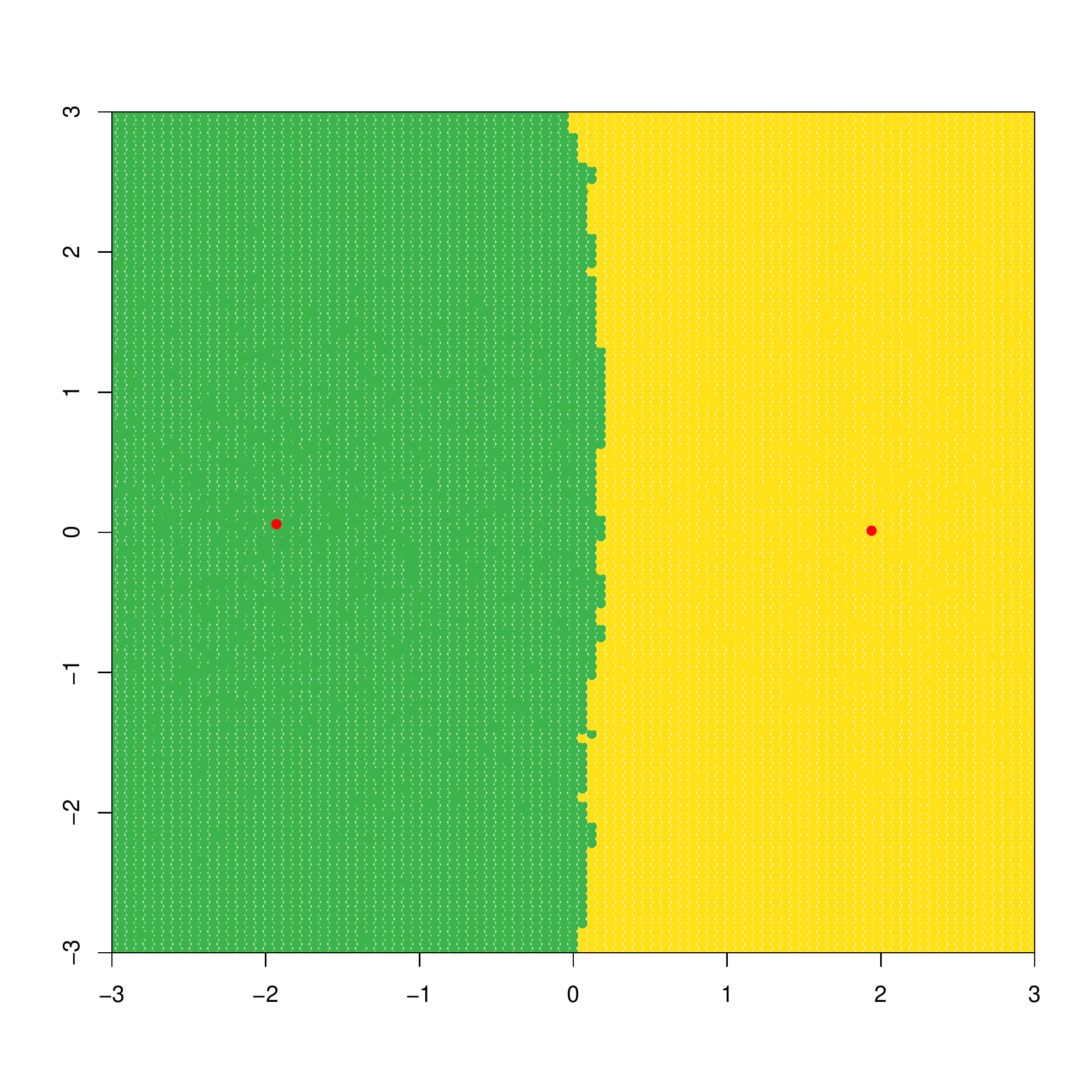}
\includegraphics[width=0.32\linewidth]{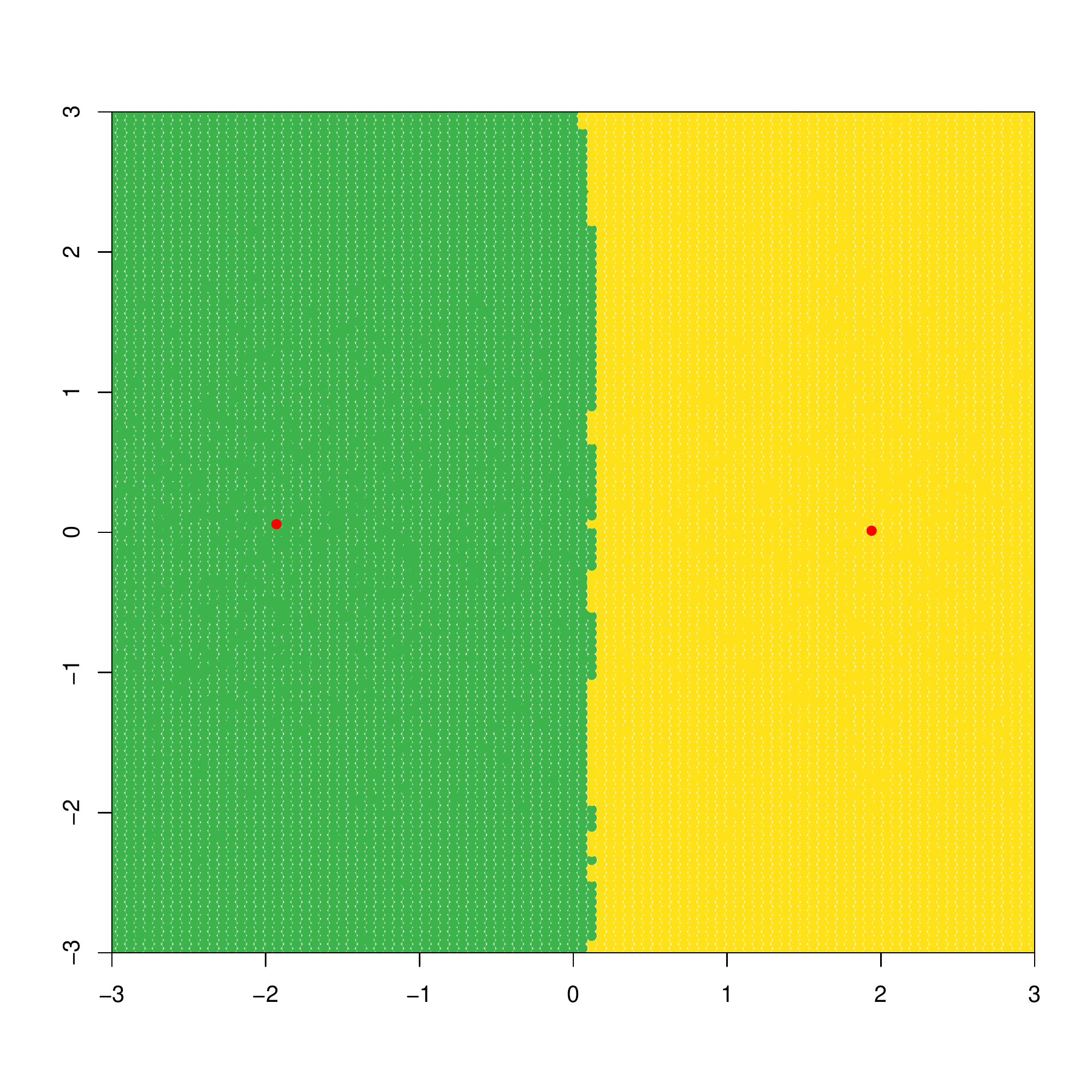}
\includegraphics[width=0.32\linewidth]{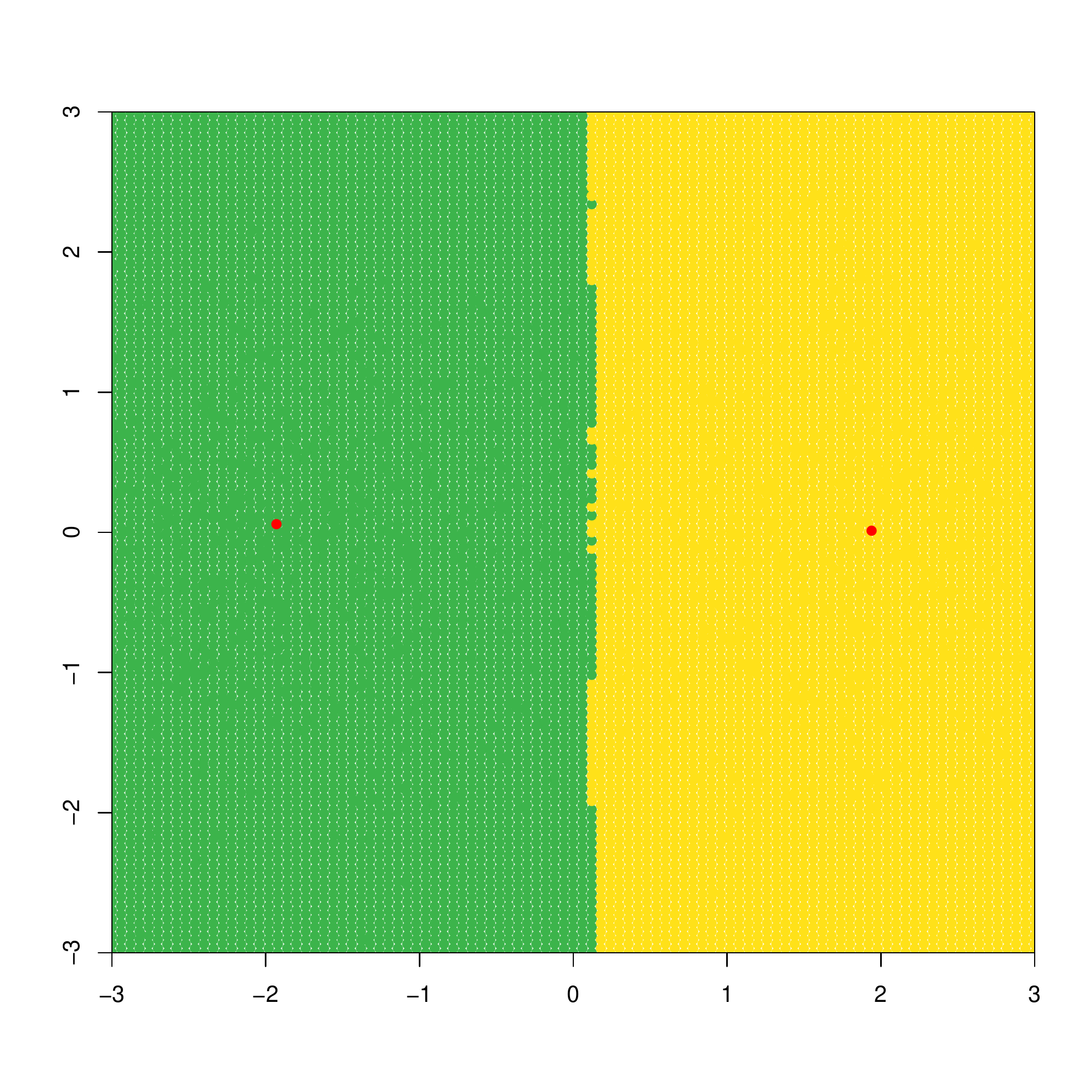}
\includegraphics[width=0.32\linewidth]{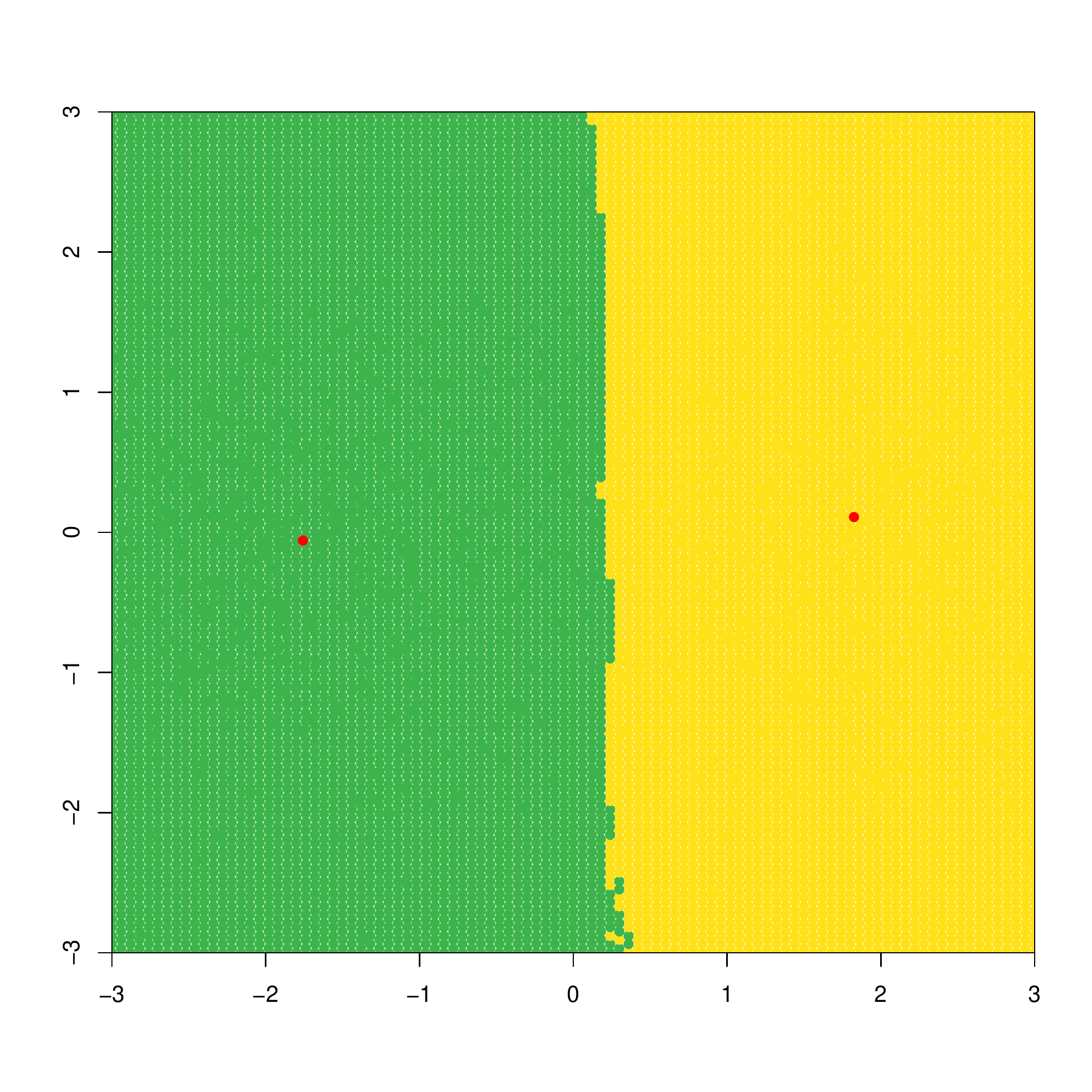}
\includegraphics[width=0.32\linewidth]{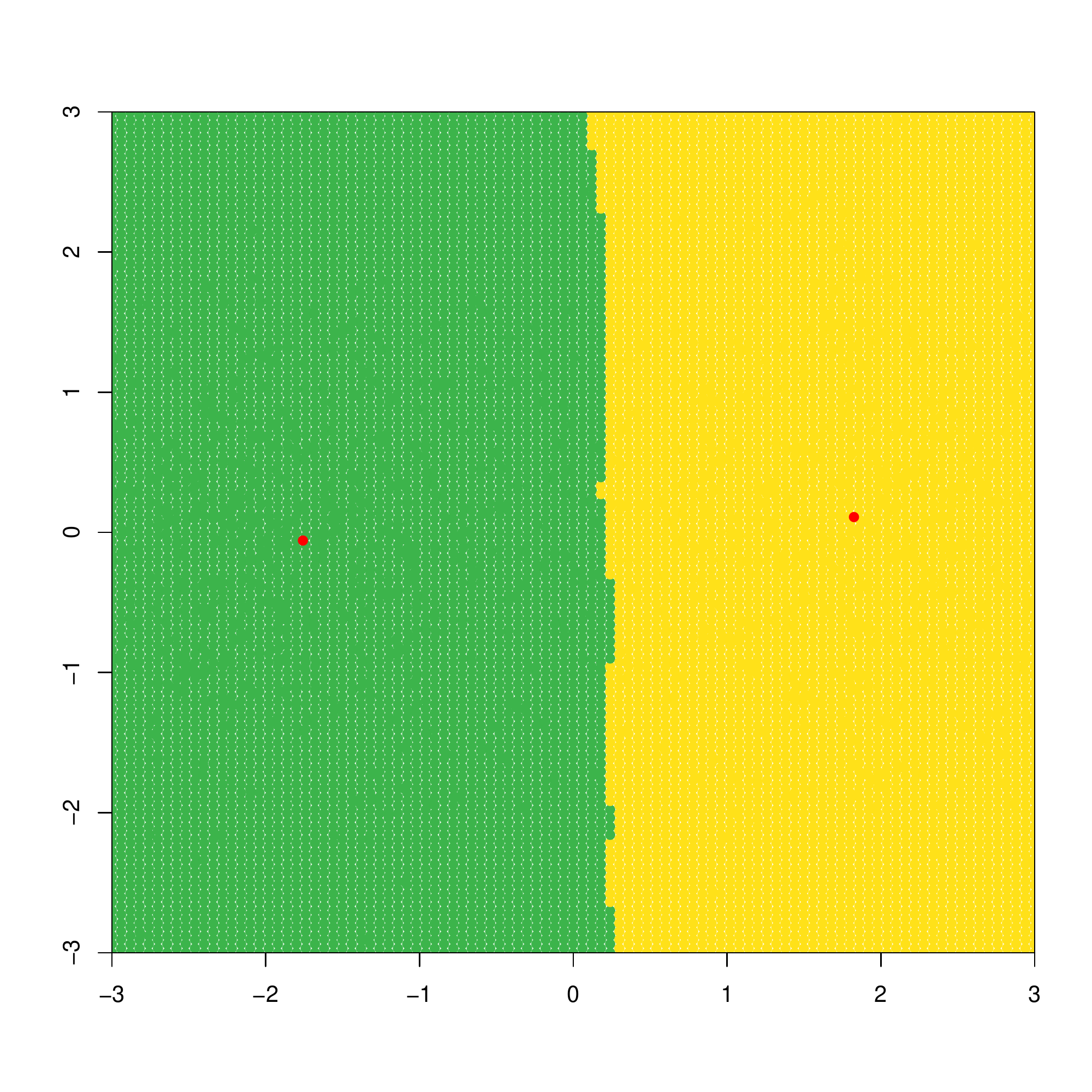}
\includegraphics[width=0.32\linewidth]{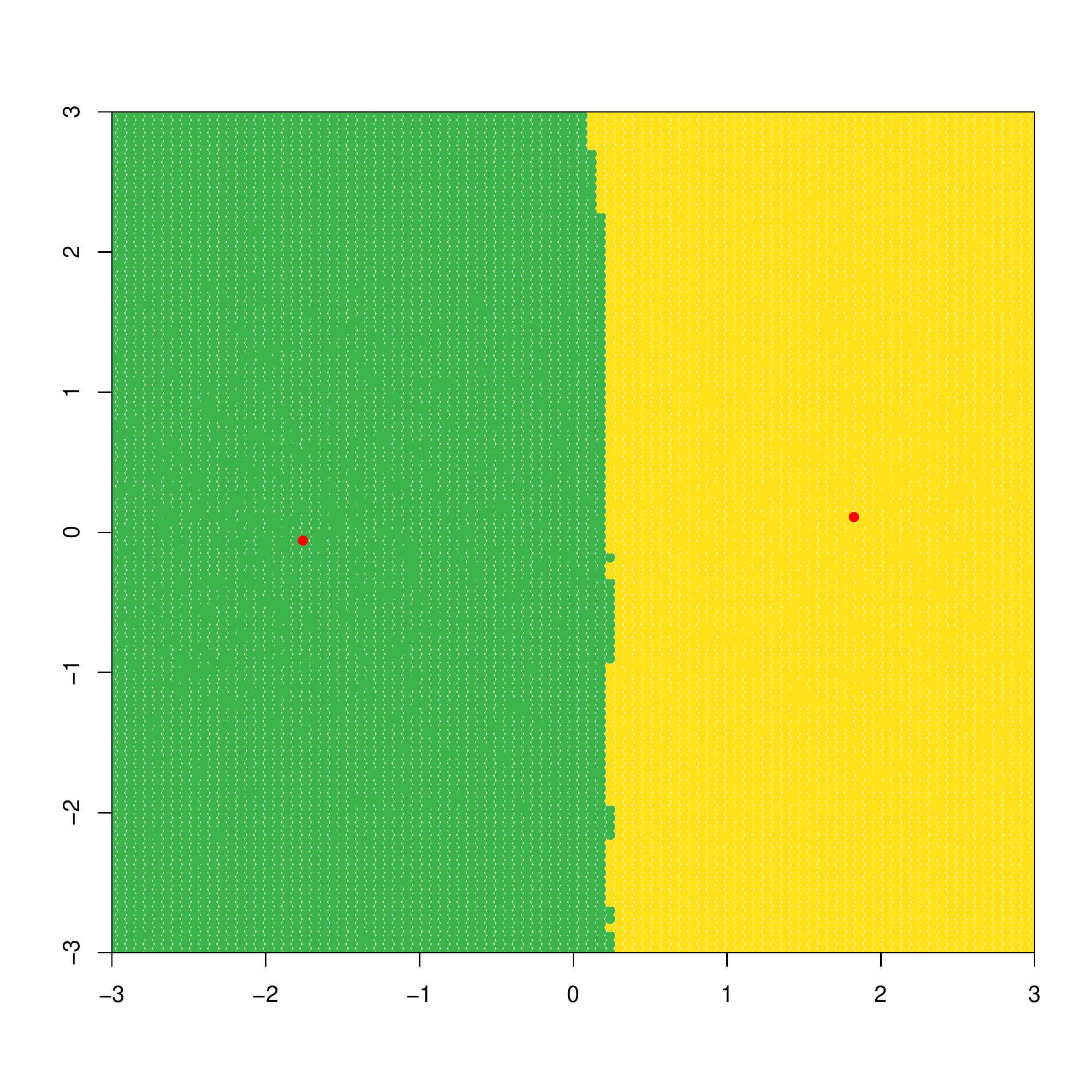}

\caption{Local depth clustering of $n=1000$ samples from the Bimodal density. The predicted local maxima are plotted in red. The parameters are $r=0.05$, $s=10,30,50$ in each column (from left to right) and $q=0.05,0.10,0.25,0.50$ in each row (from the top down).}
\label{sm:plot_clusters_bimodal_location_prob_all_points}
\end{figure}

\begin{figure}
\centering
\includegraphics[width=0.32\linewidth]{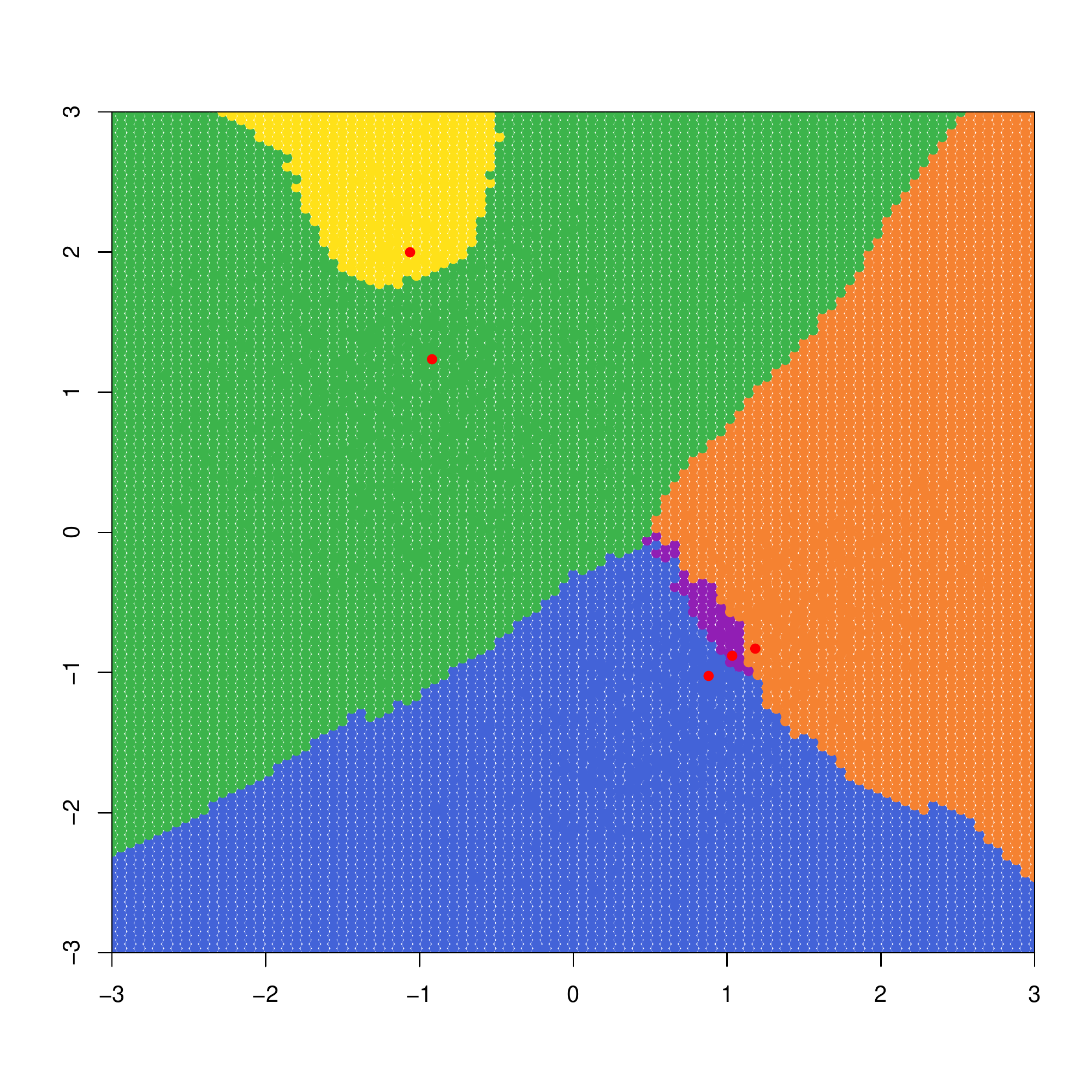}
\includegraphics[width=0.32\linewidth]{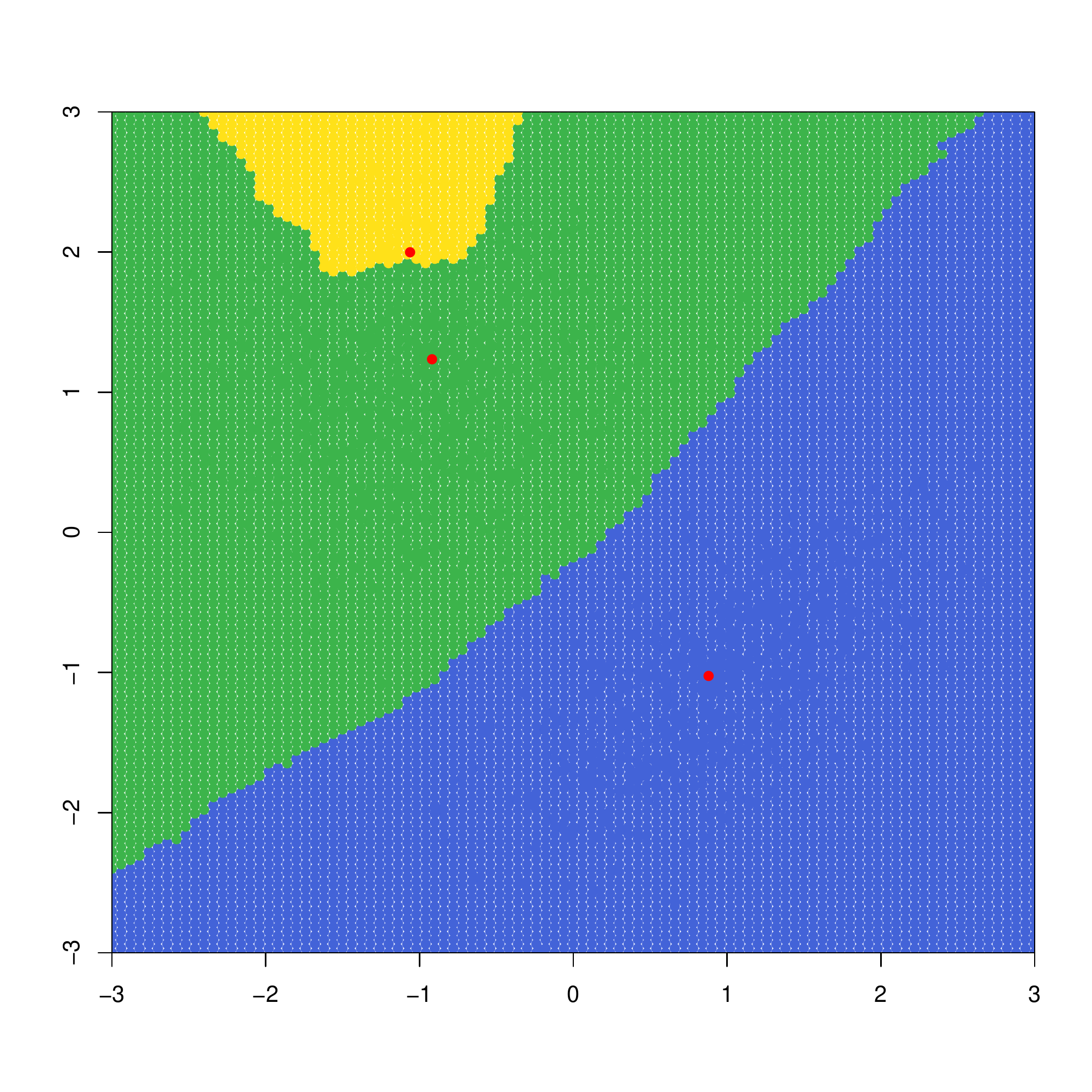}
\includegraphics[width=0.32\linewidth]{{lens-bimodal_4-prob-0.39-0.05-50-0.05-1000-ldc-all}.pdf}
\includegraphics[width=0.32\linewidth]{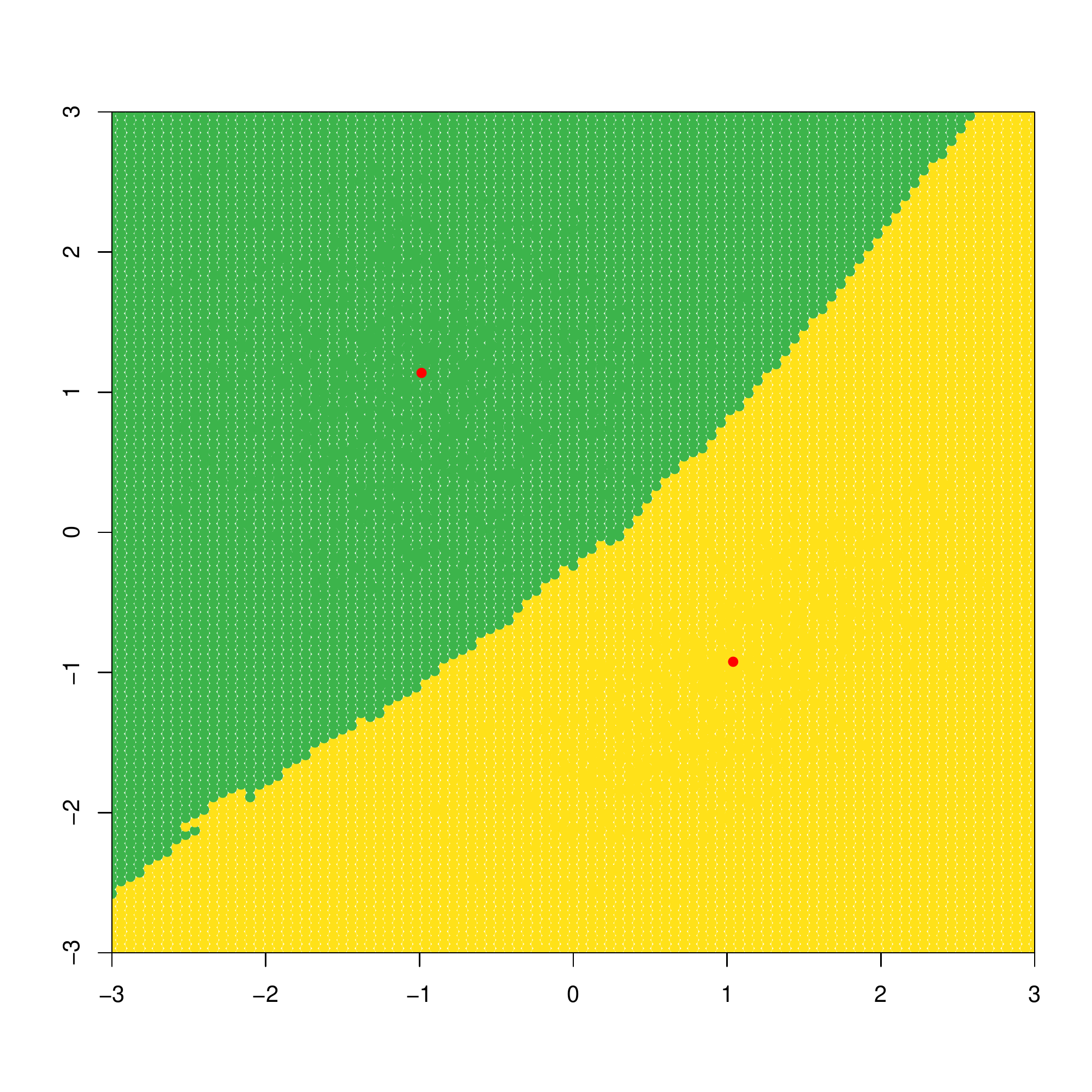}
\includegraphics[width=0.32\linewidth]{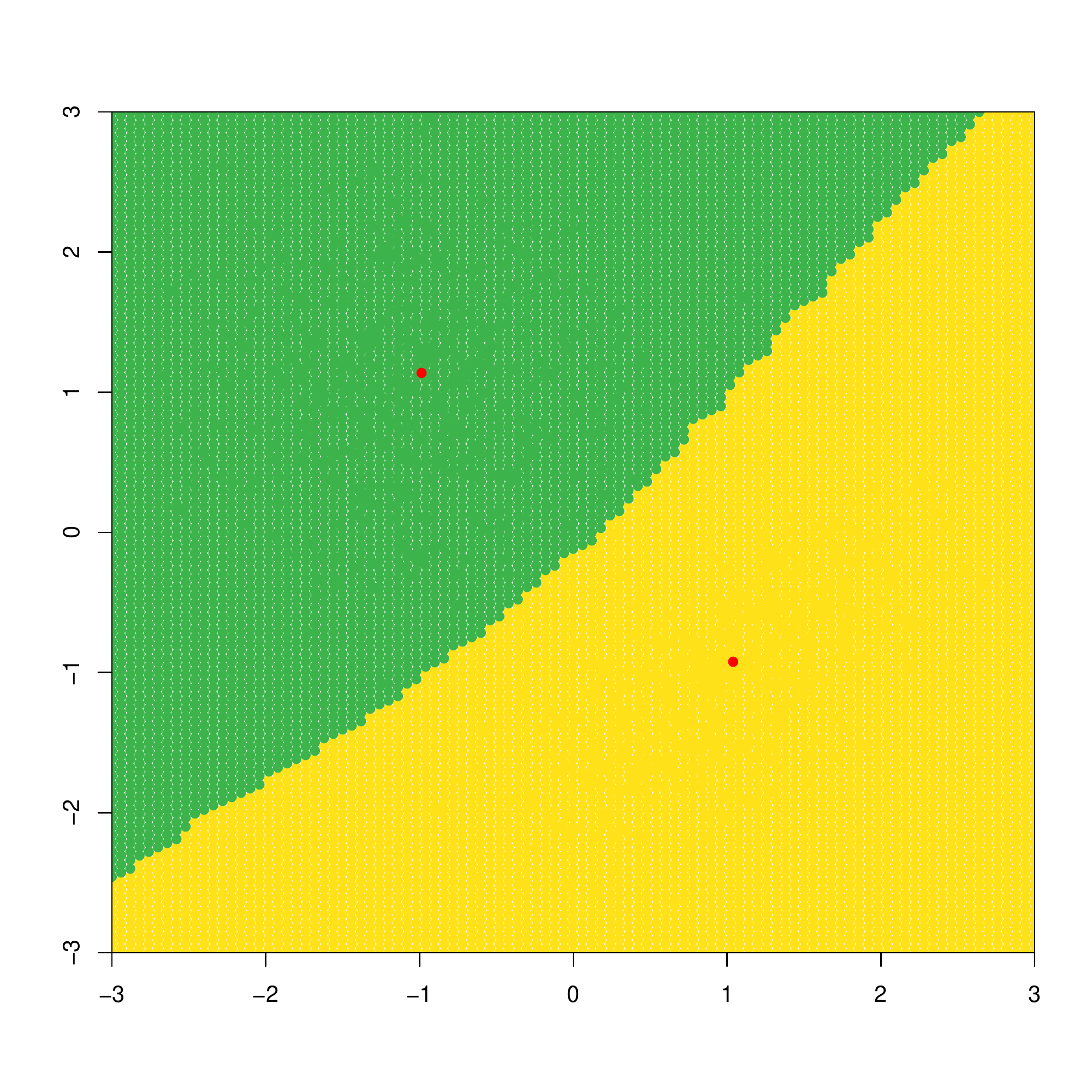}
\includegraphics[width=0.32\linewidth]{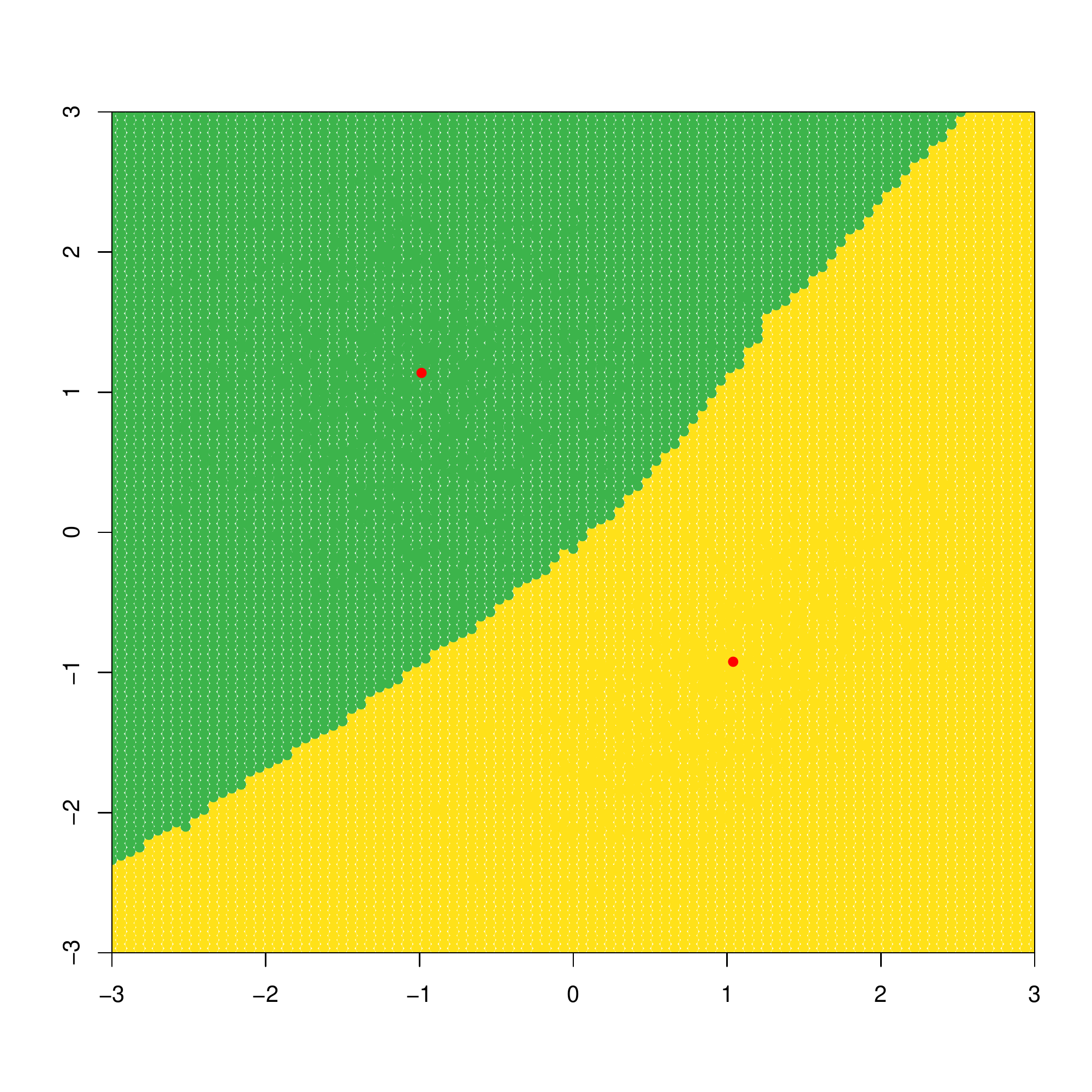}
\includegraphics[width=0.32\linewidth]{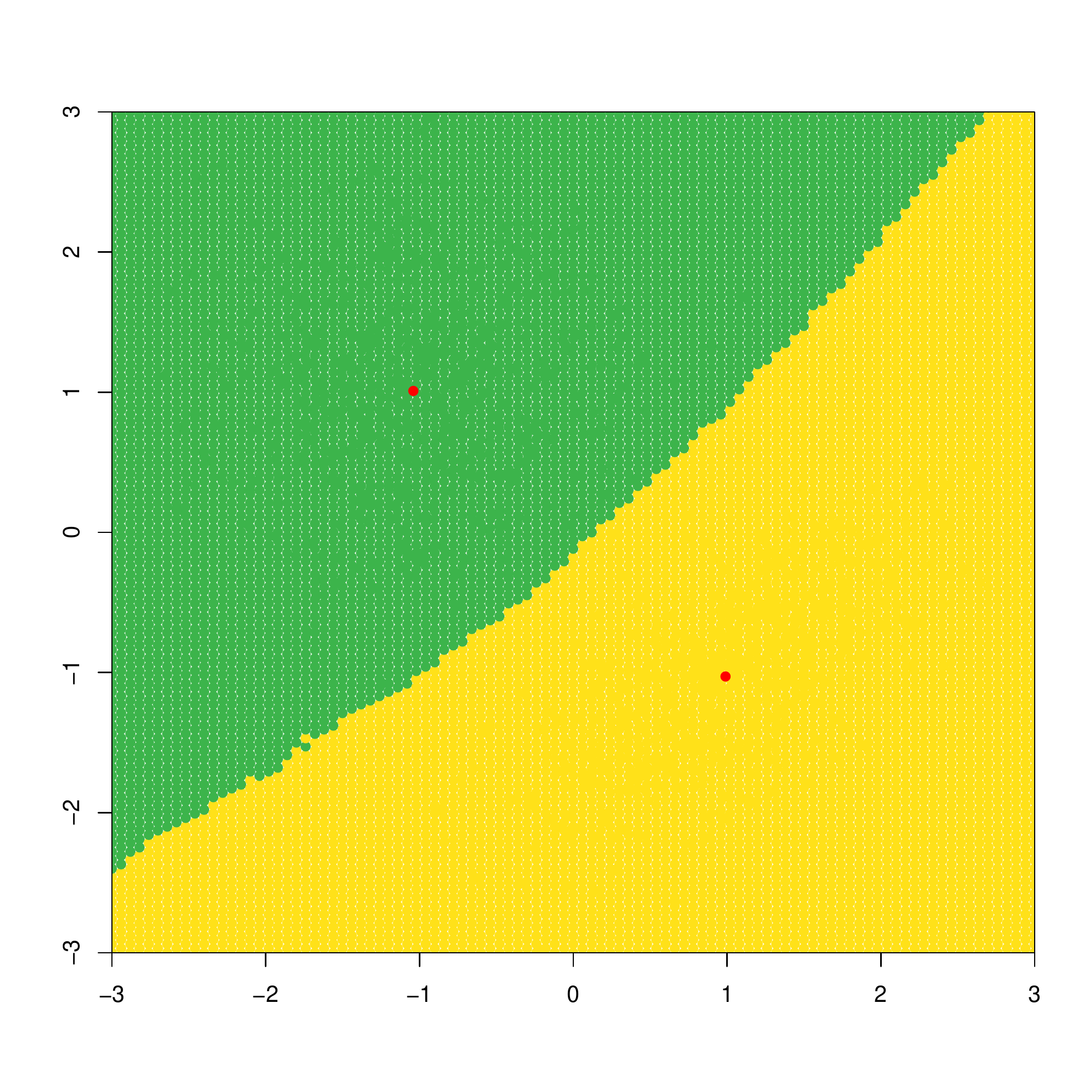}
\includegraphics[width=0.32\linewidth]{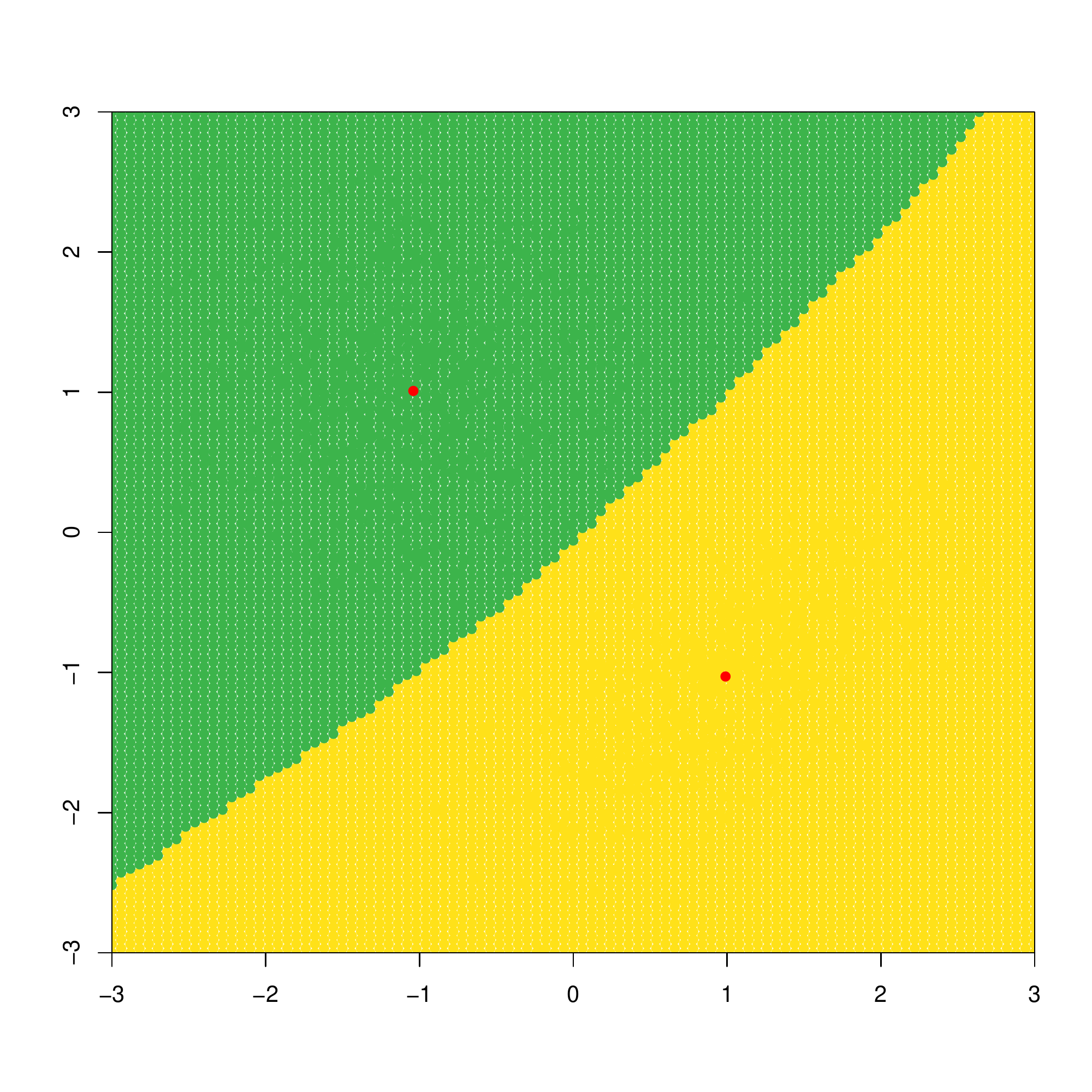}
\includegraphics[width=0.32\linewidth]{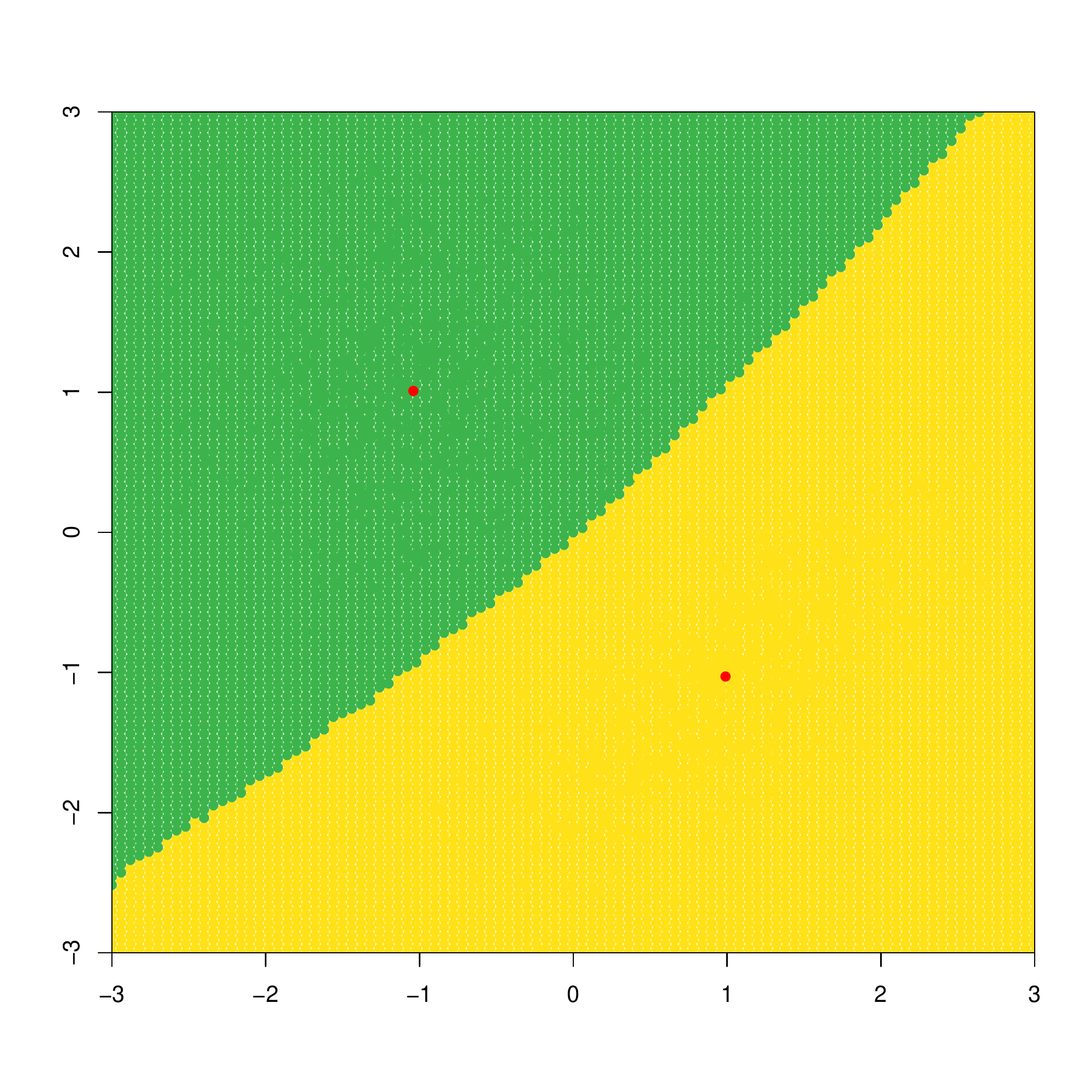}
\includegraphics[width=0.32\linewidth]{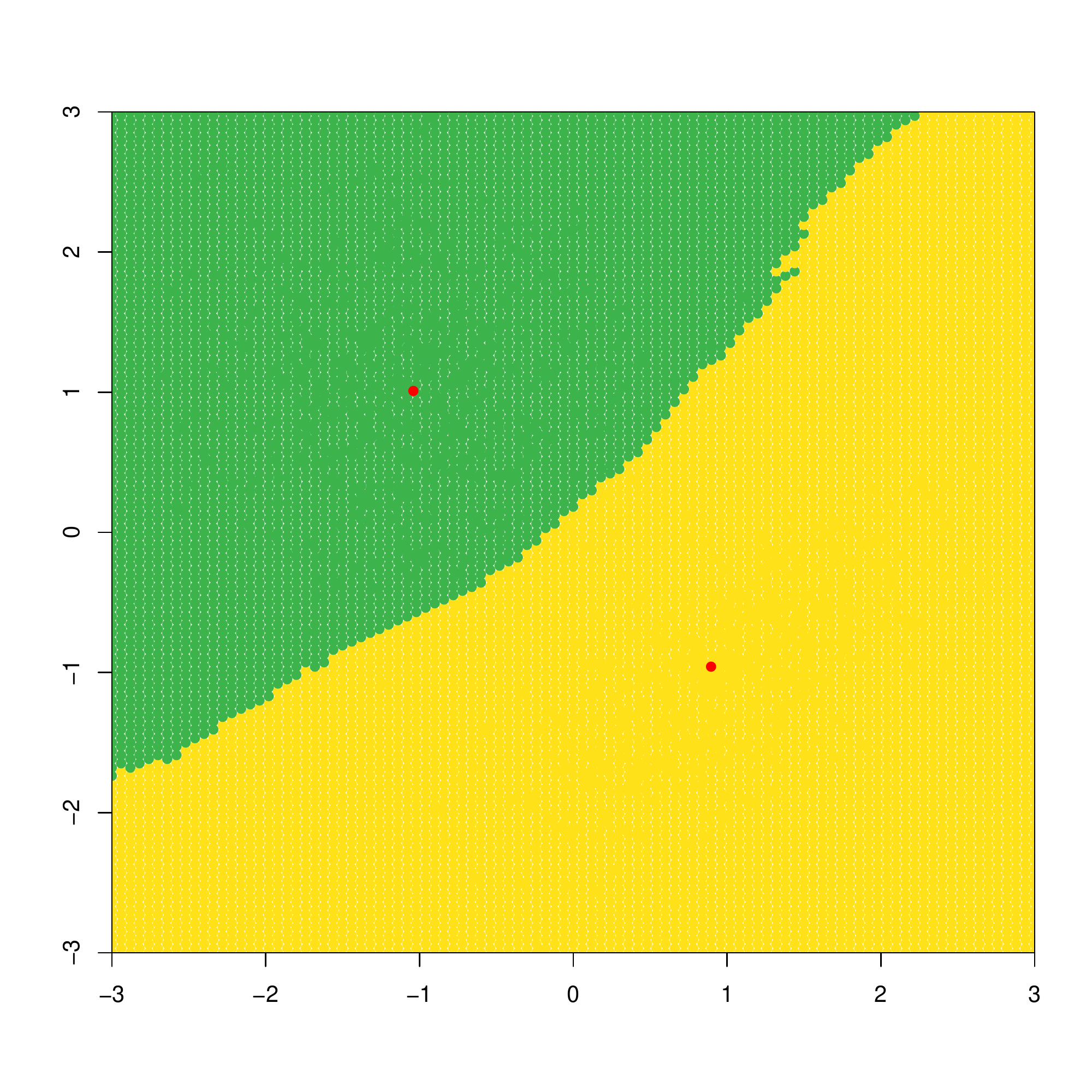}
\includegraphics[width=0.32\linewidth]{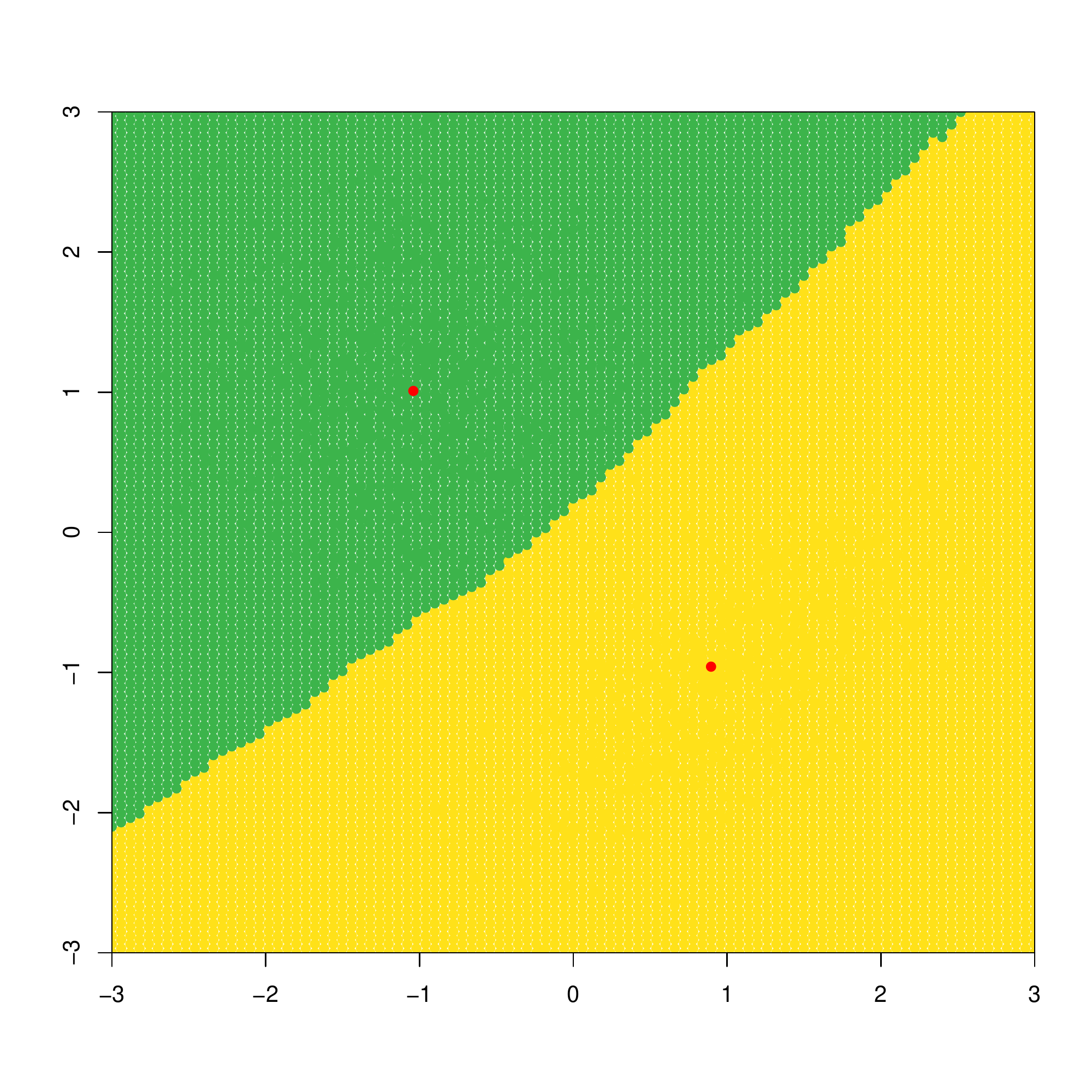}
\includegraphics[width=0.32\linewidth]{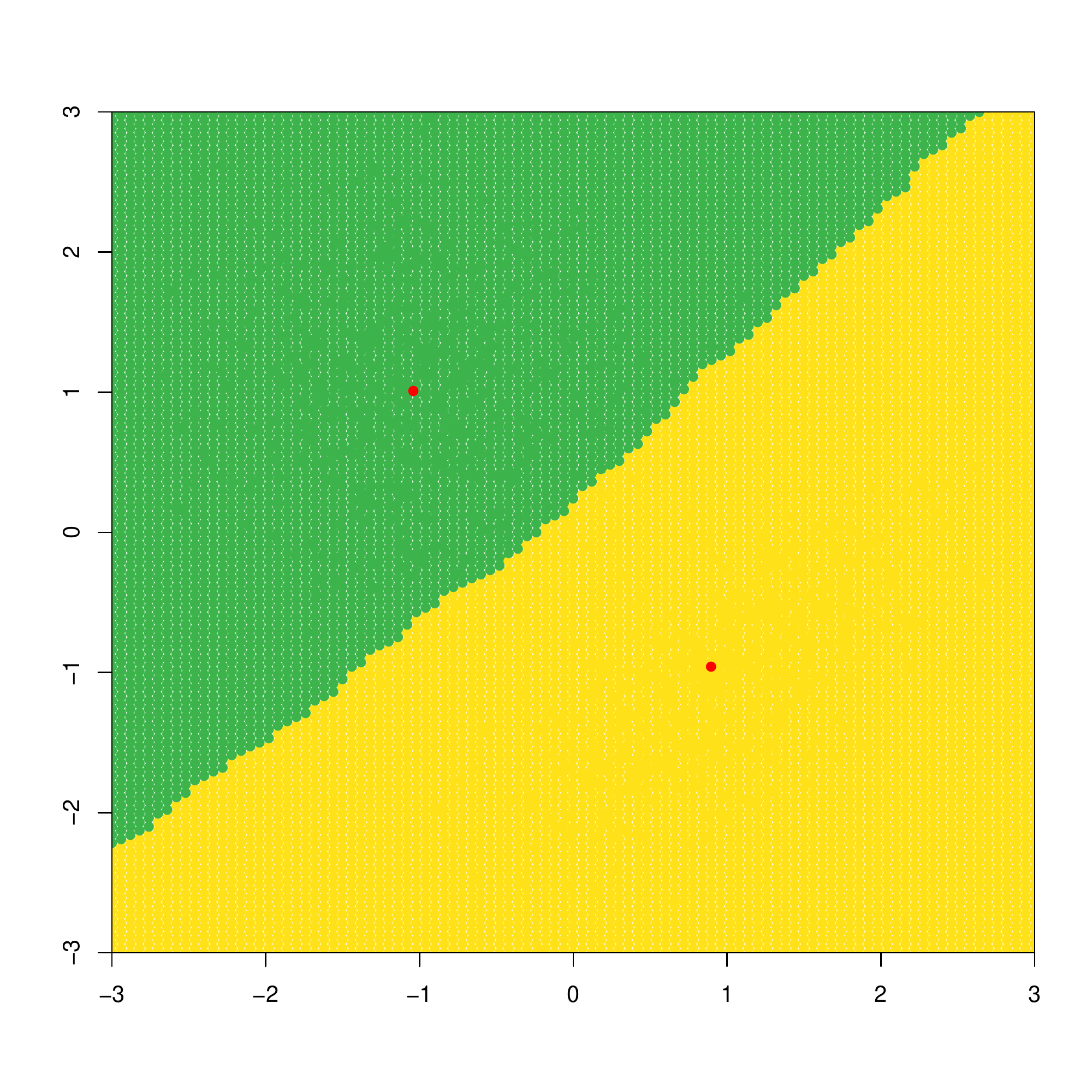}

\caption{Local depth clustering of $n=1000$ samples from the (H) Bimodal IV density. The predicted local maxima are plotted in red. The parameters are $r=0.05$, $s=10,30,50$ in each column (from left to right) and $q=0.05,0.10,0.25,0.50$ in each row (from the top down).}
\label{sm:plot_clusters_bimodal4_prob_all_points}
\end{figure}

\begin{figure}
\centering
\includegraphics[width=0.32\linewidth]{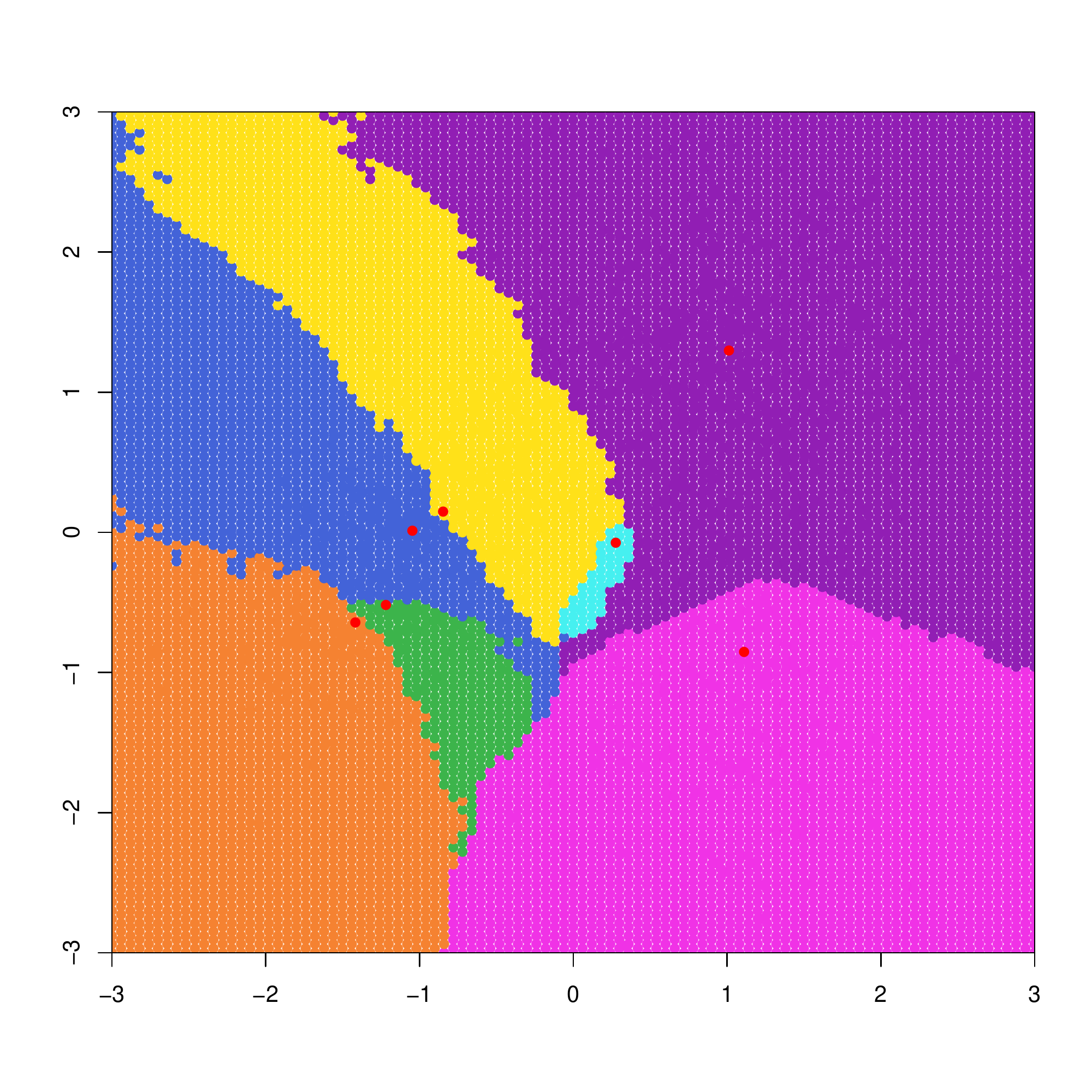}
\includegraphics[width=0.32\linewidth]{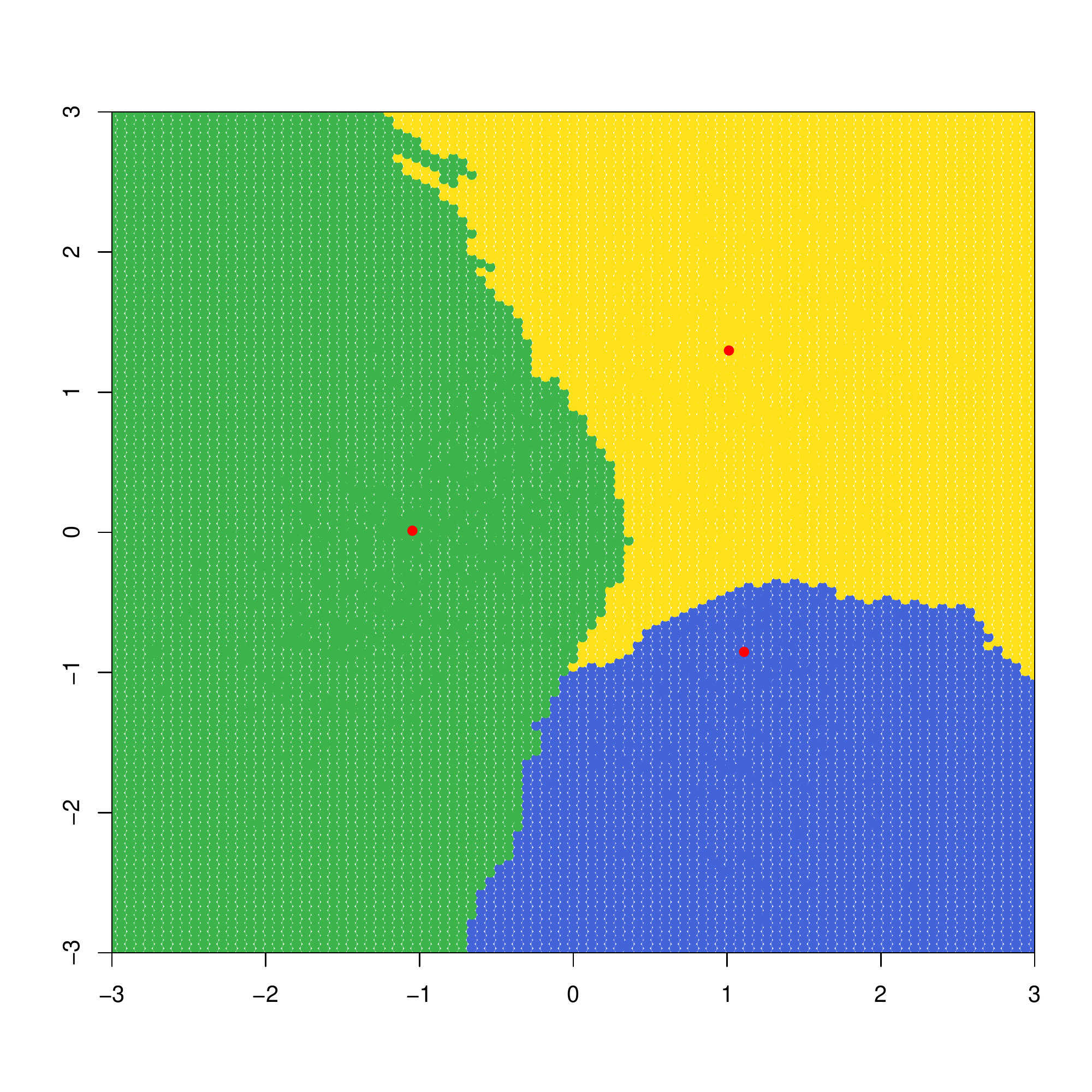}
\includegraphics[width=0.32\linewidth]{{lens-trimodal_3-prob-0.44-0.05-50-0.05-1000-ldc-all}.pdf}
\includegraphics[width=0.32\linewidth]{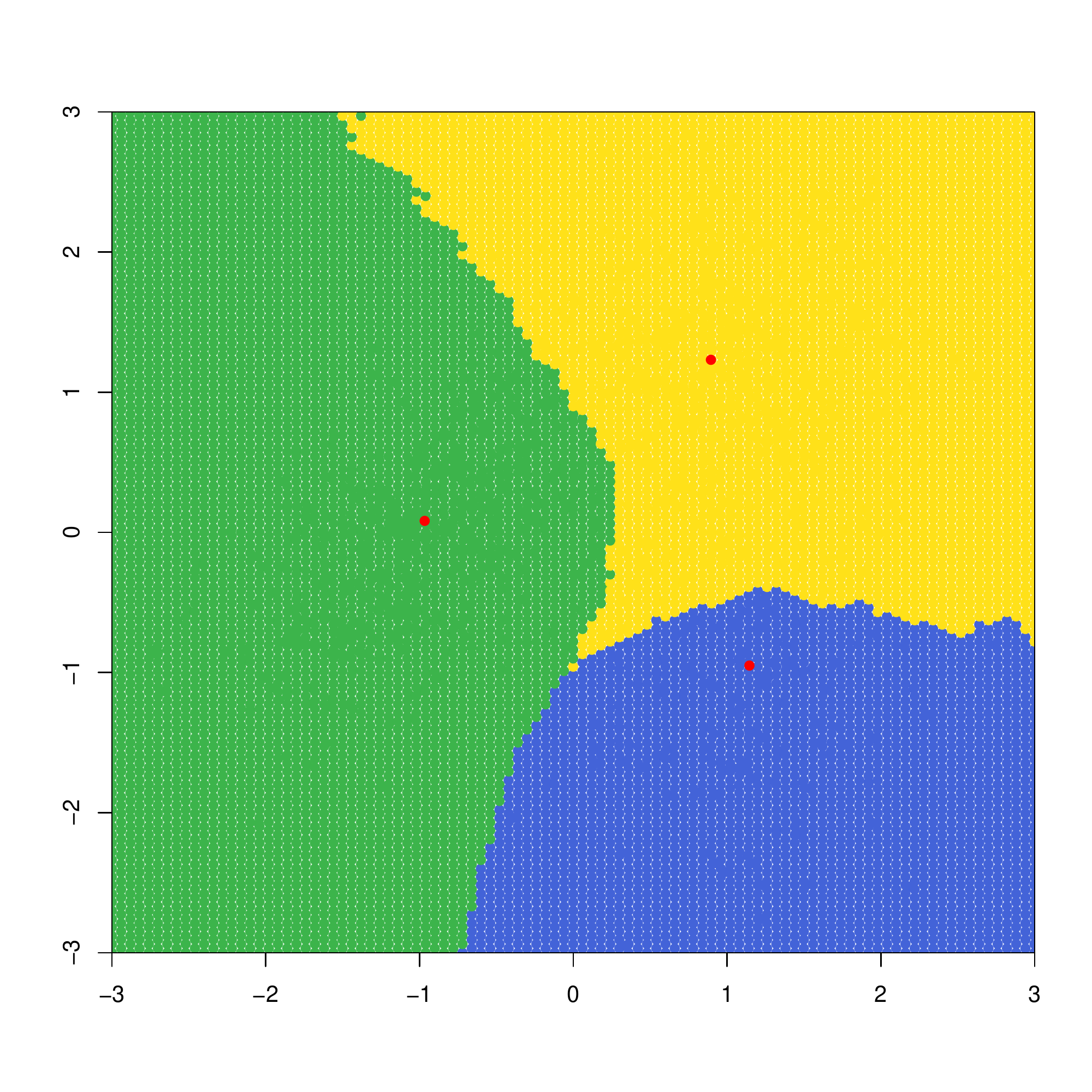}
\includegraphics[width=0.32\linewidth]{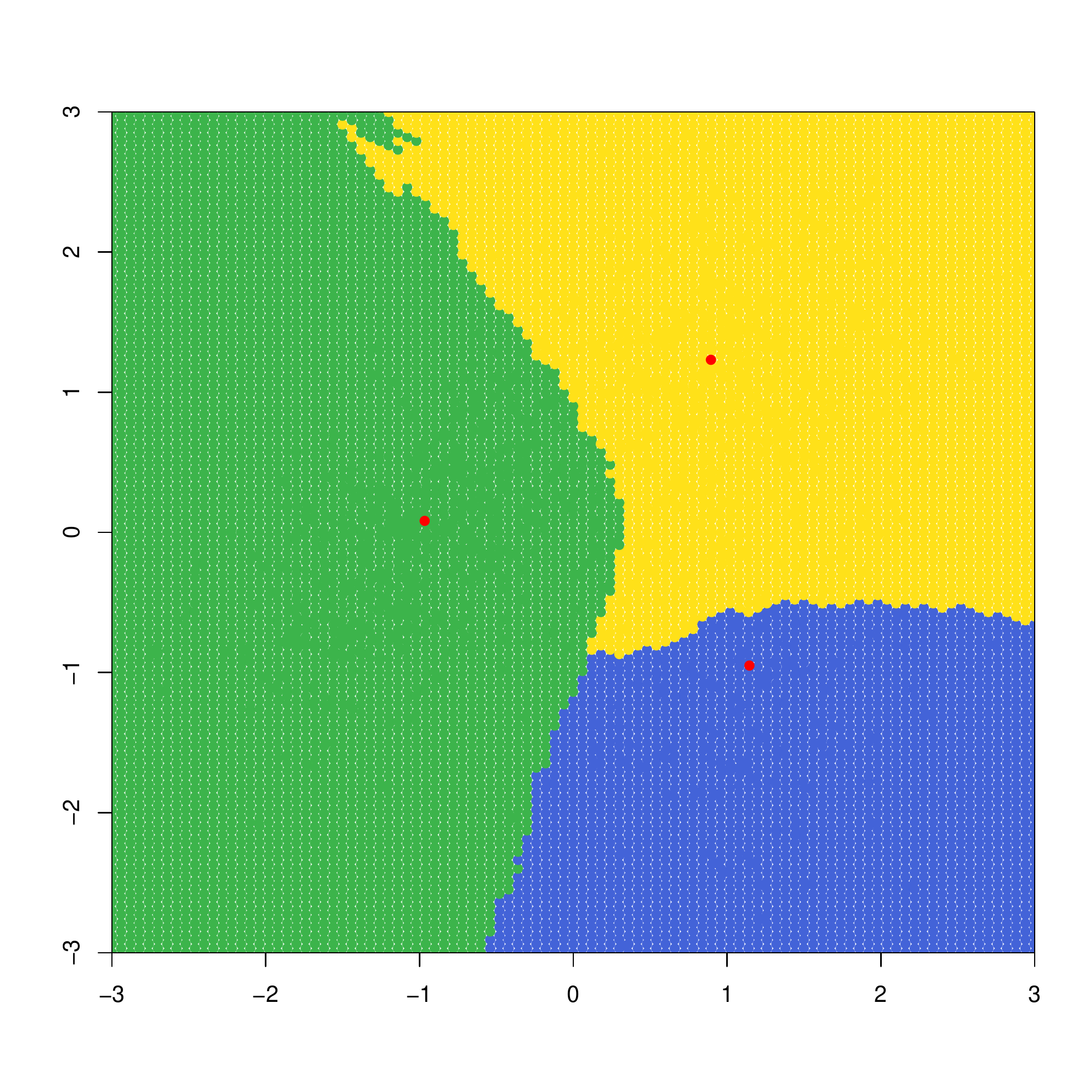}
\includegraphics[width=0.32\linewidth]{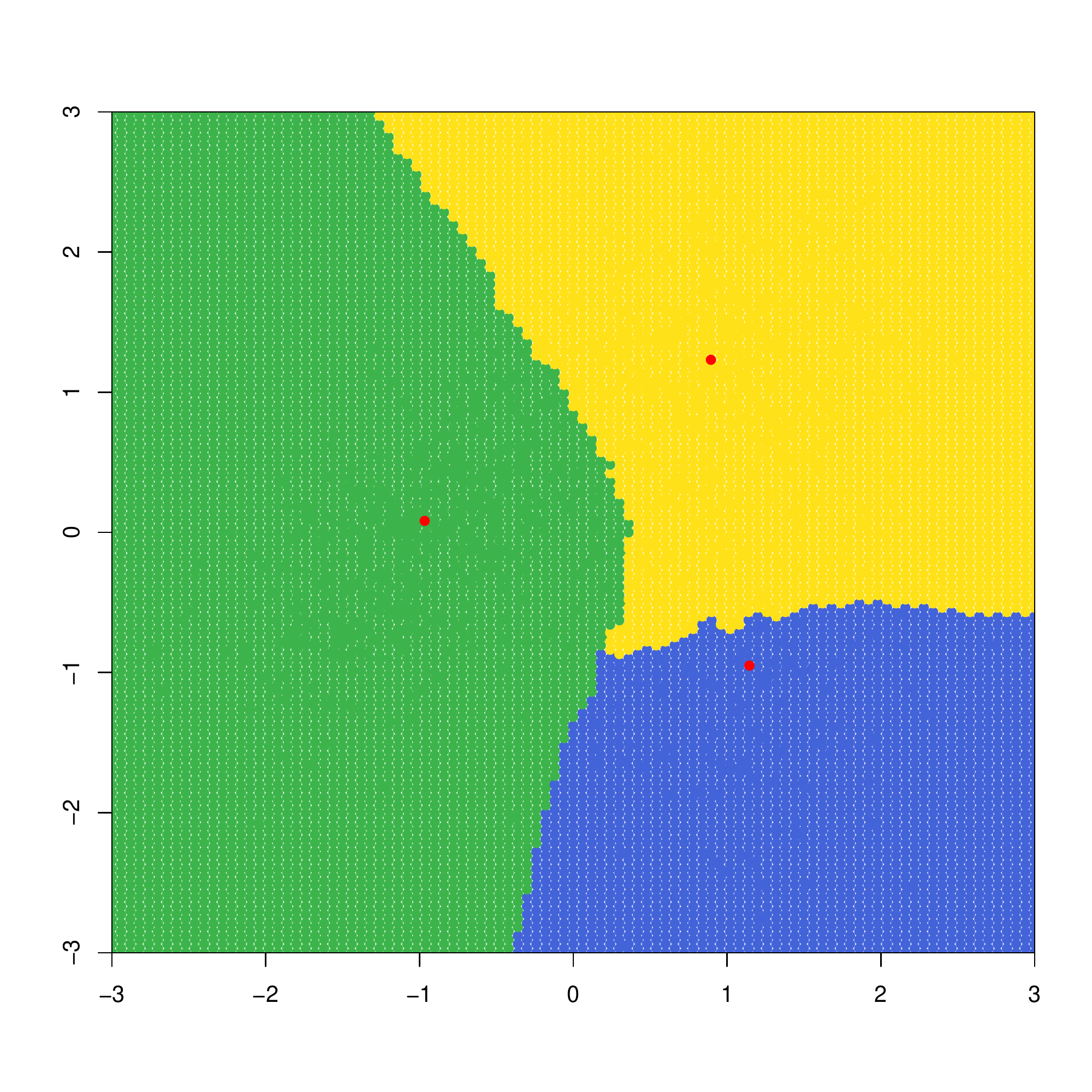}
\includegraphics[width=0.32\linewidth]{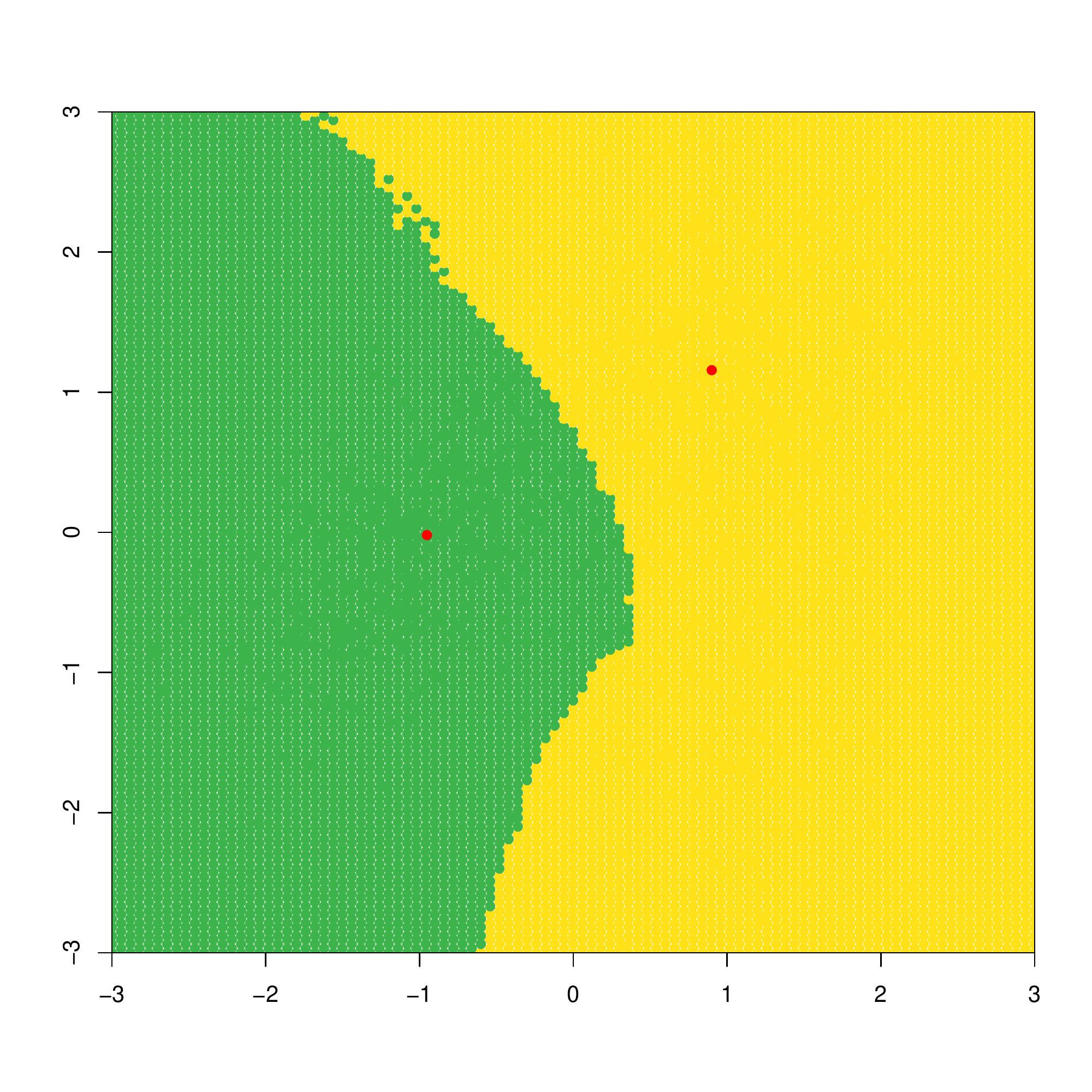}
\includegraphics[width=0.32\linewidth]{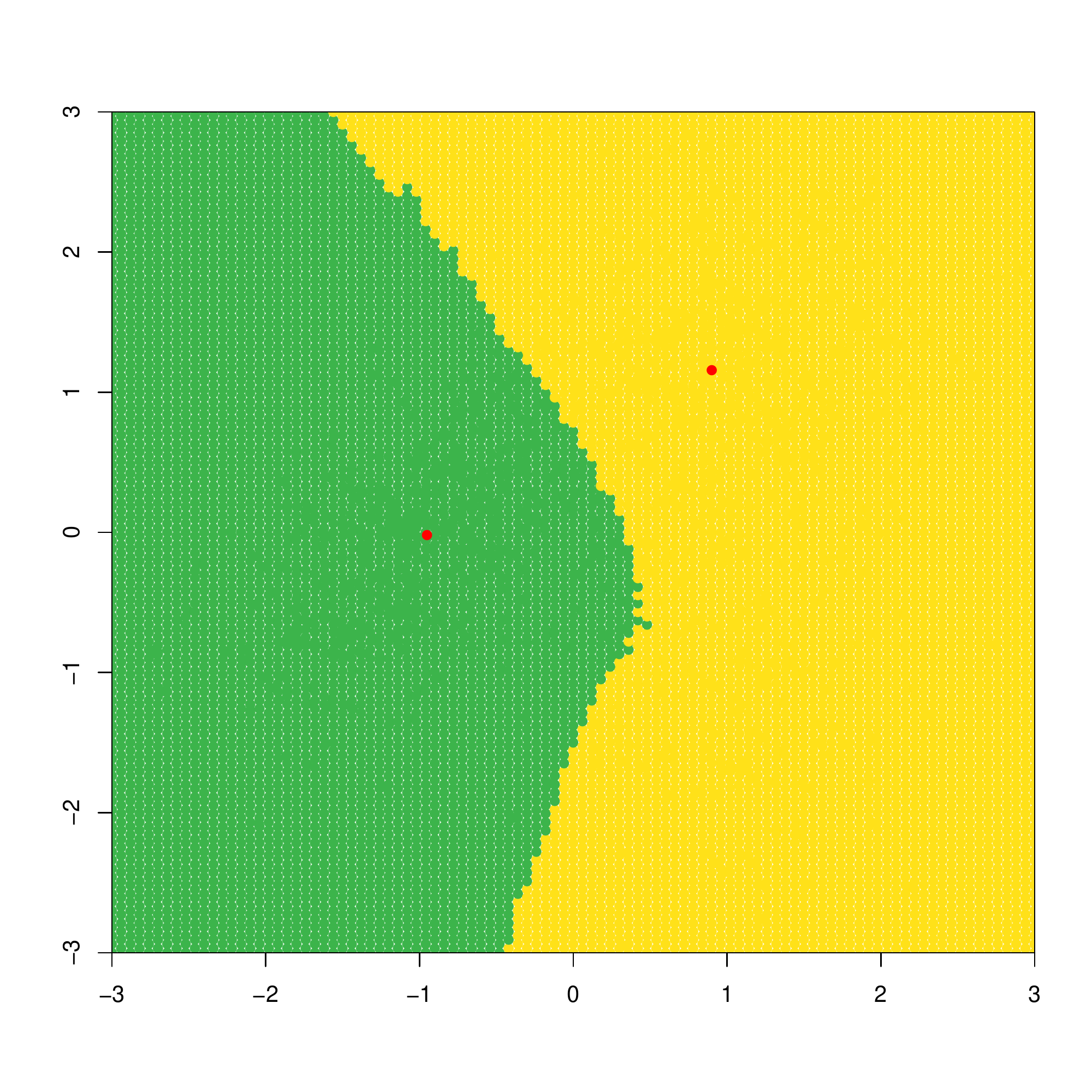}
\includegraphics[width=0.32\linewidth]{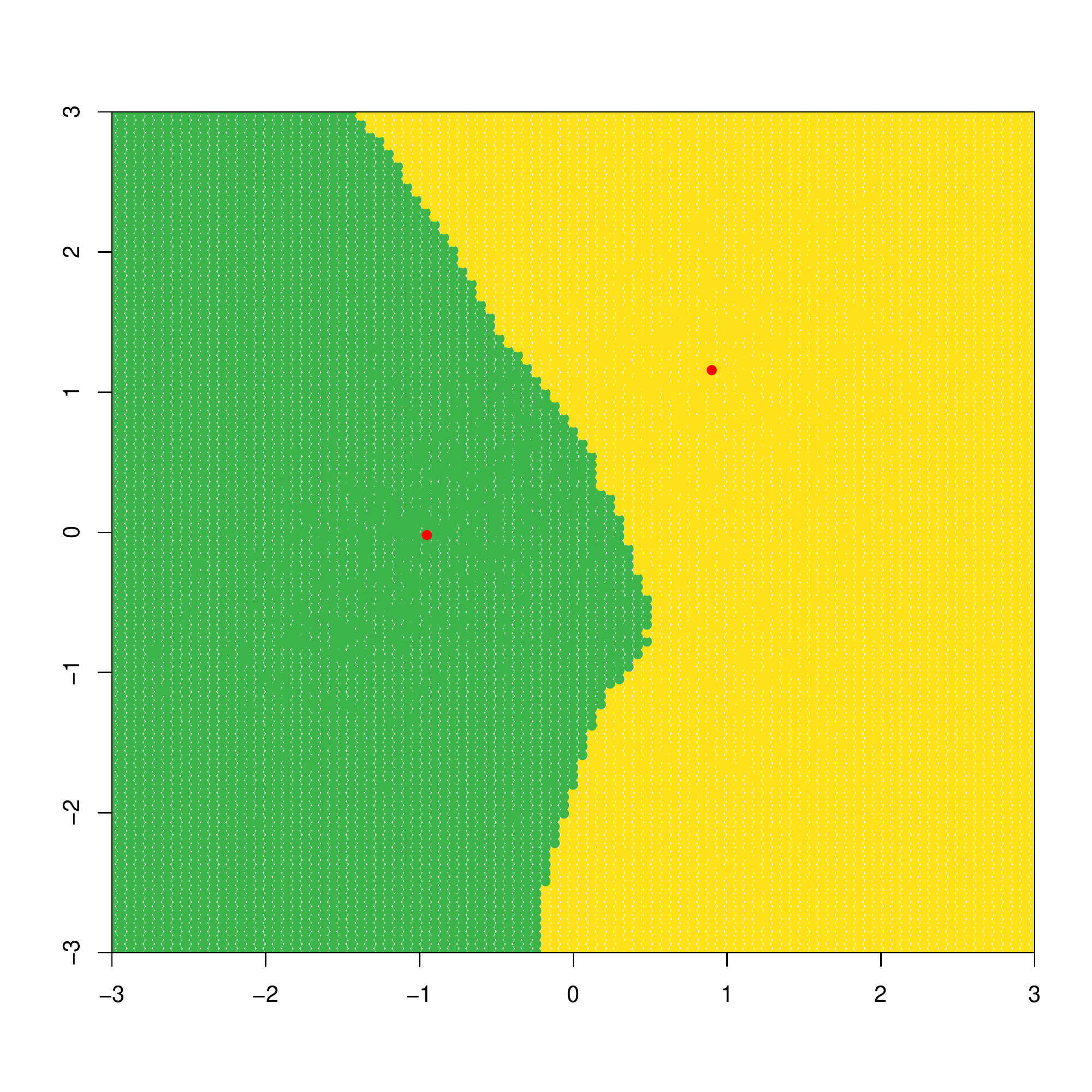}
\includegraphics[width=0.32\linewidth]{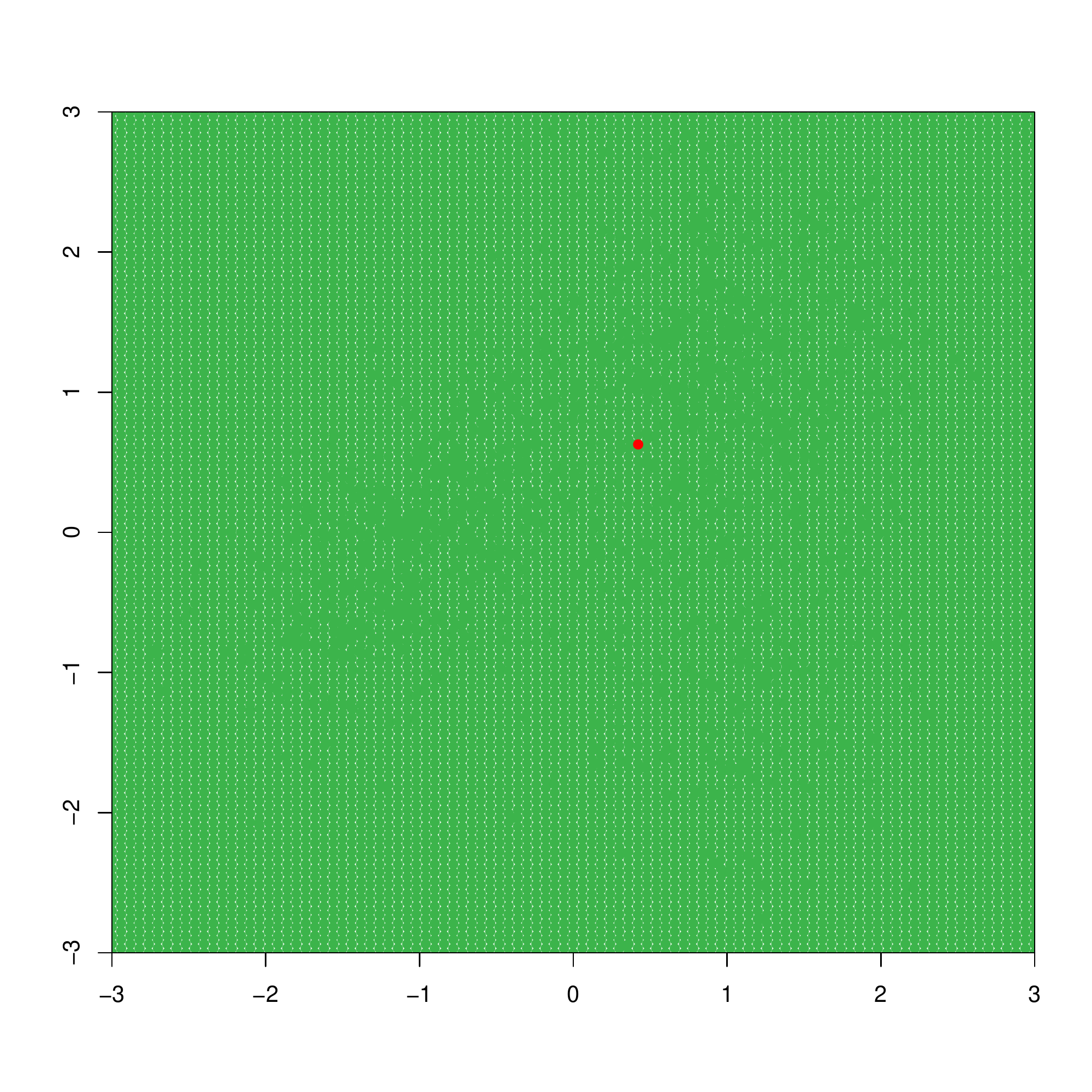}
\includegraphics[width=0.32\linewidth]{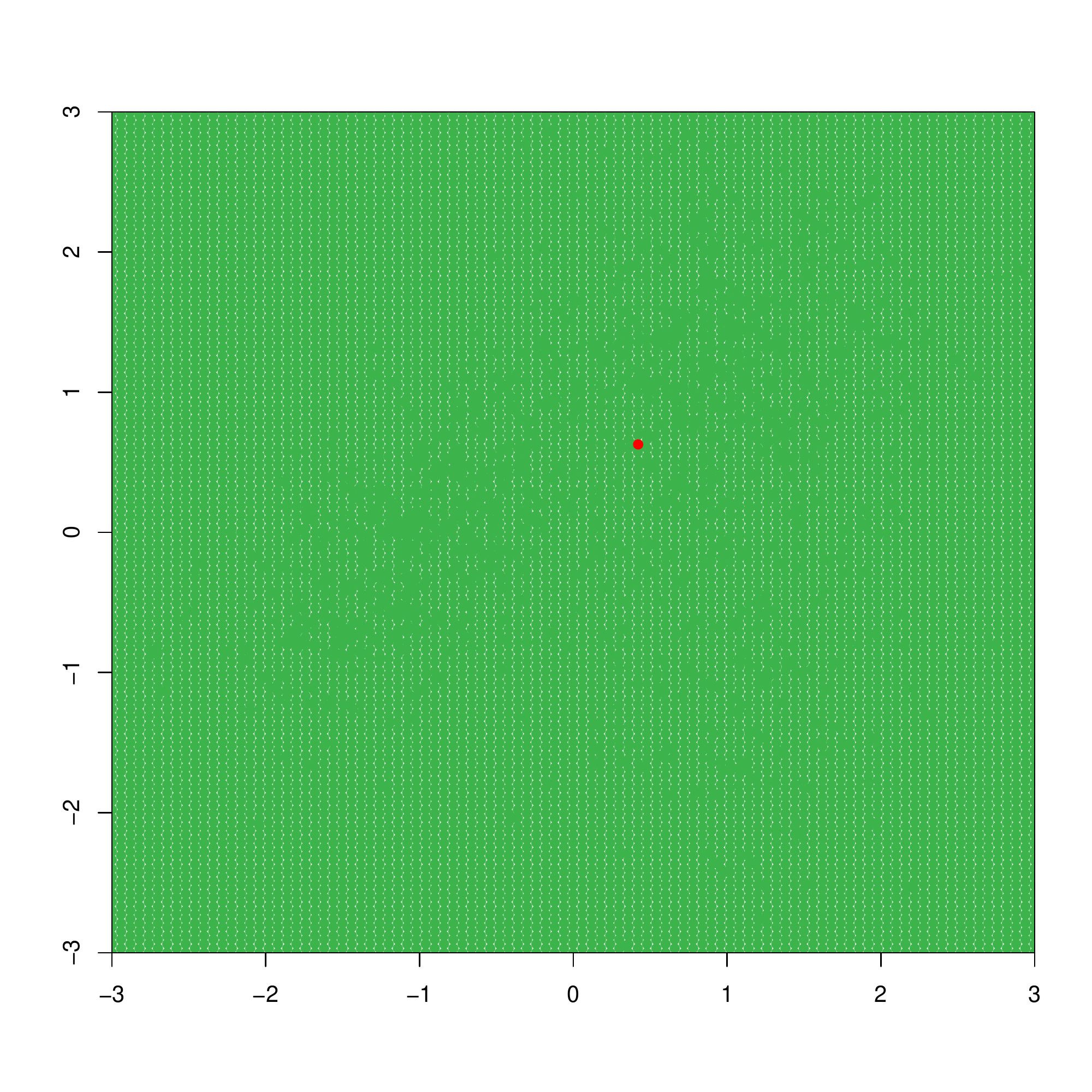}
\includegraphics[width=0.32\linewidth]{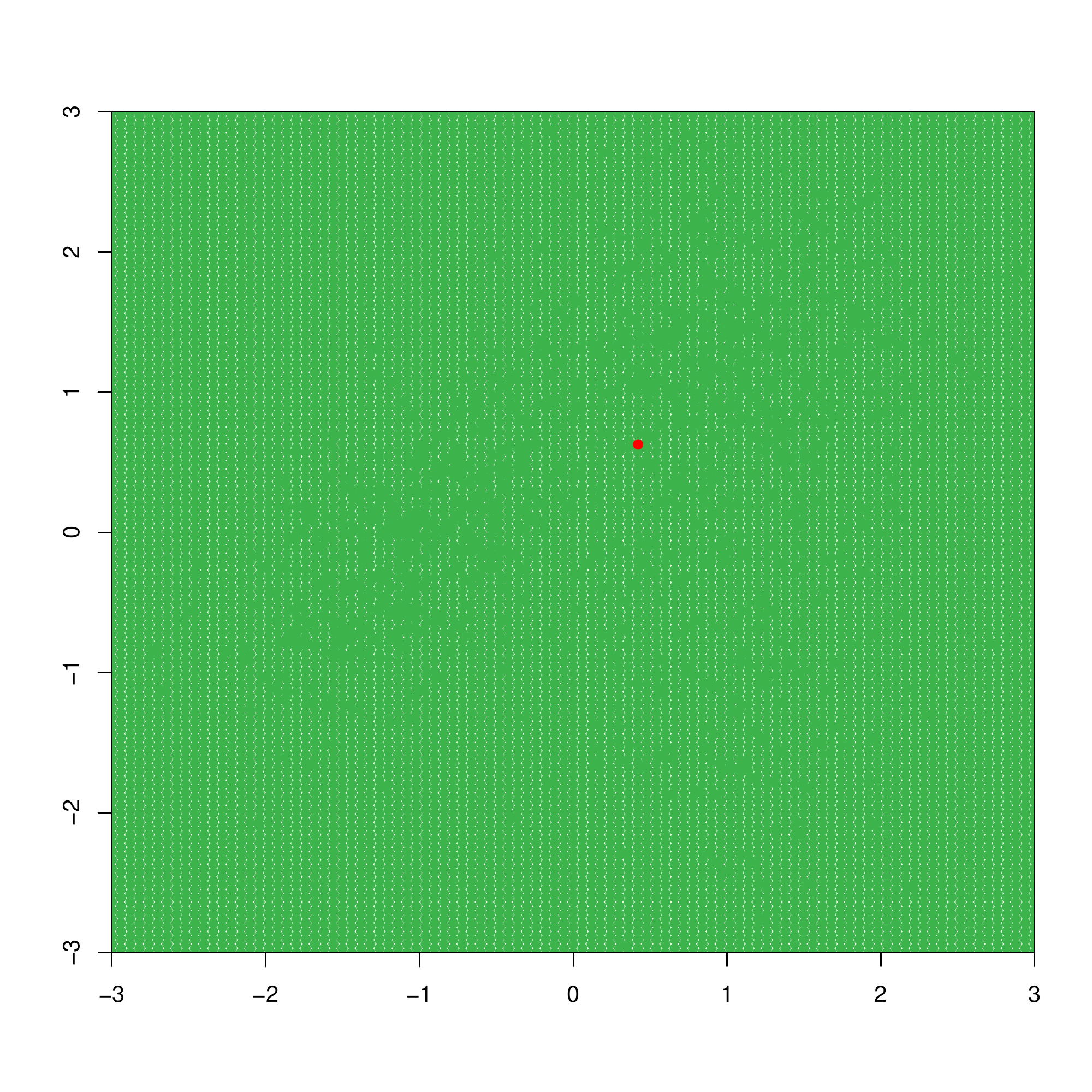}

\caption{Local depth clustering of $n=1000$ samples from the (K) Trimodal III density. The predicted local maxima are plotted in red. The parameters are $r=0.05$, $s=10,30,50$ in each column (from left to right) and $q=0.05,0.10,0.25,0.50$ in each row (from the top down).}
\label{sm:plot_clusters_trimodal3_prob_all_points}
\end{figure}

\begin{figure}
\centering
\includegraphics[width=0.32\linewidth]{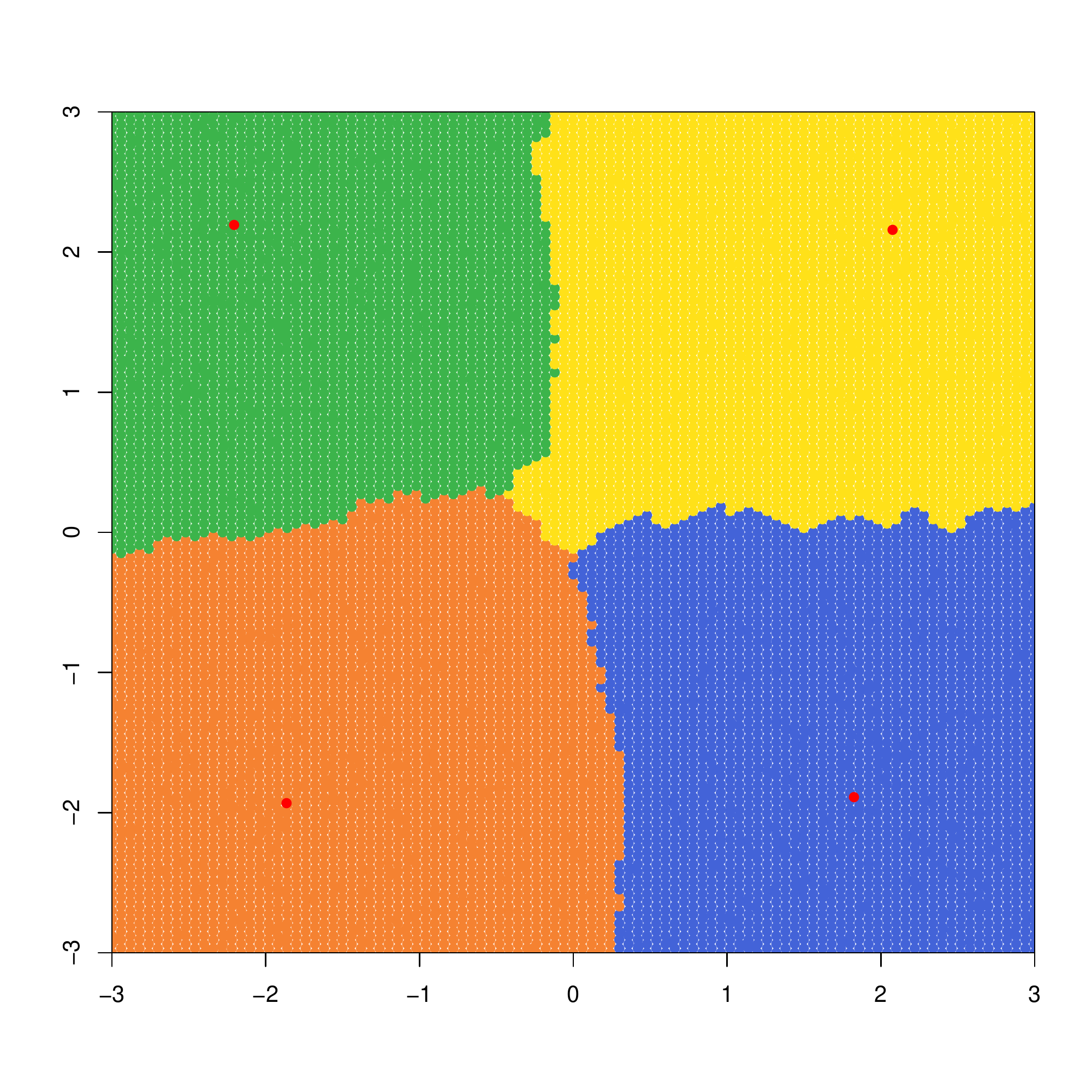}
\includegraphics[width=0.32\linewidth]{{lens-quadrimodal_location-prob-0.92-0.05-30-0.05-1000-ldc-all}.pdf}
\includegraphics[width=0.32\linewidth]{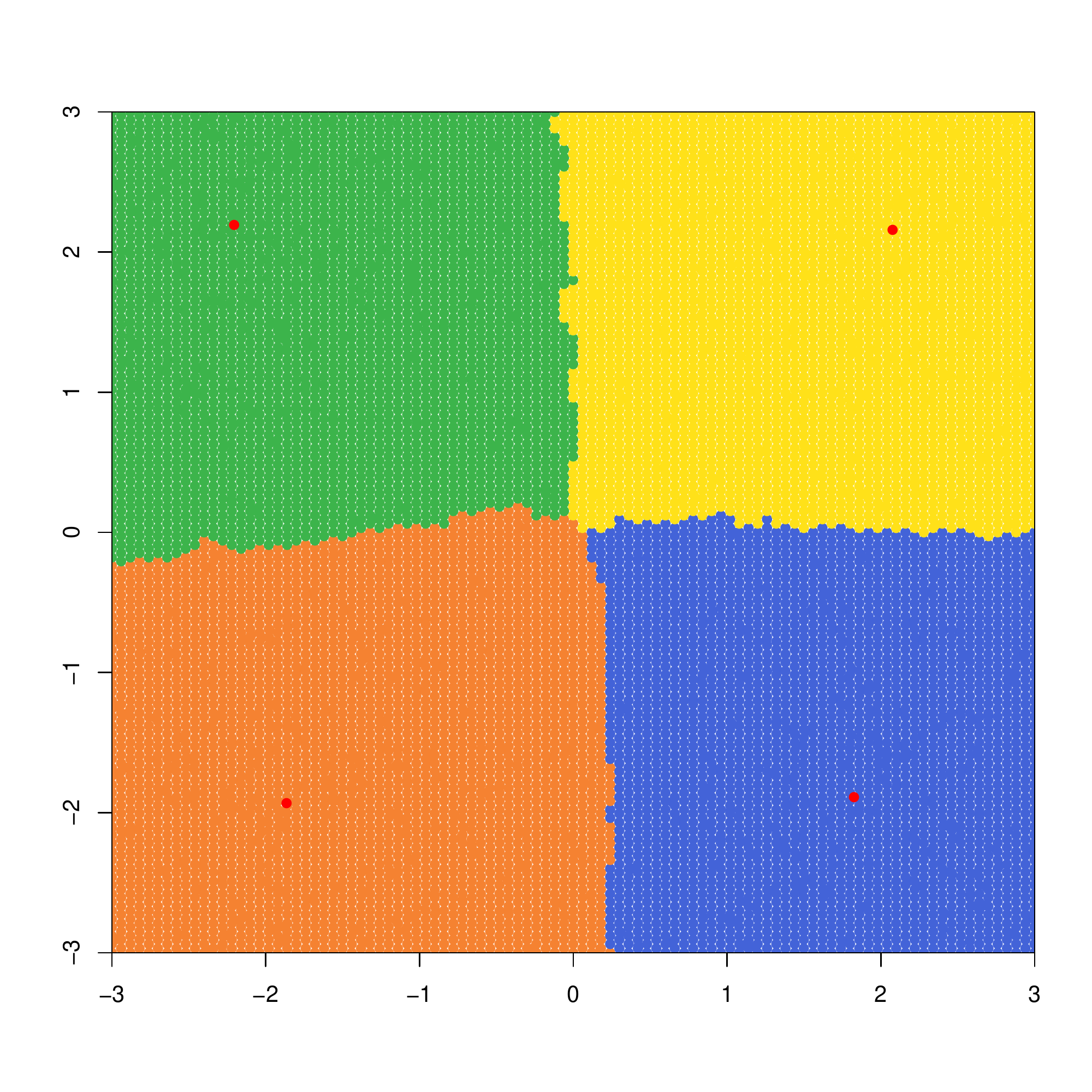}
\includegraphics[width=0.32\linewidth]{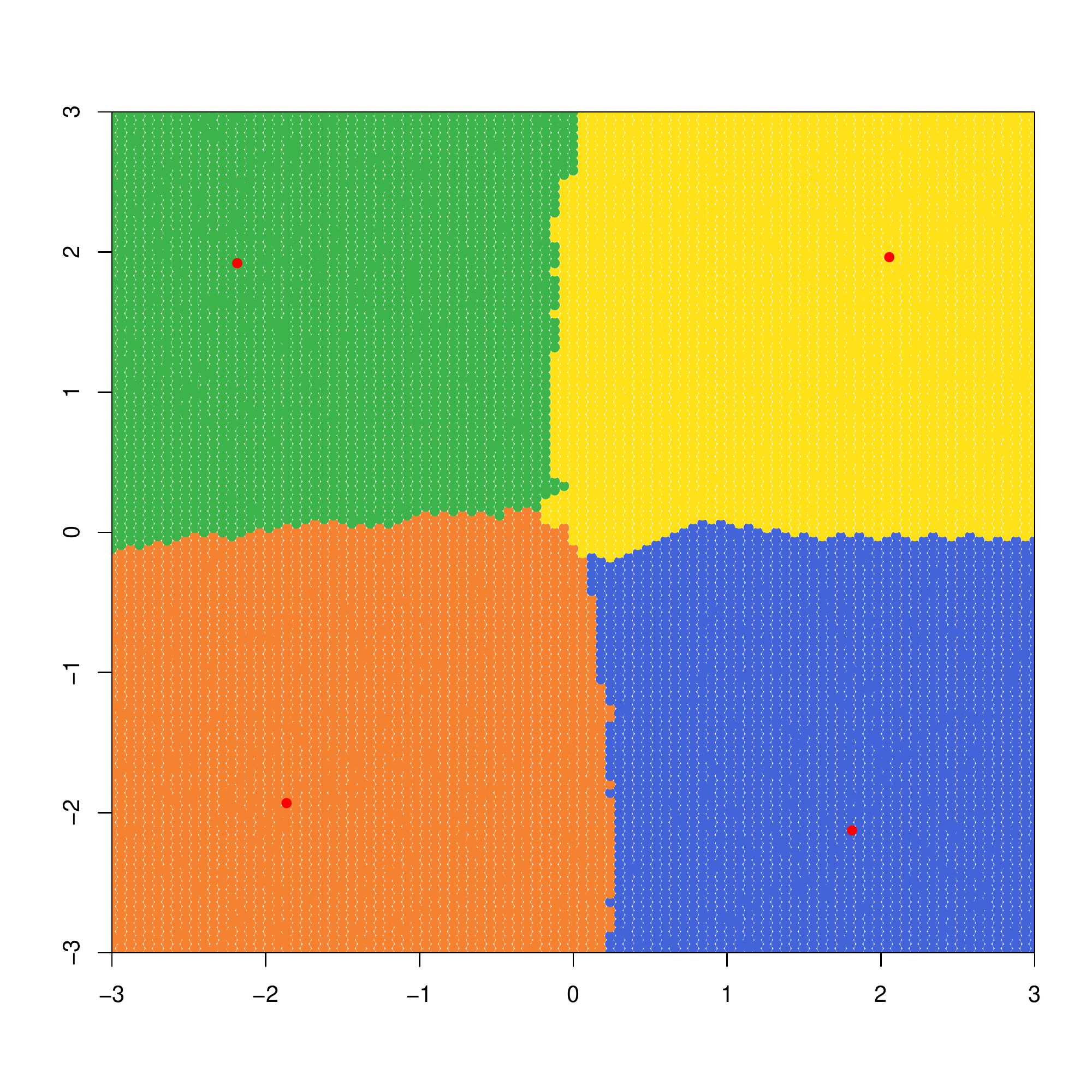}
\includegraphics[width=0.32\linewidth]{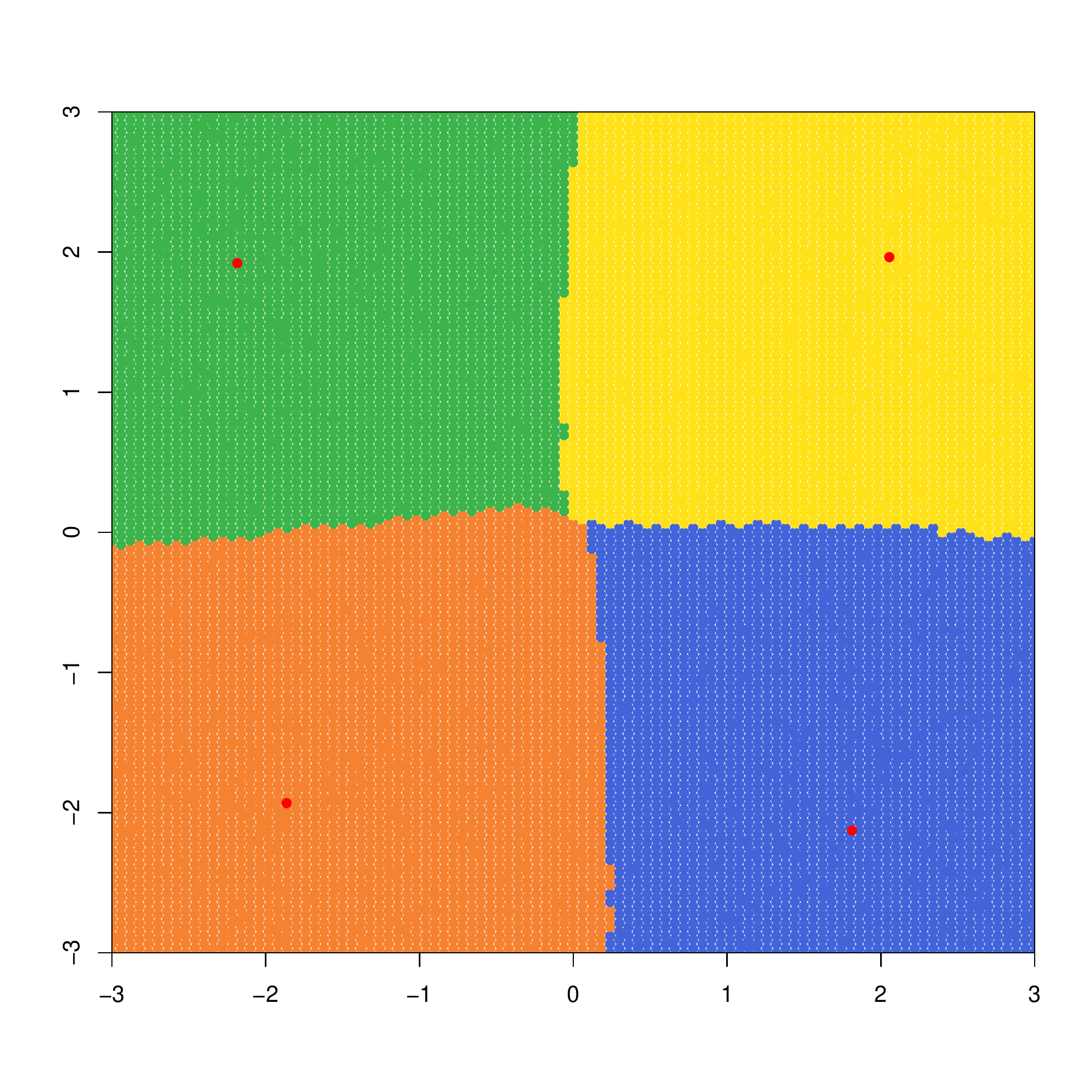}
\includegraphics[width=0.32\linewidth]{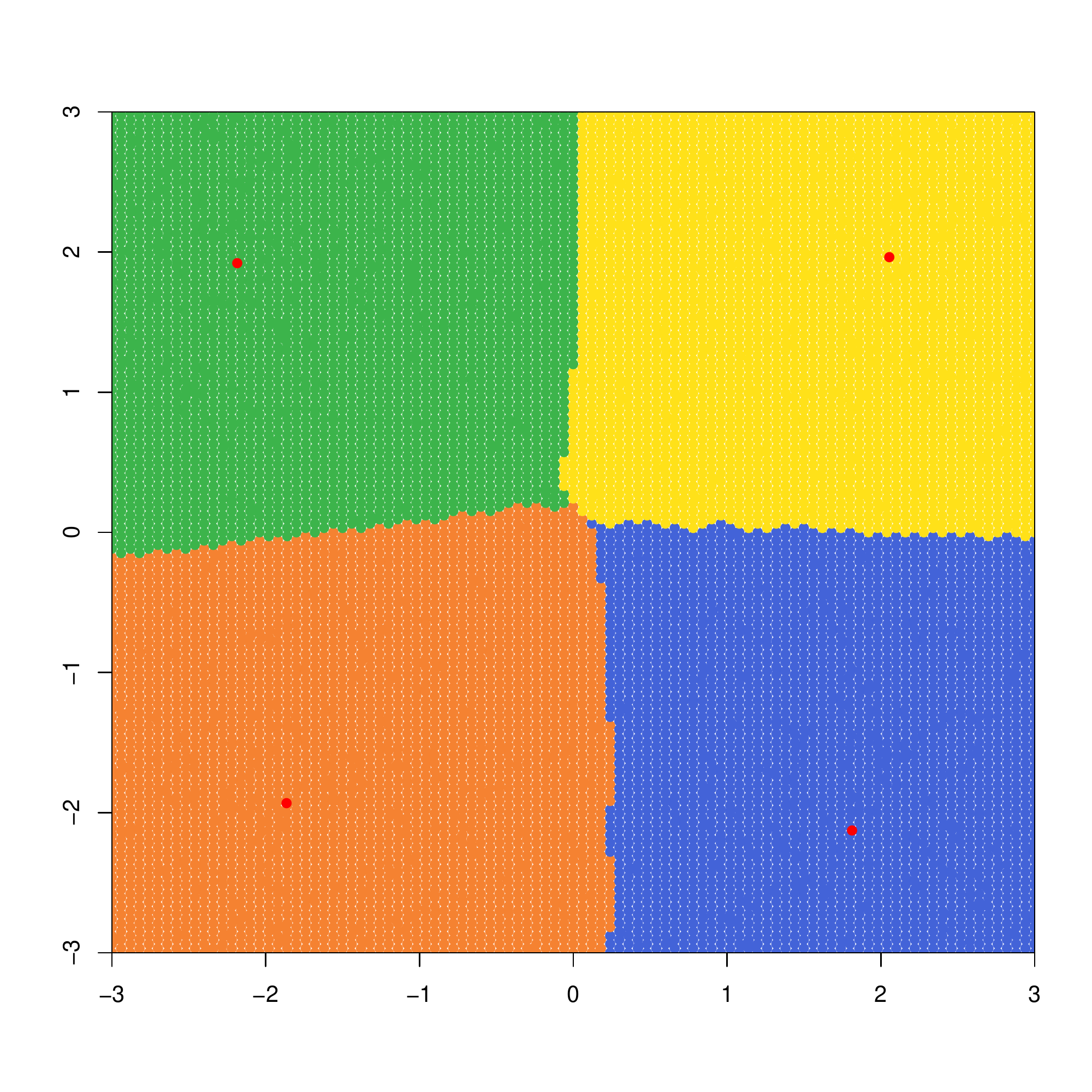}
\includegraphics[width=0.32\linewidth]{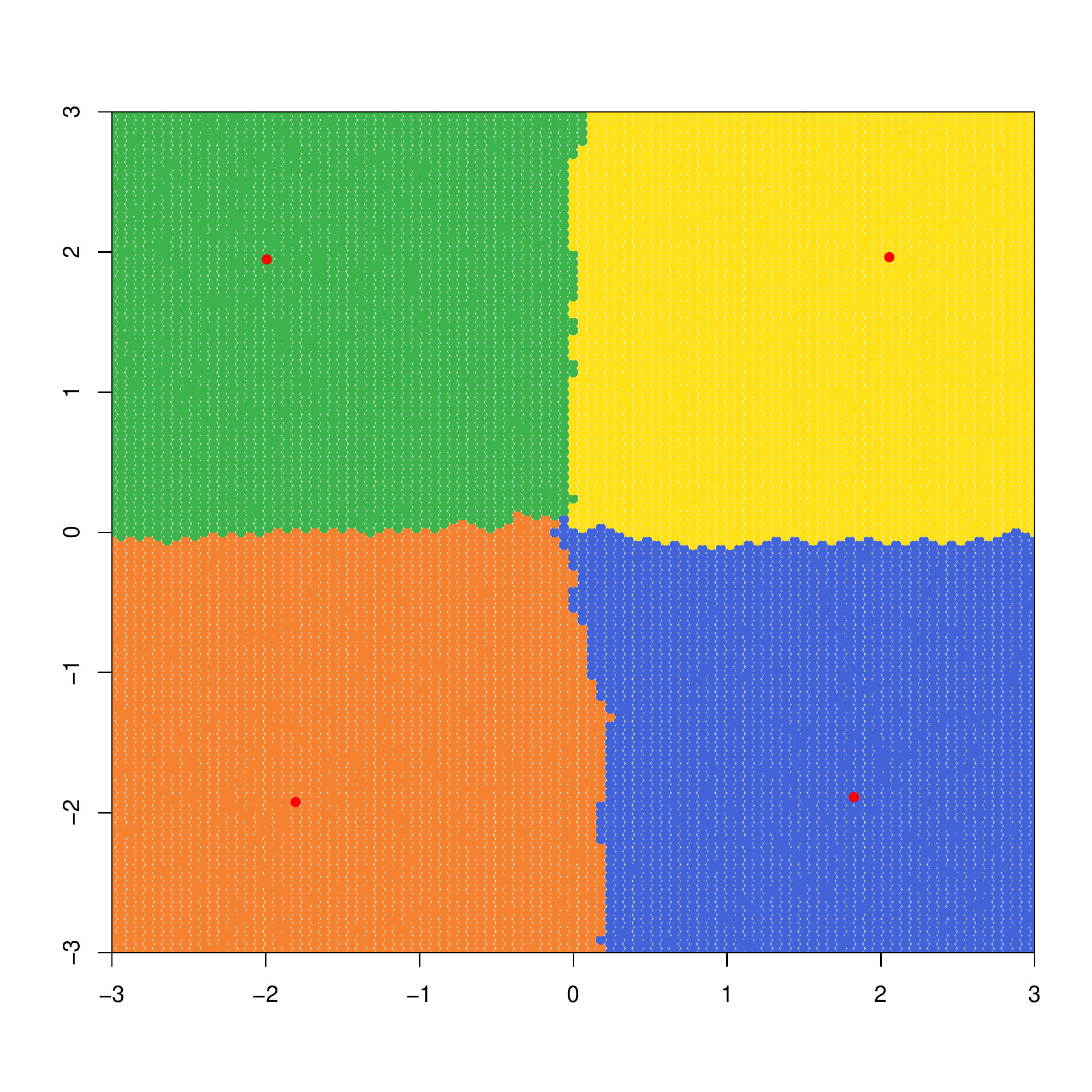}
\includegraphics[width=0.32\linewidth]{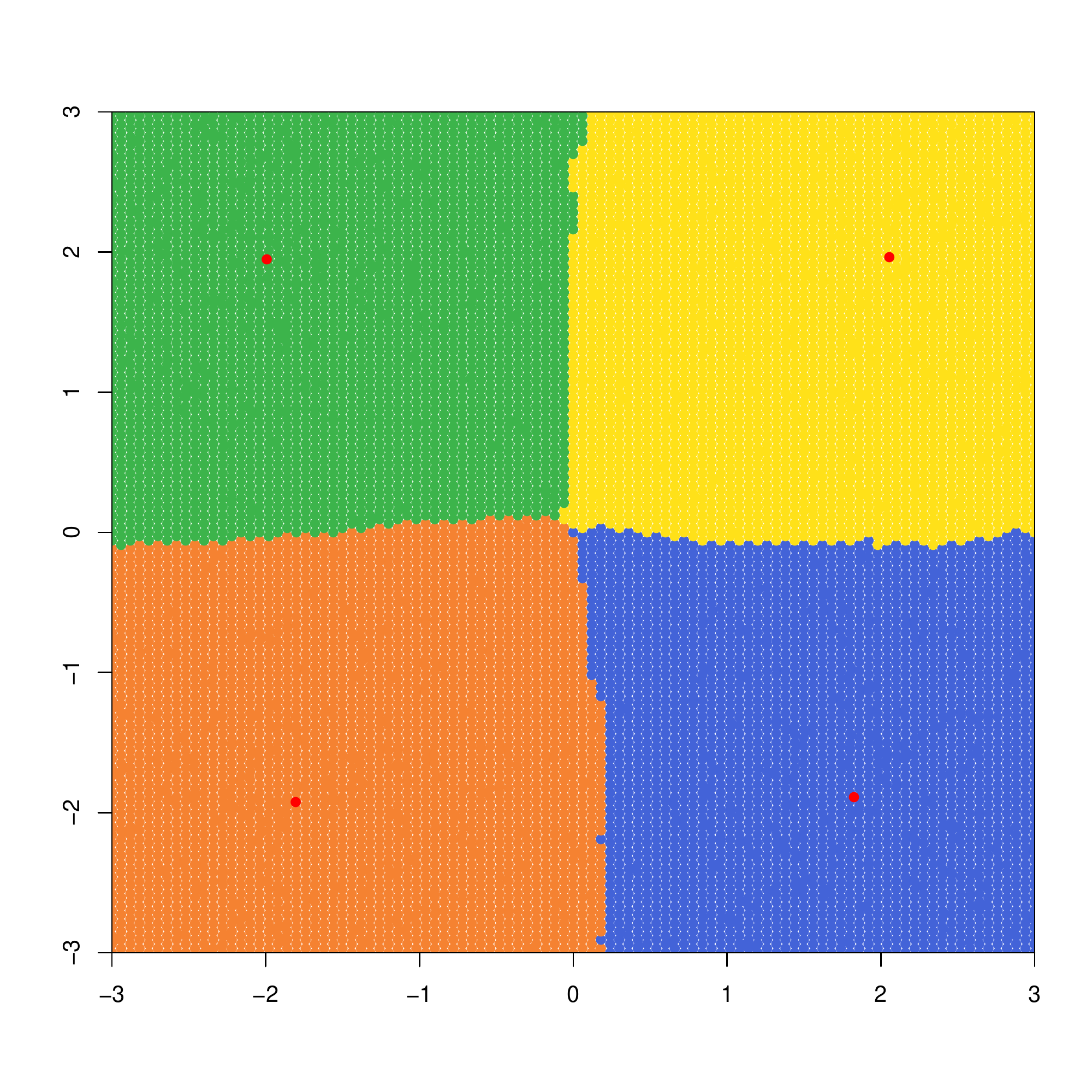}
\includegraphics[width=0.32\linewidth]{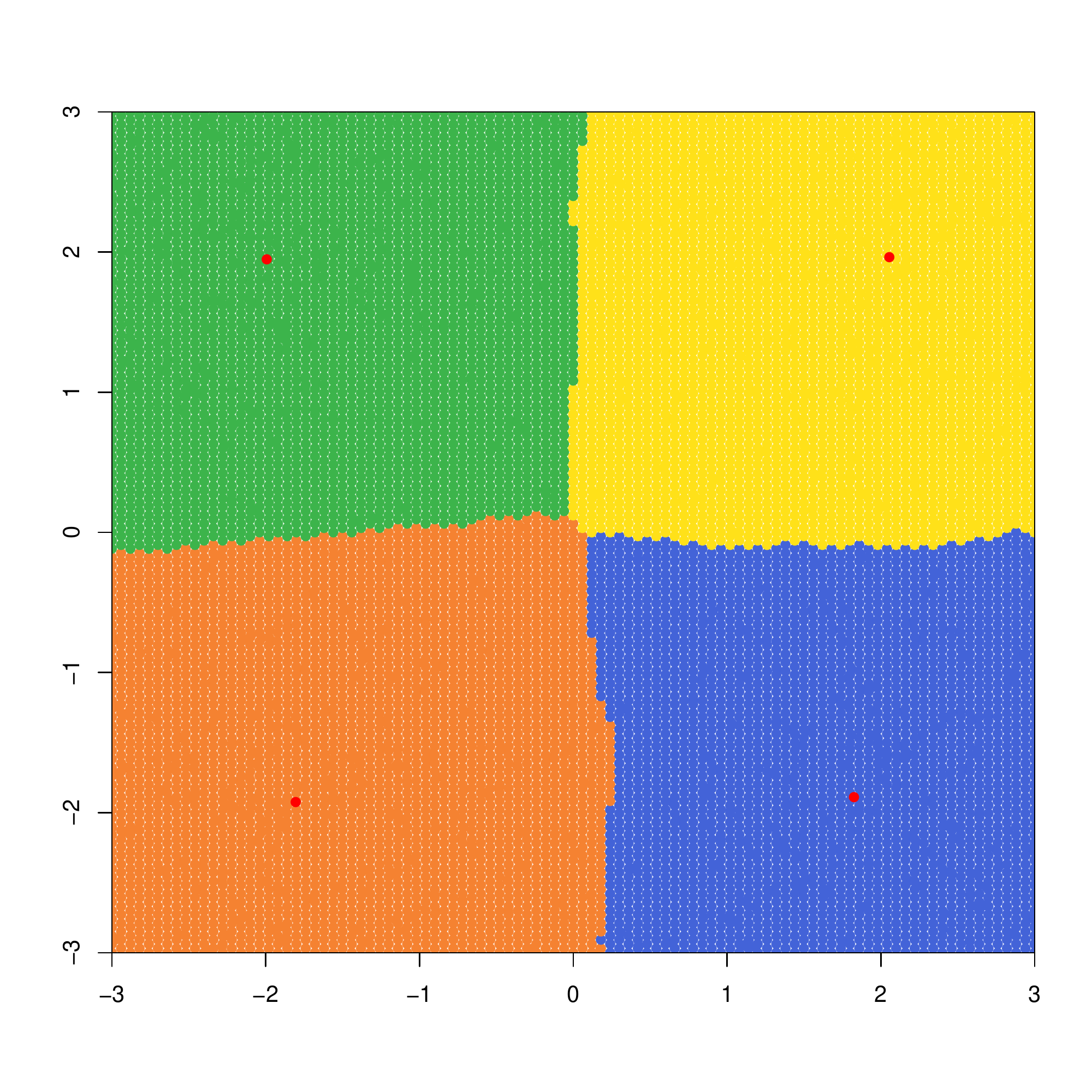}
\includegraphics[width=0.32\linewidth]{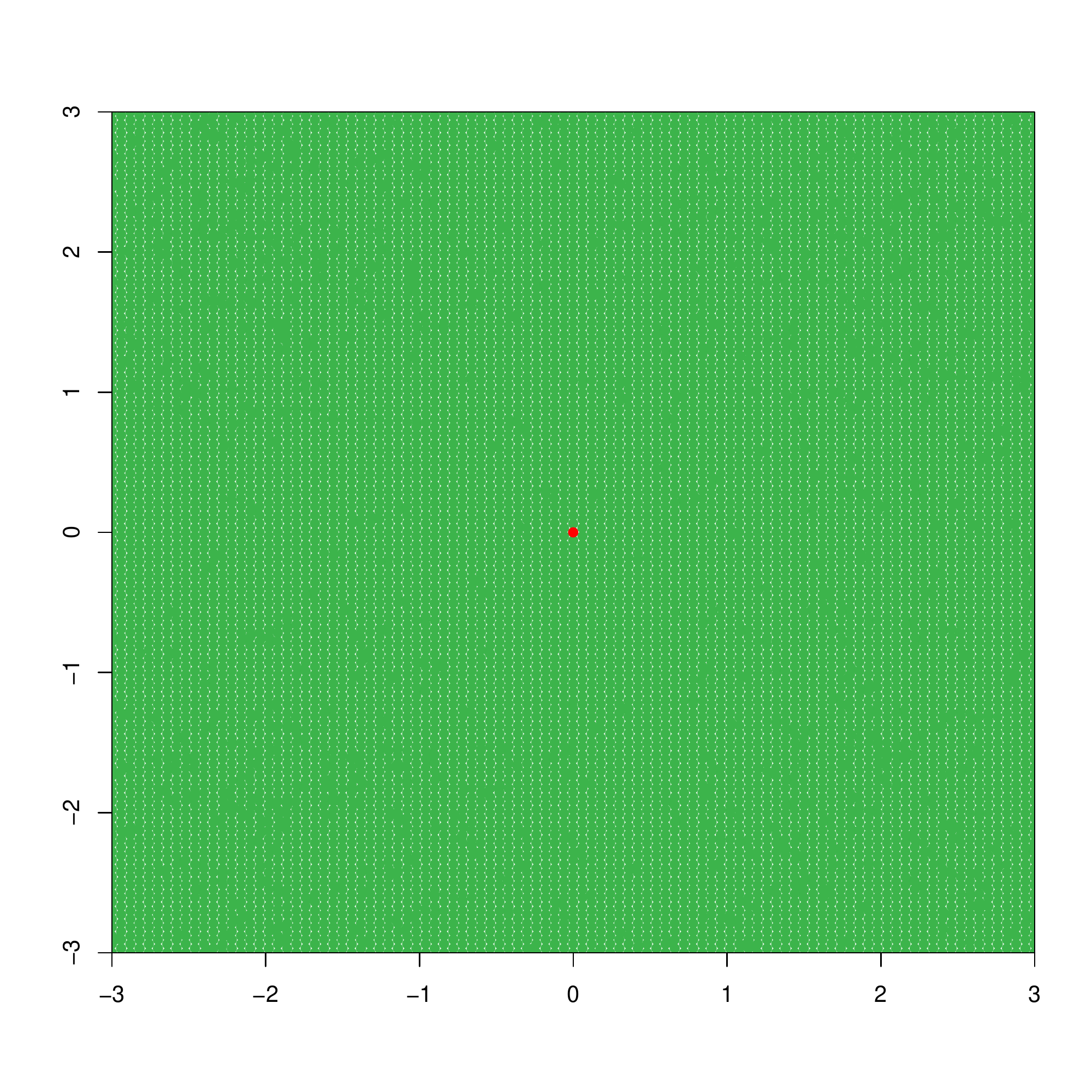}
\includegraphics[width=0.32\linewidth]{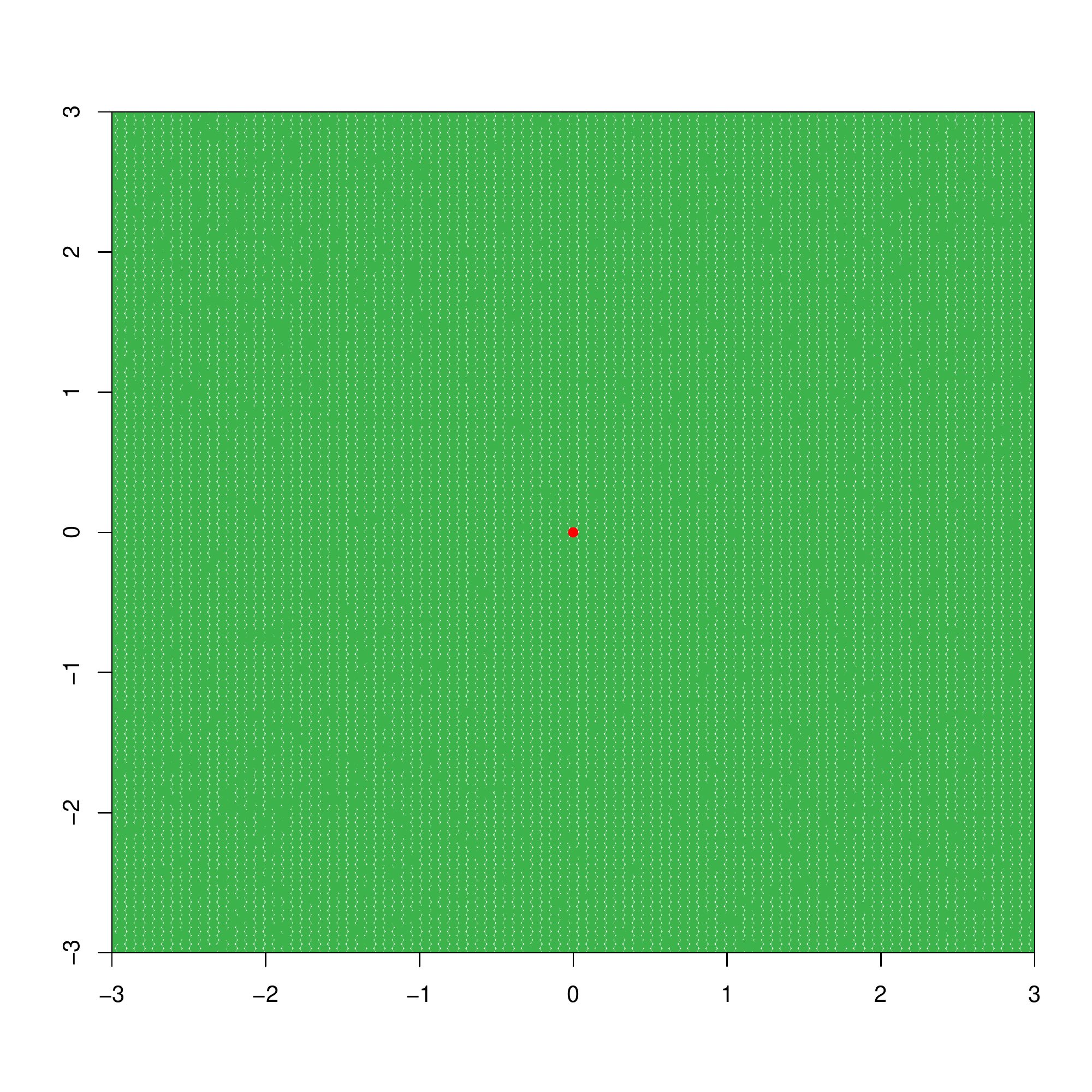}
\includegraphics[width=0.32\linewidth]{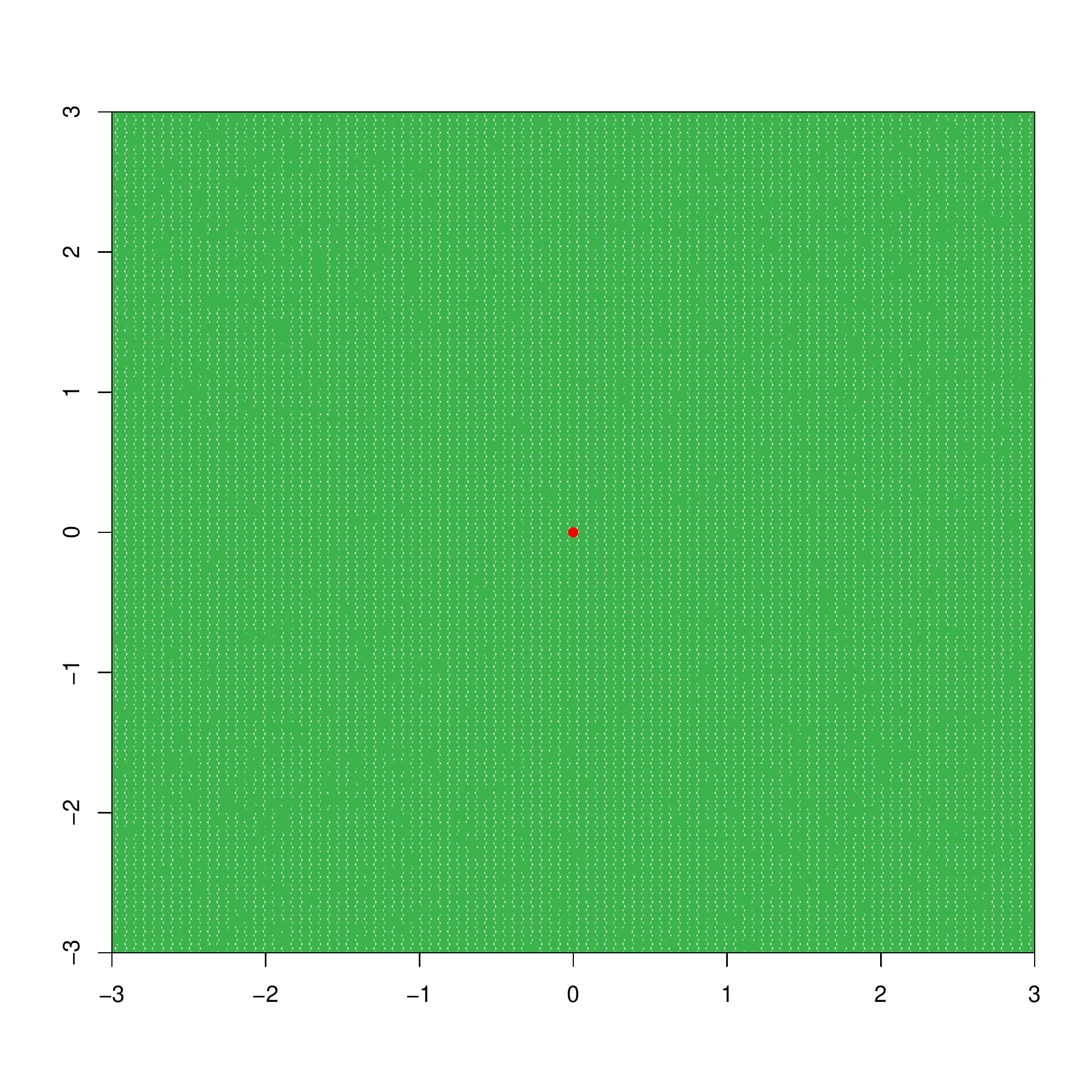}

\caption{Local depth clustering of $n=1000$ samples from the Quadrimodal density. The predicted local maxima are plotted in red. The parameters are $r=0.05$, $s=10,30,50$ in each column (from left to right) and $q=0.05,0.10,0.25,0.50$ in each row (from the top down).}
\label{sm:plot_clusters_quadrimodal_location_prob_all_points}
\end{figure}

\begin{figure}
\centering
\includegraphics[width=0.32\linewidth]{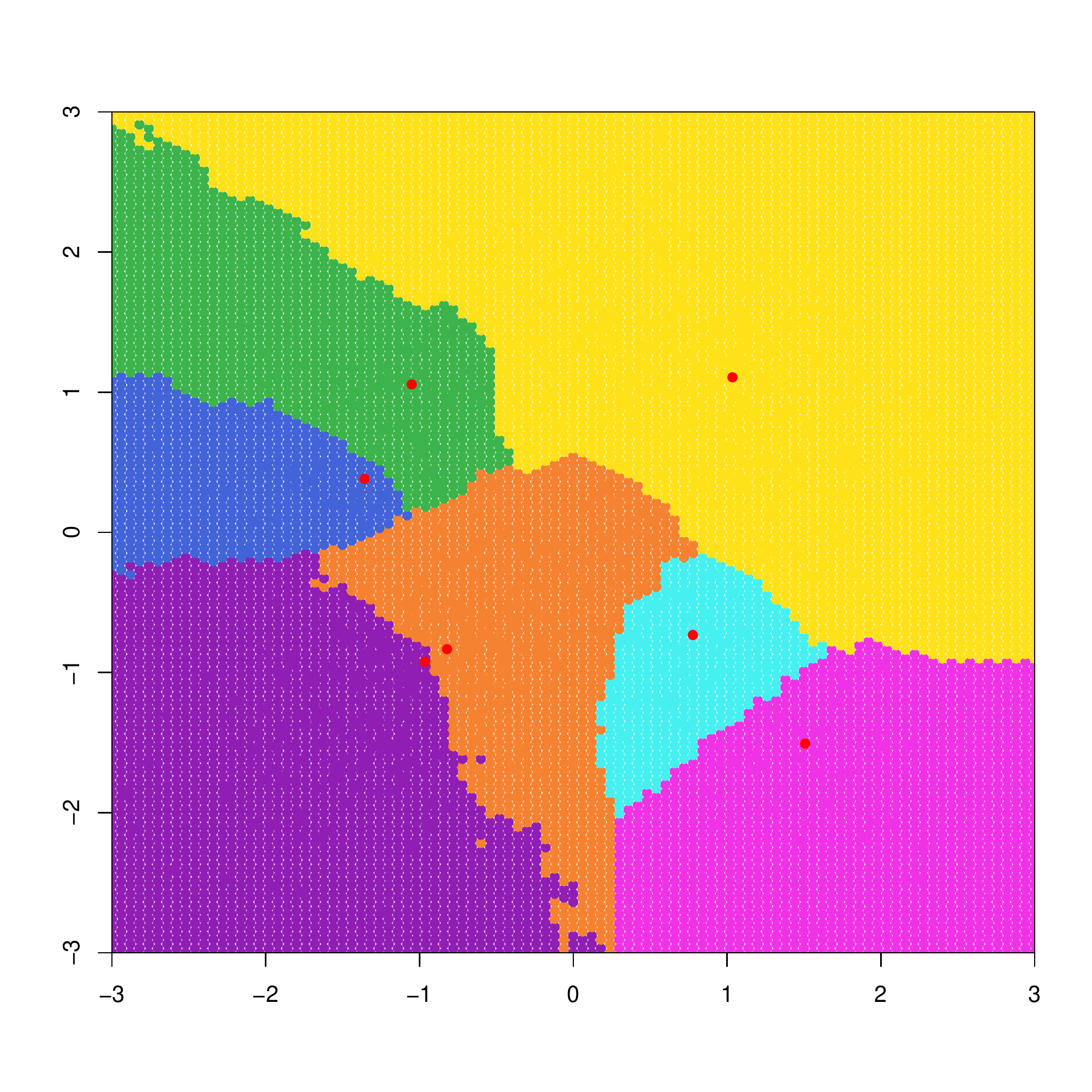}
\includegraphics[width=0.32\linewidth]{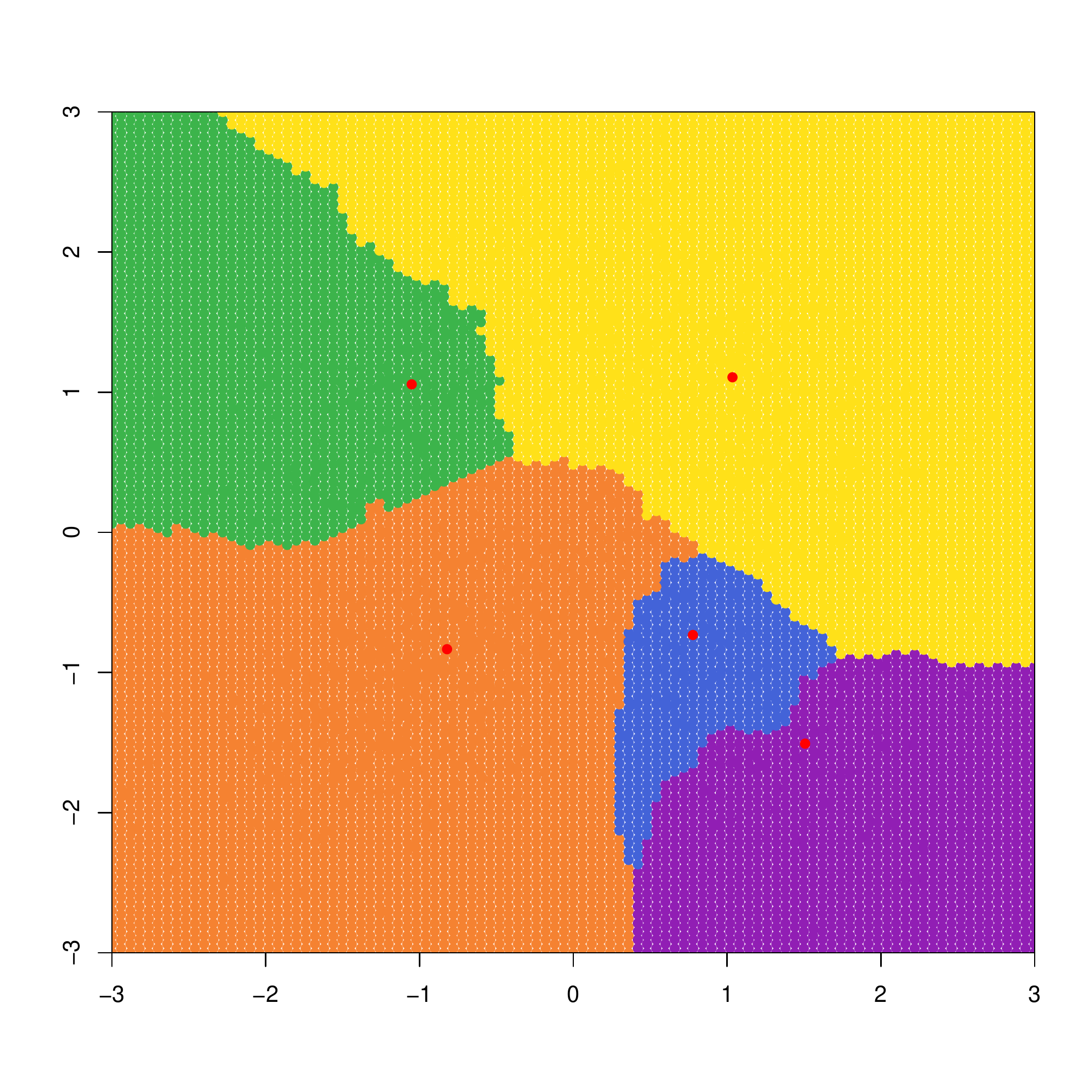}
\includegraphics[width=0.32\linewidth]{{lens-quadrimodal-prob-0.48-0.05-50-0.05-1000-ldc-all}.pdf}
\includegraphics[width=0.32\linewidth]{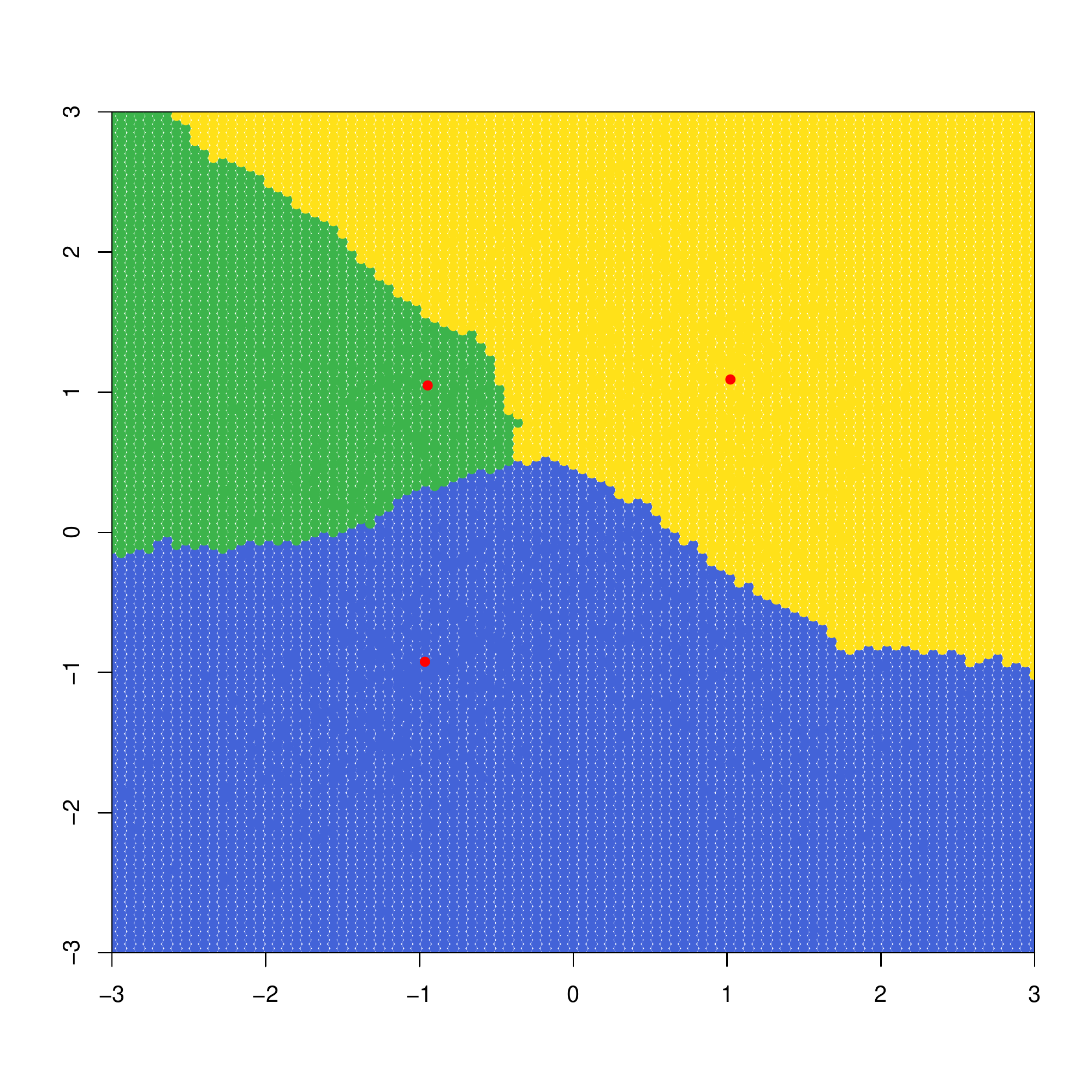}
\includegraphics[width=0.32\linewidth]{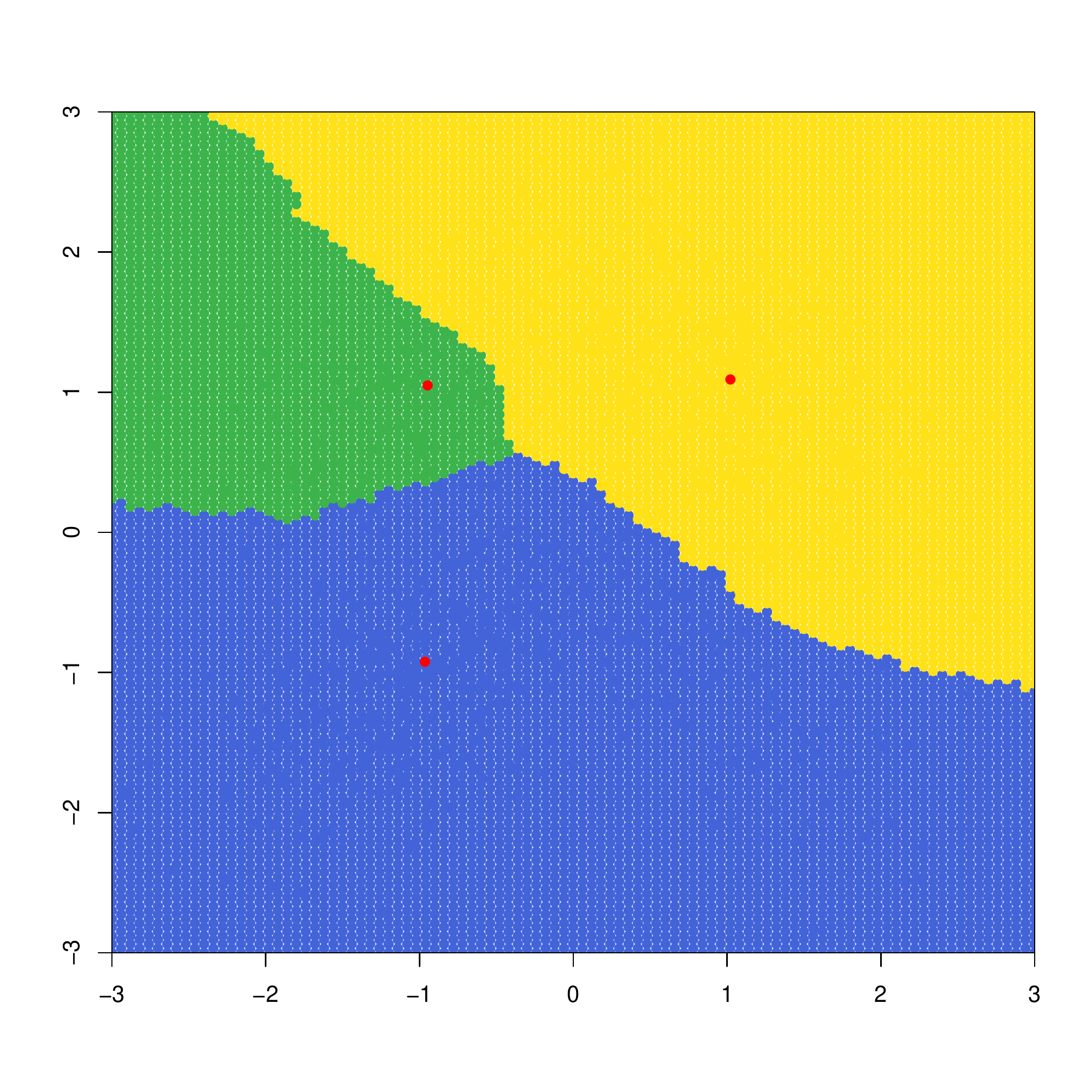}
\includegraphics[width=0.32\linewidth]{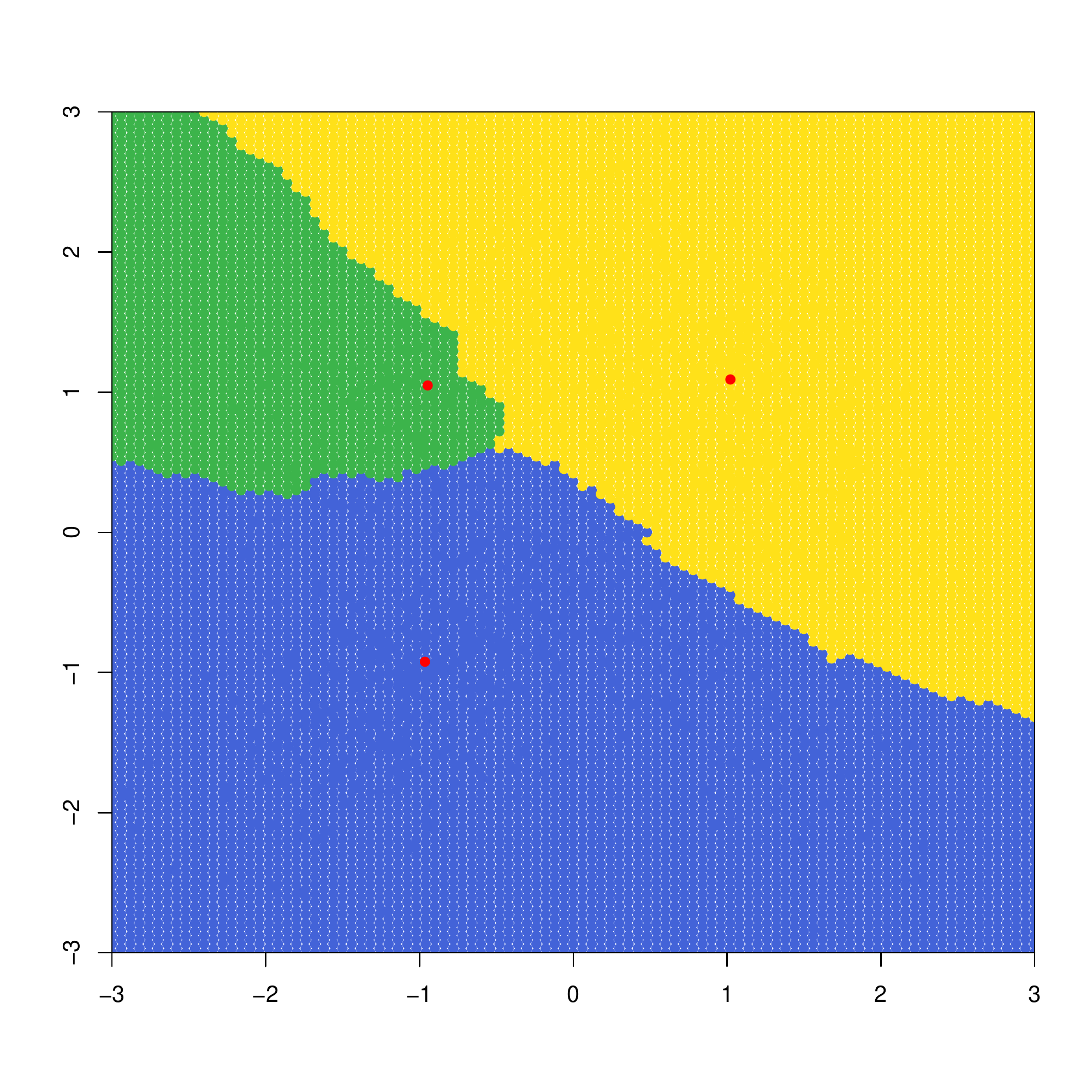}
\includegraphics[width=0.32\linewidth]{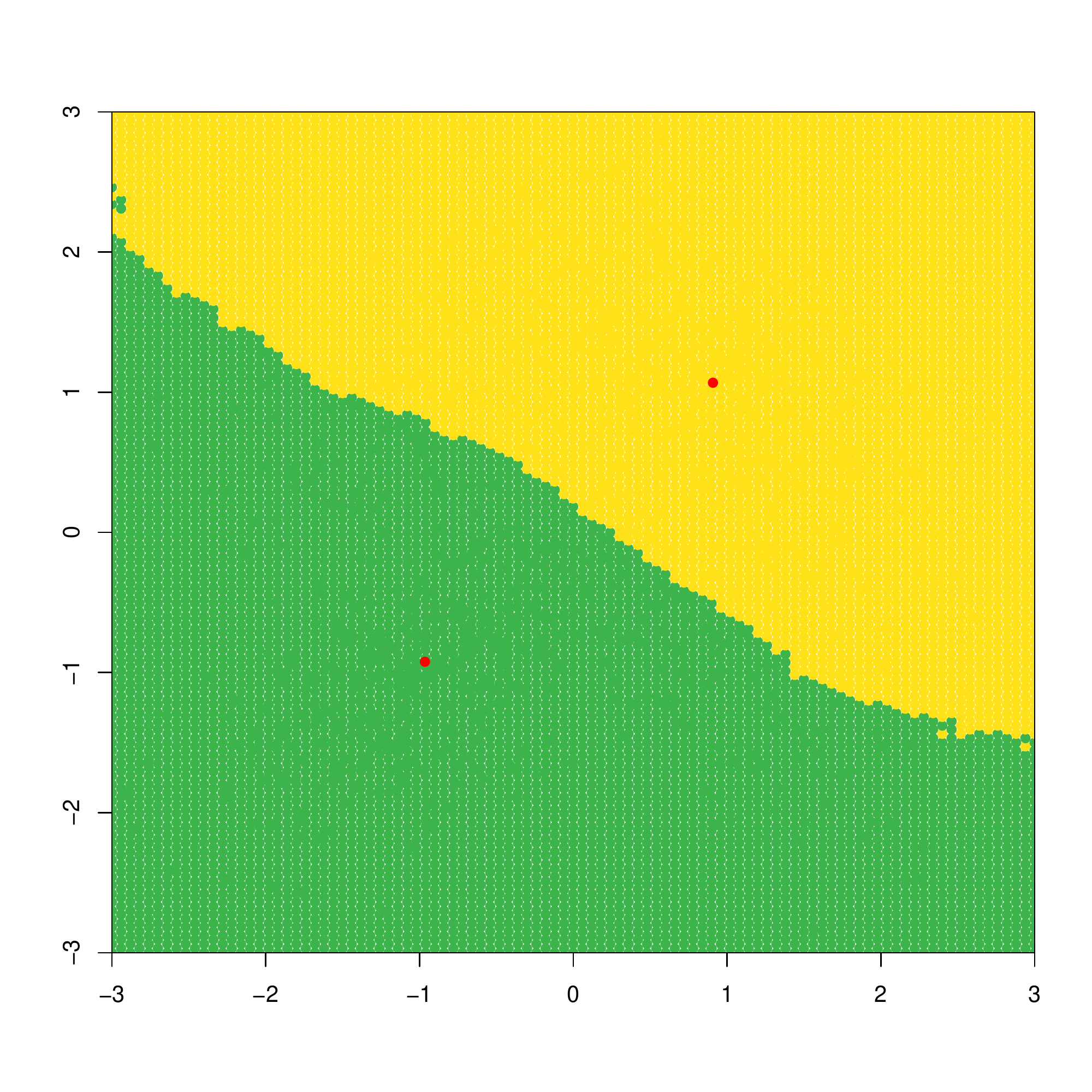}
\includegraphics[width=0.32\linewidth]{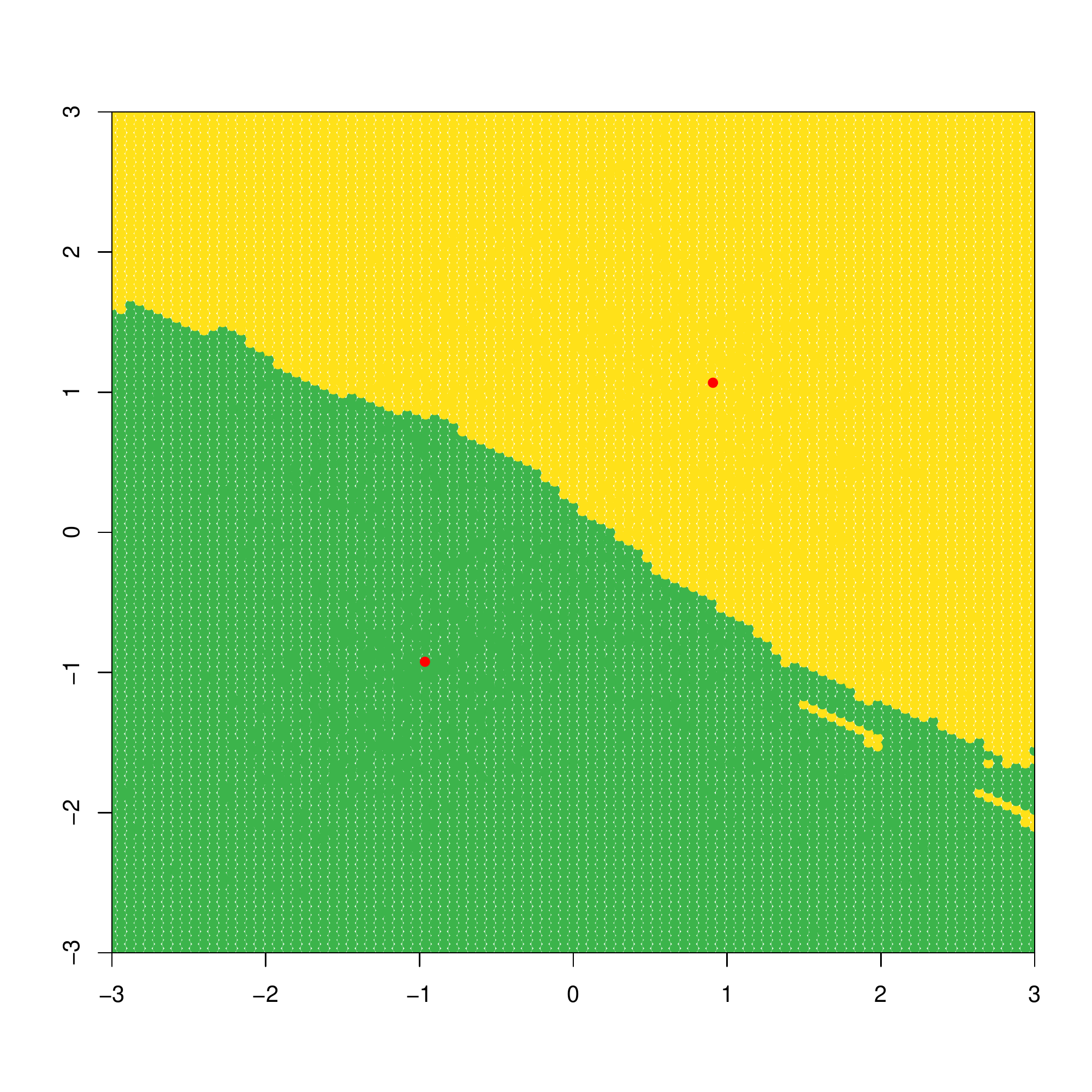}
\includegraphics[width=0.32\linewidth]{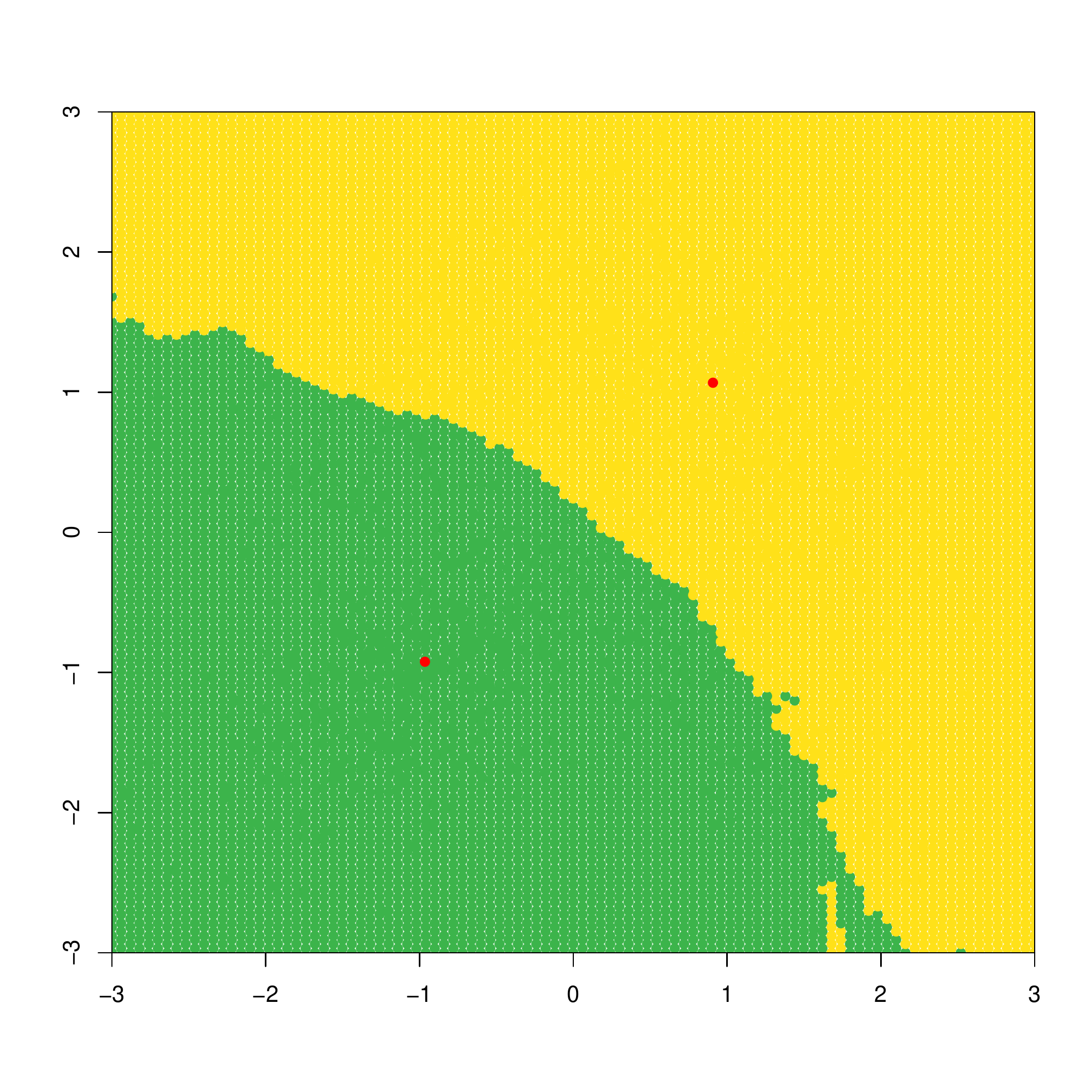}
\includegraphics[width=0.32\linewidth]{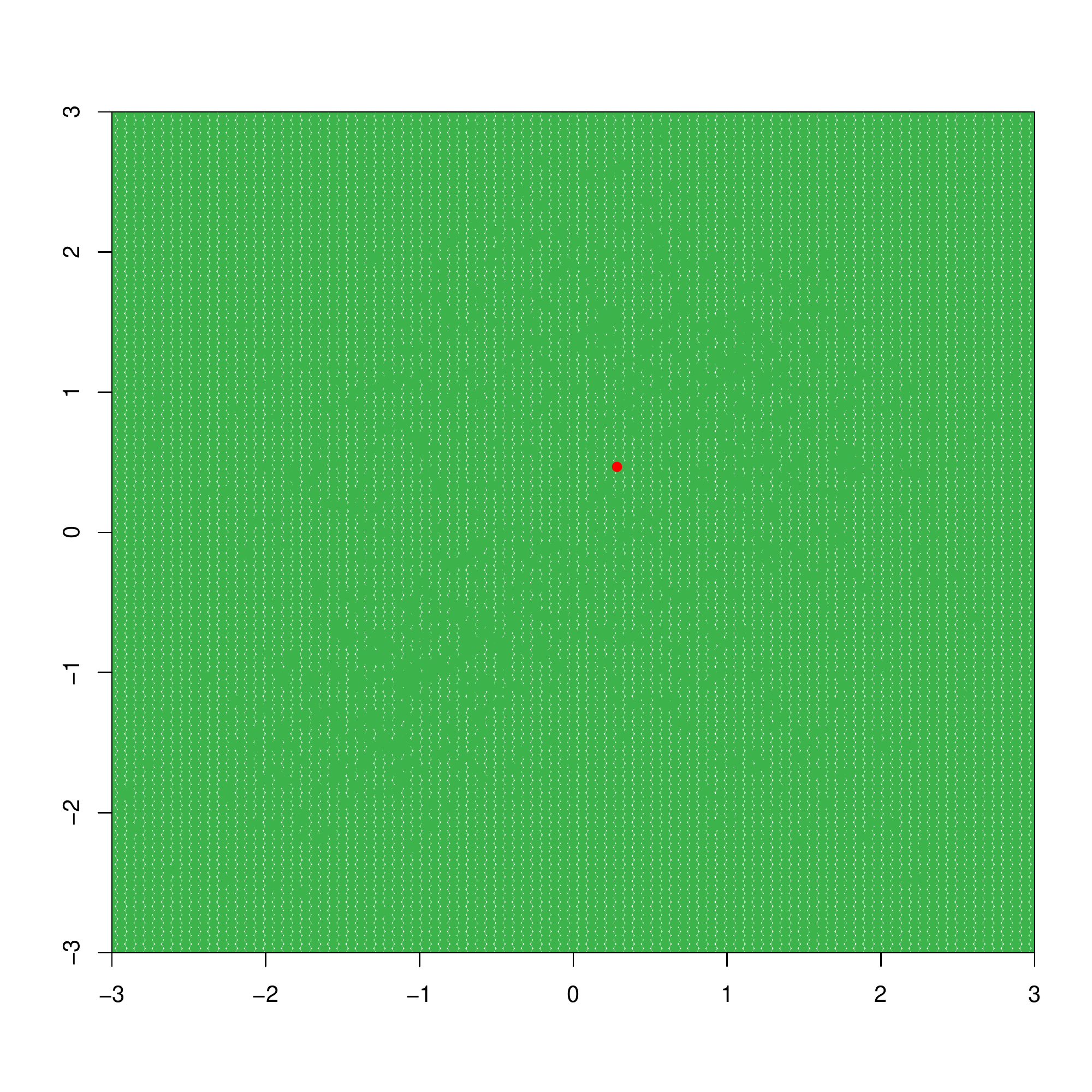}
\includegraphics[width=0.32\linewidth]{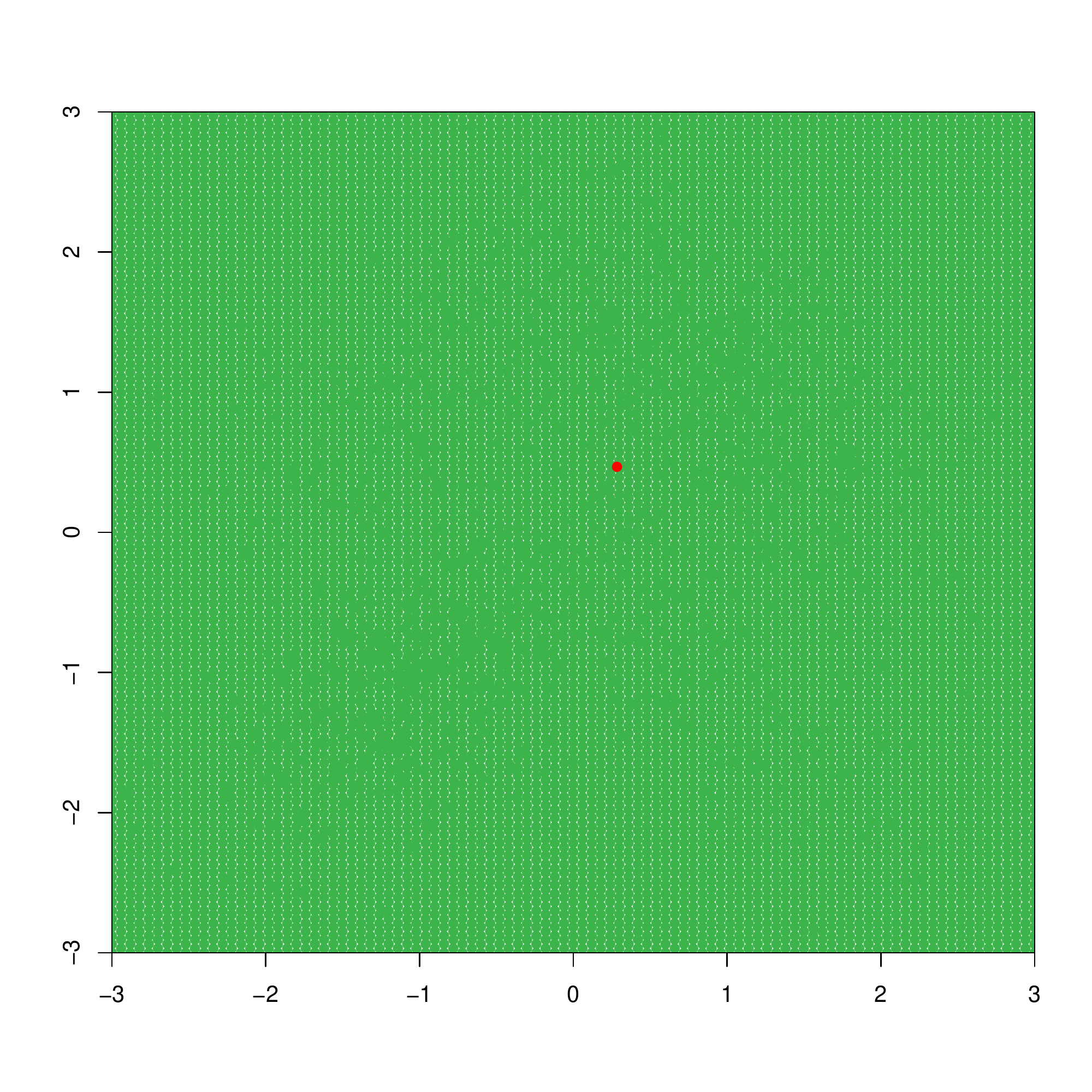}
\includegraphics[width=0.32\linewidth]{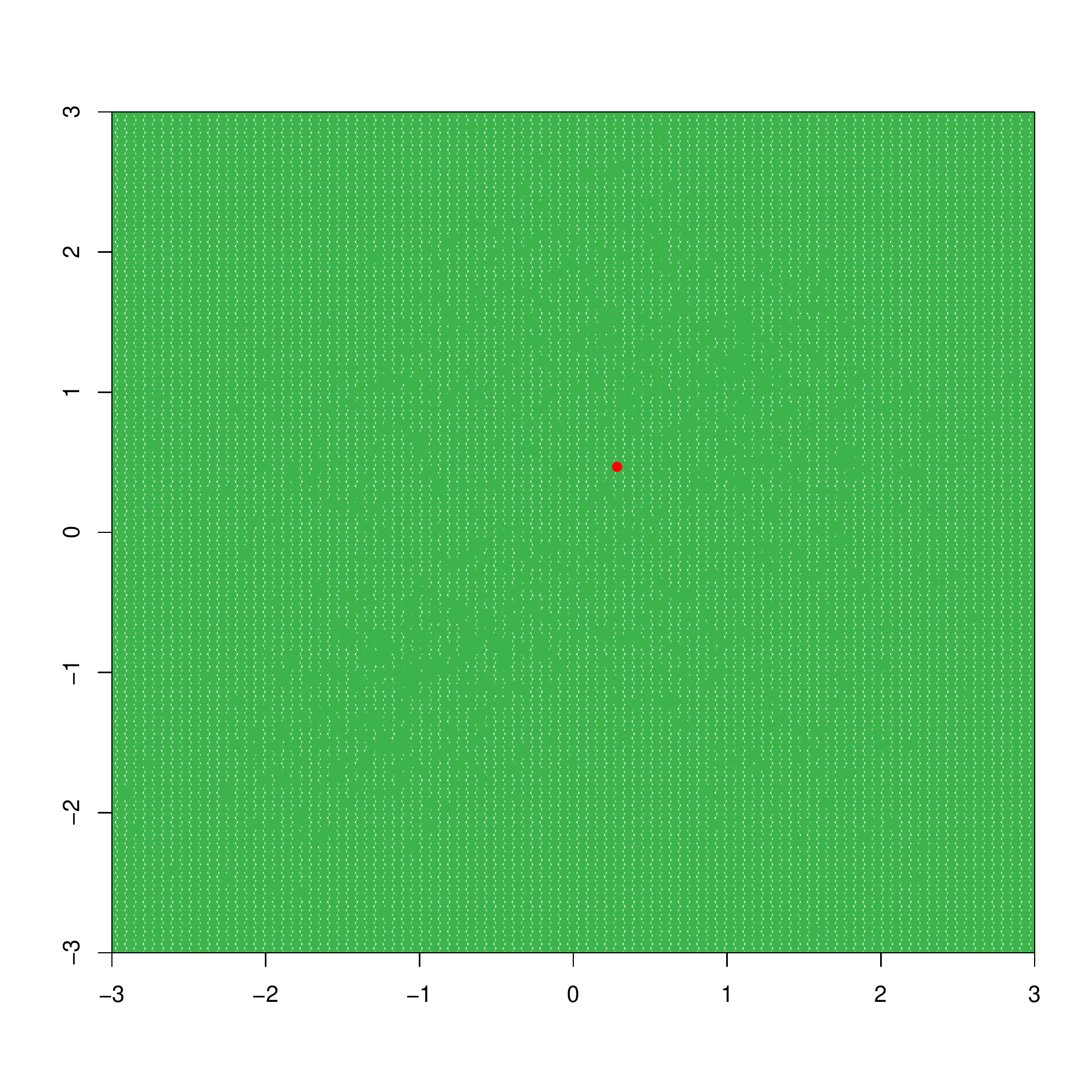}

\caption{Local depth clustering of $n=1000$ samples from the (L) Quadrimodal density. The predicted local maxima are plotted in red. The parameters are $r=0.05$, $s=10,30,50$ in each column (from left to right) and $q=0.05,0.10,0.25,0.50$ in each row (from the top down).}
\label{sm:plot_clusters_quadrimodal_prob_all_points}
\end{figure}

\begin{figure}
\centering
\includegraphics[width=0.32\linewidth]{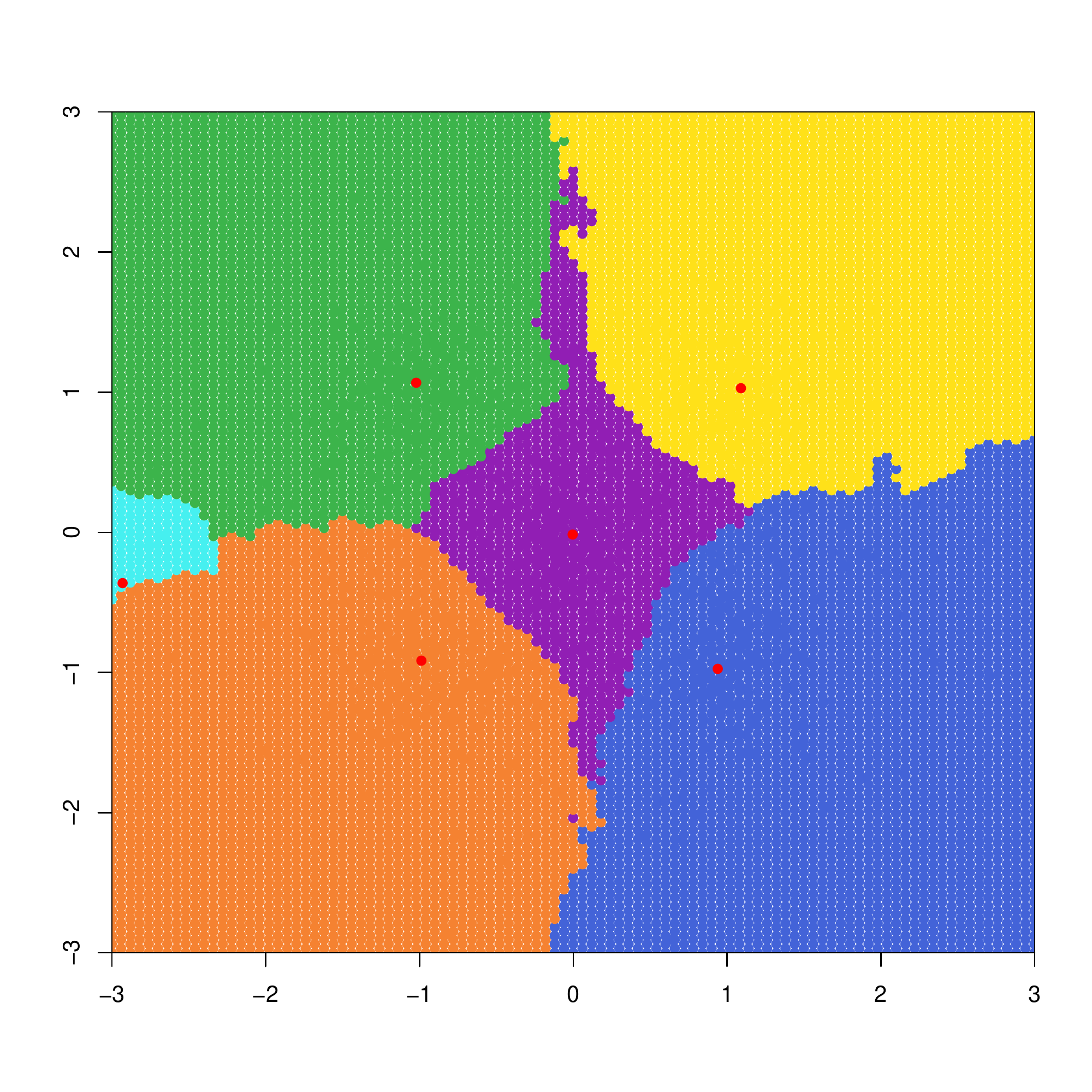}
\includegraphics[width=0.32\linewidth]{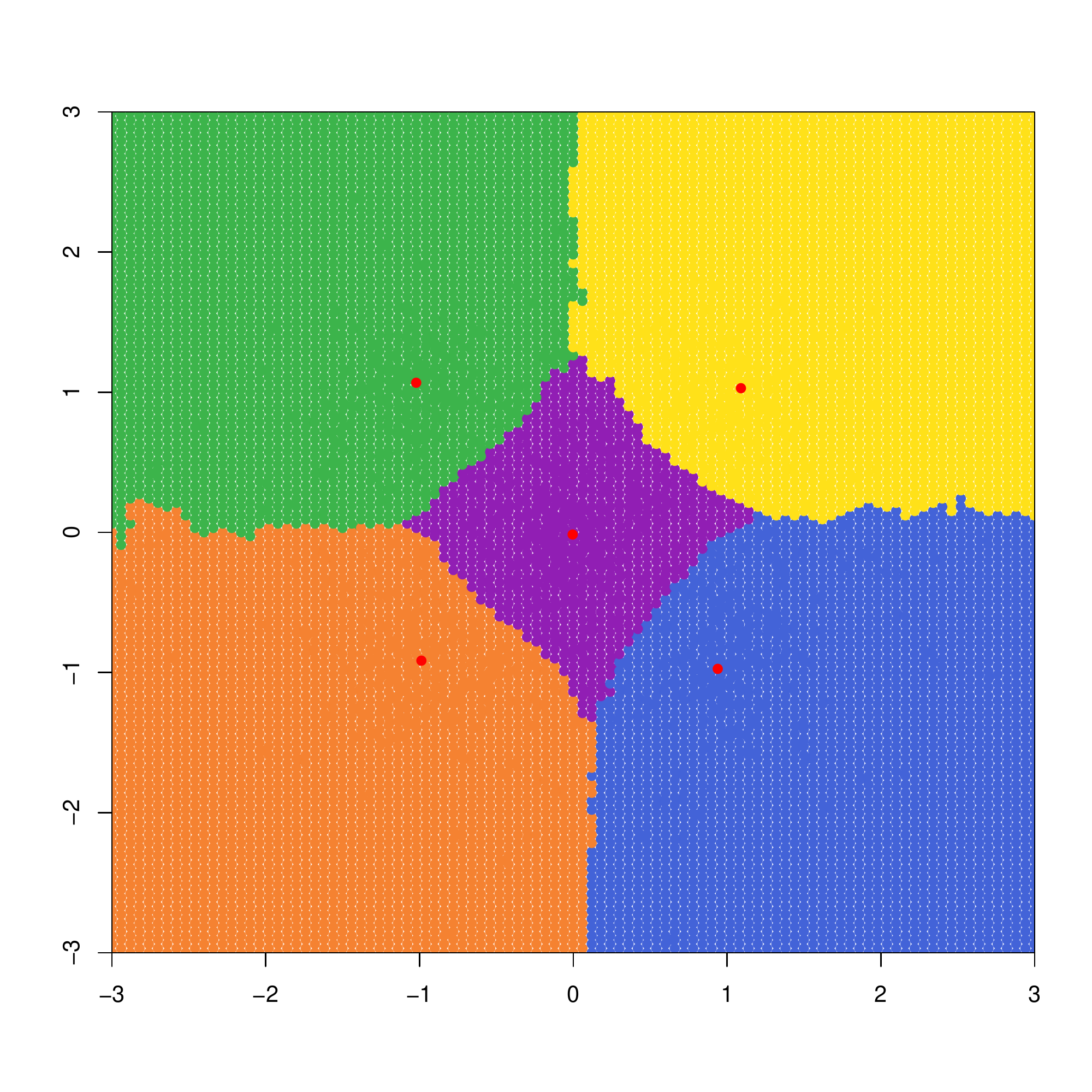}
\includegraphics[width=0.32\linewidth]{{lens-10_fountain-prob-0.38-0.05-50-0.05-1000-ldc-all}.pdf}
\includegraphics[width=0.32\linewidth]{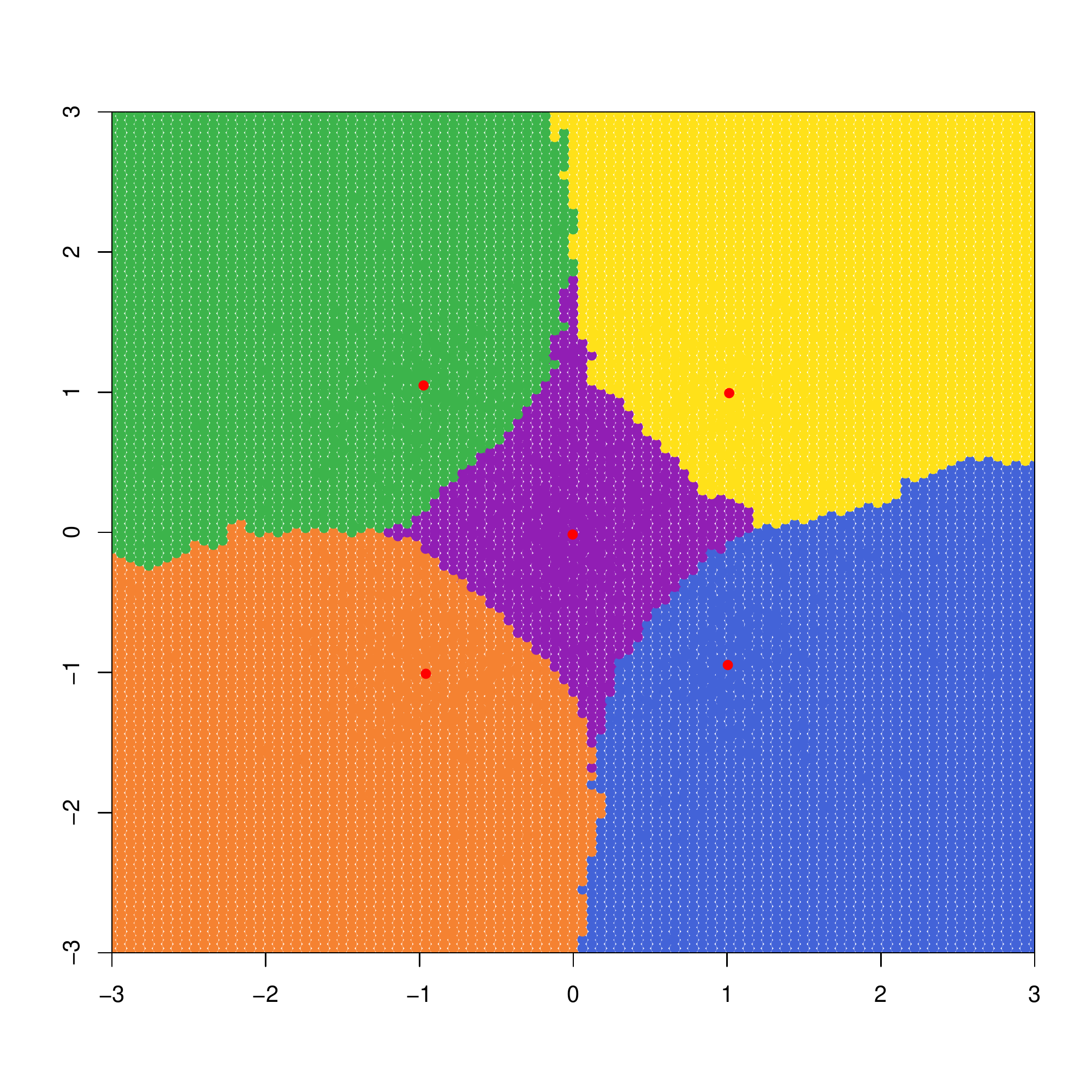}
\includegraphics[width=0.32\linewidth]{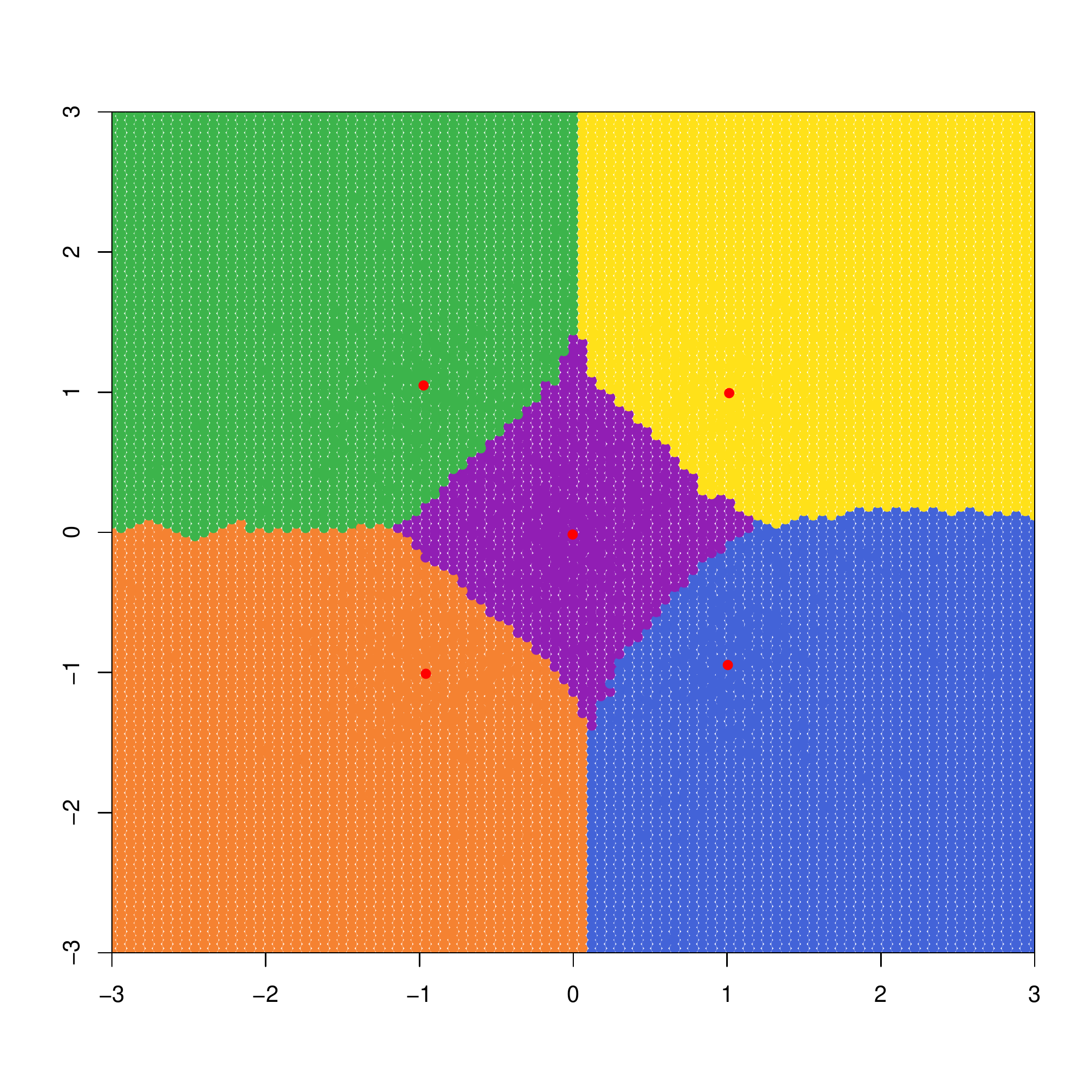}
\includegraphics[width=0.32\linewidth]{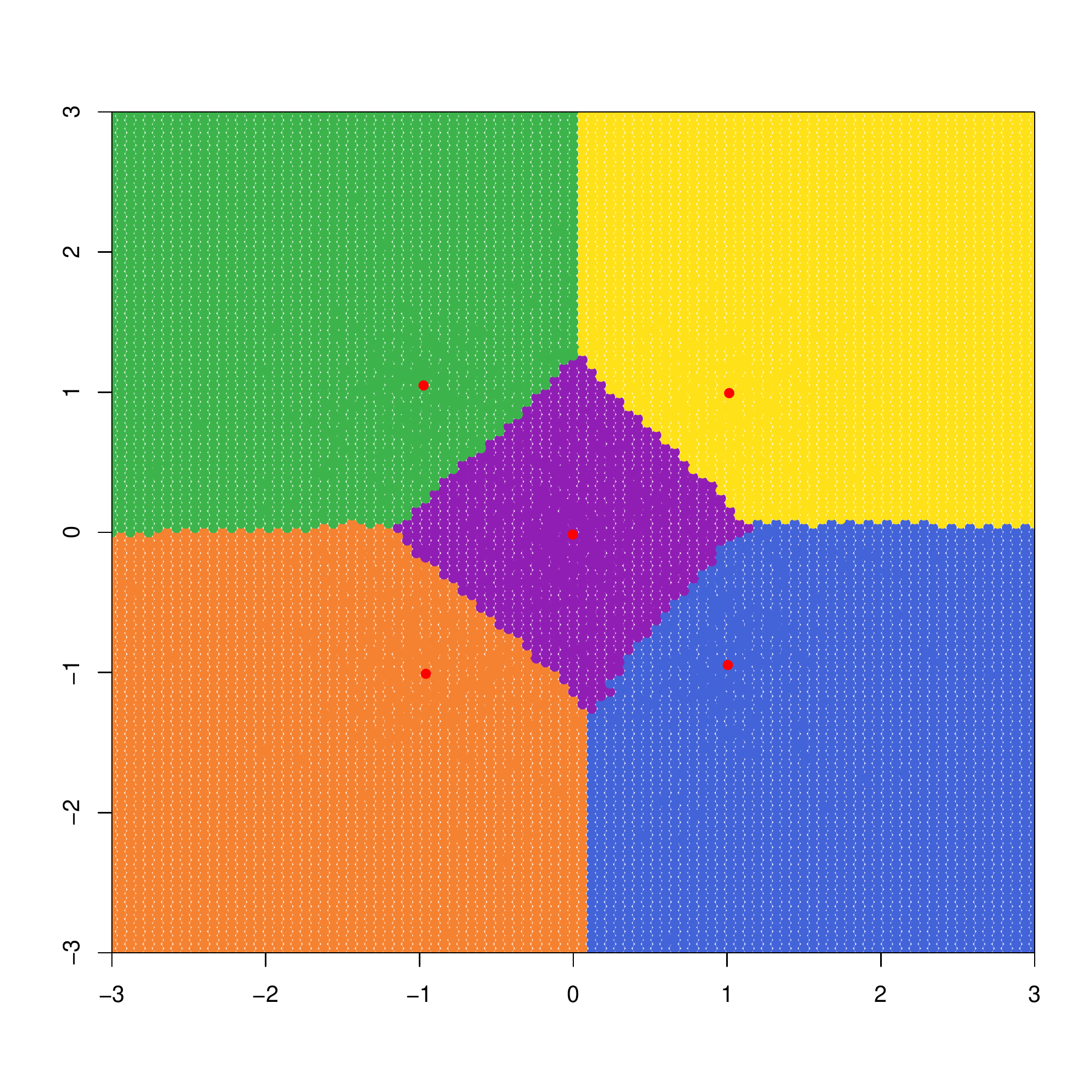}
\includegraphics[width=0.32\linewidth]{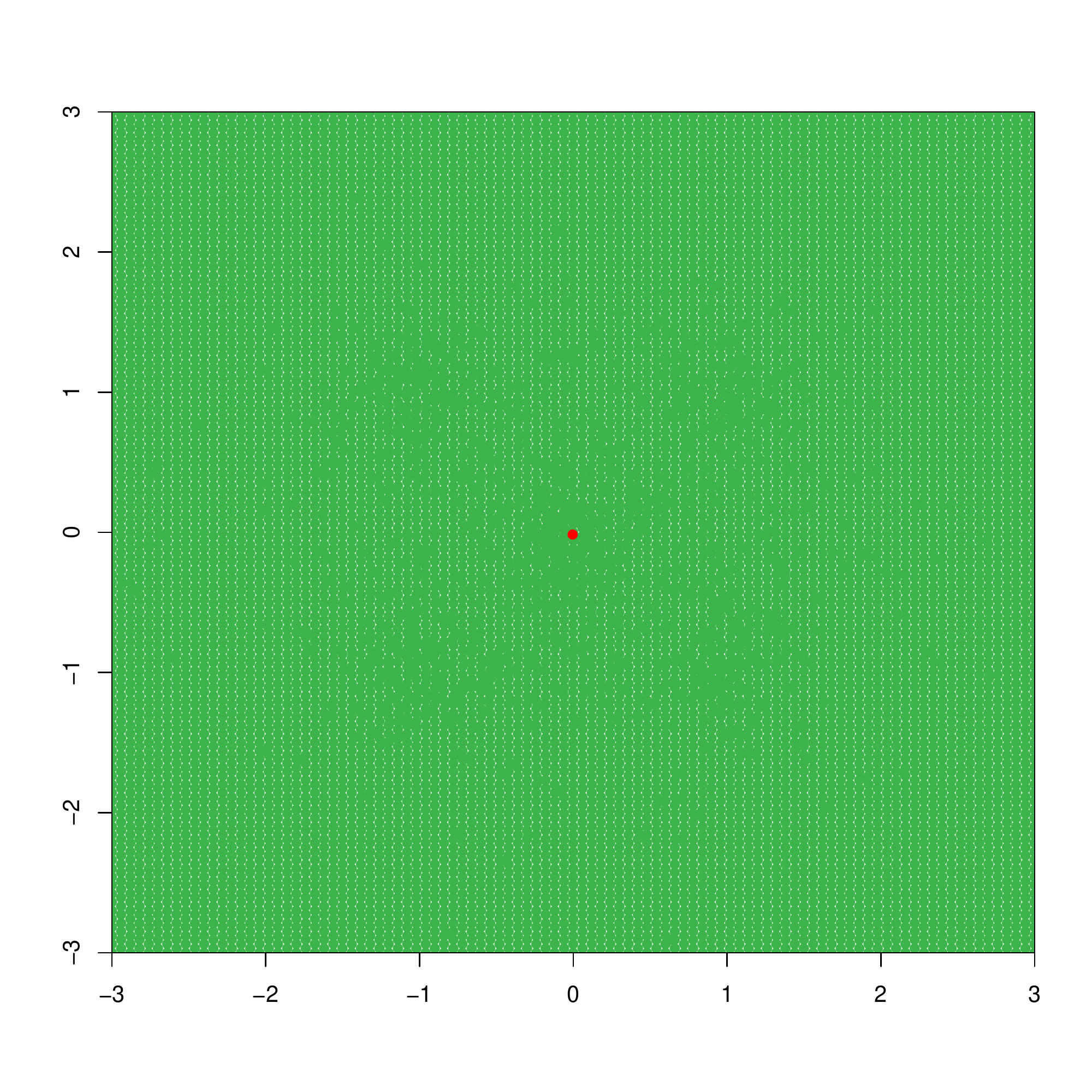}
\includegraphics[width=0.32\linewidth]{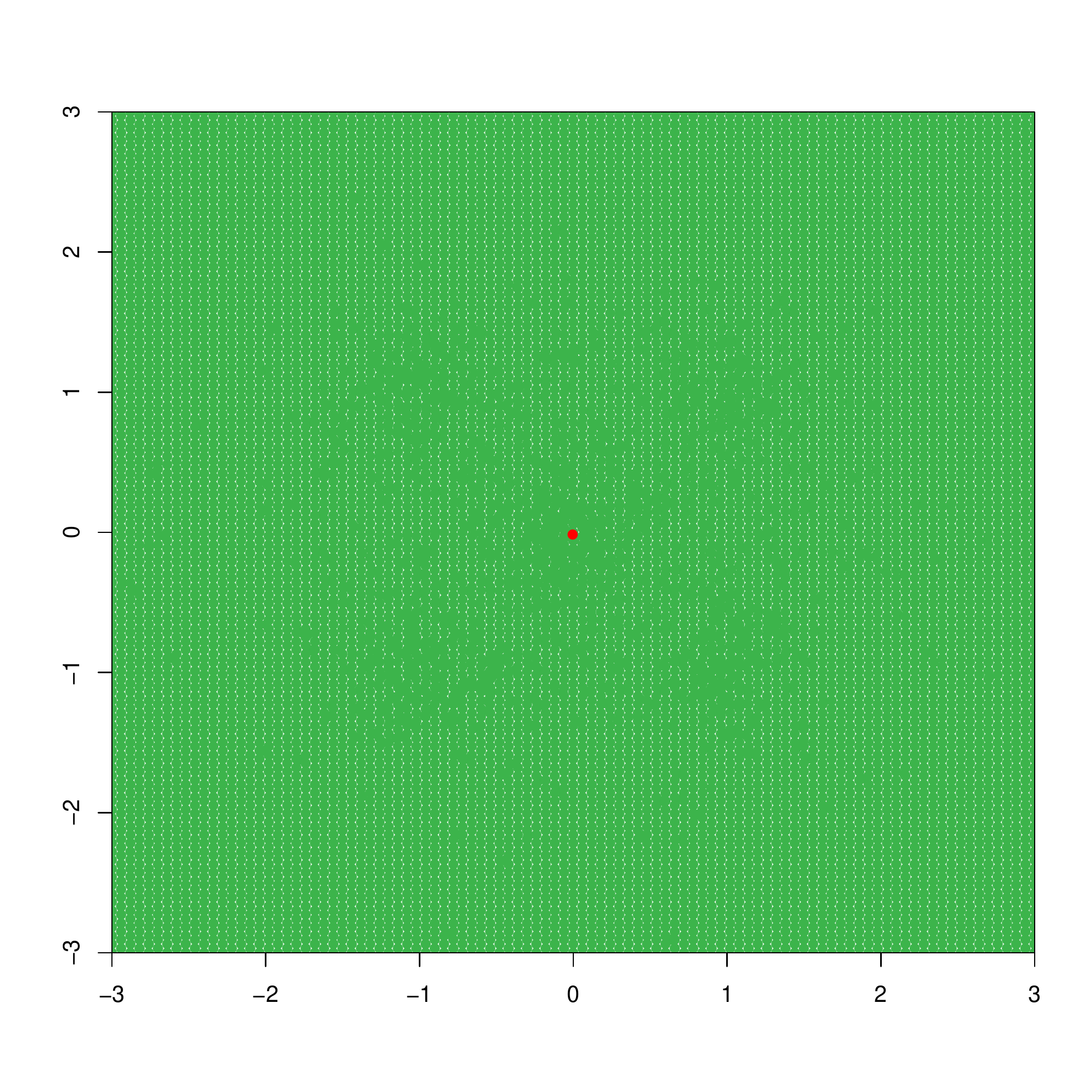}
\includegraphics[width=0.32\linewidth]{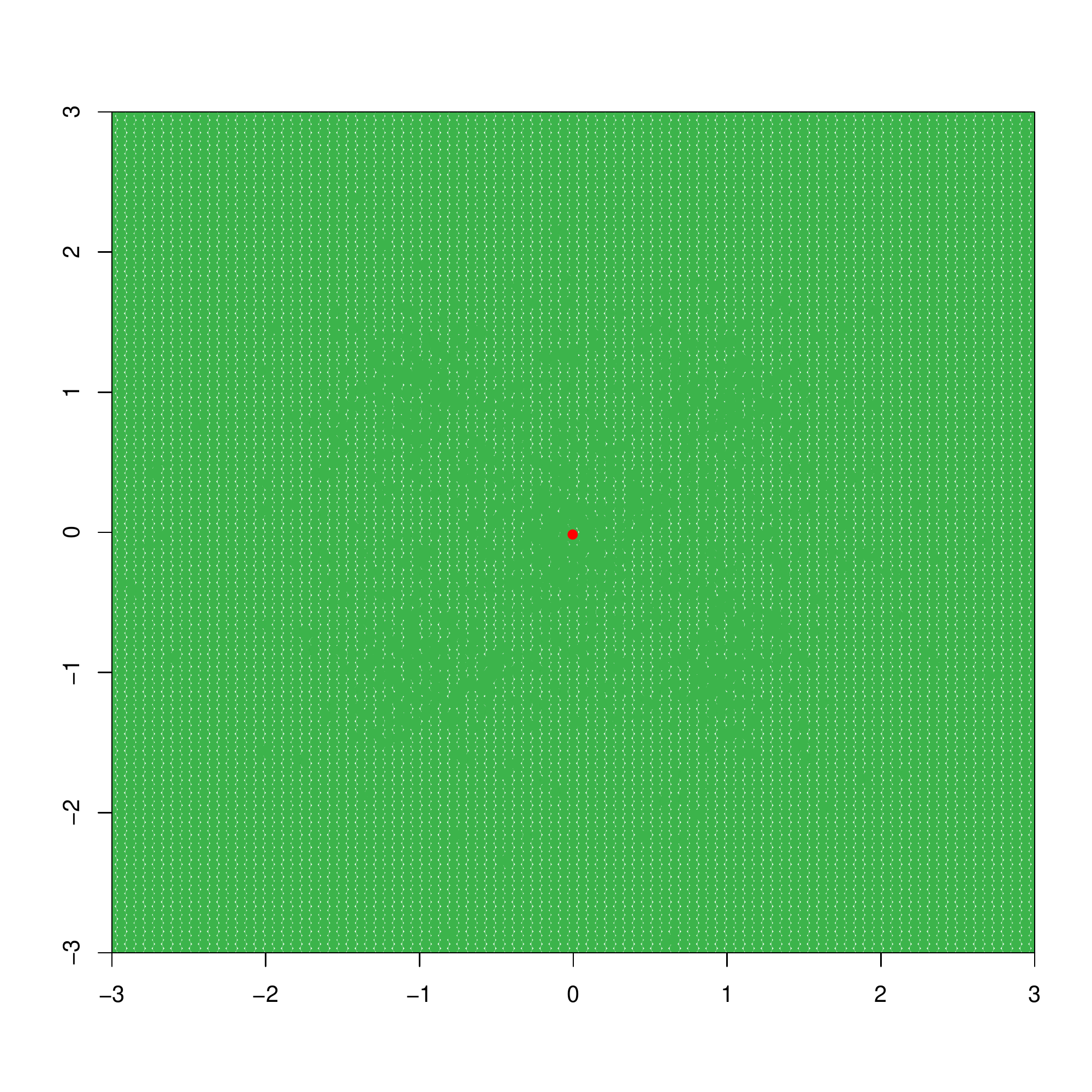}
\includegraphics[width=0.32\linewidth]{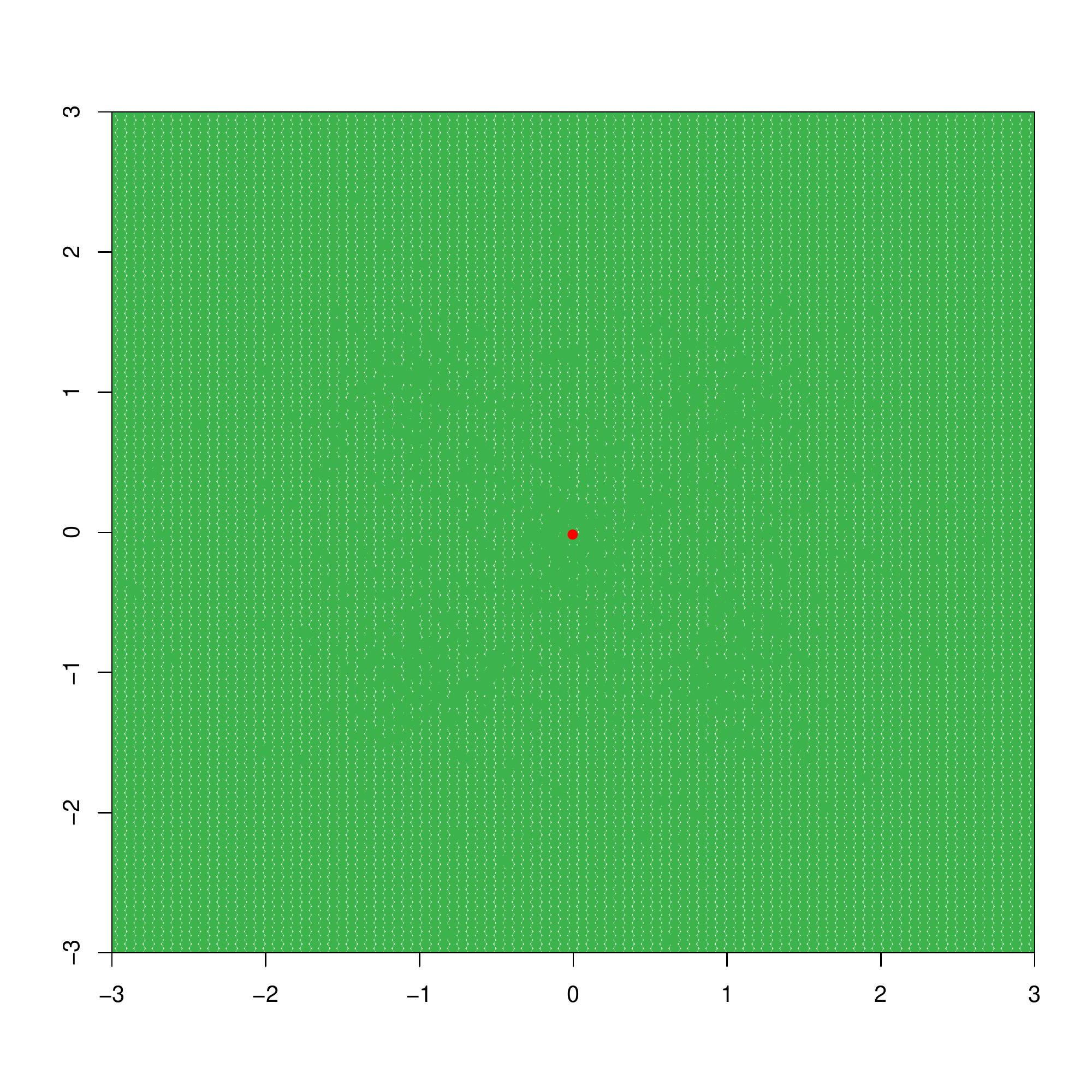}
\includegraphics[width=0.32\linewidth]{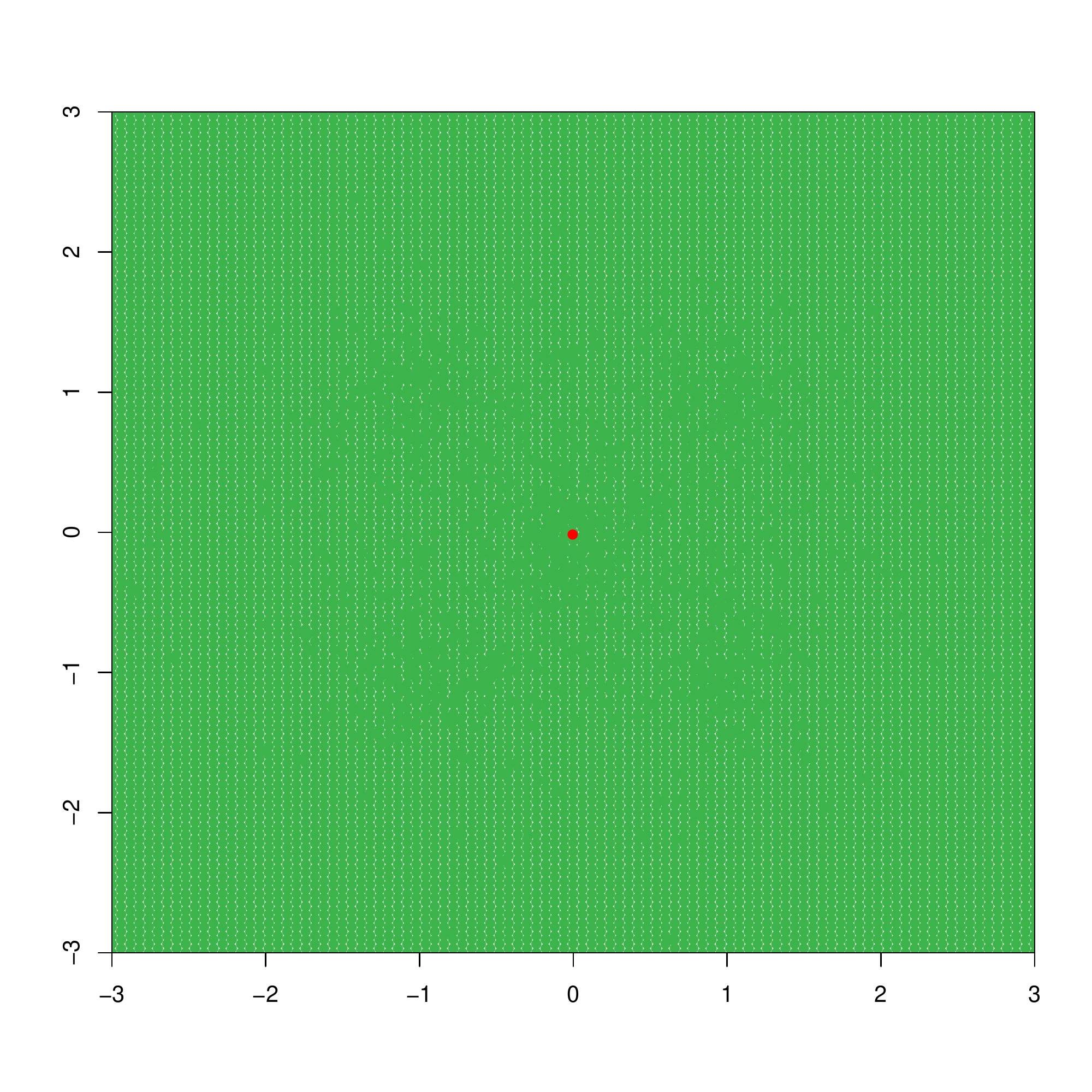}
\includegraphics[width=0.32\linewidth]{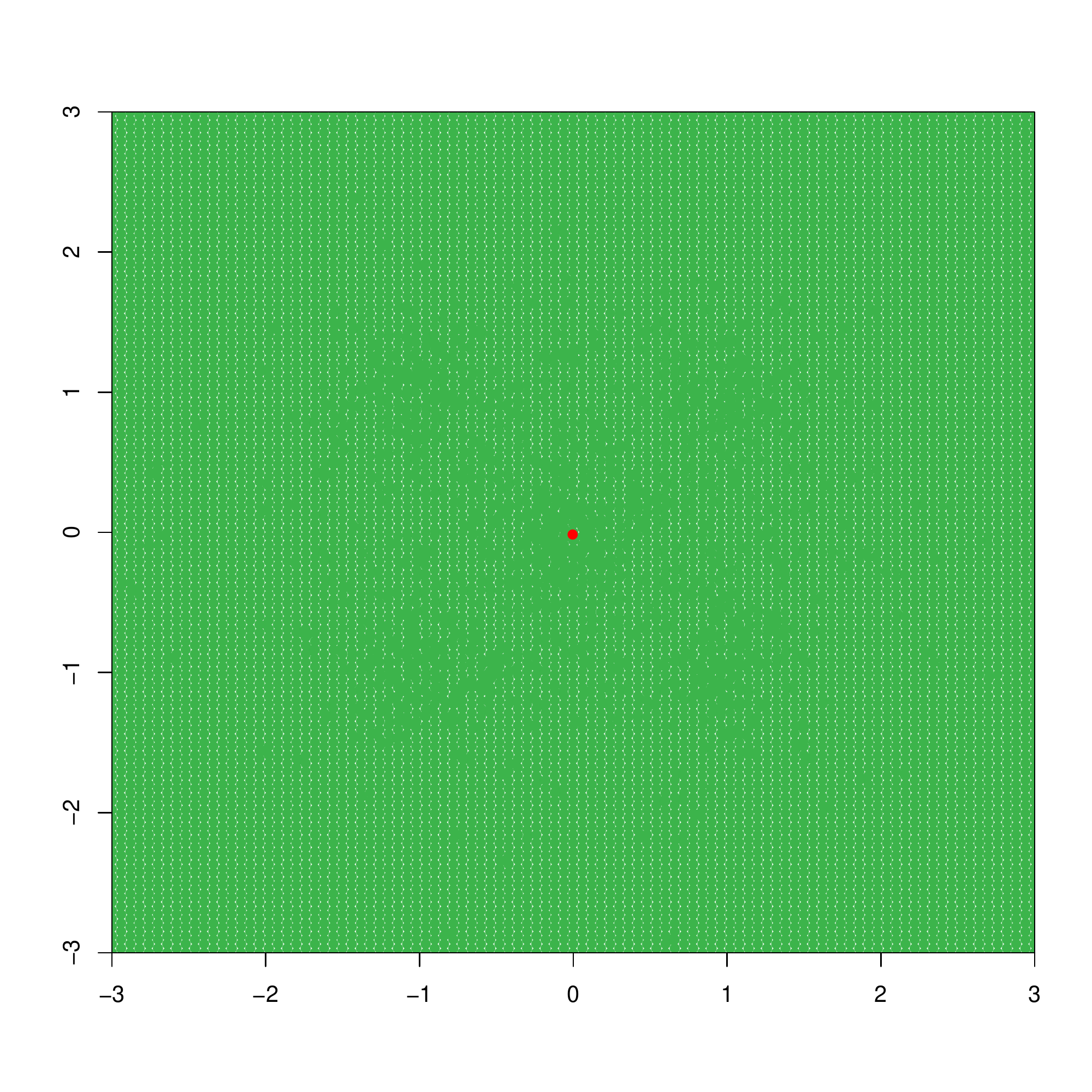}

\caption{Local depth clustering of $n=1000$ samples from the \#10 Fountain density. The predicted local maxima are plotted in red. The parameters are $r=0.05$, $s=10,30,50$ in each column (from left to right) and $q=0.05,0.10,0.25,0.50$ in each row (from the top down).}
\label{sm:plot_clusters_fountain_prob_all_points}
\end{figure}

\subsection{Numerical experiments}
\label{sm:subsection_numerical_experiments}

In this subsection, we provide additional simulation results complementing those in Section \ref{section_simulations_and_data_analysis} of the main paper. Tables \ref{sm:table_Chacon_1000} through \ref{sm:table_number_times_1000} contain results  for all the distributions in Appendix \ref{sm:subsection_true_clusters} for several choices of $q$ and $s$ beyond what is described in the main paper. Since the parameter $r$ does not affect the output of Algorithm \ref{algorithm_clusters}, we leave it fixed at $r=0.05$. The expressions LLD-$q$-$s$ and LSD-$q$-$s$ refer to LLD and LSD with parameters $q$ and $s$. To investigate the effect of sample size, in Tables \ref{sm:table_Chacon_500} through \ref{sm:table_number_times_500} $n$ is set at 500, instead of 1000. In the second part of Tables \ref{sm:table_Chacon_1000} through \ref{sm:table_normal_symmetric_1000} and \ref{sm:table_Chacon_500} through \ref{sm:table_normal_symmetric_500}, the first row refers to the case $\eta=0$ and the second row to the case $\eta=1$. From these two values, it is possible to compute the distance in probability for all values of $\eta$. In all the tables the best results are in bold face. For the probability distance, the best results for the case $\eta=1$ are in bold face. 

From the tables, we can observe that generally the best results are obtained for our proposed LLD and LSD cluster methodologies. In particular, for the distributions Mult.\ Bimodal and Mult.\ Quadrimodal, the best results are always obtained by LLD, see Tables \ref{sm:table_normal_symmetric_1000}, \ref{sm:table_number_times_1000}, \ref{sm:table_normal_symmetric_500} and \ref{sm:table_number_times_500}. For the distributions (H) Bimodal IV and Circular 2, LSD outperforms all the other procedures with the LLD being the second best, as can be seen from Tables \ref{sm:table_Chacon_1000}, \ref{sm:table_circular_1000}, \ref{sm:table_number_times_1000}, \ref{sm:table_Chacon_500}, \ref{sm:table_circular_500} and \ref{sm:table_number_times_500}. From the tables, it also clear that the best clustering results for the distributions \#10 Fountain, Circular 2 Cauchy, Circular 4 Cauchy, Bimodal, and Quadrimodal, are shared by LLD and LSD. Furthermore, the best results for the other three distributions, namely, (K) Trimodal III, (L) Quadrimodal, and Circular 3, are provided by LLD, LSD, and KDE. Turning to the true number of clusters, the best results for (K) Trimodal III are obtained using LLD while the best results for Bimodal and Circular 3 are obtained using LSD. Finally, we notice that the merging algorithm in \citet{Chazal-2013} may be used to improve the results of KDE, LLD and LSD for the circular densities in Tables \ref{sm:table_circular_1000}, \ref{sm:table_number_times_1000}, \ref{sm:table_circular_500} and \ref{sm:table_number_times_500}. As explained previously in Subsection \ref{subsection_numerical_experiments}, one can improve the performance of LSD using smaller values of $q$.

\begin{ThreePartTable}
\begin{tabularx}{10cm}{ | p{2.26cm} | p{2.8cm} | p{2.8cm} | p{2.8cm} | p{2.8cm} | }
\hline
\multicolumn{5}{|c|}{\textbf{Clustering errors (Hausdorff distance)}} \\
\hline
& (H) Bimodal IV & (K) Trimodal III & (L) Quadrimodal & \#10 Fountain \\
\hline
KDE & 0.19 (0.21) & {\bf{0.10 (0.11)}} & {\bf{0.14 (0.10)}} & 0.12 (0.06) \\
\hline
LLD-0.05-30 & 0.27 (0.17) & 0.18 (0.10) & 0.15 (0.08) & {\bf{0.06 (0.02)}} \\
\hline
LLD-0.05-50 & 0.22 (0.17) & 0.17 (0.12) & 0.16 (0.10) & {\bf{0.06 (0.01)}} \\
\hline
LLD-0.1-30  & 0.05 (0.11) & {\bf{0.10 (0.15)}} & 0.22 (0.16) & {\bf{0.06 (0.01)}} \\
\hline
LLD-0.1-50  & 0.02 (0.08) & 0.12 (0.17) & 0.26 (0.16) & {\bf{0.06 (0.01)}} \\
\hline
LSD-$0.01$-30  & 0.05 (0.11) & {\bf{0.10 (0.15)}} & 0.20 (0.15) & {\bf{0.06 (0.01)}} \\
\hline
LSD-$0.01$-50  & 0.04 (0.09) & 0.13 (0.17) & 0.26 (0.16) & {\bf{0.06 (0.01)}} \\
\hline
LSD-$0.05$-30  & {\bf{0.00 (0.00)}} & 0.25 (0.21) & 0.48 (0.02) & 0.35 (0.15) \\
\hline
LSD-$0.05$-50  & {\bf{0.00 (0.00)}} & 0.29 (0.21) & 0.48 (0.01) & 0.38 (0.15) \\
\hline
Hclust \tnotex{notetrueclusters} & 0.05 (0.09) & 0.15 (0.09) & 0.22 (0.08) & 0.29 (0.05) \\
\hline
\multicolumn{5}{|c|}{\textbf{Clustering errors (distance in probability)}} \\
\hline
& (H) Bimodal IV & (K) Trimodal III & (L) Quadrimodal & \#10 Fountain \\
\hline
KDE & 0.24 (0.28) \newline 0.37 (0.41) & 0.05 (0.04) \newline 0.08 (0.08) & 0.21 (0.20) \newline 0.29 (0.28) & 0.35 (0.30) \newline 0.42 (0.36) \\
\hline
LLD-0.05-30 & 0.33 (0.21) \newline 0.53 (0.33) & 0.10 (0.07) \newline 0.18 (0.14) & 0.31 (0.22) \newline 0.42 (0.30) & 0.07 (0.05) \newline 0.07 (0.06) \\
\hline
LLD-0.05-50 & 0.31 (0.25) \newline 0.47 (0.37) & 0.09 (0.07) \newline 0.15 (0.13) & 0.21 (0.21) \newline 0.26 (0.26) & 0.06 (0.01) \newline {\bf{0.06 (0.01)}} \\
\hline
LLD-0.1-30  & 0.09 (0.19) \newline 0.13 (0.28) & 0.04 (0.04) \newline {\bf{0.06 (0.07)}} & 0.09 (0.06) \newline {\bf{0.14 (0.08)}} & 0.06 (0.01) \newline {\bf{0.06 (0.01)}} \\
\hline
LLD-0.1-50  & 0.04 (0.12) \newline 0.05 (0.18) & 0.04 (0.03) \newline {\bf{0.06 (0.05)}} & 0.10 (0.06) \newline 0.15 (0.08) & 0.06 (0.01) \newline {\bf{0.06 (0.01)}} \\
\hline
LSD-$0.01$-30  & 0.08 (0.18) \newline 0.12 (0.27) & 0.05 (0.06) \newline 0.07 (0.09) & 0.10 (0.07) \newline {\bf{0.14 (0.10)}} & 0.06 (0.01) \newline {\bf{0.06 (0.01)}} \\
\hline
LSD-$0.01$-50  & 0.06 (0.15) \newline 0.08 (0.23) & 0.05 (0.03) \newline {\bf{0.06 (0.06)}} & 0.10 (0.07) \newline 0.16 (0.09) & 0.06 (0.01) \newline {\bf{0.06 (0.00)}} \\
\hline
LSD-$0.05$-30  & 0.00 (0.00) \newline {\bf{0.00 (0.00)}} & 0.06 (0.02) \newline 0.09 (0.05) & 0.11 (0.01) \newline 0.20 (0.01) & 0.19 (0.07) \newline 0.32 (0.14) \\
\hline
LSD-$0.05$-50  & 0.00 (0.00) \newline {\bf{0.00 (0.00)}} & 0.06 (0.02) \newline 0.10 (0.05) & 0.11 (0.01) \newline 0.20 (0.01) & 0.20 (0.06) \newline 0.35 (0.14) \\
\hline
Hclust \tnotex{notetrueclusters} & 0.05 (0.09) & 0.16 (0.09) & 0.29 (0.11) & 0.35 (0.07) \\
\hline
\caption{\label{sm:table_Chacon_1000} Mean of the clustering errors based on Hausdorff distance and distance in probability over $100$ replications with $n=1000$ samples for the densities (H) Bimodal IV, (K) Trimodal III, (L) Quadrimodal and \#10 Fountain. In parentheses the standard deviation.}
\end{tabularx}
\end{ThreePartTable}

\begin{ThreePartTable}
\begin{tabularx}{10cm}{ | p{2.26cm} | p{2.8cm} | p{2.8cm} | p{2.8cm} | p{2.8cm} | }
\hline
\multicolumn{5}{|c|}{\textbf{Clustering errors (Hausdorff distance)}} \\
\hline
& Circular 2 & Circular 2 Cauchy & Circular 3 & Circular 4 Cauchy \\
\hline
KDE & 0.57 (0.07) & 0.54 (0.08) & 0.32 (0.01) & 0.08 (0.10) \\
\hline
LLD-0.05-30 & 0.61 (0.05) & 0.56 (0.07) & 0.33 (0.01) & 0.02 (0.03) \\
\hline
LLD-0.05-50 & 0.59 (0.05) & 0.52 (0.07) & 0.34 (0.02) & 0.02 (0.01) \\
\hline
LLD-0.1-30  & 0.52 (0.06) & 0.45 (0.08) & 0.32 (0.02) & {\bf{0.01 (0.01)}} \\
\hline
LLD-0.1-50  & 0.49 (0.07) & 0.43 (0.10) & {\bf{0.31 (0.03)}} & {\bf{0.01 (0.01)}} \\
\hline
LLD-0.15-30  & 0.43 (0.12) & 0.35 (0.20) & 0.32 (0.03) & 0.02 (0.01) \\
\hline
LLD-0.15-50  & 0.39 (0.16) & {\bf{0.30 (0.21)}} & {\bf{0.31 (0.04)}}  & 0.02 (0.01) \\
\hline
LLD-0.2-30  & 0.35 (0.19) & 0.39 (0.31) & 0.34 (0.05) & 0.04 (0.05) \\
\hline
LLD-0.2-50  & 0.28 (0.21) & 0.49 (0.34) & 0.40 (0.09) & 0.04 (0.07) \\
\hline
LSD-$0.01$-30  & 0.50 (0.06) & 0.43 (0.12) & {\bf{0.31 (0.03)}} & {\bf{0.01 (0.01)}} \\
\hline
LSD-$0.01$-50  & 0.48 (0.08) & 0.41 (0.14) & {\bf{0.31 (0.03)}} & {\bf{0.01 (0.01)}} \\
\hline
LSD-$0.05$-30  & 0.26 (0.21) & 0.43 (0.31) & {\bf{0.31 (0.04)}} & 0.04 (0.06) \\
\hline
LSD-$0.05$-50  & {\bf{0.16 (0.20)}} & 0.49 (0.33) & 0.32 (0.07) & 0.04 (0.06) \\
\hline
LSD-$0.1$-30  & 0.38 (0.15) & 0.78 (0.10) & 0.36 (0.07) & 0.25 (0.20) \\
\hline
LSD-$0.1$-50  & 0.20 (0.19) & 0.80 (0.07) & 0.49 (0.12) & 0.28 (0.21) \\
\hline
Hclust \tnotex{notetrueclusters}  & 0.34 (0.12) & 0.38 (0.08) & 0.37 (0.05) & 0.24 (0.06) \\
\hline
\multicolumn{5}{|c|}{\textbf{Clustering errors (distance in probability)}} \\
\hline
& Circular 2 & Circular 2 Cauchy & Circular 3 & Circular 4 Cauchy \\
\hline
KDE & 0.31 (0.04) \newline 0.62 (0.08) & 0.32 (0.05) \newline 0.63 (0.10) & 0.29 (0.03) \newline 0.58 (0.06) & 0.17 (0.25) \newline 0.21 (0.30) \\
\hline
LLD-0.05-30 & 0.34 (0.03) \newline 0.66 (0.07) & 0.32 (0.05) \newline 0.65 (0.09) & 0.32 (0.03) \newline 0.59 (0.07) & 0.04 (0.09) \newline 0.04 (0.11) \\
\hline
LLD-0.05-50 & 0.33 (0.04) \newline 0.64 (0.07) & 0.33 (0.07) \newline 0.62 (0.09) & 0.32 (0.05) \newline 0.57 (0.08) & 0.03 (0.04) \newline 0.03 (0.05) \\
\hline
LLD-0.1-30  & 0.31 (0.06) \newline 0.59 (0.10) & 0.32 (0.12) \newline 0.55 (0.15) & 0.32 (0.06) \newline 0.55 (0.11) & 0.02 (0.01) \newline {\bf{0.02 (0.01)}} \\
\hline
LLD-0.1-50  & 0.32 (0.10) \newline 0.58 (0.12) & 0.33 (0.15) \newline 0.53 (0.18) & 0.34 (0.08) \newline 0.54 (0.14) & 0.02 (0.01) \newline {\bf{0.02 (0.01)}} \\
\hline
LLD-0.15-30  & 0.30 (0.13) \newline 0.51 (0.18) & 0.27 (0.20) \newline 0.40 (0.26) & 0.35 (0.07) \newline 0.55 (0.11) & 0.02 (0.01) \newline {\bf{0.02 (0.01)}} \\
\hline
LLD-0.15-50  & 0.30 (0.18) \newline 0.47 (0.23) & 0.24 (0.20) \newline 0.35 (0.26) & 0.38 (0.09) \newline 0.54 (0.14) & 0.03 (0.01) \newline 0.03 (0.01) \\
\hline
LLD-0.2-30  & 0.25 (0.18) \newline 0.40 (0.25) & 0.19 (0.19) \newline 0.25 (0.22) & 0.41 (0.09) \newline 0.55 (0.10) & 0.04 (0.03) \newline 0.04 (0.05) \\
\hline
LLD-0.2-50  & 0.20 (0.18) \newline 0.30 (0.26) & 0.14 (0.15) \newline 0.20 (0.17) & 0.46 (0.09) \newline 0.51 (0.10) & 0.04 (0.03) \newline 0.05 (0.06) \\
\hline
LSD-$0.01$-30  & 0.30 (0.07) \newline 0.57 (0.10) & 0.32 (0.14) \newline 0.52 (0.18) & 0.31 (0.07) \newline 0.54 (0.12) & 0.02 (0.01) \newline {\bf{0.02 (0.01)}} \\
\hline
LSD-$0.01$-50  & 0.32 (0.10) \newline 0.56 (0.13) & 0.32 (0.16) \newline 0.50 (0.20) & 0.33 (0.08) \newline 0.54 (0.14) & 0.02 (0.01) \newline {\bf{0.02 (0.01)}} \\
\hline
LSD-$0.05$-30  & 0.19 (0.18) \newline 0.29 (0.26) & 0.20 (0.18) \newline 0.26 (0.20) & 0.38 (0.08) \newline 0.54 (0.13) & 0.04 (0.03) \newline 0.05 (0.05) \\
\hline
LSD-$0.05$-50  & 0.13 (0.17) \newline {\bf{0.18 (0.23)}} & 0.15 (0.15) \newline 0.21 (0.16) & 0.40 (0.09) \newline 0.46 (0.13) & 0.04 (0.03) \newline 0.05 (0.05) \\
\hline
LSD-$0.1$-30  & 0.33 (0.13) \newline 0.57 (0.22) & 0.13 (0.11) \newline 0.21 (0.09) & 0.44 (0.08) \newline 0.56 (0.11) & 0.16 (0.09) \newline 0.28 (0.20) \\
\hline
LSD-$0.1$-50  & 0.21 (0.20) \newline 0.31 (0.31) & 0.10 (0.06) \newline {\bf{0.19 (0.05)}} & 0.43 (0.09) \newline 0.53 (0.08) & 0.18 (0.09) \newline 0.32 (0.21) \\
\hline
Hclust \tnotex{notetrueclusters} & 0.34 (0.12) & 0.38 (0.08) & {\bf{0.43 (0.05)}} & 0.34 (0.09) \\
\hline
\caption{\label{sm:table_circular_1000} Mean of the clustering errors based on Hausdorff distance and distance in probability over $100$ replications with $n=1000$ samples for the densities Circular 2, Circular 2 Cauchy, Circular 3 and Circular 4 Cauchy. In parentheses the standard deviation.}
\end{tabularx}
\end{ThreePartTable}

\begin{ThreePartTable}
\begin{tabularx}{10cm}{ | p{2.26cm} | p{2.8cm} | p{2.8cm} | p{2.8cm} | p{2.8cm} | }
\hline
\multicolumn{5}{|c|}{\textbf{Clustering errors (Hausdorff distance)}} \\
\hline
& Bimodal & Quadrimodal & Mult.\ Bimodal & Mult.\ Quadrimodal \\
\hline
KDE & 0.09 (0.15) & 0.05 (0.08) & 0.20 (0.21) & 0.09 (0.09) \\
\hline
LLD-0.05-30 & 0.25 (0.19) & 0.04 (0.06) & 0.05 (0.11) & 0.03 (0.04) \\
\hline
LLD-0.05-50 & 0.16 (0.18) & 0.02 (0.04) & {\bf{0.01 (0.04)}} & {\bf{0.02 (0.01)}} \\
\hline
LLD-0.1-30  & 0.01 (0.04) & {\bf{0.01 (0.02)}} & 0.06 (0.13) & 0.03 (0.05) \\
\hline
LLD-0.1-50  & 0.01 (0.03) & {\bf{0.01 (0.00)}} & {\bf{0.01 (0.04)}} & {\bf{0.02 (0.01)}} \\
\hline
LSD-$0.01$-30  & 0.02 (0.06) & {\bf{0.01 (0.00)}} & 0.31 (0.15) & 0.55 (0.19) \\
\hline
LSD-$0.01$-50  & 0.01 (0.04) & {\bf{0.01 (0.00)}} & 0.31 (0.18) & 0.64 (0.17) \\
\hline
LSD-$0.05$-30  & {\bf{0.00 (0.00)}} & {\bf{0.01 (0.00)}} & 0.23 (0.18) & 0.38 (0.18) \\
\hline
LSD-$0.05$-50  & {\bf{0.00 (0.00)}} & {\bf{0.01 (0.00)}} & 0.23 (0.20) & 0.48 (0.18) \\
\hline
Hclust \tnotex{notetrueclusters} & 0.06 (0.05) & 0.10 (0.05) &  0.05 (0.03) & 0.07 (0.03) \\
\hline
\multicolumn{5}{|c|}{\textbf{Clustering errors (distance in probability)}} \\
\hline
& Bimodal & Quadrimodal & Mult.\ Bimodal & Mult.\ Quadrimodal \\
\hline
KDE & 0.03 (0.04) \newline 0.05 (0.09) & 0.08 (0.15) \newline 0.11 (0.20) & 0.05 (0.06) \newline 0.08 (0.12) & 0.25 (0.29) \newline 0.31 (0.36) \\
\hline
LLD-0.05-30 & 0.06 (0.06) \newline 0.12 (0.12) & 0.08 (0.16) \newline 0.11 (0.21) & 0.024 (0.06) \newline 0.04 (0.10) & 0.08 (0.19) \newline 0.09 (0.22) \\
\hline
LLD-0.05-50 & 0.05 (0.05) \newline 0.09 (0.11) & 0.05 (0.14) \newline 0.06 (0.17) & 0.01 (0.01) \newline {\bf{0.01 (0.01)}} & 0.03 (0.01) \newline {\bf{0.03 (0.01)}} \\
\hline
LLD-0.1-30  & 0.01 (0.02) \newline 0.01 (0.03) & 0.02 (0.06) \newline 0.02 (0.07) & 0.02 (0.06) \newline 0.04 (0.10) & 0.08 (0.19) \newline 0.09 (0.22) \\
\hline
LLD-0.1-50  & 0.01 (0.01) \newline 0.01 (0.02) & 0.01 (0.00) \newline {\bf{0.01 (0.01)}} & 0.01 (0.01) \newline {\bf{0.01 (0.01)}} & 0.03 (0.01) \newline {\bf{0.03 (0.01)}} \\
\hline
LSD-$0.01$-30  & 0.01 (0.02) \newline 0.01 (0.04) & 0.01 (0.00) \newline {\bf{0.01 (0.00)}} & 0.17 (0.07) \newline 0.26 (0.14) & 0.32 (0.07) \newline 0.57 (0.15) \\
\hline
LSD-$0.01$-50  & 0.01 (0.02) \newline 0.01 (0.03) & 0.01 (0.00) \newline {\bf{0.01 (0.00)}} & 0.18 (0.07) \newline 0.28 (0.17) & 0.33 (0.05) \newline 0.64 (0.13) \\
\hline
LSD-$0.05$-30  & 0.00 (0.00) \newline {\bf{0.00 (0.00)}} & 0.01 (0.00) \newline {\bf{0.01 (0.00)}} & 0.15 (0.11) \newline 0.21 (0.17) & 0.29 (0.11) \newline 0.45 (0.17) \\
\hline
LSD-$0.05$-50  & 0.00 (0.00) \newline {\bf{0.00 (0.00)}} & 0.01 (0.01) \newline {\bf{0.01 (0.01)}} & 0.14 (0.10) \newline 0.22 (0.20) & 0.28 (0.07) \newline 0.52 (0.16) \\
\hline
Hclust \tnotex{notetrueclusters} & 0.06 (0.05) & 0.14 (0.07) & 0.05 (0.03) & 0.10 (0.04) \\
\hline
\caption{\label{sm:table_normal_symmetric_1000} Mean of the clustering based on Hausdorff distance and distance in probability over $100$ replications with $n=1000$ samples for the densities Bimodal, Quadrimodal, Mult.\ Bimodal and Mult.\ Quadrimodal. In parentheses the standard deviation.}
\end{tabularx}
\end{ThreePartTable}

\begin{ThreePartTable}
\begin{tabularx}{10cm}{ | p{2.26cm} | p{2.8cm} | p{2.8cm} | p{2.8cm} | p{2.8cm} | }
\hline
\multicolumn{5}{|c|}{\textbf{Number of times the true clusters are detected correctly}} \\
\hline
& (H) Bimodal IV & (K) Trimodal III & (L) Quadrimodal & \#10 Fountain \\
\hline
KDE & (0) 54 (46) & (5) 60 (35) & (28) 39 (33) & (0) 47 (53) \\
\hline
LLD-0.05-30 & (0) 27 (73) & (0) 24 (76) & (7) 32 (61) & (0) 99 (1) \\
\hline
LLD-0.05-50 & (0) 38 (62) & (5) 39 (56) & {\bf{(29) 45 (26)}} & {\bf{(0) 100 (0)}} \\
\hline
LLD-0.1-30 & (0) 83 (17) & {\bf{(14) 79 (7)}} & (69) 29 (2) & {\bf{(0) 100 (0)}} \\
\hline
LLD-0.1-50 & (0) 93 (7) & (20) 78 (2) & (82) 16 (2) & {\bf{(0) 100 (0)}} \\
\hline
LSD-$0.01$-30  & (0) 85 (15) & (13) 75 (12) & (65) 33 (2) & {\bf{(0) 100 (0)}} \\
\hline
LSD-$0.01$-50  & (0) 89 (11) & (21) 74 (5) & (80) 18 (2) & {\bf{(0) 100 (0)}} \\
\hline
LSD-$0.05$-30  & {\bf{(0) 100 (0)}} & (51) 49 (0) & (100) 0 (0) & (89) 11 (0) \\
\hline
LSD-$0.05$-50  & {\bf{(0) 100 (0)}} & (61) 39 (0) & (100) 0 (0) & (94) 6 (0) \\
\hline
& Circular 2 & Circular 2 Cauchy & Circular 3 & Circular 4 Cauchy \\
\hline
KDE & (0) 0 (100) & (0) 0 (100) & (0) 0 (100) & 0) 68 (32) \\
\hline
LLD-0.05-30 & (0) 0 (100) & (0) 0 (100) & (0) 0 (100) & (0) 97 (3) \\
\hline
LLD-0.05-50 & (0) 0 (100) & (0) 0 (100) & (0) 0 (100) & (0) 99 (1) \\
\hline
LLD-0.1-30  & (0) 0 (100) & (0) 2 (98) & (0) 0 (100) & {\bf{(0) 100 (0)}} \\
\hline
LLD-0.1-50  & (0) 0 (100) & (0) 5 (95) & (0) 1 (99) & {\bf{(0) 100 (0)}} \\
\hline
LLD-0.15-30  & (0) 7 (93) & (0) 27 (73) & (0) 0 (100) & {\bf{(0) 100 (0)}} \\
\hline
LLD-0.15-50  & (0) 15 (85) & (0) 37 (63) & (0) 1 (99) & {\bf{(0) 100 (0)}} \\
\hline
LLD-0.2-30  & (0) 24 (76) & {\bf{(26) 49 (25)}} & (0) 2 (98) & (4) 96 (0) \\
\hline
LLD-0.2-50  & (0) 42 (58) & (46) 41 (13) & (7) 47 (46) & (7) 93 (0) \\
\hline
LSD-$0.01$-30  & (0) 0 (100) & (0) 6 (94) & (0) 0 (100) & {\bf{(0) 100 (0)}} \\
\hline
LSD-$0.01$-50  & (0) 1 (99) & (0) 10 (90) & (0) 0 (100) & {\bf{(0) 100 (0)}} \\
\hline
LSD-$0.05$-30  & (0) 44 (56) & (30) 47 (23) & (0) 5 (95) & (5) 95 (0) \\
\hline
LSD-$0.05$-50  & {\bf{(0) 69 (31)}} & (44) 43 (13) & {\bf{(4) 51 (45)}} & (6) 94 (0) \\
\hline
LSD-$0.1$-30  & (0) 5 (95) & (91) 8 (1) & (2) 13 (85) & (64) 36 (0) \\
\hline
LSD-$0.1$-50  & (1) 54 (45) & (97) 3 (0) & (45) 36 (19) & (71) 29 (0) \\
\hline
& Bimodal & Quadrimodal & Mult.\ Bimodal & Mult.\ Quadrimodal \\
\hline
KDE & (0) 76 (24) & (0) 77 (23) & (0) 56 (44) & (0) 63 (37) \\
\hline
LLD-0.05-30 & (0) 36 (64) & (0) 80 (20) & (0) 88 (12) & (0) 92 (8) \\
\hline
LLD-0.05-50 & (0) 55 (45) & (0) 91 (9) & {\bf{(0) 99 (1)}} & {\bf{(0) 100 (0)}} \\
\hline
LLD-0.1-30 & (0) 98 (2) & (0) 99 (1) & (0) 85 (15) & (0) 93 (7) \\
\hline
LLD-0.1-50 & (0) 99 (1) & {\bf{(0) 100 (0)}} & {\bf{(0) 99 (1)}} & {\bf{(0) 100 (0)}} \\
\hline
LSD-$0.01$-30  & (0) 96 (4) & {\bf{(0) 100 (0)}} & (13) 48 (39) & (93) 6 (1) \\
\hline
LSD-$0.01$-50  & (0) 98 (2) & {\bf{(0) 100 (0)}} & (32) 49 (19) & (99) 1 (0) \\
\hline
LSD-$0.05$-30  & {\bf{(0) 100 (0)}} & {\bf{(0) 100 (0)}} & (12) 63 (25) & (77) 18 (5) \\
\hline
LSD-$0.05$-50  & {\bf{(0) 100 (0)}} & {\bf{(0) 100 (0)}} & (29) 66 (5) & (97) 3 (0) \\
\hline
\caption{\label{sm:table_number_times_1000} Number of times over $100$ replications with $n=1000$ samples that the procedure identifies the true number of clusters for the densities (H) Bimodal IV, (K) Trimodal III, (L) Quadrimodal, \#10 Fountain, Circular 2, Circular 2 Cauchy, Circular 3, Circular 4 Cauchy, Bimodal, Quadrimodal, Mult.\ Bimodal and Mult.\ Quadrimodal. In parentheses the number of times the procedure identifies a lower number of clusters (on the left) and a higher number of clusters (on the right).}
\end{tabularx}
\end{ThreePartTable}

\begin{ThreePartTable}
\begin{tabularx}{10cm}{ | p{2.26cm} | p{2.8cm} | p{2.8cm} | p{2.8cm} | p{2.8cm} | }
\hline
\multicolumn{5}{|c|}{\textbf{Clustering errors (Hausdorff distance)}} \\
\hline
& (H) Bimodal IV & (K) Trimodal III & (L) Quadrimodal & \#10 Fountain \\
\hline
KDE & 0.23 (0.20) & {\bf{0.09 (0.10)}} & {\bf{0.16 (0.12)}} & 0.14 (0.05) \\
\hline
LLD-0.05-30 & 0.30 (0.14) & 0.18 (0.11) & 0.20 (0.09) & 0.07 (0.02) \\
\hline
LLD-0.05-50 & 0.17 (0.15) & 0.20 (0.16) & 0.29 (0.14) & 0.07 (0.01) \\
\hline
LLD-0.1-30 & 0.06 (0.11) & 0.11 (0.16) & 0.29 (0.16) & {\bf{0.06 (0.01)}} \\
\hline
LLD-0.1-50 & 0.03 (0.08) & 0.20 (0.20) & 0.38 (0.14) & 0.07 (0.02) \\
\hline
LSD-$0.01$-30  & 0.07 (0.12) & 0.14 (0.17) & 0.31 (0.16) & {\bf{0.06 (0.01)}} \\
\hline
LSD-$0.01$-50  & 0.04 (0.10) & 0.23 (0.20) & 0.40 (0.13) & 0.07 (0.03) \\
\hline
LSD-$0.05$-30  & {\bf{0.01 (0.00)}} & 0.32 (0.20) & 0.47 (0.02) & 0.47 (0.18) \\
\hline
LSD-$0.05$-50  & {\bf{0.01 (0.00)}} & 0.40 (0.15) & 0.48 (0.02) & 0.54 (0.20) \\
\hline
Hclust \tnotex{notetrueclusters} & 0.03 (0.07) & 0.15 (0.09) & 0.23 (0.09) & 0.27 (0.07) \\
\hline
\multicolumn{5}{|c|}{\textbf{Clustering errors (distance in probability)}} \\
\hline
& (H) Bimodal IV & (K) Trimodal III & (L) Quadrimodal & \#10 Fountain \\
\hline
KDE & 0.30 (0.27) \newline 0.46 (0.40) & 0.05 (0.04) \newline {\bf{0.07 (0.07)}} & 0.15 (0.14) \newline 0.20 (0.20) & 0.42 (0.26) \newline 0.53 (0.32) \\
\hline
LLD-0.05-30 & 0.11 (0.06) \newline 0.21 (0.13) & 0.10 (0.07) \newline 0.17 (0.12) & 0.28 (0.20) \newline 0.36 (0.27) & 0.07 (0.08) \newline 0.08 (0.09) \\
\hline
LLD-0.05-50 & 0.06 (0.06) \newline 0.12 (0.12) & 0.10 (0.09) \newline 0.14 (0.12) & 0.15 (0.13) \newline 0.21 (0.13) & 0.07 (0.01) \newline 0.07 (0.01) \\
\hline
LLD-0.1-30 & 0.02 (0.04) \newline 0.04 (0.08) & 0.05 (0.05) \newline {\bf{0.07 (0.07)}} & 0.10 (0.05) \newline {\bf{0.16 (0.06)}} & 0.06 (0.01) \newline {\bf{0.06 (0.01)}} \\
\hline
LLD-0.1-50 & 0.01 (0.03) \newline 0.02 (0.06) & 0.07 (0.06) \newline 0.09 (0.07) & 0.10 (0.03) \newline 0.18 (0.04) & 0.07 (0.01) \newline 0.07 (0.02) \\
\hline
LSD-$0.01$-30  & 0.10 (0.19) \newline 0.15 (0.30) & 0.06 (0.05) \newline 0.08 (0.08) & 0.10 (0.03) \newline {\bf{0.16 (0.05)}} & 0.06 (0.01) \newline {\bf{0.06 (0.01)}} \\
\hline
LSD-$0.01$-50  & 0.06 (0.16) \newline 0.09 (0.24) & 0.07 (0.05) \newline 0.10 (0.07) & 0.10 (0.02) \newline 0.18 (0.04) & 0.07 (0.02) \newline 0.07 (0.02) \\
\hline
LSD-$0.05$-30  & 0.00 (0.00) \newline {\bf{0.01 (0.00)}} & 0.07 (0.02) \newline 0.11 (0.05) & 0.11 (0.01) \newline 0.20 (0.02) & 0.24 (0.06) \newline 0.44 (0.15) \\
\hline
LSD-$0.05$-50  & 0.01 (0.00) \newline {\bf{0.01 (0.00)}} & 0.08 (0.02) \newline 0.13 (0.04) & 0.11 (0.01) \newline 0.20 (0.02) & 0.26 (0.06) \newline 0.49 (0.14) \\
\hline
Hclust \tnotex{notetrueclusters} & 0.03 (0.07) & 0.15 (0.09) & 0.28 (0.10) & 0.34 (0.10) \\
\hline
\caption{\label{sm:table_Chacon_500} Mean of the clustering errors based on Hausdorff distance and distance in probability over $100$ replications with $n=500$ samples for the densities (H) Bimodal IV, (K) Trimodal III, (L) Quadrimodal and \#10 Fountain. In parentheses the standard deviation.}
\end{tabularx}
\end{ThreePartTable}

\begin{ThreePartTable}
\begin{tabularx}{10cm}{ | p{2.26cm} | p{2.8cm} | p{2.8cm} | p{2.8cm} | p{2.8cm} | }
\hline
\multicolumn{5}{|c|}{\textbf{Clustering errors (Hausdorff distance)}} \\
\hline
& Circular 2 & Circular 2 Cauchy & Circular 3 & Circular 4 Cauchy \\
\hline
KDE & 0.55 (0.06) & 0.52 (0.08) & 0.32 (0.02) & 0.10 (0.10) \\
\hline
LLD-0.05-30 & 0.58 (0.05) & 0.53 (0.07) & 0.34 (0.02) & 0.05 (0.06) \\
\hline
LLD-0.05-50 & 0.53 (0.06) & 0.49 (0.06) & 0.33 (0.04) & {\bf{0.03 (0.03)}} \\
\hline
LLD-0.1-30 & 0.51 (0.06) & 0.46 (0.09) & 0.32 (0.04) & {\bf{0.03 (0.03)}} \\
\hline
LLD-0.1-50 & 0.49 (0.06) & 0.41 (0.15) & {\bf{0.30 (0.05)}} & {\bf{0.03 (0.03)}} \\
\hline
LLD-0.15-30 & 0.45 (0.10) & 0.37 (0.17) & 0.32 (0.04) & 0.04 (0.04) \\
\hline
LLD-0.15-50 & 0.40 (0.15) & {\bf{0.30 (0.20)}} & 0.34 (0.08) & 0.05 (0.06) \\
\hline
LLD-0.2-30 & 0.39 (0.16) & 0.51 (0.28) & 0.38 (0.08) & 0.12 (0.12) \\
\hline
LLD-0.2-50 & 0.31 (0.19) & 0.59 (0.28) & 0.50 (0.14) & 0.16 (0.13) \\
\hline
LSD-$0.01$-30  & 0.50 (0.06) & 0.45 (0.08) & 0.33 (0.03) & {\bf{0.03 (0.03)}} \\
\hline
LSD-$0.01$-50  & 0.48 (0.08) & 0.43 (0.12) & 0.31 (0.06) & {\bf{0.03 (0.03)}} \\
\hline
LSD-$0.05$-30  & 0.28 (0.21) & 0.56 (0.29) & 0.36 (0.08) & 0.15 (0.14) \\
\hline
LSD-$0.05$-50  & {\bf{0.20 (0.20)}} & 0.60 (0.27) & 0.49 (0.14) & 0.19 (0.14) \\
\hline
LSD-$0.1$-30  & 0.31 (0.20) & 0.77 (0.13) & 0.46 (0.10) & 0.38 (0.19) \\
\hline
LSD-$0.1$-50  & 0.25 (0.26) & 0.80 (0.07) & 0.65 (0.13) & 0.44 (0.20) \\
\hline
Hclust \tnotex{notetrueclusters} & 0.33 (0.12) & 0.37 (0.08) & 0.37 (0.06) & 0.22 (0.07) \\
\hline
\multicolumn{5}{|c|}{\textbf{Clustering errors (distance in probability)}} \\
\hline
& Circular 2 & Circular 2 Cauchy & Circular 3 & Circular 4 Cauchy \\
\hline
KDE & 0.31 (0.04) \newline 0.62 (0.08) & 0.32 (0.06) \newline 0.61 (0.10) & 0.29 (0.03) \newline 0.56 (0.07) & 0.18 (0.23) \newline 0.23 (0.28) \\
\hline
LLD-0.05-30 & 0.34 (0.04) \newline 0.65 (0.07) & 0.33 (0.07) \newline 0.62 (0.09) & 0.34 (0.05) \newline 0.59 (0.09) & 0.09 (0.17) \newline 0.11 (0.21) \\
\hline
LLD-0.05-50 & 0.33 (0.07) \newline 0.59 (0.10) & 0.34 (0.10) \newline 0.58 (0.10) & 0.37 (0.09) \newline 0.56 (0.13) & 0.05 (0.07) \newline 0.05 (0.09) \\
\hline
LLD-0.1-30 & 0.32 (0.06) \newline 0.59 (0.10) & 0.34 (0.14) \newline 0.54 (0.15) & 0.35 (0.08) \newline 0.54 (0.13) & 0.03 (0.02) \newline {\bf{0.03 (0.02)}} \\
\hline
LLD-0.1-50 & 0.34 (0.10) \newline 0.57 (0.12) & 0.32 (0.17) \newline 0.48 (0.20) & 0.39 (0.12) \newline 0.52 (0.18) & 0.04 (0.02) \newline 0.04 (0.03) \\
\hline
LLD-0.15-30  & 0.32 (0.12) \newline 0.54 (0.16) & 0.31 (0.19) \newline 0.45 (0.24) & 0.38 (0.09) \newline 0.53 (0.14) & 0.05 (0.03) \newline 0.05 (0.04) \\
\hline
LLD-0.15-50  & 0.30 (0.16) \newline 0.47 (0.22) & 0.25 (0.19) \newline 0.34 (0.25) & 0.40 (0.09) \newline 0.46 (0.12) & 0.05 (0.03) \newline 0.06 (0.04) \\
\hline
LLD-0.2-30  & 0.29 (0.17) \newline 0.46 (0.23) & 0.23 (0.19) \newline 0.30 (0.20) & 0.44 (0.09) \newline 0.51 (0.10) & 0.09 (0.05) \newline 0.12 (0.10) \\
\hline
LLD-0.2-50  & 0.26 (0.20) \newline 0.37 (0.27) & 0.19 (0.17) \newline 0.25 (0.16) & 0.37 (0.07) \newline 0.50 (0.08) & 0.10 (0.06) \newline 0.16 (0.12) \\
\hline
LSD-$0.01$-30  & 0.32 (0.09) \newline 0.58 (0.11) & 0.34 (0.13) \newline 0.54 (0.15) & 0.34 (0.09) \newline 0.54 (0.13) & 0.03 (0.02) \newline 0.04 (0.02) \\
\hline
LSD-$0.01$-50  & 0.34 (0.11) \newline 0.56 (0.13) & 0.33 (0.16) \newline 0.51 (0.18 & 0.37 (0.11) \newline 0.50 (0.17) & 0.04 (0.02) \newline 0.04 (0.03) \\
\hline
LSD-$0.05$-30  & 0.22 (0.19) \newline 0.33 (0.26) & 0.20 (0.18) \newline 0.25 (0.16) & 0.43 (0.10) \newline 0.51 (0.13) & 0.10 (0.06) \newline 0.16 (0.12) \\
\hline
LSD-$0.05$-50  & 0.16 (0.18) \newline {\bf{0.22 (0.25)}} & 0.19 (0.17) \newline 0.25 (0.15) & 0.36 (0.08) \newline 0.48 (0.08) & 0.12 (0.06) \newline 0.19 (0.13) \\
\hline
LSD-$0.1$-30  & 0.27 (0.17) \newline 0.41 (0.28) & 0.13 (0.11) \newline 0.21 (0.08) & 0.48 (0.10) \newline 0.54 (0.10) & 0.23 (0.08) \newline 0.43 (0.17) \\
\hline
LSD-$0.1$-50  & 0.16 (0.15) \newline {\bf{0.22 (0.23)}} & 0.10 (0.06) \newline {\bf{0.19 (0.05)}} & 0.34 (0.07) \newline 0.56 (0.06) & 0.25 (0.07) \newline 0.49 (0.16) \\
\hline
Hclust \tnotex{notetrueclusters} & 0.33 (0.12) & 0.37 (0.08) & {\bf{0.44 (0.06)}} & 0.31 (0.11) \\
\hline
\caption{\label{sm:table_circular_500} Mean of the clustering errors based on Hausdorff distance and distance in probability over $100$ replications with $n=500$ samples for the densities Circular 2, Circular 2 Cauchy, Circular 3 and Circular 4 Cauchy. In parentheses the standard deviation.}
\end{tabularx}
\end{ThreePartTable}

\begin{ThreePartTable}
\begin{tabularx}{10cm}{ | p{2.26cm} | p{2.8cm} | p{2.8cm} | p{2.8cm} | p{2.8cm} | }
\hline
\multicolumn{5}{|c|}{\textbf{Clustering errors (Hausdorff distance)}} \\
\hline
& Bimodal & Quadrimodal & Mult.\ Bimodal & Mult.\ Quadrimodal \\
\hline
KDE & 0.13 (0.17) & 0.04 (0.06) & 0.45 (0.01) & 0.20 (0.06) \\
\hline
LLD-0.05-30 & 0.26 (0.17) & 0.05 (0.06) & 0.03 (0.06) & {\bf{0.04 (0.03)}} \\
\hline
LLD-0.05-50 & 0.12 (0.16) & 0.02 (0.04) & {\bf{0.02 (0.03)}} & 0.05 (0.06) \\
\hline
LLD-0.1-30 & 0.05 (0.10) & {\bf{0.01 (0.01)}} & 0.03 (0.07) & 0.05 (0.06) \\
\hline
LLD-0.1-50 & 0.02 (0.06) & {\bf{0.01 (0.01)}} & {\bf{0.02 (0.03)}} & 0.06 (0.08) \\
\hline
LSD-$0.01$-30  & 0.04 (0.09) & {\bf{0.01 (0.01)}} & 0.40 (0.17) & 0.67 (0.15) \\
\hline
LSD-$0.01$-50  & 0.02 (0.07) & {\bf{0.01 (0.01)}} & 0.46 (0.14) & 0.74 (0.10) \\
\hline
LSD-$0.05$-30  & {\bf{0.01 (0.00)}} & {\bf{0.01 (0.01)}} & 0.35 (0.21) & 0.52 (0.19) \\
\hline
LSD-$0.05$-50  & {\bf{0.01 (0.01)}} & 0.02 (0.03) & 0.40 (0.19) & 0.62 (0.19) \\
\hline
Hclust \tnotex{notetrueclusters} & 0.08 (0.07) & 0.08 (0.05) & 0.05 (0.02) & 0.07 (0.04) \\
\hline
\multicolumn{5}{|c|}{\textbf{Clustering errors (distance in probability)}} \\
\hline
& Bimodal & Quadrimodal & Mult.\ Bimodal & Mult.\ Quadrimodal \\
\hline
KDE & 0.04 (0.05) \newline 0.07 (0.10) & 0.06 (0.11) \newline 0.08 (0.15) & 0.08 (0.07) \newline 0.13 (0.12) & 0.50 (0.22) \newline 0.67 (0.27) \\
\hline
LLD-0.05-30 & 0.08 (0.06) \newline 0.14 (0.11) & 0.09 (0.17) \newline 0.12 (0.21) & 0.02 (0.02) \newline {\bf{0.02 (0.04)}} & 0.04 (0.02) \newline {\bf{0.04 (0.02)}} \\
\hline
LLD-0.05-50 & 0.04 (0.05) \newline 0.08 (0.11) & 0.03 (0.05) \newline 0.04 (0.08) & 0.02 (0.01) \newline {\bf{0.02 (0.02)}} & 0.06 (0.03) \newline 0.06 (0.06) \\
\hline
LLD-0.1-30 & 0.02 (0.03) \newline 0.03 (0.07) & 0.02 (0.01) \newline {\bf{0.02 (0.01)}} & 0.02 (0.03) \newline {\bf{0.02 (0.05)}} & 0.05 (0.03) \newline 0.06 (0.06) \\
\hline
LLD-0.1-50 & 0.01 (0.02) \newline 0.02 (0.04) & 0.02 (0.01) \newline {\bf{0.02 (0.01)}} & 0.01 (0.01) \newline {\bf{0.02 (0.02)}} & 0.06 (0.04) \newline 0.07 (0.07) \\
\hline
LSD-$0.01$-30  & 0.02 (0.03) \newline 0.03 (0.07) & 0.02 (0.01) \newline {\bf{0.02 (0.01)}} & 0.22 (0.06) \newline 0.37 (0.15) & 0.34 (0.05) \newline 0.66 (0.10) \\
\hline
LSD-$0.01$-50  & 0.01 (0.02) \newline 0.02 (0.05) & 0.02 (0.01) \newline {\bf{0.02 (0.01)}} & 0.23 (0.05) \newline 0.43 (0.12) & 0.36 (0.03) \newline 0.71 (0.06) \\
\hline
LSD-$0.05$-30  & 0.01 (0.00) \newline {\bf{0.01 (0.00)}} & 0.02 (0.01) \newline {\bf{0.02 (0.01)}} & 0.19 (0.08) \newline 0.32 (0.19) & 0.30 (0.07) \newline 0.54 (0.15) \\
\hline
LSD-$0.05$-50  & 0.01 (0.01) \newline {\bf{0.01 (0.01)}} & 0.02 (0.01) \newline {\bf{0.02 (0.02)}} & 0.20 (0.07) \newline 0.38 (0.17) & 0.32 (0.05) \newline 0.63 (0.12) \\
\hline
Hclust \tnotex{notetrueclusters} & 0.08 (0.07) & 0.11 (0.07) & 0.05 (0.02) & 0.10 (0.05) \\
\hline
\caption{\label{sm:table_normal_symmetric_500} Mean of the clustering errors based on Hausdorff distance and distance in probability over $100$ replications with $n=500$ samples for the densities Bimodal, Quadrimodal, Mult.\ Bimodal and Mult.\ Quadrimodal. In parentheses the standard deviation.}
\end{tabularx}
\end{ThreePartTable}

\begin{ThreePartTable}
\begin{tabularx}{10cm}{ | p{2.26cm} | p{2.8cm} | p{2.8cm} | p{2.8cm} | p{2.8cm} | }
\hline
\multicolumn{5}{|c|}{\textbf{Number of times the true clusters are detected correctly}} \\
\hline
& (H) Bimodal IV & (K) Trimodal III & (L) Quadrimodal & \#10 Fountain \\
\hline
KDE & (0) 42 (58) & (3) 66 (31) & {\bf{(47) 35 (18)}} & (0) 29 (71) \\
\hline
LLD-0.05-30 & (0) 15 (85) & (0) 34 (66) & (24) 32 (44) & (0) 99 (1) \\
\hline
LLD-0.05-50 & (0) 45 (55) & (17) 50 (33) & (79) 16 (5) & {\bf{(0) 100 (0)}} \\
\hline
LLD-0.1-30 & (0) 82 (18) & {\bf{(14) 78 (8)}} & (88) 11 (1) & {\bf{(0) 100 (0)}} \\
\hline
LLD-0.1-50 & (0) 92 (8) & (37) 63 (0) & (99) 1 (0) & (1) 99 (0) \\
\hline
LSD-$0.01$-30  & (0) 80 (20) & (18) 69 (13) & (89) 11 (0) & {\bf{(0) 100 (0)}} \\
\hline
LSD-$0.01$-50  & (0) 88 (12) & (41) 55 (4) & (100) 0 (0) & (2) 98 (0) \\
\hline
LSD-$0.05$-30  & {\bf{(0) 100 (0)}} & (67) 33 (0) & (100) 0 (0) & (96) 4 (0) \\
\hline
LSD-$0.05$-50  & {\bf{(0) 100 (0)}} & (84) 16 (0) & (100) 0 (0) & (99) 1 (0) \\
\hline
& Circular 2 & Circular 2 Cauchy & Circular 3 & Circular 4 Cauchy \\
\hline
KDE & (0) 0 (100) & (0) 0 (100) & (0) 0 (100) & (0) 64 (36) \\
\hline
LLD-0.05-30 & (0) 0 (100) & (0) 0 (100) & (0) 0 (100) & (1) 86 (13) \\
\hline
LLD-0.05-50 & (0) 0 (100) & (0) 0 (100) & (0) 1 (99) & (1) 97 (2) \\
\hline
LLD-0.1-30 & (0) 0 (100) & (0) 3 (97) & (0) 2 (98) & {\bf{(1) 99 (0)}} \\
\hline
LLD-0.1-50 & (0) 0 (100) & (0) 13 (87) & (0) 11 (89) & {\bf{(1) 99 (0)}} \\
\hline
LLD-0.15-30 & (0) 3 (97) & (0) 21 (79) & (0) 7 (93) & (3) 97 (0) \\
\hline
LLD-0.15-50 & (0) 15 (85) & {\bf{(0) 43 (57)}} & (8) 47 (45) & (4) 96 (0) \\
\hline
LLD-0.2-30 & (0) 15 (85) & (35) 38 (27) & (10) 38 (52) & (32) 68 (0) \\
\hline
LLD-0.2-50 & (0) 35 (65) & (53) 36 (11) & (71) 24 (5) & (44) 56 (0) \\
\hline
LSD-$0.01$-30  & (0) 0 (100) & (0) 2 (98) & (0) 1 (99) & {\bf{(1) 99 (0)}} \\
\hline
LSD-$0.01$-50  & (0) 1 (99) & (0) 8 (92) & (0) 12 (88) & {\bf{(1) 99 (0)}} \\
\hline
LSD-$0.05$-30  & (0) 41 (59) & {\bf{(48) 43 (9)}} & (9) 35 (56) & (43) 57 (0) \\
\hline
LSD-$0.05$-50  & (0) 63 (37) & (55) 38 (7) & (71) 29 (0) & (54) 46 (0) \\
\hline
LSD-$0.1$-30  & (3) 34 (63) & (89) 11 (0) & {\bf{(22) 51 (27)}} & (92) 8 (0) \\
\hline
LSD-$0.1$-50  & {\bf{(12) 65 (23)}} & (97) 3 (0) & (94) 6 (0) & (93) 7 (0) \\
\hline
& Bimodal & Quadrimodal & Mult.\ Bimodal & Mult.\ Quadrimodal \\
\hline
KDE & (0) 67 (33) & (0) 86 (14) & (0) 5 (95) & (0) 13 (87) \\
\hline
LLD-0.05-30 & (0) 29 (71) & (0) 78 (22) & (0) 95 (5) & {\bf{(0) 97 (3)}} \\
\hline
LLD-0.05-50 & (0) 65 (35) & (0) 94 (6) & {\bf{(0) 99 (1)}} & (6) 94 (0) \\
\hline
LLD-0.1-30 & (0) 87 (13) & {\bf{(0) 100 (0)}} & (0) 94 (6) & (6) 94 (0) \\
\hline
LLD-0.1-50 & (0) 96 (4) & {\bf{(0) 100 (0)}} & {\bf{(0) 99 (1)}} & (10) 90 (0) \\
\hline
LSD-$0.01$-30  & (0) 90 (10) & {\bf{(0) 100 (0)}} & (61) 36 (3) & (100) 0 (0) \\
\hline
LSD-$0.01$-50  & (0) 95 (5) & {\bf{(0) 100 (0)}} & (85) 15 (0) & (100) 0 (0) \\
\hline
LSD-$0.05$-30  & {\bf{(0) 100 (0)}} & {\bf{(0) 100 (0)}} & (56) 42 (2) & (96) 4 (0) \\
\hline
LSD-$0.05$-50  & {\bf{(0) 100 (0)}} & (1) 99 (0) & (73) 27 (0) & (100) 0 (0) \\
\hline
\caption{\label{sm:table_number_times_500} Number of times over $100$ replications with $n=500$ samples that the procedure identifies the true number of clusters for the densities (H) Bimodal IV, (K) Trimodal III, (L) Quadrimodal, \#10 Fountain, Circular 2, Circular 2 Cauchy, Circular 3, Circular 4 Cauchy, Bimodal, Quadrimodal, Mult.\ Bimodal and Mult.\ Quadrimodal. In parentheses the number of times the procedure identifies a lower number of clusters (on the left) and a higher number of clusters (on the right).}
\end{tabularx}
\end{ThreePartTable}

\end{document}